\pdfoutput=1
\documentclass[reqno]{amsart}
\usepackage{hyperref, bbm, amssymb, mathtools, url, enumerate}
\usepackage{amsmath}
\usepackage{tikz}
\usepackage{tikz-cd}
\usepackage[scr]{rsfso}
\usepackage{bm}
\usepackage{amssymb}
\usepackage{bbm}
\usetikzlibrary{matrix,arrows,decorations.pathmorphing, decorations.markings, cd}
\usepackage[capitalise]{cleveref}
\usepackage{rotate}
\usepackage[alphabetic]{amsrefs}
\usepackage{microtype}
\usepackage{aliascnt}

\DeclareMathAlphabet{\mymathbb}{U}{BOONDOX-ds}{m}{n}
\def\twocell[#1]{\arrow[#1, dash, phantom, "\Rightarrow"{scale=1.125, yshift=-.4pt, description, allow upside down, sloped, inner sep=0pt}]}

\tikzset{curve/.style={settings={#1},to path={(\tikztostart)
			.. controls ($(\tikztostart)!\pv{pos}!(\tikztotarget)!\pv{height}!270:(\tikztotarget)$)
			and ($(\tikztostart)!1-\pv{pos}!(\tikztotarget)!\pv{height}!270:(\tikztotarget)$)
			.. (\tikztotarget)\tikztonodes}},
	settings/.code={\tikzset{quiver/.cd,#1}
		\def\pv##1{\pgfkeysvalueof{/tikz/quiver/##1}}},
	quiver/.cd,pos/.initial=0.35,height/.initial=0}

\pgfdeclarelayer{bg}
\pgfsetlayers{bg,main}
\usetikzlibrary{calc}

\newcommand{\mynewtheorem}[2]{\newaliascnt{#1}{theorem}\newtheorem{#1}[#1]{#2}\aliascntresetthe{#1}}
\newtheorem{theorem}{Theorem}[section]
\mynewtheorem{conjecture}{Conjecture}
\mynewtheorem{corollary}{Corollary}
\mynewtheorem{lemma}{Lemma}
\mynewtheorem{proposition}{Proposition}

\newtheorem{introthm}{Theorem}
\newaliascnt{introcor}{introthm}
\newtheorem{introcor}[introcor]{Corollary}
\aliascntresetthe{introcor}

\theoremstyle{definition}
\newtheorem*{claim*}{Claim}
\mynewtheorem{construction}{Construction}
\mynewtheorem{variant}{Variant}
\mynewtheorem{convention}{Convention}
\mynewtheorem{definition}{Definition}

\mynewtheorem{discussion}{Discussion}
\mynewtheorem{example}{Example}
\mynewtheorem{hypothesis}{Hypothesis}
\mynewtheorem{notation}{Notation}
\mynewtheorem{observation}{Observation}

\mynewtheorem{remark}{Remark}
\newtheorem*{remark*}{Remark}
\mynewtheorem{warn}{Warning}

\newcommand{\notehelper}[3]{\textcolor{#3}{$\blacksquare$}\marginpar{\ifodd\thepage\raggedright\else\raggedleft\fi\color{#3}\tiny \textbf{#2:} #1}}

\newcommand{\Aa}{{\mathcal{A}}}
\newcommand{\Bb}{{\mathcal{B}}}
\newcommand{\Cc}{{\mathcal{C}}}
\newcommand{\Dd}{{\mathcal{D}}}
\newcommand{\Ee}{{\mathcal{E}}}
\newcommand{\Ff}{{\mathcal{F}}}
\newcommand{\Mm}{{\mathcal{M}}}
\newcommand{\Nn}{{\mathcal{N}}}
\newcommand{\Ii}{{\mathcal{I}}}

\newcommand{\Kk}{{\mathcal{K}}}
\newcommand{\Ll}{{\mathcal L}}

\newcommand{\Xx}{{\mathcal X}}
\newcommand{\Yy}{{\mathcal Y}}

\newcommand{\R} {{\mathbb{R}}}
\newcommand{\Z} {{\mathbb{Z}}}

\newcommand{\bbU}{\mathbb{U}}
\newcommand{\bbP}{\mathbb{P}}
\newcommand{\bbF}{\mathbb{F}}

\newcommand{\bbone}{{\mathbbm{1}}}

\newcommand{\cB}{\mathcal{B}}
\newcommand{\cC}{\mathcal{C}}
\newcommand{\cD}{\mathcal{D}}
\newcommand{\cE}{\mathcal{E}}
\newcommand{\cF}{\mathcal{F}}

\newcommand{\cK}{\mathcal{K}}
\newcommand{\cL}{\mathcal{L}}
\newcommand{\cM}{\mathcal{M}}
\newcommand{\cN}{\mathcal{N}}
\newcommand{\cO}{\mathcal{O}}

\newcommand{\cV}{\mathcal{V}}
\newcommand{\cX}{\mathcal{X}}
\newcommand{\cY}{\mathcal{Y}}
\newcommand{\cZ}{\mathcal{Z}}
\newcommand{\bfU}{\mathbf{U}}
\newcommand{\all}{\textup{all}}
\newcommand{\tensor}{\otimes}

\newcommand{\F}{\mathbb{F}}

\newcommand{\gl}{\textup{gl}}
\newcommand{\eq}{\textup{eq}}
\newcommand{\Bispan}{{\textup{Bispan}}}
\DeclareMathOperator{\Env}{Env}
\DeclareMathOperator{\Ext}{Ext}
\DeclareMathOperator{\Triv}{Triv}
\DeclareMathOperator{\ev}{ev}
\DeclareMathOperator{\laxlim}{lax\,lim}
\newcommand{\plaxlim}{\mathop{\laxlim^{\!\dagger\!}}}
\newcommand{\xto}{\xrightarrow}

\renewcommand{\phi}{\varphi}
\renewcommand{\epsilon}{\varepsilon}

\DeclareMathOperator{\Sp}{Sp}
\DeclareMathOperator{\Spc}{Spc}
\DeclareMathOperator{\Cat}{Cat}

\newcommand{\Rig}{\textup{Rig}}
\newcommand{\UCom}{\textup{UCom}}
\newcommand{\Glob}{\textup{Glob}}
\newcommand{\botimes}{{\rlap{$\scriptstyle\oplus$}\otimes}}

\DeclareMathOperator{\essim}{ess\,im}
\DeclareMathOperator{\Hom}{hom}
\DeclareMathOperator{\Map}{hom}
\DeclareMathOperator{\Fun}{Fun}
\DeclareMathOperator{\PSh}{PSh}

\DeclareMathOperator{\Ar}{Ar}

\DeclareMathOperator{\Stct}{St^{ct}}
\DeclareMathOperator{\CMon}{CMon}

\DeclareMathOperator{\AdTrip}{AdTrip}
\DeclareMathOperator{\Span}{Span}

\DeclareMathOperator{\Mack}{Mack}

\newcommand{\op}{{\textup{op}}}

\DeclareMathOperator{\core}{\iota}
\DeclareMathOperator{\CAlg}{\textup{CAlg}}
\newcommand{\colim}{\textup{colim}}
\newcommand{\const}{\textup{const}}
\newcommand{\id}{\textup{id}}
\newcommand{\pr}{\textup{pr}}
\newcommand{\fgt}{\textup{fgt}}
\newcommand{\diag}{\textup{diag}}
\newcommand{\cocart}{\textup{cc}}
\newcommand{\co}{\textup{cc}}

\newcommand{\cart}{\textup{ct}}
\newcommand{\ct}{\textup{ct}}
\newcommand{\BC}{\textup{BC}}
\newcommand{\Un}{\textup{Un}}
\newcommand{\Fin}{\hbox{$\mathcal F\kern-1.7pt\textit{in}$}}

\newcommand{\Orb}{{\vphantom{\textup{t}}\smash{\textup{Orb}}}}
\newcommand{\Glo}{\textup{Glo}}
\newcommand{\res}{\textup{res}}

\newcommand{\Nm}{\mathop{\vphantom{t}\smash{\textup{Nm}}}\nolimits}
\newcommand{\Ind}{\mathop{\textup{Ind}}\nolimits}
\newcommand{\Res}{\mathop{\textup{Res}}\nolimits}
\newcommand{\Inf}{\mathop{\textup{Inf}\hspace{.1em}}\nolimits}
\newcommand{\Sub}{\mathop{\textup{Sub}}\nolimits}

\newcommand{\ulhelper}[2]{\underline{\setbox0=\hbox{$#1#2$}\dp0=0.4pt \box0\relax}}
\newcommand{\ul}[1]{{\mathpalette\ulhelper{#1}}\hbox{\rule[-2pt]{0pt}{0pt}}}
\let\und=\ul

\newcommand{\Ncoprod}{{\cN\textup{-}\amalg}}
\newcommand{\Lcocart}{E\textup{-}\cocart}

\newcommand{\dcocart}{{\textup{-}\cocart}}

\newcommand{\PreOp}[2]{\mathrm{Op}^\ast_{#2}(#1)}
\newcommand{\Op}[2]{\mathrm{Op}_{#2}(#1)}
\newcommand{\PreCat}[1]{\mathrm{Cat}^\ast(#1)}
\newcommand{\ulPreCat}[1]{\und{\mathrm{Cat}}^\ast(#1)}
\newcommand{\iCat}[1]{\mathrm{Cat}(#1)}
\newcommand{\PreNmCat}[2]{\mathrm{NCat}^\ast_{#2}(#1)}
\newcommand{\NmCat}[2]{\mathrm{NCat}_{#2}(#1)}
\newcommand{\Nlax}{\cN\textup{-lax}}
\newcommand{\Nstr}{\cN\textup{-}\otimes}
\newcommand{\cR}{\mathcal{M}}
\def\Fglo{{\mathbb G}}
\def\Forb{{\mathbb O}}
\newcommand{\Ocoprod}{{\Forb\textup{-}\amalg}}

\newcommand{\lax}[1]{{#1\textup{-lax}}}
\newcommand{\neglax}[1]{{\neg #1\textup{-lax}}}
\newcommand{\negMlax}{\neglax{M}}
\newcommand{\sift}{\textup{sift}}
\newcommand{\Pmod}{\mathbb{P}^{\text{mod}}}
\newcommand{\Pder}{\mathbb{P}^{\text{der}}}
\newcommand{\Ppar}{\mathbb{P}^{\text{par}}}

\tikzset{mono/.style={>->},
         epic/.style={->>},
         norm/.style={->, shorten <=3pt,
		              postaction={decorate,
			                      decoration={markings,
                                              mark=at position 2pt with {\node[fill=white,inner sep=-1pt,circle] {$\scriptscriptstyle\otimes$};}}}}}

\newcommand{\oldrightarrownorm}{\mathrel{
\raise1.3pt\hbox{$\scriptscriptstyle\otimes$}\hskip-1.33pt{\to}}}
\newcommand{\buildarrowfromtikz}[1]{\mathrel{\begin{tikzcd}[ampersand replacement=\&, cramped, cells={nodes={inner sep=0pt}}, column sep=small] \null\arrow[r, #1]\&\null
\end{tikzcd}}}
\newcommand{\rightarrowepic}{\buildarrowfromtikz{epic}}
\newcommand{\rightarrowmono}{\buildarrowfromtikz{mono}}
\newcommand{\rightarrownorm}{\buildarrowfromtikz{norm}}
\newcommand{\buildarrowfromtikzl}[1]{\mathrel{\begin{tikzcd}[ampersand replacement=\&, cramped, cells={nodes={inner sep=0pt}}, column sep=small] \null\&\arrow[l, #1]\null
\end{tikzcd}}}
\newcommand{\leftarrowepic}{\buildarrowfromtikzl{epic}}
\newcommand{\leftarrowmono}{\buildarrowfromtikzl{mono}}

\newcommand{\sbar}{\mathbin{|}\hspace{-1pt}}
\let\smashp=\wedge

\newcommand{\shfy}{{\mathop{\textup{shfy}}}}
\newcommand{\incl}{{\mathop{\textup{incl}}}}
\newcommand{\Alg}{{\textup{Alg}}}
\def\NMon{{\textup{NMon}}}
\def\SpMack{{\textup{SpMack}}}
\def\SSet{{\textup{SSet}}}

\newcommand{\iso}{\xrightarrow{\;\smash{\raisebox{-0.25ex}{\ensuremath{\scriptstyle\sim}}}\;}}

\tikzcdset{pullback/.style = {phantom, "\lrcorner"{very near start}}}

\newcommand\noloc{%
	\nobreak
	\mspace{6mu plus 1mu}
	{:}
	\nonscript\mkern-\thinmuskip
	\mathpunct{}
	\mspace{2mu}
}

\title{Norms in equivariant homotopy theory}
\author{Tobias Lenz}
\address{T.L.: Mathematisches Institut, Rheinische Friedrich-Wilhelms-Universität Bonn, Endenicher Allee 60, 53115 Bonn, Germany}
\author{Sil Linskens}
\address{S.L.: NWF I - Mathematik, Universit\"at Regensburg, Universit\"atsstra{\ss}e 31, 93040 Regensburg, Germany}
\author{Phil Pützstück}
\address{P.P.: FB Mathematik und Informatik, Universität Münster, Einsteinstraße 62, 48149 Münster, Germany}

\begin{document}
\begin{abstract}
    We show that the $\infty$-category of \emph{normed algebras} in genuine $G$-spectra, as introduced by Bachmann--Hoyois, is modelled by strictly commutative algebras in $G$-symmetric spectra for any finite group $G$. We moreover provide an analogous description of Schwede's  \emph{ultra-commutative global ring spectra} in higher categorical terms.

    Using these new descriptions, we exhibit the $\infty$-category of ultra-commutative global ring spectra as a partially lax limit of the $\infty$-categories of genuine $G$-spectra for varying $G$, in analogy with the non-multiplicative comparison of Nardin, Pol, and the second author.

    Along the way, we establish various new results in parametrized higher algebra, which we hope to be of independent interest.
\end{abstract}

\begingroup\parskip=0pt
\maketitle
\tableofcontents
\endgroup

\section{Introduction}
Multiplicative structures on cohomology theories have been exploited fruitfully throughout the history of algebraic topology, ranging from the classical argument that the Hopf maps are stably non-trivial (using Steenrod operations on cohomology), over the simplification of Adams' solution of the Hopf invariant one problem via Adams operations on topological $K$-theory \cite{atiyah-adams}, to the resolution of the Kervaire invariant one problem by Hill, Hopkins, and Ravenel \cite{HHR2016Kervaire}.

The latter crucially relies on having a refined notion of \emph{genuine commutative algebras} in equivariant spectra, containing more information than mere $E_\infty$ algebras in the ($\infty$-)category\footnote{Throughout, we will refer to $(\infty,1)$-categories simply as \emph{categories}.} of equivariant spectra. Classically, these objects have been defined as strictly commutative algebras in a good pointset level model of $G$-spectra (for $G$ a finite group), like symmetric spectra or orthogonal spectra with $G$-action; borrowing terminology introduced by Schwede in the global case \cite{schwede2018global}, we will refer to these pointset level models as \emph{ultra-commutative equivariant ring spectra}. Many equivariant spectra of a more geometric nature, like real or complex equivariant homotopical bordism or topological $K$-theory can be explicitly written down as such ultra-commutative equivariant ring spectra, and this provides an extremely efficient way to encode a wealth of structure.

However, there are unfortunately also various $G$-spectra of a more combinatorial or categorical flavour that morally should admit a `genuine' multiplication, but where writing down an ultra-commutative structure is hard or even out-of-reach. A prominent example of this is \emph{equivariant algebraic $K$-theory} as considered for example in \cites{shimakawa, guillou-may, merling}, where it took a great deal of recent effort \cites{gmmo, yau} (involving wrestling with a significant amount of challenging categorical combinatorics) to refine this to a construction that builds ultra-commutative $G$-ring spectra from suitable categorical input. Even now, the available constructions are confined to inputs of a $1$-categorical nature, and it is not clear how one could adapt them to handle e.g.~multiplicative structures on equivariant Waldhausen $K$-theory.

On the other hand, higher category theory has allowed for a very clean understanding of multiplicative structures on \emph{non-equivariant} algebraic $K$-theory: in particular, work of Gepner--Groth--Nikolaus \cite{Gepner-Groth-Nikolaus} essentially shows that there is one and only one multiplicative refinement of the non-equivariant $K$-theory of symmetric monoidal ($\infty$-)categories. It is natural to hope for a similarly satisfying higher categorical treatment in the equivariant case. However, the work of the aforementioned authors crucially relies on the universal property of the category of spectra, of which there is no equivariant analogue; moreover, such an approach could only ever produce $E_\infty$ algebras in the category of $G$-spectra, which as mentioned above falls short of the kind of structure we are looking for.

\subsection*{Normed categories and normed algebras}
As the $E_\infty$ operad is the terminal operad, algebras over any other operad admit a `forgetful' functor from $E_\infty$ algebras. On the other hand, ultra-commutative ring spectra ought to be a more highly structured object than $E_\infty$ ring spectra and instead come with a natural forgetful functor \emph{to} $E_\infty$ ring spectra, so we should not expect to be able to describe them using ordinary higher algebra.\footnote{In fact, we will prove in Appendix~\ref{app:ordinary-HA} that for $G\not=1$ the forgetful functor from ultra-commutative $G$-ring spectra to $G$-spectra does not arise as a forgetful functor from a category of operadic algebras in $G$-spectra.} We therefore need another way to encode the extra multiplicative structure present in an ultra-commutative $G$-ring spectrum compared to an ordinary $E_\infty$ ring.

For this it is useful to first look at the corresponding structure on the zeroth homotopy groups $(\pi_0^HR)_{H\subset G}$ of such an ultra-commutative $G$-ring spectrum $R$: this structure can be completely computed by a Yoneda argument, and the extra data precisely consists of certain \emph{norm maps}
\[
    \Nm^H_K\colon \pi_0^KR\to\pi_0^HR
\]
for all subgroup inclusions $K\subset H\subset G$, which are morally given by sending a homotopy class $x\in\pi_0^KX$ to its $|H/K|$-fold power, with a specific, `twisted' action. These `twisted power operations' come from similarly defined maps
\[
    \Nm^H_K\Res^G_KR\to\Res^G_HR
\]
on the pointset level, where $\Res$ denotes restriction and $\Nm^H_K$ is now literally given by an $|H/K|$-fold smash product with a certain twisted action (the so-called \emph{Hill--Hopkins--Ravenel norm}, in honour of \cite{HHR2016Kervaire}). It is therefore quite natural to define a \emph{normed $G$-ring spectrum} as a $G$-spectrum $R$ together with `coherent' maps
\[
    \mathbb S\to R,\qquad R\smashp R\to R,\qquad \Nm^H_K\Res^G_KR\to\Res^G_HR \text{ for $K\subset H\subset G$.}
\]
Of course, the tricky part is to find a description that actually does take care of the infinite amount of homotopy coherence data implicit in the definition above. This was achieved by Bachmann and Hoyois \cite{BachmannHoyois2021Norms}; building on ideas of Hill and Hopkins \cite{hh1, hh2} and in line with the general philosophy of \emph{parametrized higher category theory} \cite{exposeI}, the key idea is to not view the category $\Sp_G$ of genuine $G$-spectra as an isolated object, but instead to consider the categories $\Sp_G$ for all finite groups $G$ at the same time, together with the contravariant functoriality in restrictions along group homomorphisms and the covariant functoriality in norm maps along subgroup inclusions. More formally, one can define\footnote{We give two equivalent definitions of $\ul\Sp^\otimes$ in the main text: one via model categories (Construction~\ref{constr:spotimes}) as outlined here and a purely $\infty$-categorical definition using spectral Mackey functors (see Theorem~\ref{thm:equivariant-uniqueness}, where we also show that the two definitions are indeed equivalent).} a certain (finite) product-preserving functor
\[
    \ul\Sp^\otimes\colon\Span_{\Forb}(\Fglo)\to\Cat
\]
from the category of \emph{spans} \cite{barwick2017spectral} of finite $1$-groupoids, whose objects are finite groupoids and whose morphisms are given by diagrams
\[
    \begin{tikzcd}
        A &[-.25em]\arrow[l] B\arrow[r, mono] & C
    \end{tikzcd}
\]
where the right-pointing morphism is in the wide subcategory $\Forb\subset\Fglo$ of faithful maps; this functor sends a finite group $G$ (viewed as a $1$-object groupoid) to the category $\und{\Sp}^\tensor(G) = \Sp_G$ of $G$-spectra, a left pointing, `backwards' morphism $G\xleftarrow{\;\;}H$ given by a group homomorphism $\phi\colon H\to G$ to the restriction functor $\phi^*\colon\Sp_G\to\Sp_H$, and a right-pointing, `forwards' morphism $H\rightarrowmono G$ given by a subgroup inclusion $H\subset G$ to the Hill--Hopkins--Ravenel norm $\Nm^G_H\colon\Sp_H\to\Sp_G$; the usual smash product of $G$-spectra is encoded in the functoriality with respect to the (right-pointing) fold maps $G\amalg G\rightarrowmono G$. We call such product-preserving functors $\Span_{\Forb}(\Fglo)\to\Cat$ \emph{normed global categories}, and refer to $\ul\Sp^\otimes$ as the \emph{normed global category of equivariant spectra}. The category of \emph{normed $G$-algebras} in any normed global category $\Cc^\otimes$ can then be rigorously defined as a certain category of functors over $\Span_{\Forb}(\Fglo)$ from $\Span(\mathbb F_G)$ to the cocartesian unstraightening $\Un^\cocart(\Cc^\otimes)$, i.e.~as a \emph{partially lax limit}, and for $\Cc^\otimes=\ul\Sp^\otimes$ this yields a rigorous implementation of the heuristic above for what a normed $G$-ring spectrum should be.\footnote{Note that for a fixed $G$, the definition only depends on the restriction of $\ul\Sp^\otimes$ to a functor $\Span(\mathbb F_G)\to\Cat$. This is the perspective taken in \cite{NardinShah}, where such product-preserving functors are called \emph{$G$-symmetric monoidal $G$-$\infty$-categories}.}

While the definition of normed algebras is considerably more involved than that of ultra-commutative $G$-ring spectra, they are often the more natural objects to consider in higher categorical contexts, and they are significantly easier to construct in this setting. As one example of this, combined work of Elmanto--Haugseng \cite{Elmanto_Haugseng_Bispans} as well as of Cnossen, Haugseng, and the first two authors of the present article \cite{CHLL_NRings} shows how various equivariant algebraic $K$-theory constructions refine (in one fell swoop) to produce normed $G$-ring spectra from input with extra symmetric monoidal structure; in particular, this gives a multiplicative equivariant algebraic $K$-theory functor producing normed $G$-ring spectra from stably symmetric monoidal categories with $G$-action.

\subsection*{Normed algebras vs.~the pointset model} In view of the above discussion, it would be desirable to be able to freely move between these two approaches, thereby combining the best of both worlds. However, while the expectation that normed $G$-ring spectra should be equivalent to ultra-commutative $G$-ring spectra has been expressed at several places in the literature \cites{BachmannHoyois2021Norms, NardinShah}, this comparison had resisted formal proof so far. As the first main result of the present paper, we close this gap. Let $\UCom_{G}$ denote the ($\infty$-)category obtained by Dwyer--Kan localizing the 1-category of strictly commutative algebras in symmetric $G$-spectra at the underlying stable equivalences. We prove the following theorem:

\begin{introthm}[See Theorem~\ref{thm:equivariant-model}]\label{introthm:equivariant-param}
    Let $G$ be a finite group. There exists an explicit equivalence
    \begin{equation}\tag{$*$}\label{introequ:equivariant-equivalence}
        \UCom_\text{$G$}\simeq\CAlg_G(\ul\Sp^\otimes),
    \end{equation}
    between the category of ultra-commutative $G$-ring spectra and the category of normed $G$-ring spectra. Moreover, this equivalence is
    natural in homomorphisms of finite groups and compatible with the forgetful functors to $\Sp_G$.
\end{introthm}

Maybe surprisingly, it is not clear that the right hand side is in fact functorial in all group homomorphisms instead of merely the injective ones, and part of the work is understanding the functoriality of $\CAlg_G(\ul\Sp^\otimes)$. Both this as well as our proof of the individual equivalences (\ref{introequ:equivariant-equivalence}) rely on \emph{global homotopy theory}, which we will recall now:

\subsection*{Global ultra-commutativity} The above examples of bordism spectra and topological $K$-theory spectra share a common feature in that they do not exist merely for a fixed group $G$, but naturally come to us as a `compatible family' of genuine $G$-spectra for all finite groups $G$.  Schwede's framework of global homotopy theory provides a rigorous way to talk about such phenomena; in its original formulation \cite{schwede2018global}, the basic objects of study are plain orthogonal or symmetric spectra, but viewed through a very fine-grained notion of weak equivalence, the \emph{global weak equivalences}, that see equivariant information for all finite groups; for every such \emph{global spectrum} $X$, we obtain a family of equivariant spectra by simply equipping $X$ with the trivial $G$-action for all finite groups $G$.\footnote{Note that, crucially, this is a procedure that would not be homotopically meaningful for the usual non-equivariant weak equivalences of symmetric or orthogonal spectra.}

One can then again define an \emph{ultra-commutative global ring spectrum} simply as a strictly commutative algebra on the $1$-categorical level, which similarly yields a family of ultra-commutative $G$-ring spectra for all $G$. The aforementioned examples of ultra-commutative equivariant bordism and topological $K$-theory spectra all arise via this procedure. In particular, despite its seeming na\"ivet\'e, this approach allows one to capture various interesting examples, and it does so in an extremely efficient and compact way. The global formalism can often be exploited fruitfully even if one is only interested in phenomena for a fixed group $G$ \cites{schwede2017powers,hausmann2022global, schwede2025snaith, LLP}.

Unfortunately, similarly to the equivariant story, the concept of global ultra-commutativity is opaque to higher categorical techniques, and this makes them hard to produce via such methods. Accordingly, one would like to have a notion of a \emph{normed global ring spectrum} defined in a higher categorical fashion. These were first introduced by the third author \cite{puetzstueck-new}. The construction relies on a certain \emph{normed global category of global spectra} $\ul\Sp_\gl^\otimes\colon\Span_{\Forb}(\Fglo)\to\Cat$; this sends the trivial group to Schwede's category $\Sp_\gl$ of global spectra, and more generally sends a finite group $G$ to the category $\Sp_\text{$G$-gl}$ of \emph{$G$-global spectra} in the sense of \cite{g-global}. The functoriality in backwards maps is again given by certain restriction functors, while the forward functoriality encodes $G$-global refinements of the smash product and the Hill--Hopkins--Ravenel norm; we refer the reader to Construction~\ref{constr:spgl} below for a precise definition. Using this, one can then define \emph{$G$-global normed ring spectra} as a certain category of functors $\Span_\Forb(\Fglo_{/G})\to\Un^{\cocart}(\ul\Sp_\gl^\otimes)$ over $\Span_{\Forb}(\Fglo)$. As in the equivariant context, we show:

\begin{introthm}[See Theorem~\ref{thm:global-model}]\label{introthm:global-param}
    For every $G$, there exists an explicit equivalence
    \[
        \UCom_\textup{$G$-gl}\simeq\CAlg_\textup{$G$-gl}(\ul\Sp^\otimes_\gl)
    \]
    between the category of $G$-global ultra-commutative ring spectra (again defined as strictly commutative algebras in a pointset model) and the category of $G$-global normed ring spectra. Moreover, these equivalences are natural in homomorphisms of finite groups and compatible with the forgetful functors to $\Sp_\textup{$G$-gl}$.
\end{introthm}

\subsection*{Global ultra-commutative ring spectra from equivariant data} In \cite{LNP}, Nardin, Pol, and the second author made precise the heuristic that global spectra encode `compatible families of equivariant spectra' by producing an equivalence between the category of global spectra and a certain partially lax limit of the categories of $H$-equivariant spectra for varying $H$:
\begin{equation}\tag{\ensuremath{\dagger}}\label{eq:spgl-as-plaxlim}
    \Sp_\gl\simeq\plaxlim_{H\in\Fglo^\op}{}\Sp_H\hskip-2pt.
\end{equation}
In \cite{Linskens2023globalization}, the second author generalized this comparison to $G$-global spectra for any finite group $G$, and moreover recast this comparison in terms of a universal property---the global category $\ul\Sp_\gl$ of ($G$-)global spectra is obtained from the global category $\ul\Sp$ of equivariant spectra by freely adjoining certain \emph{parametrized colimits}: $\ul\Sp$ is \emph{equivariantly presentable} in the sense of \cite{CLL_Clefts}, and the inclusion $\ul\Sp\hookrightarrow\ul\Sp_\gl$ is the initial example of an equivariantly cocontinuous functor to a \emph{globally presentable} global category; roughly speaking, this means that $\ul\Sp_\gl$ is obtained from $\ul\Sp$ by freely adjoining left adjoints to restrictions along \emph{surjective} homomorphisms, such that these left adjoints satisfy a Beck--Chevalley condition for pullback squares of finite groupoids. As shown in \cite{Linskens2023globalization}, any equivariantly presentable global category admits a \emph{globalization}, i.e.~an initial equivariantly cocontinuous functor to a globally presentable category, and this globalization can be computed as a partially lax limit analogous to $(\ref{eq:spgl-as-plaxlim})$.

As the second main contribution of our paper, we investigate how the process of globalization interacts with norms. In particular, we show (Theorem~\ref{thm:Normed_Rig}) that the globalization of a normed category again admits a canonical normed structure, and that globalization sends categories of \emph{equivariant} parametrized algebras to categories of \emph{global} parametrized algebras (Theorem~\ref{thm:globalize-algebras}). Together with Theorems~\ref{introthm:equivariant-param} and~\ref{introthm:global-param} above, this has the following concrete consequences, see Theorem~\ref{thm:ucom-glob}:

\begin{introthm}\label{introthm:globalize-ucom}
    The inclusion $\ul\UCom\hookrightarrow\ul\UCom_\gl$ (left adjoint to the forgetful functors) is the initial example of an equivariantly cocontinuous functor from the global category $\ul\UCom$ of equivariant ultra-commutative ring spectra to a globally presentable category.
\end{introthm}

\begin{introcor}
    There are equivalences \[\UCom_\textup{$G$-gl}\simeq\plaxlim\limits_{\phi\colon H\to G}{}\,\UCom_H\] natural in the finite group $G$ and compatible with the forgetful functors.
\end{introcor}

Similarly to the non-multiplicative comparison, we can view the special case $G=1$ as a rigorous version of the slogan that an ultra-commutative global ring spectrum encodes a compatible family of ultra-commutative equivariant ring spectra.

Note that in view of Theorem~\ref{introthm:equivariant-param}, each $\UCom_H$ is itself a certain partially lax limit of equivariant spectra. We also show that there is a more compact way to organize this information, expressing $G$-global ultra-commutative ring spectra as a single partially lax limit of equivariant spectra:

\begin{introthm}[see Corollary~\ref{cor:ucom-via-equivariant-spectra}]\label{introthm:ucomgl-from-sp}
    There exists an equivalence
    \[
        \UCom_\textup{$G$-gl}\simeq\plaxlim_{H\in\Span_{\Forb}(\Fglo_{/G})}\Sp_H
    \]
where the partially lax limit is lax away from faithful backwards maps $H\leftarrowmono K$. Moreover, this equivalence is again natural in $G$, and it lifts the equivalence
    \[
        \Sp_\textup{$G$-gl}\simeq \plaxlim_{H\in(\Fglo_{/G})^\op}\Sp_H
    \]
    from \cite{Linskens2023globalization}.
\end{introthm}

\subsection*{Tambara models} By a result of Brun \cite{brun-Tambara}, the homotopy rings $(\pi_0^HR)_{H\subset G}$ of an equivariant ultra-commutative ring spectrum $R$ naturally acquire the structure of a \emph{$G$-Tambara functor}. Such a $G$-Tambara functor encodes three kinds of functoriality between the individual rings $\pi_0^HR$:
\begin{enumerate}
    \item a contravariant functoriality along subgroup inclusions (and conjugations by elements of $G$), resulting in ring homomorphisms $\Res^H_K\colon\pi^H_0R\to\pi_0^KR$ for all $K\subset H\subset G$,
    \item a first covariant functoriality in subgroup inclusions, resulting in additive \emph{transfer homomorphisms} $\textup{Tr}^H_K\colon\pi_0^KR\to\pi_0^HR$, \textit{and}
    \item a second covariant functoriality in subgroup inclusions, resulting in the previously mentioned \emph{norm maps}, which are multiplicative (but non-additive!) maps $\Nm^H_K\colon\pi_0^KR\to\pi_0^HR$.
\end{enumerate}
These three kinds of functoriality have to satisfy various complicated compatibility conditions, which are most conveniently dealt with by hiding them in the composition law of a certain homotopy $1$-category ${\text{h\kern.5pt}}\Bispan(\mathbb F_G)$ of \emph{bispans} in finite $G$-sets. Given this category, one may define $G$-Tambara functors as product preserving functors ${\text{h\kern.5pt}}\Bispan(\mathbb F_G)\to\text{Set}$ satisfying a suitable `grouplikeness' condition.

In \cite{CHLL_NRings}, the first two authors have shown together with Cnossen and Haugseng that \emph{connective} normed $G$-ring spectra can in fact be equivalently described as (`higher') \emph{$G$-Tambara functors} in spaces. In light of Theorem~\ref{introthm:equivariant-param} above, we can recast this as follows:

\begin{introcor}[See Theorem~\ref{thm:equivariant-Tambara}]
    For any finite group $G$ we have an equivalence
    \[
        \UCom_{G,\ge0}\simeq\textup{Tamb}_G\subset\Fun^\times(\Bispan(\mathbb F_G),\Spc)
    \]
    between the category of connective ultra-commutative $G$-ring spectra and the category of $G$-Tambara functors in spaces.
\end{introcor}

Similarly, the homotopy groups of an ultra-commutative \emph{global} ring spectrum naturally come with additive transfers and multiplicative norms along {injective} group homomorphisms \cite{schwede2018global}*{§5.1}, assembling into the structure of a \emph{global Tambara functor}. However, if one wants to carry out the program of \cite{CHLL_NRings} to compare normed algebras in global spectra with (higher) global Tambara functors, one soon runs into a serious obstacle: namely, \cite{CHLL_NRings} crucially relied on having two equivalent descriptions of the normed structure on $\ul\Sp$, one in terms of Mackey functors and the other one in terms of classical pointset models and Hill--Hopkins--Ravenel norms. The argument given in \textit{op.\ cit.} that they are indeed equivalent used that the individual functors $\Sigma^\infty\colon\Spc_{G,*}\to\Sp_G$ already have (non-parametrized) universal properties in $\CAlg(\smash{\Pr^\textup{L}})$, which is no longer true for their global cousins. Fortunately, our results about the interplay of norms and globalization allow us to overcome this hurdle: we provide a Mackey functor description of the normed structure on $\ul\Sp_\gl^\otimes$ refining the non-multiplicative comparison of \cites{global-mackey,CLL_Spans, puetzstueck} (see Theorem~\ref{thm:normed-comparison}), and we use this together with the results of \cite{CHLL_NRings} to show:

\begin{introthm}[See Theorem~\ref{thm:global-Tambara}]\label{introthm:global-tambara}
    There is an equivalence
    \[
        \UCom_{\gl,\ge0}\simeq\textup{Tamb}_\gl\subset\Fun^\times(\Bispan(\Fglo,\Forb,\Forb),\Spc)
    \]
    between the category of connective ultra-commutative global ring spectra and the category of global Tambara functors in spaces.
\end{introthm}

We also provide a $G$-global generalization of this (see \textit{loc.~cit.}) as well as a version for the \emph{global commutative ring spectra with multiplicative deflations} recently considered by Blumberg, Mandell, and Yuan \cite{deflations}, see Theorem~\ref{thm:BMY-Tambara}.

\subsection*{Strategy and outline} The most natural approach to prove a comparison like Theorem~\ref{introthm:equivariant-param} or Theorem~\ref{introthm:global-param} is to use a monadicity argument: once we know both sides are monadic over the same base, it suffices to construct a comparison functor that commutes with the forgetful functors, and then verify a Beck--Chevalley condition for the free functors. And indeed, in both cases there are natural comparison maps from the model categorical construction to the higher categorical one \cite{puetzstueck-new}, which basically amount to changing the order in which we Dwyer--Kan localize or form normed algebras.

However, \emph{rigorously} carrying out this monadicity argument requires not only having `formul\ae' to compute the two free functors (in the sense of an unspecified equivalence), but to also understand how these descriptions interact with the comparison map. In particular, while Nardin and Shah \cite{NardinShah} already provide a formula for free normed $G$-algebras, it is not clear to the present authors how one could deduce Theorem~\ref{introthm:equivariant-param} from this.

In the present paper, we circumvent this issue by providing a \emph{universal formula} for the free functors, i.e.~one that is compatible with all normed functors. By construction of our comparison functor, this replaces the above problem with the problem of understanding in which sense the universal formula \emph{derives} when we pass from the model categories of (equivariant or global) spectra to their Dwyer--Kan localizations. Unfortunately, it seems that this is still not viable in the equivariant case: the formula provided by Nardin and Shah is hard to understand in model categorical terms, and we can offer no simplification of their description.

Somewhat surprisingly (and fortunately for us!), it turns out that the global case is actually easier to handle from this perspective: the formula we prove is a straightforward `genuine' version of the usual formula $\coprod_{n\ge 0} (X^{\otimes n})_{h\Sigma_n}$ for free algebras in a presentably symmetric monoidal category, and the model categorical results of \cites{g-global,Lenz-Stahlhauer} are more than enough to understand how this formula derives, and hence prove Theorem~\ref{introthm:global-param}. Accordingly, as already alluded to above, the order in which we presented  our results was purposefully misleading: we will not prove Theorem~\ref{introthm:equivariant-param} directly, but only deduce it near the end of the whole paper from the global comparison; as the localization $\Sp_\text{$G$-gl}\to\Sp_G$ is known to lift to a localization $\UCom_\text{$G$-gl}\to\UCom_G$, this reduces to similarly exhibiting $\CAlg_G(\ul\Sp^\otimes)$ as a localization of $\CAlg_\text{$G$-gl}(\ul\Sp^\otimes_\gl)$. Both this localization result as well as the formula for free normed global algebra require developing of a great deal of (parametrized) higher algebra, and this actually accounts for the bulk of the present article.

We now turn to a linear overview of the paper.

In Section~\ref{sec:basic} we develop the basics of parametrized higher algebra needed throughout this article, in particular introducing parametrized operads and normed categories; rather than defining these concepts as fibrations internal to parametrized higher category theory as in \cite{NardinShah}, we will follow \cite{BHS_Algebraic_Patterns} in viewing them as ordinary fibrations over suitable span categories. We further define parametrized categories of normed algebras and prove a representability result for them that will be crucial in later sections.

In Section~\ref{sec:free-nalg-general} we introduce a variant of the notion of \emph{distributivity}, originally studied by Nardin \cite{nardin2017thesis}, for our setup. We show that for a distributive normed category $\Cc^\otimes$, the forgetful functor $\mathbb U$ from normed algebras in $\Cc^\otimes$ to $\Cc$ itself has a left adjoint $\mathbb P$, and we provide a universal formula for this left adjoint (or, more precisely, for the composite $\mathbb U\mathbb P$). While all of this happens in a very general setting, we specialize this discussion to the context of global equivariant homotopy theory in Section~\ref{sec:free-global}.

In Section~\ref{sec:global-model} we introduce the normed global category of global spectra $\ul\Sp^\otimes_\gl$ and show that it is indeed distributive in the above sense. We go on to prove Theorem~\ref{introthm:global-param}, following the strategy outlined above.

We then shift gears in Section~\ref{sec:norms-and-globalization} and begin studying the interaction of globalization and normed structures. For this it will be convenient to exhibit the globalization construction as an instance of a more general \emph{rigidification} construction due to Abellán \cite{Abellan23}, which admits a universal property for functors \emph{into} it. We show that the rigidification of a normed category inherits a canonical normed structure, and that the resulting normed category again has a universal property. Combining this universal property with the representability result for categories of normed algebras we obtain Theorem~\ref{introthm:ucomgl-from-sp}.

Building on this, Section~\ref{sec:globalize-equivariant-alg} studies normed algebras in rigidifications of normed categories. We in particular explain how to assemble the categories $\CAlg_G(\ul\Sp^\otimes)$ into a global category, and show that the globalization of the result is given by the global category of normed algebras in $\ul\Sp_\gl^\otimes$. This allows us to deduce Theorem~\ref{introthm:equivariant-param} from its global counterpart, which then immediately yields Theorem~\ref{introthm:globalize-ucom}.

In Section~\ref{sec:tambara} we shift gears once more, and introduce alternative descriptions of the normed categories $\ul\Sp^\otimes$ and $\ul\Sp^\otimes_\gl$ in terms of \emph{spectral Mackey functors}. In the global case, this crucially relies on the results of Section~\ref{sec:norms-and-globalization} on the interaction of normed structures and rigidification. With these models at hand, we then use the results of \cite{CHLL_NRings} to describe $G$-global \emph{connective} ultra-commutative ring spectra as space-valued Tambara functors, in particular proving Theorem~\ref{introthm:global-tambara}.

The paper concludes with three short appendices: in Appendix~\ref{app:orthogonal} we compare two different pointset models of ultra-commutative global ring spectra, while Appendix~\ref{app:basechange} is devoted to the proof of parametrized versions of \emph{smooth and proper basechange} and the basechange results for functors of span categories established in \cite{CHLL_Bispans}. Finally, we make precise in Appendix~\ref{app:ordinary-HA} that ultra-commutative equivariant or global ring spectra cannot be described using ordinary higher algebra.

\subsection*{Acknowledgements}
The authors thank Clark Barwick, Bastiaan Cnossen, Stefan Schwede, and the anonymous referee for feedback on the paper. T.L.~would like to thank Jan Steinebrunner for an enlightening exchange about equifiberedness. S.L.~would like to thank Bastiaan Cnossen for many interesting discussions related to this project, and Fernando Abell\'an for explaining the relevance of \cite{Abellan23} to the notion of globalization developed in \cite{Linskens2023globalization}. It proved invaluable for the present work. P.P.~thanks Maxime Ramzi for a helpful discussion on free symmetric monoidal categories.

T.L.\ is an associate member of the Hausdorff Center for Mathematics
at the University of Bonn (DFG GZ 2047/1, project ID 390685813).

S.L.\ is an associate member of the SFB 1085 Higher Invariants.

P.P.'s contribution to this project was funded by the Deutsche Forschungsgemeinschaft (DFG, German Research Foundation) -- Project-ID 427320536 -- SFB 1442, as well as under Germany's Excellence Strategy EXC 2044/2 -390685587, Mathematics Münster: Dynamics--Geometry--Structure.

\section{Parametrized higher algebra}\label{sec:basic}

\subsection{Recollections on categories of spans}
We begin with some recollections on categories of spans,
which will be vital for almost everything in this paper.
For details, we refer the reader to \cite{HHLNa}.

\begin{definition}[\cite{barwick2017spectral}*{Definition~5.2}, \cite{HHLNa}*{Definition 2.1}]
    An \emph{adequate triple} $(\cC,\cC_B,\cC_F)$ consists of a category $\cC$
    together with two wide subcategories $\cC_B,\cC_F \subset \cC$
    of `backwards' and `forwards' maps,\footnote{\cite{barwick2017spectral}  and \cite{HHLNa} instead use the terminology `ingressive' and `egressive' for what we call forwards and backwards maps, respectively.}
    such that pullbacks of morphisms in $\cC_B$ along morphisms in $\cC_F$
    exist in $\cC$, and both $\cC_B$ and $\cC_F$ are stable under pullback.
    We write $\AdTrip$ for the category of adequate triples,
    where morphisms $(\cC,\cC_B,\cC_F) \to (\cD,\cD_B,\cD_F)$ are functors $\cC \to \cD$
    preserving forwards and backwards maps as well as pullbacks of forwards along backwards maps.
\end{definition}

\begin{example}
    If $\cC$ is any category, then $(\cC,\cC,\core\cC)$, $(\cC,\core\cC,\cC)$, and $(\cC,\core\cC,\core\cC)$ are adequate triples, where $\core\cC\subset\cC$ denotes the maximal subgroupoid.
\end{example}

\begin{example}
    If $\cC$ is a category admitting pullbacks, then $(\cC,\cC,\cC)$ forms an adequate triple.
\end{example}

Quite often, the following special case will be enough for our needs.

\begin{definition}
    We say that a category $\cC$ together with a wide subcategory $\cC_F \subset \cC$
    is a \emph{span pair} if $(\cC,\cC,\cC_F)$ is an adequate triple.
\end{definition}

The main use of adequate triples is that they present the minimum necessary conditions
to build a \emph{span category} (called the \emph{effective Burnside category} in Barwick's original treatment \cite{barwick2017spectral}). Namely by \cite{HHLNa}*{Definition~2.12}, there is a functor
\[
    \Span \colon \AdTrip \to \Cat,\ (\cC,\cC_B,\cC_F) \mapsto \Span_{B,F}(\cC),
\]
where the category $\Span_{B,F}(\cC)$ has the same objects as $\cC$, and morphisms from $X$ to $Y$ are given by spans
\[
\begin{tikzcd}[row sep=small]
	& Z \\
	X && Y
	\arrow["b"', from=1-2, to=2-1]
	\arrow["f", from=1-2, to=2-3]
\end{tikzcd}
\]
where $b$ is a morphism in $\cC_B$ and $f$ is a morphism in $\cC_F$. Composition is by taking pullbacks in $\cC$.

Given a span pair $(\cC,\cC_F)$, we will denote the span category associated to $(\cC,\cC,\cC_F)$
by $\Span_F(\cC)$, or sometimes write $\Span_{\all,F}(\cC)$ for emphasis.
Similarly, we will sometimes use $\Span_{B,\all}(\cC)$
to denote the category of spans associated to an adequate triple of the form $(\cC,\cC_B,\cC)$.
Finally, if $\cC$ admits pullbacks we will write $\Span(\cC)$ for the span category associated to $(\cC,\cC,\cC)$.

Let us collect some of the main results on span categories from \cite{HHLNa}.

\begin{theorem}\label{thm:spans}
    Let $(\cC,\cC_B,\cC_F)$ be an adequate triple.
    \begin{enumerate}
        \item The category $\AdTrip$ has all limits and they are computed in $\Cat$.
        \item The functor $\Span \colon \AdTrip \to \Cat$ is a right adjoint, hence preserves all limits.

        \item There is a natural equivalence $\Span_{B,F}(\cC)^\op \simeq \Span_{F,B}(\cC)$.

        \item We have inclusions of wide subcategories
            \[
                \cC_B^\op \hookrightarrow \Span_{B,F}(\cC)
                \qquad\text{and}\qquad
                \cC_F \hookrightarrow \Span_{B,F}(\cC),
            \]
            sending a morphism $b \colon X \to Y$ in $\cC_B$ to the `backwards span' $Y \xleftarrow{b} X = X$,
            and a morphism $f \colon Y \to Z$ in $\cC_F$ to the `forwards span' $Y = Y \xto{f} Z$, respectively.
            Moreover, in the case that $\cC_B = \core\cC$, the second inclusion
            induces an equivalence $\cC_F \simeq \Span_{\simeq,F}(\cC)$.
            Analogously, $\cC_B^\op \simeq \Span_{B,\simeq}(\cC)$.

        \item The backwards and forwards maps form an (orthogonal) factorization system
            $(\Span_{B,F}(\cC),\cC_B^\op,\cC_F)$.
    \end{enumerate}
\end{theorem}
\begin{proof}
    In the listed order, these results are \cite{HHLNa}*{Lemma 2.4, Theorem 2.18, Lemma 2.14, Proposition 2.15, and Proposition 4.9}.
\end{proof}

\subsection{Normed precategories and preoperads}

For the remainder of this section, we fix a span pair $(\cF,\cN)$.

\begin{notation}
	We will denote maps in $\cN$ as $A\rightarrownorm B$.
\end{notation}

\begin{definition}
We let $\PreOp{\cF}{\cN}$ denote the subcategory of $\Cat_{/\Span_{\cN}(\cF)}$ whose objects are the \emph{$(\cF,\cN)$-preoperads}, defined to be those functors which admit cocartesian lifts for all morphisms in $\cF^{\op}\subset  \Span_{\cN}(\cF)$,
and whose morphisms are those functors over $\Span_{\cN}(\cF)$ which preserve such lifts.
\end{definition}

\begin{warn}
    Beware that the condition that $\cX \to \Span_{\cN}(\cF)$ admits cocartesian lifts of morphisms in $\cF^{\op}$ is strictly stronger than demanding that the pullback $\cX \times_{\Span_{\cN}(\cF)} \cF^{\op}\to\cF^\op$ be a cocartesian fibration.
\end{warn}

\begin{definition}
We let $\PreNmCat{\cF}{\cN}$ denote the category $\Fun(\Span_{\cN}(\cF),\Cat)$. We call objects of this category \emph{$\cN$-normed $\cF$-precategories}. We view this as a non-full subcategory of $\PreOp{\cF}{\cN}$ via the identification of $\PreNmCat{\cF}{\cN}$ with cocartesian fibrations over $\Span_{\cN}(\cF)$ given by unstraightening.
\end{definition}

\begin{remark}\label{rem:classical_op}
    Suppose $\cF = \cN = \mathbb{F}$ is the category of finite sets. Then by \cite{BHS_Algebraic_Patterns}*{Corollary B}, operads in the sense of Lurie correspond to a full subcategory of $\PreOp{\mathbb{F}}{\mathbb{F}}$. This connection inspires our choice of notation. In §\ref{subsec:operads} we will single out the correct category of \emph{$(\cF,\cN)$-operads} which generalize Lurie's operads.
\end{remark}

\begin{definition}
Let $\cX\to \Span_{\cN}(\cF)$ be a preoperad. Note that the pullback of $\cX$ along the inclusion $\cF^{\op}\subset \Span_{\cN}(\cF)$ is a cocartesian fibration over $\cF^{\op}$. In particular we obtain a functor
\[
\mathrm{fgt}\colon \PreOp{\cF}{\cN}\to \Fun(\cF^{\op},\Cat).
\]
For consistency of notation we denote $\Fun(\cF^{\op},\Cat)$ by $\PreCat{\cF}$ and call its objects \emph{$\cF$-precategories}.
\end{definition}

\begin{remark}
    Functors $\cF^\op\to\Cat$ are typically called \emph{$\cF$-categories} in the literature, and they are the fundamental object of study in \emph{parametrized higher category theory} \cite{exposeI}. As we will explain in Remarks~\ref{rk:why-pre-1} and~\ref{rk:why-pre-2}, it will be more convenient for our purposes to reserve that name for functors satisfying an additional product preservation condition.
\end{remark}

\begin{remark}\label{rem:2cat-structure}
    We note that each of the categories $\PreOp{\cF}{\cN}, \PreNmCat{\cF}{\cN}$ and $\PreCat{\cF}$ is a subcategory of a slice of $\Cat$.
    For every $S$, the functor $S \times - \colon \Cat \to \Cat_{/S}$
    is symmetric monoidal for the cartesian symmetric monoidal structures,
    and moreover factors through $\Cat^{\cocart}_{/S}$, the category of cocartesian fibration over $S$;
    under the equivalence of the latter with $\Fun(S,\Cat)$, it is simply
    given by $\const \colon \Cat \to \Fun(S,\Cat)$.
    In particular, all of the above categories are canonically $\Cat$-modules,
    and the forgetful functors $\PreNmCat{\cF}{\cN} \to \PreOp{\cF}{\cN} \to \PreCat{\cF}$ are canonically lax $\Cat$-linear.

    All of these $\Cat$-module structures are adjoint to $\Cat$-enrichments
    in the sense of Lurie, and we denote the resulting hom categories in $\PreOp{\cF}{\cN}, \PreNmCat{\cF}{\cN}$ and $\PreCat{\cF}$ by
    \[
        \Fun_{\cF}^{\Nlax}(-,-),
        \quad \Fun_{\cF}^{\Nstr}(-,-),
        \quad \text{and} \quad \Fun_{\cF}(-,-)
    \]
    respectively.

    More specifically, all of these categories canonically inherit the structure
    of an $(\infty,2)$-category from $\Cat_{/S}$, as we now briefly recall; for a detailed overview of the following, we refer the reader
    to \cite{Heyer-Mann}*{Appendices C and D}.
    By definition, an $(\infty,2)$-category is a $\Cat$-enriched category,
    and an $(\infty,2)$-functor is a $\Cat$-enriched functor.
    There are two prominent models of enriched $\infty$-categories,
    developed by Lurie \cite{HA}*{Definition 4.2.1.28}
    and by Gepner--Haugseng \cite{Gepner-Haugseng} and Hinich \cite{Hinich}, respectively,
    which were shown to be equivalent in a strong sense by Heine \cite{Heine}.
    Using the Lurie model, if $\cC$ is a $\Cat$-module
    such that the functor $- \times c \colon \Cat \to \cC$ admits a right adjoint
    for every $c \in \cC$, then $\cC$ inherits the structure of an $(\infty,2)$-category
    \cite{Heyer-Mann}*{Example C.1.11} (although weaker conditions suffice),
    and the category of $(\infty,2)$-functors between two such $\Cat$-modules
    agrees with the category of lax $\Cat$-linear functors \cite{Heyer-Mann}*{Example C.2.2}.
    Recall also from \cite{HA}*{Corollary 7.3.2.7}
    that a right adjoint of a $\Cat$-linear functor automatically obtains a lax $\Cat$-linear structure, which thus upgrades the entire adjunction to one of $(\infty,2)$-categories.
    In particular, in this case the natural adjunction equivalence
    of mapping spaces lifts to one of mapping categories.
\end{remark}

\begin{remark}\label{rem:adj_in_functor_cat}
Let $S$ be a category. Then a morphism $L\colon F\to G$ is a left adjoint in the $(\infty,2)$-category $\Fun(S,
\Cat)$ if it is a \emph{right adjointable natural transformation}, by which we mean that
\begin{enumerate}
	\item For every $X\in S$ the functor $L_X$ admits a right adjoint $R_X$.
	\item For every morphism $f\colon X\to Y$, the \emph{Beck--Chevalley transformation}
	\[
	F(f) R_X \xRightarrow{\eta} R_Y L_Y F(f) R_X\simeq R_Y G(f)L_XR_X \xRightarrow{\epsilon} R_Y G(f)
	\]
	is an equivalence.
\end{enumerate}
This is the content of \cite{HA}*{Proposition 7.3.2.11}, after applying the unstraightening equivalence. Dually, a morphism is a right adjoint if it is a \emph{left adjointable natural transformation}. Specializing the above, we obtain a description of adjunctions in $\PreCat{\cF}$ and $\PreNmCat{\cF}{\cN}$, which we will use throughout.
\end{remark}

Our goal now is to repeat some of the basic constructions on operads in our context.

\begin{definition}\label{def:n-cocart}
    Let $\cC\colon \cF^{\op}\to \Cat$ be an $\cF$-precategory.
    \begin{enumerate}
        \item We write $\Triv(\cC)$ for the composition  $\Un^\co(\cC) \to \cF^\op \to \Span_\cN(\cF)$ of the cocartesian unstraightening of $\cC$ with the inclusion $\cF^\op \hookrightarrow \Span_\cN(\cF)$.

        \item Let $p \colon \Un^\ct(\cC) \to \cF$ denote the cartesian unstraightening of $\cC$. We define $\cC^{\Ncoprod}$ to be the functor
         \[
     	    \Span(p) \colon \Span_{\cart,\cN}(\Un^\ct(\cC)) \to \Span_\cN(\cF),
         \]
         where in the source we have taken the span category in $\Un^\ct(\cC)$ whose backwards maps are cartesian and whose forward maps are those lying over $\cN$. These two subcategories determine an adequate triple by \cite{HHLNa}*{Proposition 2.6}.
    \end{enumerate}
\end{definition}

\begin{lemma}\label{lem:unfurl}
    Let $\cC$ be an $\cF$-precategory.
    \begin{enumerate}
        \item Both $\Triv(\cC)$ and $\cC^{\Ncoprod}$ are $(\cF,\cN)$-preoperads.

        \item There exists a natural map $\incl_\cC \colon \Triv(\cC) \to \cC^{\Ncoprod}$ of $(\cF,\cN)$-preoperads induced by the pullback square
            \[\begin{tikzcd}
                {\Un^\co(\cC)} & {\cC^{\Ncoprod}} \\
                {\cF^\op} & {\Span_\cN(\cF)}
                \arrow[from=1-1, to=1-2]
                \arrow[from=1-1, to=2-1]
                \arrow["\lrcorner"{anchor=center, pos=0.125}, draw=none, from=1-1, to=2-2]
                \arrow[from=1-2, to=2-2]
                \arrow[from=2-1, to=2-2]
            \end{tikzcd}\]
        \item Suppose that $\cC$ is left $\cN$-adjointable in the sense of
            \cite[Definition 2.2.5]{Elmanto_Haugseng_Bispans},
            i.e.~that for every $n\colon A\rightarrownorm B$ in $\Nn$ the functor $n^*\coloneqq\Cc(n)\colon\Cc(B)\to\Cc(A)$ admits a left adjoint $n_!$ and that every pullback
        \begin{equation*}
            \begin{tikzcd}
                A\arrow[d,"f'"']\arrow[dr,pullback]\arrow[r, norm, "n'"] & B\arrow[d, "f"]\\
                C\arrow[r, norm, "n"'] & D
            \end{tikzcd}
        \end{equation*}
        in $\cF$ with horizontal maps in $\cN$ as indicated, is sent to a left adjointable square, meaning that the Beck--Chevalley transformation
        \[
            n'_!f^{\prime*}\to f^{*}n_!
        \]
        is an equivalence. Then $\cC^{\Ncoprod}$ is a cocartesian fibration classifying an extension of $\cC$ to an $\cN$-normed $\cF$-precategory sending the forward maps to the left adjoints.
    \end{enumerate}
\end{lemma}

\begin{proof}
  	We note that $\cC^{\Ncoprod}$ is Barwick's unfurling construction, and all claims about it
    follow immediately from \cite[Theorem 3.1, Example 3.4, Corollary 3.18]{HHLNa}.

    It remains to see that $\Triv(\cC)$ is a cocartesian fibration
    and that $\incl_\cC$ preserves cocartesian lifts of backwards maps.
    The former follows immediately from the fact that $\cF^{\op}\to \Span_{\cN}(\cF)$ is itself an $(\cF,\cN)$-operad, which in turn follows immediately from the fact that the subcategory $\cF^{\op}\subset \Span_{\cN}(\cF)$ is right cancellable.

   Finally, to see that the map $\incl_\cC$ preserves cocartesian edges over $\cF^{\op}$ we may appeal to the characterization of cocartesian edges in $\cC^{\Ncoprod}$ from Theorem 3.1 of \emph{op.~cit.}
\end{proof}

\begin{example}\label{ex:Triv-Ncoprod-rep}
    Write $\ul A$ for the $\Ff$-precategory represented by $A\in\Ff$. Then we have $\Triv(\ul A)\simeq(\Ff_{/A})^\op$ and $\ul A^{\Ncoprod}\simeq\Span_\Nn(\Ff_{/A})\coloneqq\Span(\Ff_{/A},\Ff_{/A},\Nn\times_\Ff\Ff_{/A})$ naturally in $A$. Under these identifications, the map $\incl_{\ul A}\colon\Triv(\ul A)\hookrightarrow\ul A^{\Ncoprod}$ is simply given by the inclusion of backwards maps.
\end{example}

The next crucial construction we wish to imitate is the envelope of operads. Thankfully, this has already been done by \cite{BHS_Algebraic_Patterns} in the generality of an arbitrary factorization system $(S,E,M)$. For later use, we recall the relevant definitions and results in this generality.

\begin{notation}
    We denote arrows in $M$ by $\rightarrowmono$ and arrows in $E$ by $\rightarrowepic$.
\end{notation}

\begin{definition}\label{def:e-cocart}
Consider a category $S$ equipped with a subcategory $S_0$. We let $\Cat^{S_0\dcocart}_{/S}$ denote the subcategory of $\Cat_{/S}$
on those objects $\cX \to S$ which are \emph{$S_0$-cocartesian fibrations}, i.e.~admit cocartesian lifts for all morphisms in $S_0$,
and those morphisms which preserve these lifts.
\end{definition}

Of particular interest now is the case that $S_0 = E$.

\begin{example}\label{ex:span-fact-sys}
    For $S=\Span_\Nn(\Ff)$ equipped with its standard factorization system from \cref{thm:spans}(5),
    the category $\Cat_{/S}^{E\dcocart}$ is equal to $\PreOp{\cF}{\cN}$,
    while $\Cat^\cocart_{/S}$ agrees with $\PreNmCat{\cF}{\cN}$.
\end{example}

\begin{definition}
    We denote by $\Env(1) \coloneqq \Ar_M(S)$ the full subcategory of $\Ar(S)$ spanned by the maps in $M$. We view $\Env(1)$ as a category over $S$ via the target projection.
\end{definition}

\begin{definition}\label{def:envelope}
Let $\cX\to S$ be an $E$-cocartesian fibration. We denote by $\Env(\cX)$ the following pullback
\[\begin{tikzcd}
	{\Env(\cX)} & \cX \\
	{\Env(1)} & {S\rlap.}
	\arrow[from=1-1, to=1-2]
	\arrow[from=1-1, to=2-1]
	\arrow[from=1-2, to=2-2]
	\arrow[""{name=0, anchor=center, inner sep=0}, "s", from=2-1, to=2-2]
	\arrow["\lrcorner"{anchor=center, pos=0.125}, draw=none, from=1-1, to=0]
\end{tikzcd}\]
We view $\Env(\cX)$ as a category over $S$ via the composite
\[
\Env(\cX) \to 	{\Env(1)}  \xrightarrow{t} S.
\]
\end{definition}

Note that $\Env(1)$ is $\Env(-)$ applied to the final object $1=\id_S$ of $\Cat^{\Lcocart}_{/S}$, as the notation suggests. The following is \cite[Theorem E]{BHS_Algebraic_Patterns}:

\begin{theorem}\label{thm:BHS}
	Let $(S,E,M)$ be a factorization system.
	\begin{enumerate}
		\item\label{BHS-1} Let $p\colon \cX\to S$ be an $E$-cocartesian fibration. Then $\Env(\cX)$ is a cocartesian fibration. Moreover, a morphism in $\Env(\cX)$ corresponding to a pair
		\[
            \left( f\colon x\to y,
            \begin{tikzcd}
                p(x)\arrow[d, mono]\arrow[r, "p(f)"] & p(y)\arrow[d,mono]\\
                a\arrow[r] & b
            \end{tikzcd}
            \right)
        \]
        is cocartesian if and only if $p(f)$ belongs to $E$ and $f$ is a cocartesian lift of $p(f)$.
		\item The functor $\Env(-)$ restrict to a functor
		\[
		\Env(-) \colon \Cat^{\Lcocart}_{/S} \to \Cat^{\cocart}_{/S}.
		\]
		As such it is a $\Cat$-linear left adjoint to the inclusion $\Cat^{\cocart}_{/S} \hookrightarrow  \Cat^{\Lcocart}_{/S}$, with unit $\eta$ given at $\cX$ by the identity section $e \mapsto (e, \id_{p(e)})$.
		\item The induced functor $\Env(-) \colon \Cat^{\Lcocart}_{/S} \to (\Cat^{\cocart}_{/S})_{/\Env(1)}$ is fully faithful. In particular, $\Env(-) \colon \Cat^{\Lcocart}_{/S} \to \Cat^{\cocart}_{/S}$ is  conservative.
	\end{enumerate}
\end{theorem}

\begin{proof}
The first item is the content of \cite[Proposition 2.2.2]{BHS_Algebraic_Patterns}. The adjunction in (2) is proven by combining Proposition 2.1.4 and Proposition~2.2.4 of \emph{op.\ cit,} while $\Cat$-linearity of $\Env(-)$ is proven in Proposition 5.3.4. Finally, (3) is Proposition 2.3.2 of \emph{op.~cit.}
\end{proof}

In view of \cref{ex:span-fact-sys} we may now apply this result to the case of $(\cF,\cN)$-preoperads.

\begin{definition}
We denote by
\[
\Env(-)\colon \PreOp{\cF}{\cN}\to \PreNmCat{\cF}{\cN}
\]
the left adjoint to the inclusion functor $\PreNmCat{\cF}{\cN}\hookrightarrow \PreOp{\cF}{\cN}$ obtained via the previous result.
Since it is $\Cat$-linear, this canonically upgrades to an adjunction of $(\infty,2)$-categories
in view of \cref{rem:2cat-structure}.
\end{definition}

\begin{remark}\label{rem:cocart_in_Env_1}
By \cite{HHLNb}*{Corollary~2.22}, $\Env(1) = \Ar_{\cN}(\Span_{\cN}(\cF))$ is the category of spans in $\Ar_\Nn(\Ff)$ of the form
\[\begin{tikzcd}
	X & Z & Y \\
	A & C & B
	\arrow[norm, from=1-1, to=2-1]
	\arrow[from=1-2, to=1-1]
		\arrow[norm, from=1-2, to=1-3]
		\arrow["\lrcorner"{anchor=center, pos=0.125, rotate=-90}, draw=none, from=1-2, to=2-1]
		\arrow[norm, from=1-2, to=2-2]
		\arrow[norm, from=1-3, to=2-3]
		\arrow[from=2-2, to=2-1]
		\arrow[norm, from=2-2, to=2-3]
	\end{tikzcd}\]
	where the left square is a pullback and all the marked arrows are, by our standing notational convention, in $\cN$. Unwinding the description of cocartesian edges given in Theorem~\ref{thm:BHS}, we find that such a diagram is $t$-cocartesian if and only if the map $Z \rightarrownorm Y$ is an equivalence.
\end{remark}

We now argue for the expected universal properties of $\Triv(-)$ and $(-)^{\Ncoprod}$.

\begin{proposition}\label{prop:calg}
	The constructions $\Triv(-)$ and $(-)^{\Ncoprod}$ upgrade to $\Cat$-linear left adjoints
	\[
	\Triv, (-)^{\Ncoprod} \colon \PreCat{\cF} \to \PreOp{\cF}{\cN}.
	\]
	such that the natural transformation $\incl \colon \Triv \Rightarrow (-)^{\Ncoprod}$ is $\Cat$-linear.
    Moreover:
    \begin{enumerate}
        \item $(-)^{\Ncoprod}$ preserves all limits, hence also admits a left adjoint.
        \item The functor $\Triv$ is fully faithful, and it is left adjoint to the forgetful functor $\PreOp{\cF}{\cN}\to \PreCat{\cF}$, with unit given by the evident equivalence $\id\simeq\fgt\circ\Triv$.
    \end{enumerate}
\end{proposition}

\begin{remark}
    In view of \cref{rem:2cat-structure}, $\Triv$ and $(-)^{\Ncoprod}$
    thus canonically upgrade to left adjoint 2-functors between 2-categories,
    and $\incl$ becomes a 2-natural transformation.
\end{remark}

\begin{proof}[Proof of Proposition~\ref{prop:calg}]
    We begin by showing the first addendum, that $(-)^{\Ncoprod}$ preserves all limits.
    Note that the forgetful functor $\PreOp{\cF}{\cN} \to \Cat_{/\Span_\cN(\cF)}$ is a conservative right adjoint.
    The composite of $(-)^{\Ncoprod}$ with it factors as
    \[
        \PreCat{\cF} \simeq \Cat^\cart_{/\cF} \xto{\Phi} \AdTrip_{/(\cF,\all,\cN)} \xto{\Span} \Cat_{/\Span_\cN(\cF)}
    \]
    where $\Phi$ sends a cartesian fibration $\cX \to \cF$ to the adequate triple
    $(\cX, \cX_\cart, \cX \times_\cF \cN)$ over $(\cF,\cF,\cN)$ (the former is indeed adequate by \cite[Proposition 2.6]{HHLNa}).
    By \cref{thm:spans} limits in $\AdTrip$ are computed pointwise and $\Span$ is a right adjoint,
    so that it remains to see that the functors
    \[
        \Cat_{/\cF}^\cart \to \Cat_{/\cF},\ \cX \mapsto \cX^\cart
        \quad\text{and}\quad
        \Cat_{/\cF}^\cart \to \Cat_{/\cN},\ \cX \mapsto \cX \times_\cF \cN
    \]
    preserve limits. The second functor factors as the right adjoint $\Cat_{/\cF}^\cart \to \Cat_{/\cN}^\cart$
    followed by the forgetful $\Cat_{/\cN}^\cart \to \Cat_{/\cN}$, which preserves limits
    by \cite[Theorem 1.2]{GHN17}. Analogously, the first functor actually also lands in $\Cat^\cart_{/\cF}$,
    where under the equivalence with $\Fun(\cF^\op,\Cat)$ it corresponds to postcomposing with the groupoid
    core functor, which clearly preserves limits.

	Next, we will employ the adjoint functor theorem of \cite{HTT}*{Corollary 5.5.2.9} to see that $(-)^{\Ncoprod}$ admits a right adjoint. The domain is clearly presentable and the codomain is by \cite{BHS_Algebraic_Patterns}*{Proposition 2.3.7} and so it remains to see that the functor preserves colimits.
	This can be tested after postcomposing with the conservative left adjoint $\Env(-)$ and then evaluating at every $A \in \Span_\cN(\cF)$.
	We may consider the following pullbacks, where $\pi_A$ denotes the composite $\cN_{/A} \to \cN \to \cF$:
    \[\begin{tikzcd}[cramped]
        {\Un^\ct(\pi_A^*\cC)} & {\Ar(\cN)\times_{s,\cF} \Un^\ct(\cC)} & {\Env(\cC^{\Ncoprod})} & {\cC^{\Ncoprod}} \\
        {\cN_{/A}} & {\Ar(\cN)} & {\Ar_\cN(\Span_\cN(\cF))} & {\Span_\cN(\cF)} \\
        {1} & \cN & {\Span_\cN(\cF)}
        \arrow[from=1-1, to=1-2]
        \arrow[from=1-1, to=2-1]
        \arrow[from=1-2, to=1-3]
        \arrow[from=1-2, to=2-2]
        \arrow[from=1-3, to=1-4]
        \arrow[from=1-3, to=2-3]
        \arrow[from=1-4, to=2-4]
        \arrow[""{name=0, anchor=center, inner sep=0}, from=2-1, to=2-2]
        \arrow[from=2-1, to=3-1]
        \arrow[""{name=1, anchor=center, inner sep=0}, from=2-2, to=2-3]
        \arrow["t", from=2-2, to=3-2]
        \arrow[""{name=2, anchor=center, inner sep=0}, "s", from=2-3, to=2-4]
        \arrow["t", from=2-3, to=3-3]
        \arrow[""{name=3, anchor=center, inner sep=0}, "A", from=3-1, to=3-2]
        \arrow[""{name=4, anchor=center, inner sep=0}, from=3-2, to=3-3]
        \arrow["\lrcorner"{anchor=center, pos=0.125}, draw=none, from=1-1, to=0]
        \arrow["\lrcorner"{anchor=center, pos=0.125}, draw=none, from=1-2, to=1]
        \arrow["\lrcorner"{anchor=center, pos=0.125}, draw=none, from=1-3, to=2]
        \arrow["\lrcorner"{anchor=center, pos=0.125}, draw=none, from=2-1, to=3]
        \arrow["\lrcorner"{anchor=center, pos=0.125}, draw=none, from=2-2, to=4]
    \end{tikzcd}\]
    The bottom left square is a pullback by definition, while the bottom right is by the fact that forward maps in a span category are left cancellable.\footnote{We recall that a class $\Xx$ of maps in a category $\Cc$ is called \emph{left cancellable} if given any composable maps $f,g$ in $\Cc$ such that both $fg$ and $f$ belong to $\Xx$, also $g$ belongs to $\Xx$.} The top right square is a pullback by definition of the envelope.
    The rectangle consisting of the top middle and right squares
    is a pullback by the definition of $\cC^{\Ncoprod}$ and the fact that
    $\Ar(\cN) \to \Span_\cN(\cF)$ factors as $\Ar(\cN) \xto{s} \cN \to \Span_\cN(\cF)$.
    Finally, to see that the rectangle consisting of the entire top row
    is a pullback, recall that $\Span$ is a right adjoint from the category of adequate triples.
	Viewing $\cN_{/A} \simeq \Span_{\simeq,\all}(\cN_{/A})$ as span category with backwards maps only the equivalences,
	we compute that the pullback is the span category with forwards maps $\Un^\ct(\pi_A^*\Cc)$,
	and backwards morphisms the subcategory on cartesian morphisms lying over $\core(\cN_{/A})$.
	However a cartesian map lying over an equivalence is itself an equivalence,
	and thus the pullback has no non-trivial backwards maps.

	We conclude that the following diagram commutes:
    \[\begin{tikzcd}
        {\PreCat{\cF}} & {\PreOp{\cF}{\cN}} && {\PreNmCat{\cF}{\cN}} \\
        {\Fun((\cN_{/A})^\op,\Cat)} & {\Cat^\cart_{/(\cN_{/A})}} & {\Cat_{/(\cN_{/A})}} & \Cat
        \arrow["{{(-)^{\Ncoprod}}}", from=1-1, to=1-2]
        \arrow["{{\pi_A^*}}"', from=1-1, to=2-1]
        \arrow["\Env", from=1-2, to=1-4]
        \arrow["{{\ev_A}}", from=1-4, to=2-4]
        \arrow["\sim"', "\Un^\ct", from=2-1, to=2-2]
        \arrow["\fgt", from=2-2, to=2-3]
        \arrow[from=2-3, to=2-4]
    \end{tikzcd}\]
    Each functor in the bottom composite commutes with colimits: for the first two and the last one this is clear and for $\fgt$ this was shown in \cite[Corollary A.5]{ramzi-monoidal-grothendieck}. Thus, the top composite preserves colimits for all $A\in \cF$, and hence also $(-)^{\Ncoprod}$ does.

    Next, since $(-)^{\Ncoprod}$ preserves limits,
    it is symmetric monoidal for the cartesian symmetric monoidal structures.
    By \cref{rem:2cat-structure} the $\Cat$-modules structures on both source and target are similarly
    defined by symmetric monoidal functors from $\Cat$, and so it will suffice for $\Cat$-linearity to show that the diagram
    \[
        \begin{tikzcd}[column sep=small]
            & \Cat\arrow[dl, bend right=10pt, "\const"']\arrow[dr, bend left=10pt, "\const"]\\
            \PreCat{\Ff}\arrow[rr, "(-)^\Ncoprod"'] && \PreOp{\Ff}{\Nn}
        \end{tikzcd}
    \]
    commutes as a diagram of symmetric monoidal functors or, equivalently, as a diagram of ordinary functors. This follows at once from the equivalences
    \begin{align*}
		(\mathop\const K)^{\Ncoprod}
		&= \Span_{\cart,\cN}(\Un^\ct(\mathop\const K)) \simeq \Span(K_{\simeq,\all} \times (\Ff,\Ff,\Nn))\\
		&\simeq \Span_{\simeq,\all}(K)\times \Span_\Nn(\Ff) \simeq K \times \Span_\Nn(\Ff)=\mathop\const K,
	\end{align*}
	where $K_{\simeq,\all}$ denotes $K$ viewed as adequate triple with only forwards maps.

	On the other hand, $\Triv(-)$ is clearly a $\Cat$-linear left adjoint, since we may simply restrict the $\Cat$-linear adjunction
	\[
	    i_!\colon \Cat_{/\cF^{\op}}\rightleftarrows \Cat_{/\Span_{\cN}(\cF)}\noloc i^*
	\]
    induced by the inclusion $i\colon \cF^{\op} \to \Span_{\cN}(\cF)$ to an adjunction between $\PreCat{\cF}$ and $\PreOp{\cF}{\cN}$. This also shows that the right adjoint of $\Triv$ is $\fgt = i^*$, with the evident unit. That $\Triv(-)$ is fully faithful follows immediately from the fact that $\cF^\op \times_{\Span_\cN(\cF)}\cF^\op$ is equivalent to $\cF^\op$.
	Finally, in view of the pullback in \cref{lem:unfurl} defining the natural map $\incl\colon \Triv \Rightarrow (-)^{\Ncoprod}$, it follows immediately that $\incl$ is also a $\Cat$-linear transformation.
\end{proof}

\subsection{Parametrized categories of normed algebras} As an upshot of the above, we can now give various equivalent definitions of parametrized categories of parametrized algebras:

\begin{definition}
    We denote the right adjoint of the 2-functor $(-)^{\Ncoprod}$ by
    \[
        \und{\CAlg}_\cF^\cN \colon \PreOp{\cF}{\cN} \to \PreCat{\cF}.
    \]
    We call $\und{\CAlg}_\cF^\cN(\cO)$ the \emph{$\cF$-precategory of $\cN$-normed
    algebras} in $\cO$.
    By slight abuse of notation, we will use the same notation
    for the composition with the forgetful functor
    $\PreNmCat{\cF}{\cN} \to \PreOp{\cF}{\cN}$,
    which is right adjoint to the composite ${\Env}\big((-)^{\Ncoprod}\big)$.
\end{definition}

For $\cC \in \PreNmCat{\cF}{\cN}$ and $\cD \in \PreCat{\cF}$ we thus obtain natural equivalences
\[
    \Fun_\cF(\cD,\und{\CAlg}_\cF^\cN(\cC))
    \simeq \Fun_{\cF}^{\Nstr}(\Env(\cD^{\Ncoprod}), \cC).
\]
Similarly, we obtain an equivalence
\[
    \Fun_{\cF}(\cD,\fgt(\cC)) \simeq \Fun_{\cF}^{\Nstr}(\Env(\Triv(\cD)),\cC).
\]
Given an object $A\in \cF$, we we continue to let $\ul{A}$ denote the presheaf represented by $A$, considered as an object of $\PreCat{\cF}$. As recalled in \cite[Lemma 2.2.7]{CLL_Global}, there is an equivalence
\[
    \Fun_{\cF}(\ul{A},\cD) \simeq \cD(A)
\]
which is natural in both $A$ and $\cD$. Combining this with the equivalence above and with Example~\ref{ex:Triv-Ncoprod-rep}, we obtain natural equivalences
\[\hskip-16.24pt\hfuzz=20.25pt
\begin{aligned}
    \und{\CAlg}_{\cF}^{\cN}(\cC)(A)
    &\simeq \Fun_{\cF}^{\Nstr}(\Env(\ul{A}^{\Ncoprod}),\cC)
    \simeq \Fun_{\cF}^{\Nlax}(\ul A^\Ncoprod,\Cc)\simeq\Fun_{\cF}^{\Nlax}(\Span_\Nn(\Ff_{/A}),\Cc)
    \\
    \cC(A) &\simeq \Fun_{\cF}^{\Nstr}(\Env(\Triv(\ul{A})),\cC)
    \simeq\Fun_{\cF}^{\Nlax}(\Triv(\ul A),\Cc)\simeq
    \Fun_{\cF}^{\Nlax}((\Ff_{/A})^\op,\Cc)
\end{aligned}
\]

\begin{definition}
Precomposition by the transformation
\[\Env(\incl) \colon \Env(\Triv(-)) \Rightarrow \Env((-)^{\Ncoprod})\] induces for every $\cC\in \PreNmCat{\cF}{\cN}$ a functor
\[
\bbU\colon \und{\CAlg}_{\cF}^{\cN}(\cC) \to \cC
\]
of $\cF$-precategories, defined so that the following square commutes
\[\begin{tikzcd}
{\Fun_{\cF}^{\Nstr}(\Env(\und{A}^{\Ncoprod}), \cC)} & {\und{\CAlg}_\cF^\cN(\cC)(A)} \\
	{\Fun_\cF^{\Nstr}(\Env(\Triv(\und{A})),\cC)} & {\cC(A)}
	\arrow["\sim", from=1-1, to=1-2]
	\arrow["\Env(\incl)^*"', from=1-1, to=2-1]
	\arrow["{\mathbb{U}(A)}", from=1-2, to=2-2]
	\arrow["\sim", from=2-1, to=2-2]
	\arrow[ draw=none, from=2-1, to=2-2]
\end{tikzcd}\]
naturally in $A\in \cF$.
\end{definition}

\begin{definition}
    Given an $\cF$-precategory $\cC$, we refer to $\lim_{\cF^\op} \cC$ as the underlying category of $\cC$. Note that if $\cF$ admits a final object $1$, then $\lim_{\cF^\op} \cC \simeq \cC(1)$.
\end{definition}

\begin{lemma}
    There is an equivalence natural in $\cN$-normed $\cF$-precategories $\cC$
    \[
        \CAlg_\Ff^\Nn(\Cc)
        \coloneqq \lim_{\cF^\op}\und{\CAlg}_\cF^\cN(\cC)
        \simeq\Fun^{\Nlax}_\Ff(\Span_\Nn(\Ff),\Cc)
    \]
    which is compatible with the structure maps $\CAlg_\Ff^\Nn(\Cc)\to\ul\CAlg_\Ff^\Nn(\Cc)(A)$ of the defining limit.
\end{lemma}
\begin{proof}
    We want to show that the unique diagram
    \[
        \ul{(-)}^\Ncoprod\to\const(1^\Ncoprod)
    \]
    is a colimiting cocone in $\PreOp{\Ff}{\Nn}$, where $1$ denotes the final object of $\PreCat{\Ff}$, so that $1^\Ncoprod=\Span_\Nn(\Ff)$ is the final $(\Ff,\Nn)$-preoperad. As $(-)^\Ncoprod$ is a left adjoint, it suffices for this that the colimit of the $\ul A$'s in $\PreCat{\Ff}$ is terminal. As $\Fun(\Ff^\op,\Spc)\hookrightarrow\Fun(\Ff^\op,\Cat)=\PreCat{\Ff}$ is both a left and a right adjoint, this just amounts to the well known fact that the colimit of the Yoneda embedding is terminal.
\end{proof}

For later use, let us record two further descriptions of $\ul\CAlg^\Nn_\Ff$:

\begin{lemma}\label{lemma:CAlg-slice}
    Let $A\in\Ff$ be arbitrary and let $\pi\colon\Ff_{/A}\to\Ff$ denote the forgetful functor. Then we have an equivalence natural in $\cN$-normed $\cF$-precategories $\cC$
    \[
        \ul\CAlg^\Nn_{\Ff_{/A}}(\Span(\pi)^*\Cc)\coloneqq\ul\CAlg_{\Ff_{/A}}^{\Nn\times_\Ff\Ff_{/A}}(\Span(\pi)^*\Cc)\simeq\pi^*\ul\CAlg_{\Ff}^\Nn(\Cc).
    \]
    Passing to underlying categories gives $\CAlg_{\cF_{/A}}^\cN(\Span(\pi)^*\cC) \simeq \und{\CAlg}_\cF^\cN(\cC)(A)$.
    \begin{proof}
        We have a $\Cat$-linear adjunction
        \[
            \Span(\pi)_!\coloneqq\Span(\pi)\circ-\colon \Cat_{/\Span_\Nn(\Ff_{/A})}\rightleftarrows\Cat_{/\Span_\Nn(\Ff)}:\Span(\pi)^*.
        \]
        Using that $\Span(\pi)$ is itself cocartesian over $\Ff^\op$, one easily checks that this restricts to a $\Cat$-linear adjunction $\PreOp{\Ff_{/A}}{\Nn}\rightleftarrows\PreOp{\Ff}{\Nn}$. We therefore have
        \begin{align*}
            \ul\CAlg^\Nn_\Ff(\Cc)(\pi(B))&\simeq
            \Fun_\Ff^{\Nlax}(\Span_\Nn(\Ff_{/\pi(B)}),\Cc)\\
            &\simeq
            \Fun_\Ff^{\Nlax}(\Span(\pi)_!\Span_\Nn((\Ff_{/A})_{/B}),\Cc)\\&\simeq
            \Fun_{\Ff_{/A}}^{\Nlax}(\Span_\Nn((\Ff_{/A})_{/B}),\Span(\pi)^*\Cc)\\
            &\simeq
            \ul\CAlg^\Nn_{\Ff_{/A}}(\Span(\pi)^*\Cc)(B)
        \end{align*}
        naturally in $\Cc$ and in $B\in(\Ff_{/A})^\op$.
    \end{proof}
\end{lemma}

We call an object $A\in\Ff$ \emph{productive} if every $B\in\Ff$ admits a product $A\times B$ in $\Ff$.

\begin{lemma}\label{lemma:CAlg-shift-general}
    For every productive $A\in\Ff$, we have an equivalence
    \[
        \ul\CAlg^\Nn_\Ff(\Cc)(A)\simeq\CAlg^\Nn_\Ff(\Cc(A\times-))=\CAlg^\Nn_\Ff(\Span(A\times-)^*\Cc)
    \]
    natural in $\Cc\in\PreNmCat{\Ff}{\Nn}$. Moreover, this equivalence is also natural with respect to arbitrary maps between productive objects of $\Ff$,
    and is compatible with the forgetful functors to $\cC(A)$.
\begin{proof}
    The forgetful map $\pi\colon\Ff_{/A}\to\Ff$ has a right adjoint $\rho$, given by sending $B\in\Ff$ to $\pr\colon A\times B\to A$. In particular $\rho$ is cofinal, so that restriction along $\rho^\op$ induces an equivalence
    $\cC(A) = \lim_{(\cF_{/A})^\op} \pi^*\cC \simeq \lim_{\cF^\op} \cC(A \times -)$.

    We will now see that an analogous argument works for the span categories.
    Namely, a straightforward computation using \cite[Corollary C.21]{BachmannHoyois2021Norms} shows that $\rho$ induces a \emph{left} adjoint on spans, such that the units and counits are backwards maps. Moreover, the square
\[\begin{tikzcd}
	{\Span_{\cN}(\cF)} &[2em] {\Span_{\cN}(\cF_{/A})} \\
	{\Span_{\cN}(\cF)} & {\Span_{\cN}(\cF)}
	\arrow["{\Span(\rho)}", from=1-1, to=1-2]
	\arrow[from=1-1, to=2-1,equals]
	\arrow[from=1-2, to=2-2]
	\arrow["{\Span(A\times-)}"', from=2-1, to=2-2]
\end{tikzcd}\] commutes.
It follows from \cite{Linskens2023globalization}*{Proposition~4.16} that restriction along $\Span(\rho)$ induces an equivalence
\begin{multline*}
	\ul\CAlg_{\Ff}^\Nn(\Cc)(A)
	\simeq\Fun^{\Nlax}_{\cF}(\Span_{\cN}(\cF_{/A}), \Cc)\\
	\iso \Fun^{\Nlax}_{\cF}(\Span_{\cN}(\cF),\Cc(A\times-))
	\simeq\CAlg_\Ff^\Nn(\Cc(A\times-))
\end{multline*}
and one easily checks that this has the desired naturality properties.
\end{proof}
\end{lemma}

We will now make the $\cN$-normed $\cF$-precategories $\Env(\ul{A}^{\Ncoprod}), \Env(\Triv(\ul{A}))$ and the map
$\Env(\incl)\colon \Env(\Triv(\ul{A}))\to \Env(\ul{A}^{\Ncoprod})$ explicit.
To this end, it will be convenient to restrict our attention to the case
where $\cN \subset \cF$ is \emph{left cancellable}, as will be the case in all examples of interest to us.

\begin{definition}
    A span pair $(\cF,\cN)$ is called \emph{left cancellable} if it satisfies the following property: given any composable maps $f,g$ in $\cF$ such that $fg$ and $f$ belong to $\Nn$, also $g$ belongs to $\Nn$.
\end{definition}

For the remainder of this subsection, we fix a left cancellable span pair $(\cF,\cN)$.

\begin{construction}\label{constr:n-a}
    The target projection $t \colon \Ar_\cN(\cF) \to \cF$ is a cartesian fibration
    classifying the functoriality of the slices $\cN_{/B}$ in pullbacks along arbitrary maps $A \to B$ in $\cF$. We denote this cartesian straightening by $\cN_{/-} \colon \cF^\op \to \Cat$.
    Now let $A \in \cF$ and consider the pullback on the left:
    \[\begin{tikzcd}[cramped]
        {\Ar_\cN(\cF) \times_\cF \cF_{/A}} & {\cF_{/A}} &&& A \\
        {\Ar_\cN(\cF)} & \cF && X & Y \\
        \cF &&& C & D
        \arrow[from=1-1, to=1-2]
        \arrow["p"', from=1-1, to=2-1]
        \arrow["\lrcorner"{anchor=center, pos=0.125}, draw=none, from=1-1, to=2-2]
        \arrow[from=1-2, to=2-2]
        \arrow["s", from=2-1, to=2-2]
        \arrow["t"', from=2-1, to=3-1]
        \arrow[from=2-4, to=1-5]
        \arrow[from=2-4, to=2-5]
        \arrow[norm, from=2-4, to=3-4]
        \arrow[from=2-5, to=1-5]
        \arrow[norm, from=2-5, to=3-5]
        \arrow[from=3-4, to=3-5]
    \end{tikzcd}\]
    Clearly $tp$ is a cartesian fibration, and a general morphism in it can be depicted
    as the diagram on the right. It is $tp$-cartesian precisely if the square is a pullback.
    We denote the cartesian straightening of $tp$ by
    \[
        \cF_{/A} \times_\cF \cN_{/-} \coloneqq \Stct\left(\Ar_\cN(\cF) \times_\cF \cF_{/A}\right) \colon \cF^\op \to \Cat
    \]

    In fact, $\cF_{/A} \times_\cF \cN_{/-} \in \PreCat{\cF}$
    is covariantly functorial in $A$ via postcomposition;
    this follows immediately from the fact that given $f \colon A \to B$ in $\cF$,
    the functor $f_! \colon \cF_{/A} \to \cF_{/B}$ over $\cF$ is a map of cartesian fibrations.
    We therefore obtain a functor
    \[
        \cF_{/\bullet} \times_\cF \cN_{/-} \colon \cF \to \PreCat{\cF},\ A \mapsto \cF_{/A} \times_\cF \cN_{/-}.
    \]
\end{construction}

\begin{lemma}\label{lem:n-a}
    Let $(\cF,\cN)$ be a left cancellable span pair and $\cF_{/\bullet} \times_\cF \cN_{/-}$ as constructed
    above. Then:
    \begin{enumerate}
        \item For $X \in \cF$, the $\cF$-precategory $\cF_{/X} \times_\cF \cN_{/-}$ is left $\cN$-adjointable
            in the sense of \cref{lem:unfurl}(3).
            We denote the resulting $\cN$-normed $\cF$-precategory by
            \[
                \und{\cN}^X \coloneqq (\cF_{/X} \times_\cF \cN_{/-})^{\Ncoprod}.
            \]
            Note that by construction, $\und{\cN}^X|_{\cF^\op} = \cF_{/X} \times_\cF \cN_{/-}$.

        \item The induced functor
            \[
                \und{\cN}^{(\bullet)} \colon \cF \to \PreNmCat{\cF}{\cN}
            \]
            is right $\cN$-coadjointable in the sense of \cite[Variant 2.2.10, Remark 2.2.11]{Elmanto_Haugseng_Bispans}.
            Concretely, this means that for any morphism $n \colon A \rightarrownorm B$
            the $\cN$-normed $\cF$-functor
            \[
                n_! \colon \und{\cN}^{A} \to \und{\cN}^B
            \]
            induced by postcomposition admits a right adjoint $n^*$ (given by pulling back), and every pullback square of the form
            \begin{equation}\label{diag:N-pb-repeat}
                \begin{tikzcd}
                    A\arrow[d,"f'"']\arrow[dr,pullback]\arrow[r, norm, "n'"] & B\arrow[d, "f"]\\
                    C\arrow[r, norm, "n"'] & D
                \end{tikzcd}
            \end{equation}
            is sent to a right-adjointable square, so that the following
            square commutes via the Beck--Chevalley transformation
            \[\begin{tikzcd}
                {\und{\cN}^A} & {\und{\cN}^B} \\
                {\und{\cN}^C} & {\und{\cN}^D}
                \arrow["{f'_!}"', from=1-1, to=2-1]
                \arrow["{n'^*}"', from=1-2, to=1-1]
                \arrow["{f_!}", from=1-2, to=2-2]
                \arrow["\sim"{description}, Rightarrow, from=2-1, to=1-2,shorten=5pt]
                \arrow["{n^*}", from=2-2, to=2-1]
            \end{tikzcd}
            \]
    \end{enumerate}
\end{lemma}
\begin{proof}
    Let $X \in \cF$, and consider a map $n \colon A \rightarrownorm B$ in $\cN$.
    Clearly the pullback $n^* \colon \cN_{/B} \to \cN_{/A}$
    admits a left adjoint given by postcomposition with $n$,
    and similarly one checks that we obtain an adjunction
    \[
        n_! \colon \cF_{/X} \times_\cF \cN_{/A} \to \cF_{/X} \times_\cF \cN_{/B} \noloc n^*.
    \]
    For example, the unit map at some $X \leftarrow Y \rightarrownorm A$ is given by
    the map $\phi$ in the following commutative diagram
    \[\begin{tikzcd}[cramped]
        X & Y & {Y \times_B A} & Y \\
        & A & {A \times_B A} & A \\
        & B & A
        \arrow[from=1-2, to=1-1]
        \arrow[norm, from=1-2, to=2-2]
        \arrow[norm, from=1-3, to=1-2]
        \arrow["\lrcorner"{anchor=center, pos=0.125, rotate=-90}, draw=none, from=1-3, to=2-2]
        \arrow[norm, from=1-3, to=2-3]
        \arrow["\phi"', from=1-4, to=1-3]
        \arrow["\lrcorner"{anchor=center, pos=0.125, rotate=-90}, draw=none, from=1-4, to=2-3]
        \arrow[norm, from=1-4, to=2-4]
        \arrow["n"',norm, from=2-2, to=3-2]
        \arrow[norm, from=2-3, to=2-2]
        \arrow["\lrcorner"{anchor=center, pos=0.125, rotate=-90}, draw=none, from=2-3, to=3-2]
        \arrow[norm, from=2-3, to=3-3]
        \arrow[from=2-4, to=2-3]
        \arrow["n",norm, from=3-3, to=3-2]
    \end{tikzcd}\]
    Note that $\phi$ is indeed a map in $\cN$ by left cancellability,
    and this is one of the crucial points where this assumption is needed.

    The fact that the pullback square (\ref{diag:N-pb-repeat}) induces a left-adjointable square,
    i.e.~that the induced square
    \begin{equation}\label{diag:n-a}\begin{tikzcd}[cramped]
        {\cF_{/X} \times_\cF \cN_{/A}} & {\cF_{/X} \times_\cF \cN_{/B}} \\
        {\cF_{/X} \times_\cF \cN_{/C}} & {\cF_{/X} \times_\cF \cN_{/D}}
        \arrow["{n'^*}"', from=1-2, to=1-1]
        \arrow["{f'^*}", from=2-1, to=1-1]
        \arrow["{f^*}"', from=2-2, to=1-2]
        \arrow["{n^*}", from=2-2, to=2-1]
    \end{tikzcd}\end{equation}
    is horizontally left adjointable, is a simple but tedious exercise in pullback pasting.

    For the second claim, we have to show that $\und{\cN}^{(\bullet)} \coloneqq (-)^{\Ncoprod} \circ (\cF_{/\bullet} \times_\cF \cN_{/-})$ is right $\cN$-coadjointable.
    Since 2-functors preserve adjointable squares (see e.g.~\cite[Lemma D.2.7]{Heyer-Mann}),
    it suffices to show that $\cF_{/\bullet} \times_\cF \cN_{/-} \colon \cF \to \PreCat{\cF}$
    is right $\cN$-coadjointable.

    So fix $n \colon A \rightarrownorm B$ in $\cN$. We need to show that the $\cF$-functor
    $n_! \colon \cF_{/A} \times_\cF \cN_{/-} \to \cF_{/B} \times_\cF \cN_{/-}$
    induced by postcomposing with $n$ admits a right adjoint. For some fixed $X \in \cF$, one checks that pullback along induces a right adjoint
    \[
        n^*(X) \colon \cF_{/B} \times_\cF \cN_{/X} \to \cF_{/A} \times_\cF \cN_{/X}
    \]
    to $n_!(X)$. By \Cref{rem:adj_in_functor_cat} the condition for these to assemble into a right adjoint to the $\cF$-functor
    $n_!$ is that for each $f \colon X \to Y$ the square
    \[\begin{tikzcd}[cramped]
        {\cF_{/A} \times_\cF \cN_{/X}} & {\cF_{/B} \times_\cF \cN_{/Y}} \\
        {\cF_{/B} \times_\cF \cN_{/X}} & {\cF_{/B} \times_\cF \cN_{/Y}}
        \arrow["{n_!(X)}"', from=1-1, to=2-1]
        \arrow["{f^*}"', from=1-2, to=1-1]
        \arrow["{n_!(Y)}", from=1-2, to=2-2]
        \arrow["{f^*}", from=2-2, to=2-1]
    \end{tikzcd}\]
    is vertically right adjointable, which is again a simple exercise in pullback pasting.

    Finally, it remains to see that a pullback square as in (\ref{diag:N-pb-repeat})
    \[\begin{tikzcd}[cramped]
        {\cF_{/A} \times_\cF \cN_{/-}} & {\cF_{/C} \times_\cF \cN_{/-}} \\
        {\cF_{/B} \times_\cF \cN_{/-}} & {\cF_{/D} \times_\cF \cN_{/-}}
        \arrow["{n'_!}", from=1-1, to=1-2]
        \arrow["{f'_!}"', from=1-1, to=2-1]
        \arrow["{f_!}", from=1-2, to=2-2]
        \arrow["{n'_!}"', from=2-1, to=2-2]
    \end{tikzcd}\]
    which is horizontally right adjointable, i.e.~that the Beck--Chevalley transformation
    $f'_!n'^* \Rightarrow n^*f_!$ of $\cF$-functors is an equivalence.
    This can be checked pointwise at some $X \in \cF$,
    where the argument is entirely analogous to the horizontal left adjointability
    of (\ref{diag:n-a}) considered above.
\end{proof}

\begin{proposition}\label{prop:fna}
    Let $(\cF,\cN)$ be a left-cancellable span pair. There is a natural equivalence
    \[
        \und{\cN}^{A}\simeq \Env(\ul{A}^{\Ncoprod}).
    \]
    Moreover, under this identification the map $\Env(\incl)\colon \Env(\Triv(\und{A}))\to \Env(\und{A}^{\Ncoprod})$ is given by the inclusion $\iota \und{\cN}^A\hookrightarrow \und{\cN}^A$ of the pointwise groupoid core.
\end{proposition}
\begin{proof}
    Recall from Example~\ref{ex:Triv-Ncoprod-rep} that $\und{A}^\Ncoprod$ is naturally equivalent to $\Span_{\cN}(\cF_{/A})\to \Span_{\cN}(\cF)$, so that its envelope is then given by
    \[
        \Env(\und{A}^{\Ncoprod})
        \simeq \Ar_\cN(\Span_\cN(\cF)) \times_{\Span_\cN(\cF)} \Span_\cN(\cF_{/A}).
    \]
    Identifying $\Ar_\cN(\Span_\cN(\cF))$ as in Remark~\ref{rem:cocart_in_Env_1} and appealing to \cref{thm:spans}, we see that ${\Env(\und{A}^\Ncoprod)}$ is naturally equivalent to the category
    \begin{equation}\label{eq:unstr_of_N_A}
        \Span_{\mathrm{pb},\cN}(\Ar_{\cN}(\cF)\times_{\cF}\cF_{/A}),
    \end{equation}
    whose morphisms are spans of the form
    \begin{equation}\label{diag:map-in-env-ncoprod-a}
        \begin{tikzcd}
        A & A & A \\
        X & Y & Z \\
        B & C & D\rlap.
        \arrow[equal, from=1-2, to=1-1]
        \arrow[equal, from=1-2, to=1-3]
        \arrow[from=2-1, to=1-1]
        \arrow[norm, from=2-1, to=3-1]
        \arrow[from=2-2, to=1-2]
        \arrow[from=2-2, to=2-1]
        \arrow[norm, from=2-2, to=2-3]
        \arrow["\lrcorner"{anchor=center, pos=0.125, rotate=-90}, draw=none, from=2-2, to=3-1]
        \arrow[norm, from=2-2, to=3-2]
        \arrow[from=2-3, to=1-3]
        \arrow[norm, from=2-3, to=3-3]
        \arrow[from=3-2, to=3-1]
        \arrow[norm, from=3-2, to=3-3]
        \end{tikzcd}
    \end{equation}
    On the other hand, (\ref{eq:unstr_of_N_A}) also precisely agrees with
    $\und{\cN}^A \coloneqq (\cF_{/A} \times_\cF \cN_{/-})^\Ncoprod$ by definition,
    see \cref{def:n-cocart,constr:n-a}.

    By \Cref{lem:unfurl}(2) and the definition of the envelope, we obtain pullback squares
    \[\begin{tikzcd}
        {\Env(\Triv(\und{A}))} & {\Env(\und{A}^\Ncoprod)} & {\Ar_{\cN}(\Span_{\cN}(\cF))} \\
        {(\cF_{/A})^{\op}} & {\Span_{\cN}(\cF_{/A})} & {\Span_{\cN}(\cF)}
        \arrow[hook, from=1-1, to=1-2]
        \arrow[from=1-1, to=2-1]
        \arrow[from=1-2, to=1-3]
        \arrow[from=1-2, to=2-2]
        \arrow["{s}",from=1-3, to=2-3]
        \arrow[""{name=0, anchor=center, inner sep=0}, hook, from=2-1, to=2-2]
        \arrow[""{name=1, anchor=center, inner sep=0}, from=2-2, to=2-3]
        \arrow["\lrcorner"{anchor=center, pos=0.125}, draw=none, from=1-1, to=0]
        \arrow["\lrcorner"{anchor=center, pos=0.125}, draw=none, from=1-2, to=1]
    \end{tikzcd}\]
    Note that a general morphism (\ref{diag:map-in-env-ncoprod-a})
in $\Env(\und{A}^\Ncoprod)$ is cocartesian precisely if the map $Y \rightarrownorm Z$ is an equivalence (cf.~the description of cocartesian morphisms in $\Env(1)= \Ar_\cN(\Span_\cN(\cF))$ from \cref{rem:cocart_in_Env_1}).
    Thus, $\Env(\Triv(\und{A})) \subset \Env(\und{A}^{\Ncoprod})$ is precisely the subcategory on the cocartesian morphisms.
    The identification of $\Env(\incl)$ therefore follows from the general fact that for any functor $F \colon \cC \to \Cat$ with cocartesian straightening $p \colon \cX \to \cC$, the inclusion $\cX_\cocart \hookrightarrow \cX$ of the subcategory of cocartesian morphisms corresponds under cocartesian straightening to the natural transformation $\core F \Rightarrow F$.
\end{proof}

\begin{corollary}\label{cor:forget-as-restrict}
    There is a commutative square, natural in $\cC \in \PreNmCat{\cF}{\cN}$ and $A \in \cF^\op$,
\[\begin{tikzcd}
	{\Fun_{\cF}^{\Nstr}(\und{\cN}^A, \cC)} &[1em] {\und{\CAlg}_\cF^\cN(\cC)(A)} \\
	{\Fun_\cF^{\Nstr}(\core \und{\cN}^A,\cC)} & {\cC(A),}
	\arrow["\sim", from=1-1, to=1-2]
    \arrow["{(\core \ul{\cN}^A \hookrightarrow \ul{\cN}^A)^*}"', from=1-1, to=2-1]
	\arrow["{\mathbb{U}}", from=1-2, to=2-2]
	\arrow["\sim", from=2-1, to=2-2]
	\arrow["{\ev_{(\id_A,\id_A)}}"', draw=none, from=2-1, to=2-2]
\end{tikzcd}\]
where the bottom horizontal equivalence evaluates at $(\id_A,\id_A) \in \core (\cF_{/A} \times_\cF \cN_{/A}) = \core \und{\cN}^A(A)$.
\end{corollary}
\begin{proof}
After \Cref{prop:fna}, the only non-trivial fact is the identification of the bottom arrow. However, combining the fact that $\Triv$ is fully faithful and that the natural Yoneda equivalence $\Fun_\cF(\und{A},\cC) \simeq \cC(A)$ is given by evaluating at $\id_A\in \ul{A}(A)$, we see that the bottom horizontal equivalence in the square is given by the composite
\[
    \Fun_\cF^{\Nstr}(\core \und{\cN}^A, \cC)
    \xto{\eta^*} \Fun_{\cF}(\und{A},\cC)
    \xto{\ev_{\id_A}} \cC(A),
\]
so the claim follows from the description of the unit $\eta$ of \cref{thm:BHS}.
\end{proof}

\begin{remark}\label{rem:free-normed}
	We may summarize the previous result as follows. The top equivalence of the commutative square identifies $\und{\cN}^A$ as the free $\cN$-normed $\cF$-precategory with an $\cN$-normed algebra in degree $A$, while the bottom identifies its groupoid core $\core \und{\cN}^A$ as the free $\cN$-normed $\cF$-precategory with an object in degree $A$. The commutativity of the square shows that the object $(\id_A,\id_A) \in \und{\cN}^{A}(A)$ is the universal $\cN$-normed $\cF$-algebra in degree $A$.
\end{remark}

\begin{corollary}\label{cor:iota_N_rep}
	The functor of $\cN$-normed $\cF$-precategories
	\[
	\Map_{\Span_{\cN}(\cF)}(A,-) \to \core \und{\cN}^A
	\]
	classified by $(\id_A,\id_A) \in \core (\cF_{/A} \times_\cF \cN_{/A}) \eqqcolon \core \und{\cN}^A(A)$ is an equivalence. If $A$ is terminal, it is the unique such equivalence.
\end{corollary}
\begin{proof}
	The claimed equivalence follows immediately from the identification of the equivalence
	\[
	{\Fun_\cF^{\Nstr}(\core \und{\cN}^A,\cC)}\to {\cC(A)}
	\] with $\ev_{\id_A,\id_A}$ from \cref{cor:forget-as-restrict}. In the case that $A$ is terminal,
	we see that it does not have any automorphisms in $\Span_\cN(\cF)$,
	hence by the Yoneda lemma also $\Map_{\Span_\cN(\cF)}(A,-)$ does
	not have any automorphisms, which yields the uniqueness of the equivalence.
\end{proof}

\begin{example}
    In the case of classical symmetric monoidal categories, which as mentioned in \Cref{rem:classical_op} is recovered by the special case $(\cF,\cN) = (\F,\F)$, the above results exhibit the cocartesian symmetric monoidal category of finite sets $\F^\amalg$ as the free symmetric monoidal category
    with a commutative algebra object, and its groupoid core $\core \F^\amalg$
    (which is the commutative monoid $\coprod_{n \geq 0} B\Sigma_n$)
    as the free symmetric monoidal category on one object.
    In both cases, the universal object is given by the one-point set.
\end{example}

We close this discussion by recording some convenient properties of the $\Ff$-precategories of the form $\ul\CAlg^\Nn_\Ff(\Cc)$.

\begin{definition}\label{def:fw-colims}
    Let $\cC$ be an $\cF$-precategory.
    \begin{enumerate}
        \item For a small category $K$, we say that $\cC$ admits \emph{fiberwise $K$-colimits} if $\Cc(A)$ has $K$-colimits for every $A\in \cF$, and for every $f\colon A\to B$ in $\cF$ the functor $f^* \colon \Cc(B)\to\Cc(A)$ preserves $K$-colimits.

        \item If $\Kk$ is a class of small categories, then we say that $\Cc$ has \emph{fiberwise $\Kk$-colimits} if it has fiberwise $K$-colimits for every $K\in\Kk$.

        \item An $\cF$-functor $F \colon \cC \to \cD$ is said to preserve $\cK$-colimits
            if each $F(A) \colon \cC(A) \to \cD(A)$ does.
    \end{enumerate}
    All of these definitions apply analogously to $\cN$-normed $\cF$-precategories $\cC$,
    e.g.~by viewing them as $\Span_\cN(\cF)^\op$-precategories.
\end{definition}

\begin{remark}\label{rem:fiberwise-colim-cotensor}
    Let $\cC$ be an $\cF$-precategory and $K$ an ordinary category.
    Using the internal hom $\und{\Fun}_\cF$ and the $\Cat$-tensoring of $\PreCat{\cF}$,
    we see that it is also cotensored via
    \[
        \cC^K = \und{\Fun}_\cF(\mathop{\const} K, \cC) \simeq \Fun(K,-) \circ \cC
    \]
    where the second equivalence is an instance of the natural equivalences
    \begin{align*}
        \und{\Fun}_\cF(\mathop{\const} K,\cC)(A)
        &\simeq \Fun_\cF(\mathop{\const} K, \und{\Fun}_\cF(\und{A},\cC))\\
        &\simeq \Fun(K, \Fun_\cF(\und{A},\cC))\\
        &\simeq \Fun(K,\cC(A)).
    \end{align*}
    In particular, it follows that $\cC$ admits fiberwise $K$-colimits
    precisely if the diagonal functor $\cC \to \cC^K$ admits a left adjoint,
    cf.~\Cref{rem:adj_in_functor_cat}.
    Similarly, a functor $F \colon \cC \to \cD$ preserves $K$-colimits precisely
    if the following square is vertically left adjointable:
    \[\begin{tikzcd}
        {\cC^K} & {\cD^K} \\
        \cC & \cD\rlap.
        \arrow["{F_*}", from=1-1, to=1-2]
        \arrow[from=2-1, to=1-1]
        \arrow["F", from=2-1, to=2-2]
        \arrow[from=2-2, to=1-2]
    \end{tikzcd}\]
    Again, everything applies verbatim to $\cN$-normed $\cF$-precategories.
\end{remark}

\begin{lemma}\label{lem:calg-fib-colims}
    Let $\cC \colon \Span_\cN(\cF) \to \Cat$ be an $\cN$-normed $\cF$-precategory.
    \begin{enumerate}
        \item The forgetful functor $\bbU \colon \und{\CAlg}_\cF^\cN(\cC)(A) \to \cC(A)$ is conservative.
        \item If $\cK$ is a class of small categories and $\cC$ admits fiberwise $\cK$-colimits,
            then also $\und{\CAlg}_\cF^\cN(\cC)$ admits fiberwise $\cK$-colimits,
            and $\bbU$ preserves them.
    \end{enumerate}
\end{lemma}
\begin{proof}
    The first claim is immediate from the fact that $\core \und{\cN}^A \hookrightarrow \und{\cN}^A$
    is essentially surjective, and $\bbU$ is given by restricting along this.
    For the second claim, recall that $\und{\CAlg}_\cN^\cF$ is the right adjoint
    in a $\Cat$-linear adjunction. In particular, it preserves the $\Cat$-cotensoring.
    Since it is also a 2-functor, it therefore sends the adjunction
    \[
        \colim_K \colon \cC^K \rightleftarrows \cC \noloc \diag
    \]
    witnessing that $\cC$ has fiberwise $K$-colimits for $K \in \cK$ to the analogous adjunction
    \[
        \colim_K \colon \und{\CAlg}_\cF^\cN(\cC)^K \rightleftarrows \und{\CAlg}_\cF^\cN(\cC) \noloc \diag
    \]
    which then shows that also $\und{\CAlg}_\cF^\cN(\cC)$ admits fiberwise $K$-colimits.
\end{proof}

We will later identify reasonable conditions on $\cC$ so that this forgetful functor is monadic, see \cref{cor:par_U_monadic}.

\begin{corollary}\label{cor:calg-coprod}
    Let $\cC$ be an $\cN$-normed $\cF$-precategory.
    Then $\und{\CAlg}_\cF^\cN(\cC)$ is left $\cN$-adjointable
    in the sense of \cref{lem:unfurl}(3),
    and thus defines an $\cN$-normed $\cF$-precategory
    \[
        \und{\CAlg}_\cF^\cN(\cC)^{\Ncoprod} \in \PreNmCat{\cF}{\cN}.
    \]
\end{corollary}
\begin{proof}
    This follows immediately from the fact that $\Fun_\cF^\cN(-,\cC) \colon \PreNmCat{\cF}{\cN}^\op \to \Cat$ is a 2-functor (the mapping-functor of the 2-category
    $\PreNmCat{\cF}{\cN}$), the identification $\und{\CAlg}_\cF^\cN(A) \simeq \Fun_\cF^{\cN}(\und{\cN}^A,\cC)$ from \cref{cor:forget-as-restrict},
    and the fact that $\und{\cN}^{(-)} \colon \cF \to \PreNmCat{\cF}{\cN}$
    is right $\cN$-coadjointable by \cref{lem:n-a}(3).
\end{proof}

\subsection{Parametrized operads}\label{subsec:operads}
In this section we will single out the analogue of operads amongst the $(\cF,\cN)$-preoperads. For later use we will introduce the definition in the generality of an arbitrary factorization system.

\begin{definition}\label{def:operad}
    Let $(S,E,M)$ be a factorization system such that $E$ and $S$ have finite products and the inclusion $E\hookrightarrow S$ preserves them. We call a functor $p\colon \cX\to S$ an \emph{$(S,E,M)$-operad} if
    \begin{enumerate}
        \item $p$ is $E$-cocartesian, i.e.~admits cocartesian lifts for maps in $E$,
        \item $\cX$ has finite products,
        \item $p$ preserves finite products, \emph{and}
        \item for all $X,Y\in\cX$ the projections $X\gets X\times Y\to Y$ are $p$-cocartesian.
    \end{enumerate}
    We define $\Op{S,E}{M}\subset \Cat^{\Lcocart}_{/S}$ as the full subcategory
    spanned by the $(S,E,M)$-operads.
\end{definition}

\begin{example}
    The above definition is inspired by the example where $(S,E,M) = (\Span(\F),\F^\op,\F)$ is the canonical
    factorization system on the category of spans of finite sets,
    where we will refer to the subcategory $\F^\op \subset \Span(\F)$ as the \emph{inert} morphisms.
    By work of Barkan--Haugseng--Steinebrunner, the resulting category
    of $(S,E,M)$-operads is equivalent to the usual category $\mathrm{Op}$ of operads as defined by Lurie. More specifically, it follows from \cite[Theorem 2.1.11, Proposition 2.2.6, Proposition 2.2.10]{Haugseng-Lawvere}
    that the category $\mathrm{Op}$ is equivalent to the subcategory $\mathrm{Op}(\Span(\F)) \subset \Cat_{/\Span(\F)}$
    on objects $p \colon \cO \to \Span(\F)$ satisfying
    \begin{enumerate}
        \item[$(1')$] $p$ is has cocartesian lifts of inert (i.e.~backwards)morphisms.
        \item[$(2')$] For $X \in \cO$ over $n$ in $\Span(\F)$, the $p$-cocartesian morphisms $X \to X_i$
            over the spans $\rho_i = (n \xleftarrow{i} * = *)$ exhibit $X$ as the product $\prod_{i=1}^n X_i$ in $\cO$.
        \item[$(3')$] The functor $\cO_n \to \prod_{i=1}^n \cO_1$ given by cocartesian transport over the maps $\rho_i$,
            is essentially surjective.
    \end{enumerate}
    and morphisms over $\Span(\F)$ which preserve finite products.

    Note that $(1')$ is literally the same as condition $(1)$. We claim that given this, the conditions $(2')$ and $(3')$ are equivalent to conditions $(2)$--$(4)$ of \cref{def:operad}.
    In particular, together with the above results we thus have an equivalence
    \[
        \mathrm{Op} \simeq \Op{\Span(\F),\F^\op}{\F}.
    \]
    To see the claim, fix $p \colon \cO \to \Span(\F)$ which admits cocartesian lifts of inerts.
    If $p$ also satisfies $(2')$ and $(3')$, then it will satisfy $(2)$--$(4)$ by \cite[Lemma 2.2.7]{Haugseng-Lawvere}.
    Conversely, suppose that $p$ satisfies $(2)$--$(4)$, i.e.~is a $(\Span(\F),\F^\op,\F)$-operad.
    Then it also satisfies $(2')$ and $(3')$ by the following lemma.
\end{example}

\begin{lemma}\label{lem:tem-operad}
    Let $p \colon \cX \to S$ be an $(S,E,M)$-operad.
    \begin{enumerate}
        \item Let $A,B \in S$. Then the cocartesian transport along the projections induces an equivalence
            $\cX_{A \times B} \iso \cX_A \times \cX_B$ with inverse given by taking the product in $\cX$.
            Moreover, $\cX_1$ is the terminal category.
        \item Suppose $f \colon X \to Y$ and $g \colon X \to Z$ are edges in $\cX$
            lying over $E$.
            Then $(f,g) \colon X \to  Y \times Z$ is cocartesian if and only if both $f$ and $g$ are cocartesian.

    \end{enumerate}
\end{lemma}
\begin{proof}
    It is easy to see from the assumptions on $p$ that the supposed inverse
    $\cX_A \times \cX_B \to \cX \times \cX \xto{\times} \cX$
    factors through $\cX_{A \times B}$ and defines a section of the cocartesian transport
    $\Phi \colon \cX_{A \times B} \to \cX_A \times \cX_B$.
    In particular, $\Phi$ is essentially surjective.
    To see that $\Phi$ is also fully faithful, consider $X,Y \in \cX_{A \times B}$.
    We have the following two commutative diagrams
    \[\hskip-47.17pt\hfuzz=47.18pt\begin{tikzcd}[cramped]
        &[-1em] {\cX_A(\pr_{A!}X,\pr_{A!}Y) \times \cX_B(\pr_{B!}X,\pr_{B_!}Y)} &[-1em] {\cX_A(\pr_{A!}X,\pr_{A!}Y)} &[1em] {\cX(X,\pr_{A!}Y)_{\pr_A}} \\
        {\cX_{A \times B}(X,Y)} & {\cX(X,\pr_{A!}Y)_{\pr_A} \times\cX(X,\pr_{B!}Y)_{\pr_B}} & {\cX(\pr_{A!}X,\pr_{A!}Y)} & {\cX(X,\pr_{A!}Y)} \\
        {\cX(X,Y)} & {\cX(X,\pr_{A!}Y)\times \cX(X,\pr_{B!}Y)} & {S(A,A)} & {S(A\times B,A)}
        \arrow["\sim", from=1-2, to=2-2]
        \arrow[from=1-2, to=2-2]
        \arrow["\sim", from=1-3, to=1-4]
        \arrow[from=1-3, to=2-3]
        \arrow[from=1-4, to=2-4]
        \arrow["\Phi"{description}, from=2-1, to=1-2]
        \arrow[from=2-1, to=2-2]
        \arrow[from=2-1, to=3-1]
        \arrow[from=2-2, to=3-2]
        \arrow["{(X \to\pr_{A!}X)^*}", from=2-3, to=2-4]
        \arrow[from=2-3, to=3-3]
        \arrow[from=2-4, to=3-4]
        \arrow[from=3-1, to=3-2]
        \arrow[""{name=0, anchor=center, inner sep=0}, "{\pr_A^*}", from=3-3, to=3-4]
        \arrow[pullback, from=2-3, to=0]
    \end{tikzcd}\]
    The bottom square on the right is a pullback since $X \to \pr_{A!}X$ is cocartesian.
    The vertical maps are fiber sequences, which give the top equivalence,
    which is also the left factor in the vertical equivalence in the left diagram,
    the other factor coming from the analogous diagram for $B$.
    Also the vertical maps in the left square are fiber inclusions.
    By assumption on $p$, we know that $Y \iso \pr_{A!}Y \times \pr_{B!}Y$,
    so that the bottom and hence also the top map in that square is an equivalence.
    By 2-out-of-3, we then conclude that $\Phi$ is fully faithful.

    It remains to show that the fiber over the terminal object $1 \in S$ is contractible.
    By assumption, $\cX$ has a terminal object $\textbf1$, which is contained in the fiber over 1,
    where it is in particular terminal again. It therefore suffices to show that all objects of $\cX_1$
    are equivalent (say, in $\cX$). So let $X \in \cX_1$ and pick a product $X \times \textbf1$ in $\cX$.
    By assumption, $p$ sends this to $1 \times 1 \simeq 1$, hence both projections lie over equivalences,
    so are themselves equivalences, giving $X \simeq X \times \textbf1 \simeq \textbf1$.

    For the second claim in this lemma, suppose $f \colon Y \to Z$ and $g \colon X \to Z$ are edges in $\cX$.
    If $(f,g) \colon X \to Y \times Z$ is cocartesian, then clearly $f$ and $g$ are too
    since by assumption on $p$ projections are cocartesian.
    For the converse, we obtain the following commutative diagram,
    where maps labelled `$\co$' are cocartesian (and exist by assumption), and vertical maps are fiberwise
    \[\begin{tikzcd}[cramped]
        X & {(pf,pg)_!X} & {\pr_{Y!}(pf,pg)_!X}\rlap. \\
        & {Y \times Z} & Y
        \arrow["\co", from=1-1, to=1-2]
        \arrow["{(f,g)}"', from=1-1, to=2-2]
        \arrow["\co", from=1-2, to=1-3]
        \arrow[from=1-2, to=2-2]
        \arrow[from=1-3, to=2-3]
        \arrow["{\pr_Y}"', from=2-2, to=2-3]
        \arrow["\co", draw=none, from=2-2, to=2-3]
    \end{tikzcd}\]
    Since cocartesian maps compose, we see that the top right corner is actually $(pf)_!X$,
    and hence the right vertical map is an equivalence since $f$ is cocartesian.
    But this map is the image of the middle vertical map under $(\pr_{pY})_! \colon \cX_{pY \times pZ} \to \cX_{pY}$.
    The same argument shows that its image under $(\pr_{pZ})_!$ is an equivalence,
    and since these functors induce an equivalence $\cX_{pY \times pZ} \to \cX_{pY} \times \cX_{pZ}$,
    it follows that the middle vertical map is already an equivalence,
    and hence $(f,g)$ is cocartesian.

\end{proof}

\begin{corollary}\label{cor:product_unstr}
	Let $S$ be a category with finite products and let $F\colon S\to\Cat$ be a functor with cocartesian unstraightening $p\colon\Un^{\co}(F)\to S$. Then the following are equivalent:
	\begin{enumerate}
		\item $F$ preserves finite products.
		\item $\Un^{\co}(F)$ has finite products, the functor $p$ preserves them, and for every $X,Y\in \Un^{\co}(F)$ the projections $X\gets X\times Y\to Y$ are $p$-cocartesian.
		\item $\Un^{\co}(F)$ has finite products, the functor $p$ preserves them, and an edge $X\to Y\times Z$ is cocartesian if and only if the projections $X\to Y$ and $X\to Z$ are.
		\end{enumerate}
\end{corollary}
\begin{proof}
    Note that (2) implies that $p$ is an $(S,S,\core S)$-operad.
    Hence the implications $(2) \Rightarrow (1),(3)$ follow from the above lemma.
    Moreover, the implication (3) $\Rightarrow$ (2) is clear since the identity on $X \times Y$
    is always cocartesian, hence by (3) also the projections $X \leftarrow X \times Y \to Y$ are cocartesian.

    It remains to show (1) $\Rightarrow$ (2). Let $p \colon \cX \to S$ denote the cocartesian
    unstraightening of $F$, and let $A,B\in S$, $X\in\cX_A$, and $Y\in\cX_B$ be arbitrary.
    As $F$ preserves products, there is a unique $Z\in\cX_{A\times B}$ whose cocartesian pushforwards along the projections $A\gets A\times B\to B$ are given by $X$ and $Y$, respectively.
    It therefore suffices to show that any $Y\in\cX_{A\times B}$ is the product of its cocartesian pushforwards.
    This follows by running the argument for part (1) of the previous lemma backwards.

    Namely, we let $X \in \cX$ be arbitrary and consider the following commutative diagram
    \[\begin{tikzcd}[cramped, row sep=3.5ex]
        {\cX_{A \times B}((f,g)_!X,Y)} & {\cX_A(f_!X,\pr_{A!}Y) \times \cX_B(g_!X,\pr_{B_!}Y)} \\
        {\cX(X,Y)_{(f,g)}} & {\cX(X,\pr_{A!}Y)_{f} \times\cX(X,\pr_{B!}Y)_{g}} \\
        {\cX(X,Y)} & {\cX(X,\pr_{A!}Y)\times \cX(X,\pr_{B!}Y)} \\
        {S(pX,A \times B)} & {S(pX,A) \times S(pX,B)}
        \arrow["\sim", from=1-1, to=1-2]
        \arrow["\sim"', from=1-1, to=2-1]
        \arrow["\sim", from=1-2, to=2-2]
        \arrow[from=1-2, to=2-2]
        \arrow["\psi", from=2-1, to=2-2]
        \arrow[from=2-1, to=3-1]
        \arrow[from=2-2, to=3-2]
        \arrow["\Psi", from=3-1, to=3-2]
        \arrow[from=3-1, to=4-1]
        \arrow[from=3-2, to=4-2]
        \arrow[from=4-1, to=4-2]
    \end{tikzcd}\]
    where we are interested in showing that $\Psi$ is an equivalence,
    and have taken vertical fibers over an arbitrary $(f,g) \colon pS \to A \times b$
    to reduce to showing that $\psi$ is an equivalence.
    This in turn follows by 2-out-of-3, since the top vertical maps are equivalences
    by the universal property of cocartesian morphisms,
    and the top horizontal map is an equivalence since $F$ preserves finite products.
    It remains to see that the unique object $\textbf1 \in \cX_1 = F(1) = *$ is terminal in $\cX$.
    For $X \in \cX$, we have
    \[
        \cX(X,\textbf1) \simeq \cX(X,1)_{pX \to 1} \simeq \cX_1((pX \to 1)_!X,1) = *.\qedhere
    \]
\end{proof}

Another source of examples of $(S,E,M)$-operads comes from the following result.

\begin{proposition}\label{prop:envelope-operad}
Let $(S,E,M)$ be a factorization system such that $E$ and $S$ have finite products preserved by the inclusion.
Assume moreover that finite products of maps in $M$ are again in $M$. Consider the envelope
\[
\Env\colon\Cat_{/S}^{E\dcocart}\to\Cat_{/S}^\cocart \simeq \Fun(S,\Cat).
\]
Then $p \colon \cX \to S$ in $\Cat^{E\dcocart}_{/S}$ is an $(S,E,M)$-operad if and only if $\Env(p)$ straightens to a product-preserving functor $S\to\Cat$.
    \begin{proof}
        The assumptions guarantee that $\Ar_M(S)$ has finite products, preserved by both the source and target map.
        Suppose first that $p$ is an $(S,E,M)$-operad. It follows that for every $\cX$ with finite products and any finite product preserving $p\colon \cX\to S$ the pullback $\Env(\cX) = \cX\times_S\Ar_M(S)$ again admits finite products, preserved by both projections, and in particular by $\cX\times_S\Ar_M(S)\to \Ar_M(S)\to S$.
        In view of \cref{cor:product_unstr} it only remains to show that if $p\colon\cX\to S$ is an $(S,E,M)$-operad,
        then any $(X,f\colon C\rightarrowmono A\times B) \in (\cX \times_S\Ar_M(S))_{A\times B}$
        is the product of its pushforwards to $A$ and $B$.

        For this we use the description of cocartesian edges recalled in Theorem~\ref{thm:BHS}. Namely, the cocartesian lift of $\pr_B$ starting at $(X,f)$ has second component given by the unique square
        \[
            \begin{tikzcd}
                C\arrow[d,mono,"f"']\arrow[r,epic, "p_B"] & C_B\arrow[d,mono, "j_B"]\\
                A\times B\arrow[r, "\pr_B"', epic] & B\rlap,
            \end{tikzcd}
        \]
        and first component given by the cocartesian edge $X\to p_{B!}X$; the cocartesian lift of $\pr_A$ is given analogously. As products in $\cX\times_S\Ar_M(S)$ are componentwise, we are reduced to proving that $C_A\gets C\to C_B$ is a product diagram in $S$, since then also $(p_B)_!X \leftarrow X \to (p_A)_!X$ is a product diagram in $\cX$
        by \cref{lem:tem-operad}(1).

        For this, we observe that since $E$ admits and the inclusion $E \subset S$ preserves products,
        the map $(p_A,p_B) \colon C \rightarrowepic C_A \times C_B$ lies in $E$.
        By construction, the composite
        \[\begin{tikzcd}[cramped]
            C &[1.67em] {C_A \times C_B} &[1.33em] {A \times B}
            \arrow["{(p_A,p_B)}", epic, from=1-1, to=1-2]
            \arrow["{j_A \times j_B}", mono, from=1-2, to=1-3]
        \end{tikzcd}\]
        agrees with the map $f$, which belongs to $M$. As also $j_A\times j_B$ belongs to $M$ (being a product of maps in $M$), uniqueness of factorizations shows that $(p_A,p_B)$ is an equivalence as desired.

        For the converse, suppose that $\Env(\cX)$ is (the unstraightening of a) product preserving functor,
        so that by \cref{cor:product_unstr} $\Env(\cX)$ admits finite products
        preserves by $\Env(\cX) \to \Ar_M(S) \xto{t} S$,
        and projections are cocartesian.
        Note that by definition we have the following pullback squares
        \[\begin{tikzcd}
            \cX & {\Env(\cX)} & \cX \\
            S & {\Ar_M(S)} & S \\
            & S
            \arrow[from=1-1, to=1-2]
            \arrow["p"', from=1-1, to=2-1]
            \arrow["\lrcorner"{anchor=center, pos=0.275}, draw=none, from=1-1, to=2-2,xshift=-2pt]
            \arrow[from=1-2, to=1-3]
            \arrow["q"',from=1-2, to=2-2]
            \arrow["\lrcorner"{anchor=center, pos=0.125}, draw=none, from=1-2, to=2-3,xshift=-2pt]
            \arrow["p", from=1-3, to=2-3]
            \arrow["{\id_{(-)}}"', from=2-1, to=2-2]
            \arrow["s"', from=2-2, to=2-3]
            \arrow["t"', from=2-2, to=3-2]
        \end{tikzcd}\]
        Given $X,Y \in \cX$ we can take the products $pX \times pY$ in $S$
        and $(Z,f \colon pZ \rightarrowmono pV) \simeq (X,\id_{pX}) \times (Y, \id_{pY})$ in $\Env(\cX)$.
        Since $tq$ preserves finite products, we get $pV = pX \times pY$.
        Moreover, the projection $(Z,f) \to (X,\id_{pX})$ is $tq$-cocartesian,
        so by the description of cocartesian edges in \cref{thm:BHS}(1)
        we get that $pZ \to pX$ lies in $E$ and that $Z \to X$ is $p$-cocartesian. Analogously, $pZ \to pY$ lies in $E$,
        hence by assumption $f \colon pZ \to pX \times pY$ is both in $E$ and $M$, hence an equivalence.
        In particular, $(Z,f)$ is equivalent in $\Env(\cX)$ to $(Z,\id_{pX \times pY})$,
        and it then follows from the left pullback square that $Z$ is a product
        of $X$ and $Y$ in $\cX$.
        By this construction and the uniqueness of products, it is clear that $p$ preserves products, and that projections are $p$-cocartesian. Thus $p$ is an $(S,E,M)$-operad.
    \end{proof}
\end{proposition}

We now apply these definitions to the case of $(S,E,M) = (\Span_\Nn(\Ff),\cF^\op,\cN)$.

\begin{definition}\label{def:weakly_extensive}
    We say a span pair $(\cF,\cN)$ is \emph{weakly extensive} if
    \begin{enumerate}
        \item $\cF$ has finite coproducts,
        \item\label{wext:2} the morphisms in $\cN$ are closed under finite coproducts, \emph{and}
        \item the functor
        \[
            \amalg \colon \prod_{i=1}^{n} \cF^{\cN}_{/A_i} \to \cF^{\cN}_{/\amalg_i A_i}
        \]
        is an equivalence for every $n\geq 0$ and $A_0,\dots,A_n\in \cF$, where $\cF^{\cN}_{/A_i} \subset \cF_{/A_i}$
        is the full subcategory on those maps which belong to $\cN$.
    \end{enumerate}
    We say that $(\cF,\cN)$ is \emph{extensive} if furthermore
    \begin{enumerate}\setcounter{enumi}{3}
        \item $\cN$ contains the maps $\emptyset \to A$ and $(\id,\id)\colon A \amalg A \to A$ for all $A$, \emph{and}
        \item condition (3) holds for $\cF_{/A}$ in place of the full subcategories $\cF^\cN_{/A}$, i.e.~$\cF$ is an extensive category.
    \end{enumerate}
\end{definition}

By \cite[Proposition 2.2.5]{CHLL_NRings} if $(\cF,\cN)$ is weakly extensive then $\Span_{\cN}(\cF)$ admits finite products and the inclusion $\cF^{\op}\subset \Span_{\cN}(\cF)$ preserves finite products. Hence we obtain a notion of $(\cF,\cN)$-operads by specializing the definitions above.

\begin{definition}\label{def:f-n-operad}
Let $(\cF,\cN)$ be a weakly extensive span pair. We say an $(\cF,\cN)$-preoperad $\cO$ is an \emph{$(\cF,\cN)$-operad} if it is a $(\Span_{\cN}(\cF),\cF^{\op},\cN)$-operad in the sense of \Cref{def:operad}. We denote the resulting full subcategory of $\PreOp{\cF}{\cN}$ by $\Op{\cF}{\cN}$.
\end{definition}

\begin{definition}
We say an $\cF$-precategory $\cC\colon \cF^{\op}\to \Cat$ is an \emph{$\cF$-category} if it is product preserving. This defines a full subcategory $\iCat{\cF}$ of $\PreCat{\cF}$.
\end{definition}

\begin{remark}\label{rk:why-pre-1}
    We will typically apply this to the case where $\cF=\mathbb F[T]$ is the finite coproduct completion of an arbitrary category $T$. In this case, $\mathbb F[T]$-categories can be equivalently described as functors $T^\op\to\Cat$, and are known as \emph{$T$-categories} \cites{exposeI,shah2021parametrized, nardin2016exposeIV, CLL_Global}. While it may seem silly at first to consider $\mathbb F[T]$-categories instead of $T$-precategories, the reason for this is that only the former allows to encode the interaction with norms.
\end{remark}

\begin{definition}
We say an $\cN$-normed $\cF$-precategory $\cC\colon \Span_{\cN}(\cF)\to \Cat$ is an \emph{$\cN$-normed $\cF$-category} if one of the two by \Cref{cor:product_unstr} equivalent conditions hold:
\begin{enumerate}
	\item $\cC$ preserves finite products,
	\item The cocartesian unstraightening of $\cC$ is an $(\cF,\cN)$-operad.
\end{enumerate}
We denote the resulting full subcategory of $\PreNmCat{\cF}{\cN}$ by $\NmCat{\cF}{\cN}$.
\end{definition}

\begin{remark}\label{rk:why-pre-2}
    By definition, $\cN$-normed $\cF$-precategories are nothing else than $\Span_\Nn(\Ff)^\op$-precategories, with the $\Nn$-normed $\Ff$-\emph{categories} precisely corresponding to the $\Span_\Nn(\Ff)^\op$-\emph{categories}. This is the key reason why it is advantageous to work with precategories wherever possible: the category of $\Span_\Nn(\Ff)^\op$-categories does not fit into the original framework of parametrized higher category theory \cite{exposeI} (as $\Span_\Nn(\Ff)^\op$ is not a free coproduct completion) nor into the more general toposic framework of \cites{martini2021yoneda,martiniwolf2021limits} (as product preserving presheaves on $\Span_\Nn(\Ff)^\op$ are typically not a topos), but by working with precategories we can still apply a lot of foundational results from parametrized category theory.
\end{remark}

All our constructions restrict to the above full subcategories, under reasonable assumptions on the pair $(\cF,\cN)$:

\begin{theorem}\label{thm:oper_operad}
Let $(\cF,\cN)$ be a weakly extensive span pair.
\begin{enumerate}
	\item The functor $\fgt \colon \PreOp{\cF}{\cN}\to \PreCat{\cF}$ restricts to a functor
	\[
        \hspace{4em}\fgt \colon \Op{\cF}{\cN}\to \iCat{\cF}.
	\]
	\item\label{triv_on_operad} The functor $\Triv \colon \PreCat{\cF}\to \PreOp{\cF}{\cN}$ restricts to a functor
	\[
        \hspace{4em}\Triv\colon \iCat{\cF}\to \Op{\cF}{\cN}.
	\]
	\item The functor $(-)^{\Ncoprod}\colon \PreCat{\cF}\to \PreOp{\cF}{\cN}$ restricts to a functor
	\[
        \hspace{4em}(-)^{\Ncoprod} \colon \iCat{\cF}\to \Op{\cF}{\cN}.
	\]
	\item The functor  $\Env\colon\PreOp{\cF}{\cN}\to \PreNmCat{\cF}{\cN}$ restricts to a functor
	\[
        \hspace{4em}\Env\colon\Op{\cF}{\cN}\to \NmCat{\cF}{\cN}.
	\]
	\item If $(\Ff,\Nn)$ is even extensive, then $\und{\CAlg}^{\cN}_{\cF}\colon\PreNmCat{\cF}{\cN}\to \PreCat{\cF}$ restricts to a functor
	\[
	    \hspace{4em}\und{\CAlg}^{\cN}_{\cF}\colon\NmCat{\cF}{\cN}\to \Cat({\cF}).
	\]
\end{enumerate}
\end{theorem}
\begin{proof}
~
\begin{enumerate}
    \item Suppose $p\colon \cO\to \Span_{\cN}(\cF)$ is an $(\cF,\cN)$-operad. Then the pullback of $p$ to $\cF^{\op}$ is a cocartesian fibration which satisfies the conditions of \Cref{cor:product_unstr}(2). We conclude by said corollary that its straightening is a product preserving functor, i.e.~$\fgt(\cO)$ is an $\cF$-category.

    \item Suppose $\Cc$ is an $\cF$-category. Then its unstraightening satisfies the conditions of \Cref{cor:product_unstr}(2) as a functor to $\cF^{\op}$. However the inclusion $\cF^{\op}\hookrightarrow \Span_{\cN}(\cF)$ is product preserving, and so $\Triv(\Cc)$ is an $(\cF,\cN)$-operad.

    \item Consider the cartesian unstraightening $p\colon \Un^{\ct}(\Cc)\to \Ff$ of an $\Ff$-category $\Cc$.
        By the dual of \Cref{cor:product_unstr}, $\Un^{\ct}(\Cc)$ admits coproducts. We will verify the assumptions of \cite[Corollary C.21(2)]{BachmannHoyois2021Norms} to show that $\Span_{\cart,\Nn}(\Un^{\ct}(\Cc))$ admits products, computed as coproducts in $\Un^{\ct}(\Cc)$. To begin we show that
        \[
            \hspace{4em}\amalg\colon \Un^{\ct}(\Cc)\times \Un^{\ct}(\Cc)\to \Un^{\ct}(\Cc)
        \]
        is a functor of adequate triples. The fact that $\Un^{\ct}(\Cc)$ preserves maps lying over maps in $\Nn$ follows from the fact that $p$ preserves coproducts and \Cref{def:weakly_extensive}(2). That $\amalg$ preserves cartesian edges follows by \Cref{cor:product_unstr}$^{\op}(3)$. From this it follows immediately that $\amalg$ preserves ambigressive pullback squares, since any square in $\Un^{\ct}(\Cc)$ over a pullback square in $\Ff$ such that two parallel sides are $p$-cartesian is a pullback square, see \cite[proof of Proposition 2.6]{HHLNa}.

        Next, we show that the unit and counit of the $\amalg \dashv \Delta$ adjunction are cartesian edges. In the first case this follows from the fact that the diagonal is $p$-cartesian, once again using \Cref{cor:product_unstr}$^{\op}(3)$. In the second case this follows from the fact that the coprojections are $p$-cartesian by (2) of the above cited lemma. Finally, we have to show that the unit and counit naturality squares for maps in $\cN$ are pullback squares. This  follows immediately from the above characterization of pullback squares in $\Un^{\ct}(\Cc)$ combined with the fact that the squares
        \[\hspace{4em}\begin{tikzcd}
            {A\amalg A} & A && A & {A\amalg B} \\
            {B\amalg B} & B && {A'} & {A'\amalg B'}
            \arrow["\nabla", from=1-1, to=1-2]
            \arrow[norm, from=1-1, to=2-1]
            \arrow[norm, from=1-2, to=2-2]
            \arrow[from=1-4, to=1-5]
            \arrow[norm, from=1-4, to=2-4]
            \arrow[norm, from=1-5, to=2-5]
            \arrow["\nabla", from=2-1, to=2-2]
            \arrow[from=2-4, to=2-5]
        \end{tikzcd}\]
        are pullbacks in $\cF$ by \cite[Remark 2.2.3 and Proposition 2.2.4]{CHLL_NRings}. We conclude that $\Cc^{\Ncoprod}= \Span_{\cart,\cN}(\Un^{\ct}(\Cc))$ admits products computed as coproducts in $\Un^{\ct}(\Cc)$.
        Using this, the fact that $p \colon \Un^\ct(\cC) \to \cF$ satisfies the dual of \cref{cor:product_unstr},
        and that backwards maps in $\Span_{\cart,\all}(\Un^\ct(\cC))$ are $\Span(p)$-cocartesian,
        it follows that $\cC^{\Ncoprod}$ is an $(\cF,\cN)$-operad.

    \item The fact that $\Env$ restricts as claimed follows immediately by applying \Cref{prop:envelope-operad}
        to $(S,E,M) = (\Span_\cN(\cF),\cF^\op,\cN)$. The conditions of this cited proposition
        are satisfied by the definition of a weakly extensive span pair, see~\cref{def:weakly_extensive}.

    \item Let $X,Y\in\Ff$ arbitrary. Then the coproduct diagram $X\hookrightarrow X\amalg Y\hookleftarrow Y$ lifts through the forgetful functor $\pi\colon\Ff_{/X\amalg Y}\to\Ff$ to a coproduct diagram in $\Ff_{/X\amalg Y}$. It will therefore suffice to show that $\pi^*\ul\CAlg^\Nn_\Ff(\Cc)$ is an $\Ff_{/X\amalg Y}$-category. In view of Lemma~\ref{lemma:CAlg-slice} we may therefore assume, after replacing $(\Ff,\Nn)$ by $(\Ff_{/X\amalg Y},\Ff_{/X\amalg  Y}\times_\Ff\Nn)$, that $X\amalg Y$ is final.

 	In this case the extensivity of $\cF$ implies that $\cF\simeq \cF_{/X\amalg Y} \simeq \cF_{/X}\times \cF_{/Y}$. Under this equivalence the objects $X$ and $Y$ correspond to the objects $(\id_X,\emptyset)$ and $(\emptyset,\id_Y)$ respectively. Now note that the equivalence $\cF_{/\emptyset}\simeq \ast$ shows that $\emptyset$ is a strict initial object in any extensive category, i.e.~every map $X\to \emptyset$ is an equivalence. A simple computation shows that products by strict initial objects always exist (and are given by $\emptyset$), and so we conclude that for every $A\in \cF$, the products $X\times A$ and $Y\times A$ exists, and that
 	 \[
        \hspace{4em}X\times A\hookrightarrow \underbrace{(X\amalg Y)\times A}_{\simeq A}\hookleftarrow Y\times A
 	\]
 	is a coproduct diagram in $\cF$. In particular, we obtain an equivalence $\Cc((X\amalg Y)\times-)\iso\Cc(X\times-)\times\Cc(Y\times-)$ via the restrictions for any $\Nn$-normed $\Ff$-category $\Cc$. Since $\CAlg_\cF^\cN(-) = \ul\CAlg^\Nn_\Ff(-)(1)$ is corepresented on $\NmCat{\Ff}{\Nn}$, we conclude that
 	\[
        \hspace{4em}\CAlg^\Nn_\Ff\big(\Cc((X\amalg Y)\times-)\big)
        \iso \CAlg^\Nn_\Ff\big(\Cc(X\times-)\big) \times \CAlg^\Nn_\Ff\big(\Cc(Y\times-)\big)
 	\]
 	via the restrictions. Lemma~\ref{lemma:CAlg-shift-general} then identifies this with the natural map $\ul\CAlg^\Nn_\Ff(\Cc)(X\amalg Y)\to\ul\CAlg^\Nn_\Ff(\Cc)(X)\times\ul\CAlg^\Nn_\Ff(\Cc)(Y)$, finishing the proof. \qedhere
\end{enumerate}
\end{proof}

\begin{corollary}\label{cor:calg-cocomplete}
    Let $(\cF,\cN)$ be an extensive left cancellable span pair,
    and suppose that $\cC$ is an $\cN$-normed $\cF$-category which admits fiberwise sifted colimits.
    Then $\und{\CAlg}_\cF^\cN(\cC)$ is a fiberwise cocomplete $\cF$-category
    which is left $\cN$-adjointable.
\end{corollary}
\begin{proof}
    By \cref{thm:oper_operad}(5) and \cref{lem:calg-fib-colims}
    $\und{\CAlg}_\cF^\cN(\cC)$ is an $\cF$-category admitting fiberwise sifted colimits.
    Moreover, \cref{cor:calg-coprod} (which requires the assumption of left-cancellability) shows that it is left $\cN$-adjointable.
    Since $\cN$ contains all fold maps by extensivity
    and $\und{\CAlg}_\cF^\cN(\cC)$ is an $\cF$-category (and not just a precategory),
    this also gives fiberwise finite coproducts by \cite{CLL_Global}*{Proposition 4.2.14}.
\end{proof}

\section{Distributivity and free functors}\label{sec:free-nalg-general}
Let $\Cc$ be an $\cN$-normed $\cF$-precategory, and consider the forgetful functor
\[
    \bbU\colon \ul{\CAlg}_{\Ff}^{\cN}(\Cc) \to \Cc
\]
from the precategory of $\Nn$-normed algebras in $\Cc$ to $\Cc$ itself. In this section we will provide conditions for the existence of a left adjoint of $\mathbb U$, i.e.~for the existence of free normed algebras in $\Cc$, and prove a formula computing it.

Recall that even in the classical case of symmetric monoidal categories, one needs to impose conditions on the interaction between colimits and the symmetric monoidal structure to understand free commutative algebras. More specifically, we may require that the tensor product preserves colimits in each variable separately. This condition can be split up into the tensor product preserving sifted colimits as a functor $\Cc\times \Cc \to\Cc$, and requiring that coproducts should \emph{distribute} over the tensor product in the sense that we have specific equivalences
\[
    (A\amalg B)\otimes(C\amalg D)\simeq (A\otimes C)\amalg(A\otimes D)\amalg (B\otimes C)\amalg (B\otimes D).
\]
We begin by introducing an analogue of this in our context.

\subsection{Distributivity}

Phrasing the condition of distributivity for $\cN$-normed $\cF$-categories requires many more conditions on the pair $(\cF,\cN)$. For later use we in fact introduce a more general context, in which we single out a class $\cR\subset \cF$ of `indexed coproducts' which we would like our norms to `distribute' along.

\begin{definition}\label{def:dist}
	We call a triple $(\cF,\cN,\cR)$ of a category $\cF$ together with two wide subcategories $\cN,\cR\subset\cF$ a \emph{distributive context} if
	\begin{enumerate}
        \item\label{dist:1} $(\cF,\cN)$ is a weakly extensive left cancellable span pair.\label{cond:distr-ext-sp}
		\item\label{dist:2} $(\cF,\cR)$ is a span pair.
		\item\label{dist:3} For every $n\colon A\rightarrownorm B$ in $\cN$, the pullback functor $n^*\colon\cR_{/B}\to \cR_{/A}$ admits a right adjoint $n_*$, and for any pullback
		\[
		\begin{tikzcd}
			A'\arrow[r, "n'", norm]\arrow[dr,pullback]\arrow[d, "g"'] & B'\arrow[d, "f"]\\
			A\arrow[r, "n"', norm] & B
		\end{tikzcd}
		\]
		in $\cF$ with $n,n'\in\cN$ as indicated, the Beck--Chevalley transformation
		\[
		f^*n_*\to n'_*g^*
		\]
		between functors $\cR_{/A}\to \cR_{/B'}$ is an equivalence.
		\item\label{dist:4} Both $\cR$ and $\cF$ admit finite coproducts, preserved by the inclusion $\cR\hookrightarrow\cF$.
		\item $\cR$ is the right class of a factorization system $(\Ee,\cR)$ on $\cF$.
		\item $(\cF,\cE,\cN)$ is an adequate triple.
		\item\label{dist:6} Maps in $\cE$ are closed under coproducts in $\cF$.
	\end{enumerate}
\end{definition}

\begin{remark}\label{rk:bispan-triple-alt}
	In the language of \cites{Elmanto_Haugseng_Bispans,CHLL_Bispans, CHLL_NRings}, the first three conditions (sans left cancellability) say that $(\cF,\cN,\cR)$ is a \emph{bispan triple}.
\end{remark}

\begin{notation}
    Recall that we denote maps in $\cN$ by $A\rightarrownorm B$. We will further denote maps in $\cR$ by $A\rightarrowmono B$ and maps in $\cE$ by $A\rightarrowepic B$.
\end{notation}

\begin{remark}\label{rem:orbital-dist}
	Let $T$ be any category. Recall that a wide subcategory $P\subset T$ is called an \emph{orbital subcategory} \cite{CLL_Global}*{Definition~4.2.2} if the finite coproduct completions $\mathbb F[P]\subset\mathbb F[T]$ form a span pair; we call $T$ itself \emph{orbital} if it is an orbital subcategory of itself, i.e.~if $\mathbb F[T]$ has all pullbacks.

	If $P\subset T$ is orbital, then the span pair $(\mathbb F[T],\mathbb F[P])$ is always extensive \cite{CHLL_NRings}*{Example~3.1.9}. Assume now that we are given a factorization system $(E,M)$ on $T$. Then $\mathbb F[M]$ is the right half of a factorization system on $\mathbb F[T]$, with the left class given by those maps in $\mathbb F[T]$ that decompose as finite coproducts of maps in $E$ (i.e.~compared to $\mathbb F[E]$ we exclude fold maps).

	We therefore see that given a category $T$ with a wide subcategory $P\subset T$ and a factorization system $(M,E)$, the triple $(\mathbb F[T], \mathbb F[P], \mathbb F[M])$ is a distributive context whenever the following assumptions are satisfied:
	\begin{enumerate}
		\item Both $P$ and $M$ are orbital subcategories of $T$.
		\item $(T,E)$ is a span pair.
		\item For every $n\colon x\to y$ in $\mathbb F[P]$ the pullback $n^*\colon\mathbb F[M]_{/y}\to\mathbb F[M]_{/y}$ admits a right adjoint $n_*$ satisfying the basechange condition from Remark~\ref{rk:bispan-triple-alt}.
	\end{enumerate}
	We further remark that in the case that $T=M=P$, the second condition is automatic and the third condition simplifies to $\mathbb F[T]$ being locally cartesian closed.
\end{remark}

\begin{example}\label{ex:gog}
    We specialize the previous remark to $T=M=P=\Glo$ the ($2$-)\hskip0ptcategory of finite connected $1$-groupoids, so that $\Fglo\coloneqq\mathbb F[\Glo]$ is the category of finite $1$-groupoids (which clearly has all pullbacks). As $\Fglo$ is locally cartesian closed (e.g.~by straightening--unstraightening), we see that $(\Fglo,\Fglo,\Fglo)$ is a distributive context. On the other hand, the subcategory $\Orb\subset\Glo$ of \emph{faithful} functors is an orbital subcategory by \cite{CLL_Global}*{Example~4.2.5}, so writing $\Forb\coloneqq\mathbb F[\Orb]$ we also obtain a distributive context $(\Fglo,\Forb,\Fglo)$. This latter context is the example we are most interested in the bulk of this paper.
\end{example}

\begin{example}\label{ex:goo}
    Varying the previous example, we can also consider the factorization system on $\Glo$ with right class given by $\Orb$ and left class $\text{Surj}$ given by the full functors of connected finite $1$-groupoids \cite{LNP}*{Remark~6.14}. We will write $\mathbb E$ for the left class of the corresponding factorization system on $\Fglo$ with right class $\Forb$.

    By \cite{CLL_Global}*{Lemma~5.2.3}, we have an equivalence $\Forb_{/G}\simeq\Fun(G,\mathbb F)$ natural in $G\in\Fglo^\op$; in particular, the pullback maps $\Forb_{/G}\to\Forb_{/H}$ have right adjoints satisfying the Beck--Chevalley condition, so that also $(\Fglo, \Fglo, \Forb)$ and $(\Fglo,\Forb,\Forb)$ are distributive contexts. The latter context will be our main example throughout Section~\ref{sec:globalize-equivariant-alg}.
\end{example}

\begin{example}\label{ex:FG3}
    Write $\mathbb F_G$ for the category of finite $G$-sets, which is equivalently the finite coproduct completion of the category of finite \emph{transitive} $G$-sets. As $\mathbb F_G$ has pullbacks and is locally cartesian closed, we see that also $(\mathbb F_G,\mathbb F_G,\mathbb F_G)$ is a distributive context.
\end{example}

Let us now discuss the notion of distributivity for $\cN$-normed $\cF$-categories associated to a distributive context. For this we need some definitions.

\begin{definition}[{{cf.~\cite[Definition 2.4.1, 2.4.3]{Elmanto_Haugseng_Bispans}}}]\label{def:distributivity_diagram}
	Consider a distributive context $(\cF,\cN,\cR)$
	together with morphisms $m \colon A \rightarrowmono B$ in $\cR$ and $n\colon B\rightarrownorm C$ in $\cN$.
	We call a diagram
	\begin{equation}\label{distr_diagram}
		\begin{tikzcd}[row sep=scriptsize]
			& X & Y \\
			A \\
			& B & C
			\arrow["{n'}", norm, from=1-2, to=1-3]
			\arrow["\epsilon"', from=1-2, to=2-1]
			\arrow["\lrcorner"{anchor=center, pos=0.125}, draw=none, from=1-2, to=3-3]
			\arrow["m'", mono, from=1-3, to=3-3]
			\arrow["{m}"', mono, from=2-1, to=3-2]
			\arrow["n", norm, from=3-2, to=3-3]
			\arrow["{m''}",mono, from=1-2,to=3-2]
		\end{tikzcd}
	\end{equation}
	a \emph{distributivity diagram} for $m$ and $n$ if the square is a pullback, $m'$ is again in $\cR$, and the composite
	\[
	\Map_{\cF_{/C}}(\phi,m')
	\xto{n^*} \Map_{\cF_{/B}}(n^*\phi,m'')
	\xto{\epsilon_*} \Map_{\cF_{/B}}(n^*\phi,m)
	\]
	is an equivalence for all $\phi\colon D\to C$ in $\cF$. To connect to \Cref{def:dist}, the universal property of a distributivity diagram precisely witnesses $m'$ as a right adjoint object for pullback along $n$ and ensures that the Beck--Chevalley transformations in
	$(\ref{dist:3})$ are equivalences \cite{Elmanto_Haugseng_Bispans}*{Remark~2.4.5 and Lemma~2.4.6}. As such, distributivity diagrams are necessarily unique if they exist.
\end{definition}

\begin{remark}
    Note that since we check the universal property of a distributivity diagram against all maps $\phi$ in $\cF$ (and not just $\phi$ in $\cR$), this notion is invariant under enlarging $\cR$ and $\cN$, i.e.~if $(\cF,\cN',\cR')$ is another distributive context such that $\cR'\supset\cR$, $\cN'\supset\cN$, then a distributivity diagram with respect to $(\cF,\cN,\cR)$ is also a distributivity diagram with respect to $(\cF,\cN',\cR')$.
\end{remark}

\begin{remark}\label{rem:dist_diag_comm_in_span}
    We observe that given a distributivity diagram (\ref{distr_diagram}), the diagram
    \[\begin{tikzcd}
        Y & Y & C \\
        X & X & B \\
        A & A & B
        \arrow[Rightarrow, no head, from=1-1, to=1-2]
        \arrow["{m'}", mono, from=1-2, to=1-3]
        \arrow["{n'}", norm, from=2-1, to=1-1]
        \arrow[Rightarrow, no head, from=2-1, to=2-2]
        \arrow["\epsilon"', from=2-1, to=3-1]
        \arrow["{n'}", norm, from=2-2, to=1-2]
        \arrow["\lrcorner"{anchor=center, pos=0.125, rotate=90}, draw=none, from=2-2, to=1-3]
        \arrow["{m''}"', mono, from=2-2, to=2-3]
        \arrow["\lrcorner"{anchor=center, pos=0.125, rotate=-90}, draw=none, from=2-2, to=3-1]
        \arrow["\epsilon", from=2-2, to=3-2]
        \arrow["n"', norm, from=2-3, to=1-3]
        \arrow[Rightarrow, no head, from=2-3, to=3-3]
        \arrow[Rightarrow, no head, from=3-2, to=3-1]
        \arrow["m", mono, from=3-2, to=3-3]
    \end{tikzcd}\]
    defines a commutative square in $\Span_{\cN,\all}(\cF) \simeq \Span_{\cN}(\cF)^{\op}$. By \cite{Elmanto_Haugseng_Bispans}*{Proposition 2.5.9} it is in fact a pullback square in $\Span_\cN(\cF)^\op$. Moreover, as shown there, a pullback of the cospan above exists in $\Span_{\cN}(\cF)^{\op}$ if and only if a distributivity diagram for $m$ and $n$ exists.

    On the other hand, the inclusion $\Ff\hookrightarrow\Span_\Nn(\Ff)^\op$ sends pullback squares of the form
    \[
        \begin{tikzcd}
            A\arrow[r]\arrow[dr,pullback]\arrow[d,mono] & B\arrow[d,mono]\\
            C\arrow[r] & D
        \end{tikzcd}
    \]
    to pullback squares by Proposition~2.5.7 of \emph{op.\ cit.} By writing any map in $\Span_{\cN}(\cF)^{\op}$ as a composite of maps in $\cN^{\op}$ and $\cF$, we conclude that $(\cF,\cN,\cR)$ satisfies $(\ref{dist:2})$ and $(\ref{dist:3})$ of \Cref{def:distributivity} if and only if $(\Span_{\cN}(\cF)^{\op},\cM)$ is a span pair.
\end{remark}

\begin{remark}\label{rk:goo-ggg-pb}
	Note that the inclusion $\Forb\hookrightarrow\Fglo$ preserves pullbacks (as $\Forb$ is left cancellable and closed under basechange); as also the notions of distributivity diagrams agree, we conclude that the inclusion $\Span(\Forb)\hookrightarrow\Span_\Forb(\Fglo)^\op$ preserves pullbacks along forward maps.
\end{remark}

We are now ready to introduce the main definition.

\begin{definition}\label{def:distributivity}
	Let $(\cF,\cN,\cR)$ be a distributive context. We say an $\cN$-normed $\cF$-precategory $\cC \colon \Span_\cN(\cF) \to \Cat$ is \emph{$\cR$-distributive} if
	\begin{enumerate}
		\item $\cC$ has fiberwise sifted colimits (Definition~\ref{def:fw-colims}), i.e.~it factors through the subcategory $\Cat(\sift) \subset \Cat$ of categories admitting sifted colimits and functors preserving them;\label{cond:fw-sifted}
		\item given a map $m \colon A \to B$ in $\cR$, the induced restriction $m^* \colon \cC(B) \to \cC(A)$
		admits a left adjoint $m_! \colon \cC(A) \to \cC(B)$;
		\item for all $f\colon C\to B$ in $\cF$, $\Cc$ sends the pullback square
		\[\begin{tikzcd}
			{C \times_B A} & A \\
			C & B
			\arrow["{f'}", from=1-1, to=1-2]
			\arrow["{m'}"', mono, from=1-1, to=2-1]
			\arrow["\lrcorner"{anchor=center, pos=0.125}, draw=none, from=1-1, to=2-2]
			\arrow["m", mono, from=1-2, to=2-2]
			\arrow["f"', from=2-1, to=2-2]
		\end{tikzcd}\]
		in $\cF$ to a (vertically) left adjointable square, i.e.~the Beck--Chevalley transformation $m'_!f'^* \to f^*m_!$ is an equivalence; and
        \item for every distributivity diagram (\ref{distr_diagram}), the commutative square
        \[\begin{tikzcd}[cramped,sep=scriptsize]
            {\cC(Y)} && {\cC(C)} \\
            {\cC(X)} \\
            {\cC(A)} && {\cC(B)}
            \arrow["{m'^*}"', from=1-3, to=1-1]
            \arrow["{n'_\tensor}", from=2-1, to=1-1]
            \arrow["{\epsilon^*}", from=3-1, to=2-1]
            \arrow["{n_\tensor}"', from=3-3, to=1-3]
            \arrow["{m^*}", from=3-3, to=3-1]
        \end{tikzcd}\]
		obtained by applying $\Cc$ to the commutative square in $\Span_{\cN}(\cF)$ from \Cref{rem:dist_diag_comm_in_span} is horizontally left adjointable, i.e.~the Beck--Chevalley transformation
		$m'_!n'_\tensor \epsilon^* \to n_\tensor m_!$
		is an equivalence.\label{cond:distributivity}
	\end{enumerate}
	In the special case where $\cR = \cF$, we will simply call $\cC$ a \emph{distributive} $\cN$-normed $\cF$-precategory.
\end{definition}

\begin{remark}
	If $(\cF,\cN)$ is extensive and $\cC$ is an $\cN$-normed $\cF$-category (i.e.~preserves finite products), then each $\Cc(A)$ carries a natural symmetric monoidal structure by restricting $\Cc$ along the unique product-preserving functor $\Span(\mathbb F)\to\Span_\cN(\cF)$ sending the singleton to $A$.

    Condition (1) then implies that the tensor product preserves sifted colimits (as a functor $\cC(A)\times\cC(A)\to\cC(A)$ or, equivalently, in each variable separately). On the other hand, if also $\cR$ contains fold maps, then specializing the distributivity condition $(\ref{cond:distributivity})$ to the case where $n\colon A\amalg A\to A$ and $m\colon (A\amalg A)\amalg (A\amalg A)\to A\amalg A$ are fold maps, shows (after a straightforward, but somewhat tedious computation) that the tensor product preserves finite coproducts in each variable separately, hence preserves all colimits in each variable separately.

	In general, one should think of $(-)_!$ as `summation,' while $(-)_\tensor$ is `multiplication.' The final condition implies that we may distribute these operations, analogously to how one distributes addition and multiplication in rings. For a more precise explanation of this analogy we refer the reader to \cite[Section 1.2]{Elmanto_Haugseng_Bispans}.
\end{remark}

\begin{remark}\label{rem:distributivity_is_base_change}
	Our notion of distributivity is inspired by \cite{Elmanto_Haugseng_Bispans}. By the discussion in \Cref{rem:dist_diag_comm_in_span}, see also Theorem~2.5.2 of \emph{op.\ cit.}, we can rephrase the third and fourth conditions as a Beck--Chevalley condition for maps in $\Span_{\Nn}(\Ff)^\op$: they are equivalent to every pullback square
	\[
	\begin{tikzcd}
		A\arrow[r]\arrow[dr,pullback]\arrow[d,mono] & B\arrow[d,mono]\\
		C\arrow[r] & D
	\end{tikzcd}
	\]
	in $\Span_{\Nn,\all}(\Ff)\simeq\Span_{\Nn}(\Ff)^\op$ with the vertical maps in $\cR\subset\Ff\subset \Span_{\Nn,\all}(\Ff)$ being vertically left adjointable. We also say that the left adjoints for maps in $\cR$ \emph{satisfy basechange} with respect to all maps in $\Span_\Nn(\Ff)^\op$.
\end{remark}

\begin{remark}\label{rk:nardin-shah}
	For an (`atomic') orbital category $T$ such that $\mathbb F[T]$ is locally cartesian closed, we expect that an $\mathbb{F}[T]$-normed $\mathbb{F}[T]$-category is distributive if and only if it is a distributive $T$-symmetric monoidal $T$-category in the sense of \cite{NardinShah}*{Definition~3.2.3}.
\end{remark}

\begin{example}\label{ex:gog-terminology}
	Consider the distributive context $(\Fglo,\Forb,\Fglo)$ from Example~\ref{ex:gog}; we will refer to ($\Fglo$-)distributive $\Forb$-normed $\Fglo$-categories as \emph{distributive normed global categories}.\footnote{From a systematic point of view, we should refer to $\Forb$-normed $\Fglo$-categories as \emph{equivariantly normed global categories} instead, and reserve the name \emph{(globally) normed global categories} for $\Fglo$-normed $\Fglo$-categories. However, as the latter will play virtually no role in this paper, we have opted for the above terminology.}
    In Theorem~\ref{thm:spgl-dist} we will show that the categories of \emph{$G$-global spectra} from \cite{g-global}*{Chapter 3} assemble into a distributive normed global category.
\end{example}

\begin{example}\label{ex:borel-distributive}
    If $\Cc\colon\Span(\mathbb F)\to\Cat$ is any symmetric monoidal category, then we can form the right Kan extension $\Cc^\flat$ along the unique product-preserving functor $\Span(\mathbb F)\hookrightarrow\Span_{\Forb}(\Fglo)$ which preserves the final object. One can give the following description of the resulting functoriality, c.f.~\cite[Example 3.1]{puetzstueck-new}:
    on objects, $\cC^\flat$ is given by $G\mapsto G\text{-}\Cc\coloneqq\Fun(G,\Cc)$ with the evident functoriality in backwards maps. If $G\amalg G\to G$ is a fold map, then the functoriality in the corresponding forward map is given by the tensor product in $\Cc$, with the diagonal action. In other words, the restriction of $\cC^\flat$ along $\Span(\F) \to \Span_\Forb(\Fglo),\ * \mapsto G$ classifies the pointwise symmetric monoidal structure on $G\text{-}\cC$. Finally, if $G$ is a finite group, and $H\subset G$ is a subgroup, then the norm $\Nm^G_H$ of an $H$-object is given by $X^{\otimes|G/H|}$ with a specific $G$-action. If $\Cc$ is actually a symmetric monoidal $1$-category, the $G$-action can be explicitly written down in terms of the symmetry isomorphisms of $\Cc$, and $\Nm^G_H$ is known as the \emph{symmetric monoidal norm}.

    We call $\Cc^\flat$ the \emph{Borel normed global category} associated to $\Cc$; it is distributive whenever $\Cc$ is cocomplete and the tensor product preserves colimits in each variable separately, see \cite{Elmanto_Haugseng_Bispans}*{Proposition~3.2.8}.
\end{example}

\begin{example}
	Consider the distributive context $(\Fglo,\Forb,\Forb)$ from Example~\ref{ex:goo}; in analogy with the terminology of \cite{CLL_Clefts} we will refer to $\Forb$-distributive $\Forb$-normed $\Fglo$-categories as \emph{equivariantly distributive normed global categories}.

    Every distributive normed global category is in particular also equivariantly distributive. In Proposition~\ref{prop:sp-equ-distr} we will moreover show that the categories of genuine $G$-spectra assemble into an equivariantly distributive normed global category (which however, crucially, is \emph{not} $\Fglo$-distributive).
\end{example}

\begin{example}\label{ex:G-distributivity}
    Recall the distributive context $(\mathbb F_G,\mathbb F_G,\mathbb F_G)$ from Example~\ref{ex:FG3}. Product preserving functors $\Span(\mathbb F_G)\to\Cat$ are known as \emph{$G$-symmetric monoidal $G$-categories} \cite{NardinShah} or \emph{normed $G$-categories} \cite{CHLL_NRings}. Taking Remark~\ref{rk:nardin-shah} for granted, the corresponding notion of distributivity is explored in \cites{nardin2017thesis, NardinShah}. Contrary to what one might expect, this distributive context will play no serious role in the present paper.
\end{example}

The next goal of this section is to construct free algebras in distributive $\cN$-normed $\cF$-precategories. The result we will prove is:

\begin{theorem}\label{free_omnibus}
	Let $\cC$ be a distributive $\cN$-normed $\cF$-precategory. Then:
    \begin{enumerate}
    \item  The forgetful functor
	\[
	\bbU \colon \und{\CAlg}_\cF^\cN(\cC) \to \cC
	\]
	admits a parametrized left adjoint $\bbP$, which exhibits $ \und{\CAlg}_\cF^\cN(\cC)(A)$ as monadic over $\cC(A)$ for all $A\in \cF$. Moreover, the functors $\bbP$ are oplax natural in normed functors of $\cN$-normed $\cF$-precategories.

    \item\label{omnibus:formula_for_free} If $\cF$ admits a final object $1\in \cC$, we may compute $\bbU\bbP\colon \cC(1)\to \cC(1)$ as the composite
	\[
	\cC(1)
	\simeq \Fun_\cF^{\Nstr}(\core \und{\cN}^1,\cC)
	\xto{\res} \Fun_\cF(\fgt(\core \und{\cN}^1),\cC)
	\xto{\colim} \Fun_\cF(\ul{1},\cC) \simeq \cC(1),
	\]
	where $\colim$ denotes the left adjoint to restriction along $\fgt(\core \und{\cN}^1)\to \ul{1}$. This composite is again oplax natural in normed functors, as is the identification with $\mathbb U\mathbb P$.
    \end{enumerate}
\end{theorem}

We will in fact prove this theorem by much more abstract and general considerations in parametrized category theory. We begin with this now, and only return to the question of constructing free algebras in distributive normed categories
in \cref{subsec:free-normed-algebras}.

\subsection{Distributed colimits}\label{subsec:cleft_colimits}
This subsection and the one following it will be quite different from the rest of this paper in that we will require a good deal of parametrized higher category theory. The reader more interested in our applications to equivariant homotopy theory may take the previous theorem on faith and proceed to \Cref{sec:free-global}.

\begin{convention}\label{cleft-convention}
    Throughout, we fix a factorization system $(T,E,M)$ such that $M$ and $T$ have finite coproducts, and the inclusion $M\hookrightarrow T$ preserves them. The reader is welcome to think of $T = \Span_{\cN}(\cF)^{\op}$ for a weakly extensive span pair $(\cF,\cN)$, even more so if it is part of a distributive context.
\end{convention}

\begin{warn}\label{warn:t-cat}
    Compared to the rest of this paper, our basic object of study will be \emph{contravariant} functors $T^\op\to\Cat$ instead of covariant functors $\Span_\Nn(\Ff)\to\Cat$. This results in various conventions being the opposite of the conventions used elsewhere in this paper: for example, when we specialize $T$ as indicated above, then the norms will correspond to the \emph{left} class $E$ of the factorization system on $T$.

    This is unfortunate, but unavoidable: we will heavily rely on the existing param\-etrized literature, and in particular on \cites{martiniwolf2021limits,CLL_Clefts}\footnote{Note that Martini--Wolf work in the setting of $\cB$-parametrized category theory for a topos $\cB$. To apply their results to our setting of $T$-precategories, we pick $\cB = \PSh(T)$.}, which use the contravariant convention (for good reason).
\end{warn}

\begin{definition}
    We define the $T$-precategory $\ulPreCat{T}$ given in level $A\in T$ by $\Fun((T_{/A})^{\op},\Cat)$,
    with functoriality induced by postcomposition.
\end{definition}

\begin{definition}
    We write $\und{\bfU}_M^\times \subset \ulPreCat{T}$ for the full subcategory
    given in degree $A$ by the essential image of the fully faithful inclusion
    \[
        \Fun^\times((M_{/A})^\op,\Cat)\hookrightarrow\Fun((M_{/A})^\op,\Cat)\xhookrightarrow{\text{incl}_!}\Fun((T_{/A})^\op,\Cat)=\ulPreCat{T}(A).
    \]
\end{definition}

To see that this definition makes sense, we first have to show:

\begin{lemma}\label{lem:um-subcat}
    $\und{\bfU}_M^\times$ is indeed a subcategory of $\ulPreCat{T}$.
    Moreover, a $T_{/A}$-precategory $\cC$ is contained in $\und{\bfU}_M^\times(A)$ if and only if it is a $T_{/A}$-category and lies in the image of $\incl_!$.
    \begin{proof}
        Note that the inclusion $\incl_!\colon\PSh(M_{/A})\hookrightarrow \PSh(T_{/A})$ preserves pullbacks by \cite[Lemma 3.14 and Proposition 3.33]{CLL_Clefts} and the final object for trivial reasons, hence it preserves all finite limits. Moreover both its right adjoint as well as the inclusion
        \[
            P_\Sigma(M_{/A}) = \Fun^\times((M_{/A})^\op,\Spc) \hookrightarrow \PSh(M_{/A})
        \]
        preserves all limits. Thus, all of these functors induce functors on complete Segal objects by \cite{martini2021yoneda}*{Lemma~3.3.1}, where $\incl_!$ is again left adjoint to restriction.
        Identifying categories with complete Segal objects in spaces,
        we may thus work with coefficients in $\Spc$ instead of $\Cat$.

        Now the left Kan extension $\incl_! \colon \PSh(M_{/A}) \hookrightarrow \PSh(T_{/A})$
        restricts from free cocompletions to free sifted cocompletions.
        In fact, since $M_{/A} \subset T_{/A}$ preserves finite coproducts,
        we see that if $\incl_!X$ preserves finite products, so does $X = (\incl_!X)|_{M_{/A}}$,
        so that we get the following pullback square of large categories
        \[\begin{tikzcd}
            {P_\Sigma(M_{/A})} & {P_{\Sigma}(T_{/A})} \\
            {\PSh(M_{/A})} & {\PSh(T_{/A})}\rlap.
            \arrow["{\incl_!}", hook, from=1-1, to=1-2]
            \arrow[hook, from=1-1, to=2-1]
            \arrow[pullback, from=1-1, to=2-2]
            \arrow[hook, from=1-2, to=2-2]
            \arrow["{\incl_!}"', hook, from=2-1, to=2-2]
        \end{tikzcd}\]
        We need to show that the composite inclusion assembles into a $T$-subcategory.
        Because the square is a pullback, it suffices to check
        this for the bottom horizontal inclusion,
        which was done in \cite[Theorem 4.24]{CLL_Clefts},
        and for the right vertical inclusion.
        For this it remains to see that given a morphism $f \colon A \to B$ in $T$,
        restriction along $(T_{/f})^\op \colon (T_{/A})^\op \to (T_{/B})^\op$
        preserves finite product preserving presheaves.
        This follows from the fact that $T_{/f}$ preserves finite coproducts,
        as for both source and target they are computed in $T$.
        The addendum follows immediately from the fact that the above square is a pullback.
    \end{proof}
\end{lemma}

Given any full subcategory $\ul{\textbf{U}}\subset\ul\Cat^*(T)$, the general theory of parametrized category theory as developed by Martini--Wolf \cite{martiniwolf2021limits} gives us a notion of \emph{$\ul{\textbf{\textup U}}$-cocompleteness} of $T$-precategories, which we will now recall.

\begin{definition}[cf.~\cite{martiniwolf2021limits}*{Proposition~4.1.12 and Proposition~4.2.4}]
    Let $T$ be any category, and let $\Cc,\Kk$ be $T$-precategories. We say that $\Cc$ \emph{admits $\Kk$-colimits}, if the functor $\const\colon\Cc\to\ul\Fun_T(\Kk,\Cc)$ admits a parametrized left adjoint $\colim_\cK$, i.e.~an internal left adjoint in the $(\infty,2)$-category $\Cat^*(T)$. Here $\ul\Fun_T$ denotes the internal hom in the cartesian closed category $\Cat^*(T)$.

    Moreover, we say that a functor $F\colon\Cc\to\Dd$ of $T$-precategories \emph{preserves $\cK$-colimits} if the square
    \[
        \begin{tikzcd}
            \Cc\arrow[r, "F"]\arrow[d,"\const"'] &[2.5em] \Dd\arrow[d,"\const"]\\
            \ul\Fun_T(\Kk,\Cc)\arrow[r, "{\ul\Fun_T(\Kk,F)}"'] & \ul\Fun_T(\Kk,\Dd)
        \end{tikzcd}
    \]
    is vertically left adjointable.
\end{definition}

\begin{example}[see~\cite{martiniwolf2021limits}*{Example~4.1.14 and Remark~4.1.15}]\label{ex:fiberwise}
    If $\Kk=\const(K)$ for some category $K$, then $\Cc$ has $\Kk$-colimits if and only if it \emph{has fiberwise $K$-colimits} in the sense of Definition~\ref{def:fw-colims}, i.e.~every $\Cc(A)$ has $K$-colimits, and for every $f\colon A\to B$ in $T$ the restriction functor $f^*\colon\Cc(B)\to\Cc(A)$ preserves $K$-colimits.
\end{example}

\begin{definition}[see \cite{martiniwolf2021limits}*{Definition~5.2.1}]\label{defi:param-colims}
    Let $\ul{\textbf U}\subset\und{\Cat}^*(T)$ be a full subcategory. We say that a $T$-precategory $\Cc$ \emph{is $\ul{\textbf{\textup U}}$-cocomplete} if $\pi_A^*\Cc$ has $\Kk$-colimits for every $A\in T$ and $\Kk\in\ul{\textbf{U}}(A)$; here $\pi_A\colon T_{/A}\to T$ is the forgetful functor, with induced functor $\pi_A^*\colon\Cat^*(T)\to\Cat^*(T_{/A})$.

    Similarly, we say that $F\colon\Cc\to\Dd$ \emph{preserves $\ul{\textbf{\textup U}}$-colimits} (or \emph{is $\ul{\textbf{\textup U}}$-cocontinuous}) if $\pi_A^*F$ preserves $\Kk$-colimits for every $A\in T$ and $\Kk\in\ul{\textbf{U}}(A)$.
\end{definition}

\begin{example}
    Assume $\ul{\textbf{U}}$ is a subcategory of $A\mapsto T_{/A}$ (viewed as a subcategory of $\ul\Cat^*(T)$ by means of the Yoneda embedding). Then \cite{CLL_Global}*{Lemma 2.3.14} shows that being $\ul{\textbf{U}}$-cocomplete is equivalent to the following `pointwise' conditions:
    \begin{enumerate}
        \item For every $A\in T$ and every object $f\colon B\to A$ of $\ul{\textbf U}(A)$, the functor $f^*\colon\Cc(A)\to\Cc(B)$ has a left adjoint $f_!$.
        \item For every pullback square
        \[
            \begin{tikzcd}
                B'\arrow[d,"\beta"']\arrow[dr,pullback] \arrow[r, "f'"] & A'\arrow[d, "\alpha"]\\
                B\arrow[r, "f"'] & A
            \end{tikzcd}
        \]
        the Beck--Chevalley transformation $f'_!\beta^*\to\alpha^*f_!$
        is an equivalence (note that $f'_!$ exists, as $f'=\alpha^*(f)$ in $\ul{\textbf{U}}$).
    \end{enumerate}
\end{example}

\begin{example}\label{ex:distributivity-as-cocompleteness}
    Let $U\subset T$ be a wide subcategory closed under basechange (i.e.~$(T,U)$ is a span pair) and define $\ul{\textbf U}(A)$ for each $A\in T$ as the full subcategory spanned by the maps $B\to A$ in $U$; we will also refer to $\ul{\textbf{U}}$-cocomplete categories as \emph{$U$-cocomplete}.

    Specializing the previous example, we see that a $T$-precategory $\Cc$ is $\ul{\textbf U}$-cocomplete if and only if it is left $U$-adjointable, i.e.~for every $f\colon B\to A$ in $U$ the restriction $f^*\colon\Cc(A)\to\Cc(B)$ admits a left adjoint, satisfying basechange with respect to pullbacks in $T$.

    In particular, \Cref{rem:distributivity_is_base_change} shows that if $(\Ff,\Nn,\Mm)$ is a distributive context, then the last two conditions in the definition of $\Mm$-distributivity are equivalent to $\Mm$-cocompleteness in the above sense. Combining this with \Cref{ex:fiberwise}, we may note that any cocomplete $\Span_\cN(\cF)^\op$-precategory is a distributive $\cN$-normed $\cF$-precategory.
\end{example}

Our goal for the remainder of this subsection will be to prove the following theorem, which provides a concrete characterization of $\und{\bfU}_M^\times$-cocompleteness:

\begin{theorem}\label{thm:CatM-cocomplete}
    Let $(T,E,M)$ be a factorization system such that $(T,M)$ is a span pair and $M$ and $T$ have finite coproducts, preserved by the inclusion $M\hookrightarrow T$. Then a $T$-precategory $\Cc$ is $\und{\bfU}_M^\times$-cocomplete provided that
    \begin{enumerate}
    	\item it has fiberwise sifted colimits, that is the functor $\Cc\colon T^{\op}\to \Cat$ factors through the non-full subcategory $\Cat(\mathrm{sift}) \subset \Cat$, {and}
    	\item for every $m\colon A\to B$ in $M$ the restriction $m^*$ has a left adjoint $m_!$ satisfying basechange for pullbacks in $T$, i.e.~for every pullback
    	\[\begin{tikzcd}
    		{A'} \arrow[dr,pullback] & A \\
    		{B'} & B
    		\arrow["g", from=1-1, to=1-2]
    		\arrow["m'"', from=1-1, to=2-1, mono]
    		\arrow["m", from=1-2, to=2-2, mono]
    		\arrow["f", from=2-1, to=2-2]
    	\end{tikzcd}\]
    	in $T$ such that $m$ (and so $m'$) is in $M$, the associated Beck--Chevalley transformation
    	\[
    	m'_!g^* \to f^* m_!
    	\]
    	is an equivalence.
    \end{enumerate}
\end{theorem}

The proof of \Cref{thm:CatM-cocomplete} will require some preparations.

\begin{lemma}[{cf.~\cite{CLL_Clefts}*{Lemma~4.7--Lemma~4.9}}]\label{lemma:colimit-closure}
    Let $T$ be arbitrary, let $K$ be a small (non-parametrized) category, and let $\Cc$ be a $T$-precategory which has fiberwise $K$-colimits. Then for every $A \in T$, the class of $T_{/A}$-precategories $\Ii$ such that $\pi_A^*\Cc$ has $\Ii$-colimits is closed under $K$-colimits.
    \begin{proof}
        Replacing $T$ by $T_{/A}$, we may assume that $A$ is terminal.
        Let $\Ii_\bullet\colon K\to\Cat_{T}$ be a diagram such that $\Cc$ has $\Ii_k$-colimits for all $k\in K$, i.e.~each
        \[
            \ul\Fun_T(1,\Cc)\xrightarrow{\,\textup{const}\,}\ul\Fun_T(\Ii_k,\Cc)
        \]
        has a left adjoint. We have to show that
        \[
            \ul\Fun_T(1,\Cc)\xrightarrow{\,\textup{const}\,}\ul\Fun_T(\colim_K \Ii_\bullet,\Cc)\simeq\lim\nolimits_{K^\op}\ul\Fun_T(\Ii_\bullet,\Cc)
        \]
        has a left adjoint. For this we will show more generally that given $g=(g_k)\colon\Cc\to\lim_{K^\op}\Dd_k$ such that each individual $g_k$ has a left adjoint $f_k$, then the total map has a left adjoint.

        Note first that we have a pointwise left adjoint by \cite{descent-lim}*{Theorem B${}^\op$}. By \cite{martiniwolf2021limits}*{Corollary~3.2.11} it then only remains to show that these pointwise left adjoints $f$ satisfy the Beck--Chevalley condition with respect to restriction along maps in $T$. For this, we unravel the construction of $f$ from \cite{descent-lim} partially as follows:
        \begin{enumerate}
            \item For $(X_k)\in\lim_{K^\op}\Dd_k(B)$, the value $f(X_k)$ is given as colimit over $K$ of some diagram with $k\mapsto f_k(X_k)$.
            \item The counit $fg=\colim_{k\in K} f_kg_k(X)\to X$ is induced by a cocone given at $k\in K$ by the counit of $f_k\dashv g_k$.
            \item The unit $X\to gf(X)$ is given after restricting to each component by the composite $X\to g_kf_k(X)\to g_k\colim_{k\in K} f_k(X)$.
        \end{enumerate}
        It follows that the Beck--Chevalley map for commuting $f$ with restriction along $\alpha\colon A\to B$ in $T$ is given by the composite
        \begin{align*}
            f\alpha^*((X_k)_{k\in K})=
            \colim_k f_k\alpha^*X_k\xrightarrow[\smash{\raise3.5pt\hbox{$\scriptstyle\sim$}}]{\colim\,\BC_!}{}&\colim_k \alpha^*f_k(X_k)\\
            \xrightarrow[\smash{\raise3.5pt\hbox{$\scriptstyle\sim$}}]{\,\BC_!\,}{}&\alpha^*\big(\colim_k f_k(X_k)\big)\simeq \alpha^*f((X_k)_{k\in K}).
        \end{align*}
        Here the first Beck--Chevalley map is an equivalence since each $f_k$ is a $T$-functor, while the second Beck--Chevalley map is an equivalence as restrictions in $\Cc$ preserve $K$-colimits by assumption.
    \end{proof}
\end{lemma}

\begin{lemma}
    Assume $M$ has finite coproducts. Then the inclusion \[\Fun^\times(M^\op,\Cat)\hookrightarrow\Fun(M^\op,\Cat)\] has a left adjoint which we denote by $\shfy$. The functor $\shfy$ sends a finite coproduct $\coprod_{i=1}^n \ul A_k$ of representables to the functor $\ul A$ represented by the coproduct $A\coloneqq\coprod_{i=1}^n A_k$ in $M$.
\end{lemma}
\begin{proof}
    We can identify the source with complete Segal objects in the free sifted cocompletion
    $P_\Sigma(M) = \Fun^\times(M^\op,\Spc)$ of $M$, which is presentable by \cite{HTT}*{Proposition~5.5.8.10(1)}.
    Hence also $\Fun^\times(M^\op,\Cat)$ is presentable, and since the inclusion
    clearly preserves filtered colimits and all limits,
    the existence of the left adjoint $\shfy$ follows via the Adjoint Functor Theorem.
    The statement about representables follows at once by comparing corepresented functors using the Yoneda Lemma.
\end{proof}

\begin{lemma}\label{lemma:basic-colim}
    Let $(T,E,M)$ be as in \Cref{thm:CatM-cocomplete}, let $\Cc$ be a $T$-precategory such that restrictions along maps in $M$ admit left adjoints satisfying the Beck--Chevalley condition, and let $A\in T$. Then $\pi_A^*\Cc$ has
    $\incl_!\shfy(\coprod_{k=1}^n B_k\times I_k)$-colimits for all $B_1,\dots,B_n\in M_{/A}$ and all categories $I_1,\dots,I_n$ with terminal objects.
\end{lemma}

\begin{warn}
    It is tempting to assume that $\Cc$ is a $T$-category and then try and argue via adjointness here. However, while this extra assumption would guarantee equivalences $\Fun_{T_{/A}}(\incl_!\shfy(X),\pi_A^*\Cc)\simeq\Fun_{M_{/A}}(X,\pi_A^*(\Cc|_M))$, this does not lift to a statement about $T_{/A}$-parametrized functor categories (for example, the categories on the right do not even a priori assemble into a $T_{/A}$-category).
\end{warn}

    \begin{proof}[Proof of Lemma~\ref{lemma:basic-colim}]
        We have to show that the restriction
        \[
        \pi_{A}^*\Cc\hskip0pt minus .5pt\to\hskip0pt minus .5pt\ul\Fun_{T_{/A}}\hskip 0pt minus 2pt\Big(\mathop{\text{incl}}\nolimits_!\shfy\Big(\coprod_{k=1}^n\hskip 0pt minus 1.5pt B_k\hskip0pt minus .5pt\times\hskip 0pt minus .5pt I_k\Big),\pi_{A}^*\Cc\Big)\]
        admits a left adjoint. To this end, we observe that it factors as
        \[
            \pi_A^*\Cc\to\ul\Fun_{T_{/A}}\hskip-1pt minus 1pt\Big(\hskip-1pt minus 1pt\mathop{\text{incl}}\nolimits_!\shfy\Big(\coprod_{k=1}^nB_k\Big),\pi_A^*\Cc\Big)\to\ul\Fun_{T_{/A}}\hskip-1pt minus 1pt\Big(\hskip-1pt minus 1pt\mathop{\text{incl}}\nolimits_!\shfy\Big(\coprod_{k=1}^n B_k\times I_k\Big),\pi_A^*\Cc\Big),
        \]
        so it will be enough to construct left adjoints to these two maps individually.

        For the first map, we observe that $\shfy(\coprod_{k=1}^nB_k)$ is the object represented by the coproduct $B\coloneqq B_1\amalg\cdots\amalg B_n$ in $M_{/A}$, and since left Kan extension preserves representables,
        also $\text{incl}_!\shfy(\coprod_{k=1}^nB_k)$ is represented by $B$ in $T_{/A}$.
        Thus the first map above can be identified with $\pi_A^*\Cc\to\ul\Fun_{T_{/A}}(\ul B,\pi_A^*\Cc)$,
        which has a left adjoint by assumption on $\Cc$.

        For the second map, we observe that the projection $\coprod_{k=1}^n B_k\times I_k\to\coprod_{k=1}^n B_k$ admits a right adjoint by picking out the terminal objects of the $I_k$'s. It will therefore suffice that $\text{incl}_!\circ\shfy$ sends adjoint functors of $M_{/A}$-precategories to adjoint functors of $T_{/A}$-precategories.

        For this, we will upgrade each of the functors
        \begin{align*}
            \Fun((M_{/A})^\op,\Cat)&\xrightarrow{\shfy}\Fun^\times((M_{/A})^\op,\Cat)\hookrightarrow\Fun((M_{/A})^\op,\Cat)\\
            &\xrightarrow{\text{incl}_!}\Fun((T_{/A})^\op,\Cat)
        \end{align*}
       to functors of $(\infty,2)$-categories. This is trivial for the middle inclusion. As this inclusion moreover preserves cotensors, \cite[Lemma~A.2.14(2)]{mazel-gee-stern} shows that it admits an $(\infty,2)$-left adjoint $L$; the underlying functor of $L$ is then necessarily left adjoint to the underlying functor of the inclusion, i.e.\ equivalent to $\shfy$, providing the desired enrichment of the first functor in the above composite. 
       Finally, we obtain the enrichment on $\textup{incl}_!$ in the same way from the evident enrichment of its right adjoint (or alternatively by identifying it with restriction along the localization $T_{/A}\to M_{/A}$).
    \end{proof}

\begin{proof}[Proof of Theorem~\ref{thm:CatM-cocomplete}]
    Let $A\in T$ be arbitrary. We want to show that $\pi_A^*\Cc$ has all $\Fun^\times((M_{/A})^\op,\Cat)$-colimits.
    For this we first note that by the previous lemma it has $\text{incl}_!\shfy\big(\coprod_{k=1}^rB_k\times[n_k]\big)$-colimits for all $B_k\in M_{/A}$ and $n_k\ge0$. To complete the proof, it will now suffice by Lemma~\ref{lemma:colimit-closure} that $\Fun^\times((M_{/A})^\op,\Cat)\subset\Fun((T_{/A})^\op,\Cat)$ is generated under sifted colimits by objects of this form.

    As $\text{incl}_!\colon\Fun((M_{/A})^\op,\Cat)\to\Fun((T_{/A})^\op,\Cat)$ is a fully faithful left adjoint, it will suffice to prove the corresponding statement for $\Fun^\times((M_{/A})^\op,\Cat)\subset\Fun((M_{/A})^\op,\Cat)$. Identifying $\Cat$ with complete Segal spaces, we see that the right hand side is a localization of $\PSh(\Delta\times M_{/A})$, so every object in it can be written as a (single) colimit of objects of the form $B\times[n]$ for $B\in M_{/A}$ and $n\ge 0$. Using the Bousfield--Kan formula, we may rewrite this colimit as a $\Delta^\op$-shaped colimit of coproducts. Applying the left adjoint $\shfy \colon \Fun((M_{/A})^\op,\Cat)\to\Fun^\times((M_{/A})^\op,\Cat)$, we then see that any object of $\Fun^\times((M_{/A})^\op,\Cat)$ is a sifted colimit of objects of the form
    \[
        \shfy\left(\coprod_{i\in I} B_i\times[n_i]\right).
    \]
    Filtering $I$ by its finite subsets, we can write this as a filtered (hence a fortiori sifted) colimit of objects of above form, finishing the proof.
\end{proof}

\subsection{Distributed Kan extensions}\label{subsec:cleft-Kan}
As in Convention~\ref{cleft-convention}, we let $(T,E,M)$ denote a factorization system such that $(T,M)$ is a span pair, $M$ and $T$ have finite coproducts, and the inclusion $M \hookrightarrow T$ preserves them. The reader is again welcome to think of $T = \Span_{\cN}(\cF)^{\op}$ for a weakly extensive span pair $(\cF,\cN)$.

Having just set up a basic theory of colimits in this setting, we now turn our attention to left Kan extensions:

\begin{remark}\label{rk:pointwise-Kan}
    Let $f\colon\Xx\to\Yy$ be a $T$-functor and $\Cc$ be a $T$-precategory. \cite{martiniwolf2021limits}*{Theorem~6.3.5} provides a sufficient criterion for the existence of a \emph{left Kan extension}, i.e.~a left adjoint $f_!$ to the restriction $f^*\colon\ul\Fun_T(\Yy,\Cc)\to\ul\Fun_T(\Xx,\Cc)$: analogously to the non-parametrized situation, it is enough that $\pi_A^*\Cc$ has $(\pi_A^*\Xx\times_{\pi_A^*\Yy}\pi_A^*\Yy_{/Y})$-colimits for every $A\in T$ and $Y\in\Yy(A)$. If this criterion is satisfied, we will say that the left Kan extension is \emph{pointwise}. Remark~6.3.6 of \emph{op.\ cit.}\ shows that in this case a parametrized analogue of Kan's\ \emph{pointwise formula} holds: the square
    \[
        \begin{tikzcd}
            \Fun_{T_{/A}}(\pi_A^*\Xx,\pi_A^*\Cc)\arrow[d,"\fgt^*"'] &[.5em] \arrow[l, "f^*"'] \Fun_{T_{/A}}(\pi_A^*\Yy,\pi_A^*\Cc)\arrow[d, "Y^*"]\\
            \Fun_{T_{/A}}(\pi_A^*\Xx_{/Y},\pi_A^*\Cc) &\arrow[l, "\const"] \Fun_{T_{/A}}(\ul A,\pi_A^*\Cc)\simeq\Cc(A)
        \end{tikzcd}
    \]
    is horizontally left adjointable, where we abbreviate $\pi_A^*\Xx_{/Y}\coloneqq\pi_A^*\Xx\times_{\pi_A^*\Yy}\pi_A^*\Yy_{/Y}$.
\end{remark}

As the main result of this subsection, we will give a criterion for when we can compute a left Kan extension along a $T$-functor $f$ in terms of the left Kan extension along the underlying $M$-functor $f|_M$.
To state this result in full glory, we first have to upgrade the assignment $f\mapsto f|_M$ to a parametrized functor:

\begin{construction}
    Let $\alpha\colon S\to T$ be any functor. The functor $\alpha^*\colon\PreCat{T}\to\PreCat{S}$ preserves products, yielding for any $T$-precategories $\Xx,\Cc$ a natural equivalence $\alpha^*(\Xx\times\Cc)\simeq (\alpha^*\Xx)\times(\alpha^*\Cc)$. Adjoining over in one variable yields maps $\alpha^*\ul\Fun_T(\Xx,\Cc)\to\ul\Fun_S(\alpha^*\Xx,\alpha^*\Cc)$ natural in $\Xx$ and $\Cc$.

    Specializing to the inclusion $M\hookrightarrow T$, we obtain functors \[\res\colon\ul\Fun_T(\Xx,\Cc)|_M\to\ul\Fun_M(\Xx|_M,\Cc|_M)\] and naturality squares
    \begin{equation}\label{diag:BC-res}
        \begin{tikzcd}
            \ul\Fun_T(\Yy,\Cc)|_M\arrow[d,"\res"']\arrow[r, "f^*"] & \ul\Fun_T(\Xx,\Cc)|_M\arrow[d,"\res"]\\
            \ul\Fun_M(\Yy|_M,\Cc|_M)\arrow[r, "f|_M^*"'] & \ul\Fun_M(\Xx|_M,\Cc|_M)
        \end{tikzcd}
    \end{equation}
    for any $f\colon\Xx\to\Yy$.
\end{construction}

\begin{lemma}\label{lemma:BC-res-pw}
    Let $A\in M$ arbitrary. Then the above square agrees in degree $A$ up to equivalence with the analogous naturality square
    \begin{equation}\label{diag:BC-res-pw}
        \begin{tikzcd}
            \Fun_{T_{/A}}(\pi_A^*\Yy,\pi_A^*\Cc)\arrow[d, "\res"']\arrow[r, "f^*"] & \Fun_{T_{/A}}(\pi_A^*\Xx,\pi_A^*\Cc)\arrow[d, "\res"]\\
            \Fun_{M_{/A}}(\pi_A^*\Yy|_M,\pi_A^*\Cc|_M)\arrow[r, "f|_M^*"'] & \Fun_{M_{/A}}(\pi_A^*\Xx|_M,\pi_A^*\Cc|_M)
        \end{tikzcd}
    \end{equation}
    where we write $\pi_A$ for both of the forgetful functors $T_{/A}\to T$ and $M_{/A}\to M$.
    \begin{proof}
        \cite{CLL_Global}*{Corollary~2.2.11(ii)} shows that the comparison maps
        \begin{align*}
            \pi_A^*\ul\Fun_T(-,-)&\Rightarrow\ul\Fun_{T_{/A}}(\pi_A^*(-),\pi_A^*(-))\\
            \pi_A^*\ul\Fun_M(-,-)&\Rightarrow\ul\Fun_{M_{/A}}(\pi_A^*(-),\pi_A^*(-))
        \end{align*}
        are invertible. As Beck--Chevalley maps compose, this identifies the two naturality squares
        \[\hskip-26pt\hfuzz=26.05pt
        \footnotesize
            \begin{tikzcd}[cramped]
                \pi_A^*\ul\Fun_T(\Yy,\Cc)|_M\arrow[d,"\pi_A^*\res"']\arrow[r, "f^*"] & \pi_A^*\ul\Fun_T(\Xx,\Cc)|_M\arrow[d,"\pi_A^*\res"]\\
                \pi_A^*\ul\Fun_M(\Yy|_M,\Cc|_M)\arrow[r, "f|_M^*"'] & \pi_A^*\ul\Fun_M(\Xx|_M,\Cc|_M)
            \end{tikzcd}
            \qquad
            \begin{tikzcd}[cramped]
                \smash{\ul\Fun_{T_{/A}}}(\pi_A^*\Yy,\pi_A^*\Cc)|_{M_{/A}}\arrow[d, "\res"']\arrow[r, "f^*"] & \smash{\ul\Fun_{T_{/A}}}(\pi_A^*\Xx,\pi_A^*\Cc)|_{M_{/A}}\arrow[d, "\res"]\\
                \smash{\ul\Fun_{M_{/A}}}(\pi_A^*\Yy|_M,\pi_A^*\Cc|_M)\arrow[r, "f|_M^*"'] & \smash{\ul\Fun_{M_{/A}}}(\pi_A^*\Xx|_M,\pi_A^*\Cc|_M)
            \end{tikzcd}
        \]
        and the lemma follows by evaluating at the terminal object $\id_A\in M_{/A}$.
    \end{proof}
\end{lemma}

We can now state our criterion:

\begin{theorem}\label{thm:cleft-Kan}
	Let $f\colon\Xx\to\Yy$ be a map of small $T$-categories. Assume that for every $e\colon B\rightarrowepic C$ in $E\subset T$, the map $e^*\colon\Yy(C)\to\Yy(B)$ is a right fibration and the naturality square
	\[
	\begin{tikzcd}
		\Xx(C)\arrow[d,"e^*"']\arrow[r, "f"] & \Yy(C)\arrow[d,"e^*"]\\
		\Xx(B)\arrow[r, "f"'] & \Yy(B)
	\end{tikzcd}
	\]
    is a pullback. Then the pointwise left Kan extension functors
    \[
        f_!\colon \und{\Fun}_T(\Xx,\Cc)\to\und{\Fun}_T(\Yy,\Cc)\quad\text{and}\quad
        f|_{M!}\colon\ul\Fun_M(\Xx|_M,\Cc|_M)\to\ul\Fun_M(\Yy|_M,\Cc|_M)
    \]
    exist for every $\und{\bfU}_M^\times$-cocomplete $T$-precategory $\Cc$, and the Beck--Chevalley transformations
    populating the following squares are equivalences:
    \[\hskip-11.11pt\hfuzz=11.12pt\begin{tikzcd}[cramped]
        {\und{\Fun}_T(\cX,\cC)|_{M}} & {\und{\Fun}_T(\cY,\cC)|_M} &[-1em] {\Fun_T(\cX,\cC)} & {\Fun_T(\cY,\cC)} \\
        {\und{\Fun}_M(\cX|_M,\cC|_M)} & {\und{\Fun}_M(\cY|_M, \cC|_M)} & {\Fun_M(\cX|_M,\cC|_M)} & {\Fun_M(\cY|_M,\cC|_M)}\rlap.
        \arrow["{f_!|_M}", from=1-1, to=1-2]
        \arrow["\res"', from=1-1, to=2-1]
        \arrow["\res", from=1-2, to=2-2]
        \arrow["{f_!}", from=1-3, to=1-4]
        \arrow["\res"', from=1-3, to=2-3]
        \arrow["\res", from=1-4, to=2-4]        \arrow["\BC_!"{description}, Rightarrow, from=2-1, to=1-2,shorten=7pt]
        \arrow["{f|_{M!}}"', from=2-1, to=2-2]
        \arrow["\BC_!"{description}, Rightarrow, from=2-3, to=1-4,shorten=7pt]
        \arrow["{f|_{M!}}"', from=2-3, to=2-4]
    \end{tikzcd}\]
    Moreover, if $f$ is fully faithful, then so are $f_!$ and $f|_{M!}$.
\end{theorem}

\begin{remark}
    The pullback condition in the above theorem corresponds to the notion of \emph{equifiberedness} from \cite{BS_Equifibred, BHS_Algebraic_Patterns}.
\end{remark}

The proof of Theorem~\ref{thm:cleft-Kan} will consist of two steps: first, we will establish an even more general (but harder to verify) criterion ensuring that the $T$-parametrized Kan extension exists and can be computed in $M$-categories, and then we will show that the assumptions of the theorem imply the assumptions of the general criterion. So let's begin with the former:

\begin{lemma}\label{lem:cleft-Kan}
    Let $f\colon\Xx\to\Yy$ be a functor of small $T$-precategories such that for all $A\in T,Y\in\Yy(A)$ the slice $T_{/A}$-precategory $\pi_A^*(\Xx)_{/Y}$ is contained in $\und{\bfU}_M^\times(A) \subset \und{\Cat}_T(A)$.
    Then the parametrized pointwise left Kan extensions $f_!\colon\und{\Fun}_T(\Xx,\Cc)\to\und{\Fun}_T(\Yy,\Cc)$ and $f|_{M!}$ exist for every $\und{\bfU}_M^\times$-cocomplete $T$-precategory $\Cc$,
    and the Beck--Chevalley maps considered in \cref{thm:cleft-Kan} are equivalences.
    Moreover, if $f$ is fully faithful, then so are $f_!$ and $f|_{M!}$.
\end{lemma}
\begin{proof}
    The assumption on the slices and the fact that $\cC$ is $\und{\bfU}_M^\times$-cocomplete
    guarantees by \cite[Theorem 6.3.5]{martiniwolf2021limits} (recalled in Remark~\ref{rk:pointwise-Kan} above)
    the existence of the pointwise parametrized Kan extensions
    \begin{align*}
        f_! \colon &\und{\Fun}_T(\cX,\cC) \to \und{\Fun}_T(\cY,\cC)
        \\
        f|_{M!} \colon &\und{\Fun}_M(\cX|_M,\cC|_M) \to \und{\Fun}_M(\cY|_M,\cC|_M).
    \end{align*}
    The cited theorem further shows that both of these functors are fully faithful whenever $f$ is.

    To show that the Beck--Chevalley transformation populating the left hand square is invertible, we may argue pointwise for every $A\in T$. By Lemma~\ref{lemma:BC-res-pw}, we are then reduced to showing that the Beck--Chevalley transformation associated to
    \begin{equation}\label{diag:to-paste-top}
        \begin{tikzcd}[cramped, row sep = scriptsize]
            \Fun_{T_{/A}}(\pi_A^*\Xx,\pi_A^*\Cc)\arrow[d, "\text{res}"'] &[.5em] \arrow[l, "f^*"'] \Fun_{T_{/A}}(\pi_A^*\Yy,\pi_A^*\Cc)\arrow[d, "\text{res}"]\\
            \Fun_{M_{/A}}(\pi_A^*\Xx|_M,\pi_A^*\Cc|_M) &\arrow[l, "f^*"] \Fun_{M_{/A}}(\pi_A^*\Yy|_M,\pi_A^*\Cc|_M)
        \end{tikzcd}
    \end{equation}
    is an equivalence, i.e.~this square is (horizontally) left adjointable.
    The composites
    \begin{multline}\label{eq:res-joint-conservative}
        \Fun_{M_{/A}}(\pi_A^*\Yy|_M,\pi_A^*\Cc|_M)\xrightarrow{\;M_{/f}^*\;}
        \Fun_{M_{/B}}(M_{/f}^*\pi_A^*\Yy|_M,M_{/f}^*\pi_A^*\Cc|_M)
        \\\simeq
        \Fun_{M_{/B}}(\pi_B^*\Yy|_M,\pi_B^*\Cc|_M)
        \xrightarrow{Y^*}\Fun_{M_{/B}}(\ul B, \pi_B^*\Cc|_M)
    \end{multline}
    for varying $f\colon B\to A$ in $M$ and $Y\in \Cc(B)$ are jointly conservative, so it suffices to check that the Beck--Chevalley map is invertible after applying each of these.

    Consider now the pasting
    \begin{equation}
        \begin{tikzcd}[cramped,row sep=scriptsize]
            \Fun_{T_{/A}}(\pi_A^*\Xx,\pi_A^*\Cc)\arrow[d, "\text{res}"'] &[.5em] \arrow[l, "f^*"'] \Fun_{T_{/A}}(\pi_A^*\Yy,\pi_A^*\Cc)\arrow[d, "\text{res}"]\\
            \Fun_{M_{/A}}(\pi_A^*\Xx|_M,\pi_A^*\Cc|_M)\arrow[d] &\arrow[l, "f^*"'] \Fun_{M_{/A}}(\pi_A^*\Yy|_M,\pi_A^*\Cc|_M)\arrow[d]\\
            \Fun_{M_{/B}}(\pi_B^*\Xx|_M,\pi_B^*\Cc|_M)\arrow[d] &\arrow[l, "f^*"']
            \Fun_{M_{/B}}(\pi_B^*\Yy|_M,\pi_B^*\Cc|_M)\arrow[d, "Y^*"]\\
            \Fun_{M_{/B}}((\pi_B^*\Xx_{/Y})|_M,\pi_B^*\Cc|_M) & \arrow[l, "\const"'] \Fun_{M_{/B}}(\ul B,\pi_B^*\Cc|_M)
        \end{tikzcd}
    \end{equation}
    where the vertical maps in the middle square are the composite of the first two maps in $(\ref{eq:res-joint-conservative})$. The middle square is left adjointable by the lemma, as $M_{/f}$ agrees up to equivalence with the forgetful functor $(M_{/A})_{/f}\to M_{/B}$. On the other hand, the bottom square is left adjointable by the pointwise formula (Remark~\ref{rk:pointwise-Kan}). As Beck--Chevalley maps compose, we are therefore reduced to showing that the total pasting is left adjointable.

    By naturality and functoriality of $\res$, this pasting can be rewritten as
    \[
        \begin{tikzcd}[cramped,row sep = scriptsize]
            \Fun_{T_{/A}}(\pi_A^*\Xx,\pi_A^*\Cc)\arrow[d] &[.5em] \arrow[l, "f^*"'] \Fun_{T_{/A}}(\pi_A^*\Yy,\pi_A^*\Cc)\arrow[d]\\
            \Fun_{T_{/B}}(\pi_B^*\Xx,\pi_B^*\Cc)\arrow[d] &[.5em] \arrow[l, "f^*"'] \Fun_{T_{/B}}(\pi_B^*\Yy,\pi_B^*\Cc)\arrow[d, "Y^*"]\\
            \Fun_{T_{/B}}(\pi_B^*\Xx_{/Y},\pi_B^*\Cc)\arrow[d,"\res"'] &\arrow[l, "\const"'] \Fun_{T_{/B}}(\ul B,\pi_B^*\Cc)\arrow[d,"\res"]\\
            \Fun_{M_{/B}}((\pi_B^*\Xx_{/Y})|_M,\pi_B^*\Cc|_M) & \arrow[l,"\const"'] \Fun_{M_{/B}}(\ul B|_M,\pi_B^*\Cc|_M)
        \end{tikzcd}
    \]
    with the top maps induced by restriction along $T_{/f}$. This time, the top square is left adjointable by the lemma, while the middle square is so by the pointwise formula, so we are reduced to showing that the bottom square is left adjointable. For this, we will show that both vertical maps are actually equivalences.

    As $M\subset T$ is left cancellable (being the right class of a factorization system), the inclusion $i\colon M_{/B}\hookrightarrow T_{/B}$ is fully faithful, hence so is the left Kan extension functor $i_!$. By $\Cat$-linearity of $i_!\dashv i^*$, we conclude that $\res\colon\Fun_{T_{/B}}(\Aa,\Bb)\to\Fun_{M_{/B}}(i^*\Aa,i^*\Bb)$ is invertible whenever the $T_{/B}$-precategory $\Aa$ is left Kan extended from $M_{/B}$. But this holds for $\ul B$ (for trivial reasons) as well as for the slice $\pi_B^*\Xx_{/Y}$ (by assumption). This finishes the proof that the first of the two Beck--Chevalley transformations is invertible.

    For the right hand square, recall that any $T$-precategory can be extended uniquely to a limit preserving functor $\PSh(T)^\op\to\Cat$. By~\cite{martiniwolf2021limits}*{Corollary 3.2.11}, we then see that $f^*\colon\ul\Fun_T(\Yy,\Cc)(A)\to\ul\Fun_T(\Xx,\Cc)(A)$ admits a left adjoint for every $A\in\PSh(T)$, satisfying basechange for restrictions along arbitrary maps in $\PSh(T)$. Consider now the pasting
    \[
        \begin{tikzcd}[row sep=scriptsize]
            \ul\Fun_T(\Xx,\Cc)(1_T)\arrow[d, "\alpha^*"'] &\arrow[l, "f^*"']  \ul\Fun_T(\Yy,\Cc)(1_T)\arrow[d, "\alpha^*"]\\
            \ul\Fun_T(\Xx,\Cc)(i_!1_M)\arrow[d,"\sim"'] &\arrow[l,"f^*"'] \ul\Fun_T(\Yy,\Cc)(i_!1_M)\arrow[d,"\sim"]\\
            \ul\Fun_T(\Xx,\Cc)|_M(1_M)\arrow[d,"\res"'] &\arrow[l, "f^*"'] \ul\Fun_T(\Yy,\Cc)|_M(1_M)\arrow[d,"\res"]\\
            \ul\Fun_M(\Xx|_M,\Cc|_M)(1_M) &\arrow[l, "f^*"']\ul\Fun_M(\Yy|_M,\Cc|_M)(1_M)
        \end{tikzcd}
    \]
    where we write $1_T\in\PSh(T), 1_M\in\PSh(M)$ for the terminal objects, $i_!\colon\PSh(M)\to\PSh(T)$ for the left Kan extension functor, and $\alpha\colon i_!1_M\to 1_T$ for the unique map. We want to show that the total rectangle is left adjointable. By what we just said, the top rectangle is left adjointable, while the above shows (after evaluating at $1_M$) that the bottom rectangle is left adjointable. The middle rectangle is again left adjointable since the vertical maps are invertible. Using once more that Beck--Chevalley maps compose, we conclude.
\end{proof}

To prove that the comma categories are contained in $\und{\bfU}_M^\times$ we use the following criteria.

\begin{lemma}
	A $T_{/A}$-precategory $\Cc$ is left Kan extended from $M_{/A}$ if and only if the restriction functor $e^*\colon\Cc(C)\to\Cc(B)$ is an equivalence for every $e\colon B\rightarrowepic C$ in $T_{/A}\times_TE$.
	\begin{proof}
        By \cite{CLL_Clefts}*{proof of Proposition~3.33}, the inclusion $i\colon M_{/A}\hookrightarrow T_{/A}$ admits a left adjoint $\lambda$, which is a localization at $T_{/A}\times_TE$. We can then identify $i_!\simeq\lambda^*$, and the claim is an instance of the universal property of localization.
	\end{proof}
\end{lemma}

\begin{lemma}
    Let $f\colon\Xx\to\Yy$ be a functor of $T$-categories satisfying the assumptions of Theorem~\ref{thm:cleft-Kan}.
    Then it also satisfies the assumptions of \cref{lem:cleft-Kan}, i.e.~for all $A \in T, Y \in \cY(A)$
    the slice $T_{/A}$-category $(\pi_A^*\Xx)_{/Y}$ is contained in $\und{\bfU}_M^\times(A) \subset \und{\Cat}_T(A)$.
	\begin{proof}
        As $\Xx$ and $\Yy$ are $T$-categories, $\smash{(\pi_A^*\Xx)_{/Y}}$ is a $\smash{T_{/A}}$-category, and so the restriction to $M_{/A}$ is an $M_{/A}$-category. To prove that $(\pi_A^*\Xx)_{/Y}$ is in $\und{\bfU}_M^\times(A)$ it therefore suffices to prove that it is left Kan extended from $M_{/A}$, see \cref{lem:um-subcat}.
        By the previous lemma, we have to show that $e^*\colon(\pi_A^*\Xx)_{/Y}(C)\to(\pi_A^*\Xx)_{/Y}(B)$ is an equivalence for every map
		\[
		\begin{tikzcd}[column sep=small]
			B\arrow[dr, bend right=15pt]\arrow[rr, "e"] && C\arrow[dl, bend left=15pt,"\gamma"]\\
			&A
		\end{tikzcd}
		\]
		in $T_{/A}\times_TE$. Plugging in the definitions, this is the map
		\begin{equation}\label{eq:slice-map}
			\Xx(C)\times_{\Yy(C)}\Ar(\Yy(C))\times_{\Yy(C)}\{\gamma^*Y\}\to
			\Xx(B)\times_{\Yy(B)}\Ar(\Yy(B))\times_{\Yy(B)}\{e^*\gamma^*Y\}
		\end{equation}
		induced by $e^*\colon\Xx(C)\to\Xx(B)$ and $e^*\colon\Yy(C)\to\Yy(B)$. As the latter is a right fibration,
		\begin{equation}\label{eq:right-fib-equiv}
			e^*\colon\Ar(\Yy(C))\times_{\Yy(C)}\{\gamma^*y\}\to\Ar(\Yy(B))\times_{\Yy(B)}\{e^*\gamma^*Y\}
		\end{equation}
		is an equivalence, see e.g.~\cite{kerodon}*{Tag \href{https://kerodon.net/tag/0199}{\texttt{0199}}}. Rewriting
		\begin{align*}
			\hskip-8.48pt\hfuzz=8.5pt\Xx(C)\times_{\Yy(C)}\Ar(\Yy(C))\times_{\Yy(C)}\{\gamma^*y\}
			&\simeq
			\Xx(B)\times_{\Yy(B)}\Yy(C)\times_{\Yy(C)}\Ar(\Yy(C))\times_{\Yy(C)}\{\gamma^*Y\}\hfuzz=8.5pt\\
			&\simeq
			\Xx(B)\times_{\Yy(B)}\Ar(\Yy(C))\times_{\Yy(C)}\{\gamma^*Y\}
		\end{align*}
		using the second assumption and the pasting law for pullbacks, we then conclude that $(\ref{eq:slice-map})$ is a basechange of the equivalence $(\ref{eq:right-fib-equiv})$, hence itself an equivalence.
	\end{proof}
\end{lemma}

The lemma above, when combined with \Cref{lem:cleft-Kan}, amounts to a proof of \Cref{thm:cleft-Kan}.

\subsection{Free normed algebras}\label{subsec:free-normed-algebras}
We now return to a distributive context $(\cF,\cN,\cR)$, and explain how we can apply the previous abstract results to deduce the existence of free algebras. Namely, we will specialize the above to obtain a notion of $\und{\bfU}_\cR^\times$-cocompleteness of $\cN$-normed $\cF$-precategories $\cC$,
which holds whenever $\cC$ is $\cR$-distributive.
In order to do this, we first have to produce a factorization system on $\Span_{\cN}(\cF)^\op = \Span_{\cN,\all}(\cF)$
with right class given by $\cR \subset \cF \subset \Span_{\cN}(\cF)^\op$. For this, note first that there are two wide subcategory inclusions
\[
    \Span_{\Ee,\cN}(\cF) \subset \Span_{\cN}(\cF)
    \quad\text{and}\quad
    \cR^{\op}\simeq \Span_{\cR,\simeq}(\cF) \subset \Span_{\cN}(\cF).
\]

\begin{proposition}\label{prop:part_fact_on_spans}
Suppose $(\cF,\cN,\cR)$ is a distributive context. Then the pair
$(\cR^\op, \Span_{\Ee,\cN}(\cF))$ defines a factorization system on $\Span_\cN(\cF)$.
\end{proposition}
\begin{proof}
    In view of \cite{BS_Equifibred}*{Proposition A.0.4},
    the notion of a factorization system $(\cC,\cC_L,\cC_R)$ on a category $\cC$
    is equivalent to the statement that $(\cC_L,\cC_R)$ \emph{uniquely factors} $\cC$
    in the sense of Definition A.0.1 of \emph{op.\ cit.}
    So we want to show that $(\cR^\op,\Span_{\Ee,\cN}(\cF))$ uniquely factors $\Span_\cN(\cF)$.

    We know from \cref{thm:spans}(5) above that $(\cF^\op,\cN)$ uniquely factors $\Span_\cN(\cF)$,
    and by assumption that $(\cR^\op,\cE^\op)$ uniquely factors $\cF^\op$.
    Therefore, an application of Proposition A.0.2 of \emph{op.\ cit.\ }shows that the triple
    $(\cR^\op,\cE^\op,\cN)$ uniquely factors $\Span_\cN(\cF)$.
    Now again \cref{thm:spans}(5) tells us that $(\cE^\op,\cN)$ uniquely factors $\Span_{\Ee,\cN}(\cF)$,
    and hence another application of their Proposition A.0.2 yields that $(\cR^{\op},\Span_{\Ee,\cN}(\cF))$ uniquely factors $\Span_{\cN}(\cF)$, as desired.
\end{proof}

Taking opposites, we obtain the desired
factorization system $(\Span_{\cN,\cE}(\cF), \cR)$ on $\Span_{\cN}(\cF)^\op$. Since $(\cF,\cN)$ is a weakly extensive span pair, \cite[Proposition 2.2.5(1)]{CHLL_NRings} shows that the coproduct in $\cF$ gives a product in $\Span_\cN(\cF)$.
In other words, $\Span_{\cN}(\cF)^\op$ admits coproducts
and the inclusion $\cF \subset \Span_{\cN}(\cF)^\op$ preserves them.
Moreover, since we assumed that $\cR$ has finite coproducts preserved by $\cR \subset \cF$,
we have now verified all assumptions necessary to apply the definitions and results of the previous
subsection to the current setting. In particular, we obtain a notion of $\und{\bfU}_\cR^\times$-cocompleteness
for $\cN$-normed $\cF$-precategories.
Using \cref{thm:CatM-cocomplete}, we can then show:

\begin{proposition}\label{prop:distributivity}
    Let $(\cF,\cN,\cR)$ be a distributive context and let $\cC$ be an $\cR$-distributive $\Nn$-normed $\Ff$-precategory.
    Then $\cC$ is $\und{\bfU}_\cR^\times$-cocomplete.
\end{proposition}
\begin{proof}
   By \Cref{rem:dist_diag_comm_in_span}, $(\Span_{\cN}(\cF)^\op,\cR)$ is a span pair, and so we may sensibly apply \cref{thm:CatM-cocomplete}
    to the factorization system $(\Span_{\cN,\cE}(\cF),\cR)$ and our $\cR$-distributive
    $\cN$-normed $\cF$-precategory $\cC$.
    Point (1) of \cref{def:distributivity} precisely spells out the definition of $\cC$ having fiberwise
    sifted colimits, while point (2) is identical to the second condition of \cref{thm:CatM-cocomplete}. Finally, we have seen in Example~\ref{ex:distributivity-as-cocompleteness} that points (3) and (4) of our definition of $\cR$-distributivity are equivalent to the adjoints $m_!$ satisfying basechange along all maps in $\Span_{\cN}(\cF)^{\op}$.
\end{proof}

In particular, we may consider the case $\cM = \cF$ (hence $\cE = \core\cF$ and $\Span_{\cE,\cN}(\cF) = \cN$)
and apply the results of \cref{subsec:cleft-Kan} to distributive $\Nn$-normed $\Ff$-precategories. We will explain now how this allows us to produce free normed algebras:

\begin{lemma}[cf.~\cite{NardinShah}*{Proposition~4.3.3 and Theorem~4.3.4}]\label{lem:left_kan_along_env}
Suppose $f\colon \cO\to \cO'$ is a map of $(\cF,\cN)$-operads and that $\cC$ is a distributive $\cN$-normed $\cF$-precategory. Then the restriction
\[
\Env(f)^*\colon {\Fun_{\cF}^{\Nstr}(\Env(\cO'), \cC)} \to {\Fun_{\cF}^{\Nstr}(\Env(\cO), \cC)}
\]
along the envelope of $f$ admits a left adjoint given by pointwise left Kan extension. Moreover, the Beck--Chevalley map
\[
    \begin{tikzcd}        \Fun_\cF^{\Nstr}(\Env(\cO),\cC)\arrow[d,"\fgt"']\arrow[r, "\Env(f)_!"] &[1em] \Fun_\cF^{\Nstr}(\Env(\cO'),\cC)\arrow[d, "\fgt"]\\       \Fun_\cF(\Env(\cO)|_{\Ff^\op},\Cc|_{\Ff^\op})\arrow[r, "\Env(f)_!"']\arrow[ur,"\BC_!"{description},Rightarrow,shorten=10pt] & \Fun_\cF(\Env(\cO')|_{\Ff^\op},\Cc|_{\Ff^\op})
    \end{tikzcd}
\]
is an equivalence.
\end{lemma}

\begin{proof}
We will show that a left adjoint to $\Env(f)^*$ exists by applying \Cref{thm:cleft-Kan} for the factorization system $(\cN^{\op},\cF)$ on $\Span_{\cN}(\cF)^{\op}$. To check the hypotheses, we note that by the previous proposition, $\Cc$ is $\ul{\bfU}_{\cF}^{\times}$-cocomplete. Also, the source and target of $\Env(f)$ are both $\cN$-normed $\cF$-categories by \Cref{thm:oper_operad}.

Next, consider a map $n\colon B\rightarrownorm C$ in $\cN$. Note that for an $\cN$-normed $\cF$-precategory\footnote{Also known as a $\Span_{\cN}(\cF)^{\op}$-precategory.} restriction along $n$, viewed as a morphism in $\cN^{\op}$, is what we have typically denoted $n_\otimes$. We may then observe that by the characterization of the cocartesian morphisms in $\Env(\cO)$ from \Cref{thm:BHS} (see also \cite[Proposition 2.3.3]{BHS_Algebraic_Patterns}), for an arbitrary $(\cF,\cN)$-preoperad $\cO$ the functor $n_\otimes\colon\Env(\cO)(B)\to\Env(\cO)(C)$ is equivalent to the basechange of the postcomposition functor $\cN_{/B}\to \cN_{/C}$, and hence a right fibration.

Finally, unpacking the definitions we find that the naturality square
\[
\begin{tikzcd}
	\Env(\cO)(B) \arrow[r]\arrow[d, "n_\otimes"'] & 	\Env(\cO')(B)\arrow[d, "n_\otimes"]\\
	\Env(\cO)(C) \arrow[r] & \Env(\cO')(C)
	\end{tikzcd}
	\]
	can be written as the pullback
	\[\hspace{-9.9444pt}
	\left(
	\begin{tikzcd}[cramped]
	\cO\arrow[r]\arrow[d,equal] & \cO' \arrow[d,equal]\\
	\cO\arrow[r] 							   & \cO'\vphantom{_{/B}}
	\end{tikzcd}
	\right)
	\longrightarrow
	\left(
	\begin{tikzcd}[cramped]
		\Span_{\cN}(\cF)\arrow[d,equal]\arrow[r,equal] & \Span_{\cN}(\cF)\arrow[d,equal]\\
		\Span_{\cN}(\cF)\arrow[r,equal] & \Span_{\cN}(\cF)\vphantom{_{/B}}
	\end{tikzcd}
	\right)
	\longleftarrow
	\left(\,
	\begin{tikzcd}[cramped]
		\vphantom{E}\vphantom{_{/B}}\smash{\cN_{/B}}\arrow[d]\arrow[r,equal] & \smash{\cN_{/B}}\vphantom{_{/B}}\arrow[d]\\
		\vphantom{E}\vphantom{_{/B}}\smash{\cN_{/C}}\arrow[r,equal] & \vphantom{E}\smash{\cN_{/C}}\vphantom{_{/B}}
	\end{tikzcd}
	\right)
	\]
	of pullback squares, so it is itself a pullback. We conclude by applying \Cref{thm:cleft-Kan}.
\end{proof}

\begin{theorem}\label{thm:free-algs}
    Let $\cC$ be a distributive $\cN$-normed $\cF$-precategory. Then the forgetful functor
    \[
        \bbU \colon \und{\CAlg}_\cF^\cN(\cC) \to \cC
    \]
    admits a left adjoint $\bbP$, and for each $A \in \cF$ we have an equivalence of adjunctions natural in $A$
    \[\begin{tikzcd}
            {\Fun_{\cF}^{\Nstr}(\und{\cN}^A, \cC)} & {\und{\CAlg}_\cF^\cN(\cC)(A)} \\
            {\Fun_\cF^{\Nstr}(\core \und{\cN}^A,\cC)} & {\cC(A)}
            \arrow["\sim", from=1-1, to=1-2]
            \arrow[""{name=0, anchor=center, inner sep=0}, "{i_A^*}", shift left=3, from=1-1, to=2-1]
            \arrow[""{name=1, anchor=center, inner sep=0}, "{{\mathbb{U}}(A)}", shift left=3, from=1-2, to=2-2]
            \arrow[""{name=2, anchor=center, inner sep=0}, "{(i_A)_!}", shift left=3, from=2-1, to=1-1]
            \arrow["\sim"', from=2-1, to=2-2]
            \arrow[""{name=3, anchor=center, inner sep=0}, "{\bbP(A)}", shift left=3, from=2-2, to=1-2]
            \arrow["\dashv"{anchor=center}, draw=none, from=2, to=0]
            \arrow["\dashv"{anchor=center}, draw=none, from=3, to=1]
    \end{tikzcd}\]
    Moreover, $\bbP$ is oplax natural with respect to normed functors $\cC \to \cD$.
\end{theorem}

\begin{proof}
    Recall from \cref{prop:fna} that $i_A$ is equivalent to the map \[\Env(\incl)\colon \Env(\Triv(\ul{A}))\to \Env(\ul{A}^{\Ncoprod})\] and that the diagram
    \[\begin{tikzcd}
    	{\Fun_{\cF}^{\Nstr}(\Env(\und{A}^{\Ncoprod}), \cC)} & {\und{\CAlg}_\cF^\cN(\cC)(A)} \\
    	{\Fun_\cF^{\Nstr}(\Env(\Triv(\und{A})),\cC)} & {\cC(A).}
    	\arrow["\sim", from=1-1, to=1-2]
    	\arrow["\Env(\incl)^*"', from=1-1, to=2-1]
    	\arrow["{\mathbb{U}(A)}", from=1-2, to=2-2]
    	\arrow["\sim", from=2-1, to=2-2]
    	\arrow[ draw=none, from=2-1, to=2-2]
    \end{tikzcd}\]
    commutes naturally in $A$ by definition. Note that $\incl\colon \Triv(\ul{A}) \to \ul{A}^{\Ncoprod}$ is a map of $(\cF,\cN)$-operads by \Cref{thm:oper_operad}. So by the previous lemma, we deduce the existence of a left adjoint $\bbP(A) \coloneqq (i_A)_!$ to restriction which clearly makes the diagram above commute.

    It remains to check that the pointwise left adjoints constructed above are natural in $A$. Said differently, we have to show that the various $\bbP(A)$ assemble into a \emph{parametrized} left adjoint of $\bbU$, i.e.~that given a map $f\colon A\to B$ in $\cF$ with induced functor $f\colon \ul{\cN}^A\to \ul{\cN}^B$, the Beck--Chevalley transformation filling the diagram
    \[\begin{tikzcd}
    	{{\Fun_{\cF}^{\Nstr}(\iota\und{\cN}^B, \cC)} } & {{\Fun_{\cF}^{\Nstr}(\und{\cN}^B, \cC)} } \\
    	{{\Fun_{\cF}^{\Nstr}(\iota\und{\cN}^A, \cC)} } & {{\Fun_{\cF}^{\Nstr}(\und{\cN}^A, \cC)} }
    	\arrow["{(i_B)_!}", from=1-1, to=1-2]
    	\arrow["{(\iota f)^*}"', from=1-1, to=2-1]
    	\arrow["{f^*}", from=1-2, to=2-2]
    	\arrow[shorten <=5pt, shorten >=5pt, Rightarrow, from=2-1, to=1-2, "\BC_!"{description}]
    	\arrow["{(i_A)_!}"', from=2-1, to=2-2]
    \end{tikzcd}\]
    is invertible. However, we note that the square
    \[\begin{tikzcd}
    	{\iota\ul{\cN}^A} & {\ul{\cN}^A} \\
    	{\iota\ul{\cN}^B} & {\ul{\cN}^B}
    	\arrow["{i_A}", from=1-1, to=1-2]
    	\arrow["\iota f"', from=1-1, to=2-1]
    	\arrow["\lrcorner"{anchor=center, pos=0.125}, draw=none, from=1-1, to=2-2]
    	\arrow["f", from=1-2, to=2-2]
    	\arrow["{i_B}", from=2-1, to=2-2]
    \end{tikzcd}\]
    is a pullback since $f\colon \ul{\cN}^A\to \ul{\cN}^B$ is conservative.
    On the other hand, $f \colon \und{\cN}^A \to \und{\cN}^B$ is a (pointwise) right fibration, and
    hence the claim follows from parametrized smooth/proper basechange (see~\cref{prop:lfib-smooth} and \cref{prop:smooth-proper-basechange}).

    Finally, the oplax naturality is obtained from the naturality of $\bbU$ by passing to left adjoints, see \cite[Corollary F]{HHLNb}.
\end{proof}

As promised, we can now deduce:

\begin{corollary}\label{cor:par_U_monadic}
Suppose $\Cc$ is a distributive $\Nn$-normed $\Ff$-precategory and let $A\in \cF$. Then the forgetful functor $\mathbb{U}\colon \und{\CAlg}^{\cN}_{\cF}(\Cc)(A)\to \Cc(A)$ is a monadic right adjoint.
\begin{proof}
    Since $\bbU$ admits a left adjoint $\bbP$ by the previous result,
    and moreover is conservative and preserves sifted colimits by \cref{lem:calg-fib-colims},
    this follows from the Barr--Beck--Lurie theorem, see \cite[Theorem 4.7.3.5]{HA}.
\end{proof}
\end{corollary}

Finally we may obtain the advertised formula for free algebras.

\begin{corollary}\label{cor:univ-formula-term-ex}
	Assume $\Ff$ has a terminal object $1$, and let $\Cc$ be any distributive $\Nn$-normed $\Ff$-precategory.
    Then $\mathbb U\mathbb P\colon\Cc(1)\to\Cc(1)$ is equivalent to the functor
	\[
        \cC(1)
        \simeq \Fun_\cF^{\Nstr}(\core \und{\cN}^1,\cC)
        \xto{\;\res\;} \Fun_\cF(\core \cN_{/-},\cC)
        \xto{\;\colim\;} \cC(1)
	\]
    where $\cN_{/-} \colon \cF^\op \to \Cat$ is given by $A\mapsto\cN_{/A}$ with functoriality via pullback. Moreover, this equivalence is itself oplax natural in normed functors $\Cc\to\Dd$.
\end{corollary}
\begin{proof}
    The result follows from the following commutative diagram:
    \[\begin{tikzcd}
        {\cC(1)} &[.5em] {\und{\CAlg}_\cF^\cN(\cC)(1)} &[.5em] {\cC(1)} \\
        {\Fun^{\Nstr}_\cF(\core \und{\cN}^1, \cC)} & {\Fun_\cF^{\Nstr}(\und{\cN}^1,\cC)} & {\Fun_\cF^{\Nstr}(\core \und{\cN}^1,\cC)} \\
        {\Fun_\cF(\core \cN_{/-},\cC)} & {\Fun_\cF(\cN_{/-}, \cC)} & {\Fun_\cF(\core \cN_{/-},\cC)} \\
        {\Fun_\cF(\core \cN_{/-},\cC)} & {\cC(1)} & {\cC(1).}
        \arrow["{\bbP(1)}", from=1-1, to=1-2]
        \arrow["{\bbU(1)}", from=1-2, to=1-3]
        \arrow["{\ev_{\id_1}}"', from=2-1, to=1-1]
        \arrow["\sim", draw=none, from=2-1, to=1-1]
        \arrow["{i_!}", from=2-1, to=2-2]
        \arrow["\res", from=2-1, to=3-1]
        \arrow["{\ev_{\id_1}}", from=2-2, to=1-2]
        \arrow["\sim"', draw=none, from=2-2, to=1-2]
        \arrow["{i^*}", from=2-2, to=2-3]
        \arrow["\res"', from=2-2, to=3-2]
        \arrow["{\ev_{\id_1}}", from=2-3, to=1-3]
        \arrow["\sim"', draw=none, from=2-3, to=1-3]
        \arrow["\res"', from=2-3, to=3-3]
        \arrow["{(i|_{\cF^\op})_!}", from=3-1, to=3-2]
        \arrow[Rightarrow, no head, from=3-1, to=4-1]
        \arrow["{(i|_{\cF^\op})^*}", from=3-2, to=3-3]
        \arrow["{\ev_{\id_1}}"', from=3-2, to=4-2]
        \arrow["{\ev_{\id_1}}"', from=3-3, to=4-3]
        \arrow["\colim", from=4-1, to=4-2]
        \arrow[Rightarrow, no head, from=4-2, to=4-3]
    \end{tikzcd}\]
 	The middle left square commutes by \cref{thm:cleft-Kan} while the left bottom square commutes by the pointwise formula for left Kan extensions (Remark~\ref{rk:pointwise-Kan}), since the comma category $\core \cN_{/-} \times_{\cN_{/-}} (\cN_{/-})_{/\id_1}$ is equivalent to $\core \cN_{/-}$. Commutativity of the rest of the diagram is clear. Since the right vertical composite is the identity on $\cC(1)$, the claim follows.
    All maps except the horizontal ones on the left are natural in normed functors $\cC^\tensor \to \cD^\tensor$.
    The squares on the left are (by definition) horizontally right adjointable,
    and the right adjoints are all natural with respect to normed functors $\cC^\tensor \to \cD^\tensor$,
    so that the left adjoints inherit an oplax naturality
    using \cite[Corollary F]{HHLNb}.
\end{proof}

Combining the three previous results proves \Cref{free_omnibus}.

\section{Free normed global algebras}\label{sec:free-global}

Recall from \cref{ex:gog,ex:gog-terminology} that we have a distributive context $(\Fglo,\Forb,\Fglo)$,
where $\Fglo$ is the ($\infty$-)category of finite 1-groupoids and
$\Forb$ denotes the wide subcategory of $\Fglo$ spanned by the faithful functors.
Following the conventions in Example~\ref{ex:G-distributivity}, we refer to $\Fglo$-(pre)categories as \emph{global (pre)categories},
and to $\Forb$-normed $\Fglo$-(pre)categories as \emph{normed global (pre)categories.}
Given such normed global precategories $\cC$ and $\Dd$, we write
\[
    \Fun_{\gl}^{\otimes}(\cC,\Dd) \coloneqq \Fun_{\Fglo}^{{\raisebox{-.65pt}{${\scriptstyle\Forb}$\hspace{.2pt}}\mathrel{\textup{-}}\otimes}}(\Cc,\Dd)\quad
    \text{and}\quad \und{\CAlg}_\gl(\cC) \coloneqq \und{\CAlg}_\Fglo^\Forb(\cC)
\]
for the category of normed global functors and the global precategory of (equivariantly) normed algebras in $\cC$, respectively. Moreover, it will be convenient to introduce the following suggestive notation for the functoriality of a normed global category.

\begin{notation}
    Let $\cC$ be a normed global category.
    \begin{enumerate}
    	 \item For every injective group homomorphism $H \hookrightarrow G$
    	we have a \emph{restriction functor}
    	\[
    	\Res_H^G\colon \Cc(G)\to \Cc(H)
    	\]
    	induced by the span $G\hookleftarrow H = H$.
            A left adjoint, if it exists, gives an \emph{induction functor}
            \[
                \Ind_H^G \colon \cC(H) \to \cC(G).
            \]

        \item For every surjective group homomorphism $K \rightarrowepic G$
        we have an \emph{inflation functor}
            \[
                \Inf^G_K\colon \Cc(G) \to \Cc(K)
            \]
            induced by the span $G\leftarrowepic K = K$; we will use the same notation in the case that $K$ is the empty groupoid. A left adjoint, if it exists, gives a \emph{subduction functor}
            \[
                \Sub_K^G \colon \cC(K) \to \cC(G).
            \]
        \item For every injective group homomorphism $H \hookrightarrow G$
            we have \emph{norm functor}
            \[
                \Nm_H^G \colon \Cc(H)\to \Cc(G)
            \]
            induced by the span $H=H\hookrightarrow G$. We will employ the same notation when $H$ is the empty groupoid, in which case this is given for a normed global category by the inclusion of the symmetric monoidal unit.
    \end{enumerate}
    Note that in each case we commit a slight abuse of notation,
    as the functors actually depend on the specific morphisms of groupoids induced
    by the maps of groups, which are left implicit.
\end{notation}

Recall from Example~\ref{ex:gog} the notion of distributive normed global (pre)categories.
We will find use for the following slight weakening of this notion.

\begin{definition}
    We say a normed global precategory $\Cc$ is \emph{weakly distributive} if it satisfies the following conditions:
    \begin{enumerate}
    	\item Each $\Cc(A)$ has sequential colimits, and these are preserved by the restriction and norm functors.
        \item Given a map $f\colon G\to K$ in $\Fglo$, the restriction $f^*\colon\Cc(K)\to\Cc(G)$ has a left adjoint $f_!$.
        \item The left adjoints satisfy basechange along all maps in $\Span_{\Forb}(\Fglo)^\op$. Just as in \Cref{def:distributivity} this decomposes into two separate conditions.
    \end{enumerate}
\end{definition}

\begin{remark}
    The only difference between weak distributivity in the sense of the preceding definition and distributivity in the sense of \cref{def:distributivity} is that for the former we do not assume the existence of all fiberwise sifted colimits, but only the sequential ones. This additional flexibility will become useful in the next section in order to apply our theorem to a certain normed global $1$-category whose underlying category is given by the category of \emph{flat symmetric spectra}, since we do not even know whether the latter admits reflexive coequalizers.
\end{remark}

Let us quickly record the following stability property of weakly distributive global categories, which will be useful later.

\begin{lemma}\label{lem:shift-distr}
	If $\cC$ is a weakly distributive normed global precategory
	and $G \in \Fglo$, then $\cC(- \times G)$ is again weakly distributive.
\end{lemma}
\begin{proof}
	That $\cC(- \times G)$ again has fiberwise sequential colimits is immediate, as is the existence of left adjoints to restrictions. Finally, basechange for such maps will follow easily from the fact that
	\[
	-\otimes G \coloneqq \Span_{\Forb}(-\times G) \colon \Span_{\Forb}(\Fglo) \to \Span_{\Forb}(\Fglo)
	\]
	preserves pullbacks along backwards maps, for which we argue now. It suffices to show that $-\otimes G$ preserves pullbacks of forwards and backwards maps along backwards maps separately. The latter is immediate from the fact $- \times G \colon \Fglo \to \Fglo$ preserves pullbacks and $\Fglo^\op \subset \Span_\Forb(\Fglo)$
    sends pullbacks in $\Fglo$ to pullbacks in $\Span_\Forb(\Fglo)$.
    For the former, it suffices to show that $-\times G\colon \Fglo\to \Fglo$ preserves distributivity diagrams. Borrowing the notation of \Cref{def:distributivity_diagram}, this follows immediately from the following commutative diagram
	\[\hskip-37pt\hfuzz=37pt\begin{tikzcd}[cramped]
		{\Map_{\mathcal{F}_{/C\times G}}(\varphi,m'\times G)} &[.5em] {\Map_{\mathcal{F}_{/B\times G}}((n\times G)^*\varphi,m''\times G)} &[.5em] {\Map_{\mathcal{F}_{/B\times G}}((n\times G)^*\varphi,m\times G)} \\
		{\Map_{\mathcal{F}_{/C}}(\varphi,m')} & {\Map_{\mathcal{F}_{/B}}(n^*\varphi,m'')} & {\Map_{\mathcal{F}_{/B}}(n^*\varphi, m).}
		\arrow["{(n\times G)^*}", from=1-1, to=1-2]
		\arrow["\sim", from=1-1, to=2-1]
		\arrow["{(\epsilon \times G)_*}", from=1-2, to=1-3]
		\arrow["\sim", from=1-2, to=2-2]
		\arrow["\sim", from=1-3, to=2-3]
		\arrow["{n^*}", from=2-1, to=2-2]
		\arrow["\epsilon_*", from=2-2, to=2-3]
	\end{tikzcd}\]
	The vertical equivalences are the adjunction equivalences, while the left and right squares commutes by basechange and naturality, respectively.
\end{proof}

In this section we will prove the key computational input for Theorem~\ref{introthm:global-param}:

\begin{theorem}\label{thm:free-algebra-formula}
    Let $\Cc$ be a weakly distributive normed global category and let $G$ be any finite group. Then $\mathbb U\colon \ul\CAlg_\gl(\Cc)(G)\to\Cc(G)$ has a left adjoint $\mathbb P$, and there exists an equivalence
    \begin{equation}\label{eq:da-formula}
        \psi\colon\coprod_{n=0}^\infty \Sub^{G}_{\Sigma_n\times G}\Nm^{\Sigma_n\times G}_{\Sigma_{n-1}\times G}\Inf^G_{\Sigma_{n-1}\times G}(X) \iso\mathbb U\mathbb P(X)
    \end{equation}
    natural in $X\in\Cc(G)$, with the convention that $\Sigma_{-1}=\emptyset$ is the empty groupoid. Moreover, these equivalences can be chosen oplax naturally in $\Cc$. In particular, we have for every further weakly distributive normed global category $\Dd$, every normed global functor $F\colon\Cc\to\Dd$, and every $X\in\Cc(G)$ a commutative diagram
    \[
        \begin{tikzcd}
            \coprod_{n=0}^\infty \Sub^{G}_{\Sigma_n \times G}\Nm^{\Sigma_n \times G}_{\Sigma_{n-1} \times G}\Inf^{G}_{\Sigma_{n-1} \times G}F(X)\arrow[r, "\psi", "\sim"']\arrow[d, "\BC_!"'] & \mathbb U\mathbb P(F(X))\arrow[d, "\BC_!"]\\
            F\left(\coprod_{n=0}^\infty \Sub^{G}_{\Sigma_n \times G}\Nm^{\Sigma_n \times G}_{\Sigma_{n-1} \times G}\Inf^G_{\Sigma_{n-1} \times G} X\right)\arrow[r, "F(\psi)"', "\sim"] & F(\mathbb U\mathbb PX)\rlap.
        \end{tikzcd}
    \]
\end{theorem}

\begin{remark}
As per our conventions explained above, the zeroth summand $\Sub^G_{\Sigma_0\times G}\Nm^{\Sigma_0\times G}_\emptyset\Inf^G_\emptyset$ is the constant functor with value the symmetric monoidal unit in $\Cc(G)$.
\end{remark}

\begin{example}
    Assume $\cC$ is a symmetric monoidal category such that the tensor product commutes with colimits in each variable. Applying the above to $G=1$ and the Borel category $\Cc^\flat$ (Example~\ref{ex:borel-distributive}) recovers the classical formula
    \[
        \coprod_{n=0}^\infty (X^{\otimes n})_{h\Sigma_n}
        \simeq \mathbb U\mathbb P(X),
    \]
    with the usual convention that $X^{\otimes0}=\bbone$ is the symmetric monoidal unit. Accordingly, we suggest to think of the left hand side of $(\ref{eq:da-formula})$ as a \emph{genuine symmetric powers} construction.
\end{example}

\subsection{Understanding $\und{\Forb}^1$}

The proof of \Cref{thm:free-algebra-formula} will be an application of \Cref{free_omnibus} in the case of the distributive
context $(\cF,\cN,\cR) = (\Fglo,\Forb,\Fglo)$ from Example~\ref{ex:gog}. Therefore we will require an understanding of $\und{\cN}^1 = \und{\Forb}^1 = (\Forb_{/-})^{\Ocoprod}$.

\begin{definition}
    We define $\bbF_{G} \coloneqq \Fun({G},\F)$ to be the category of finite ${G}$-sets for a finite groupoid ${G}\in\Fglo$. These clearly assemble into a strict functor $\ul{\bbF}\colon\Fglo\to \Cat_1$ of 2-categories which we call the \emph{global category of finite equivariant sets}.
\end{definition}

\begin{lemma}[{{\cite[Lemma 5.2.3]{CLL_Global}}}]\label{ex:glo}
    There is an equivalence of global categories $\und{\F} \simeq \Forb_{/-}$.
    In particular, applying $(-)^{\Ocoprod}$ gives an equivalence of normed global categories
    $\und{\F}^{\Ocoprod} \simeq \und{\Forb}^1$.\qed
\end{lemma}

We will require the following `monoidal Borelification principle.'
Recall the terminology of \cref{ex:borel-distributive}.

\begin{theorem}[{{\cite[Theorem C]{puetzstueck-new}}}]\label{thm:mon-borel}
    The (monoidal) global Borelification functors sit in a pullback square
    \[\begin{tikzcd}[cramped]
        {\CMon(\Cat)} & \Cat \\
        {\NmCat{\Fglo}{\Forb}} & {\Cat(\Fglo)}
        \arrow["\fgt", from=1-1, to=1-2]
        \arrow["{(-)^\flat}"', from=1-1, to=2-1]
        \arrow["\lrcorner"{anchor=center, pos=0.125}, draw=none, from=1-1, to=2-2]
        \arrow["{(-)^\flat}", from=1-2, to=2-2]
        \arrow["\fgt", from=2-1, to=2-2]
    \end{tikzcd}\]
    In particular, a normed global category is Borel precisely if its underlying global category is Borel.\qed
\end{theorem}

\begin{remark}
    The functor $\und{\F}\colon \Fglo^\op \to \Cat$ is right Kan extended from its value at the trivial group $1\in\Fglo$, i.e.~$\und{\F} \simeq \F^\flat$ is the Borel global category associated to $\mathbb F$ in the sense of Example~\ref{ex:borel-distributive}.

    By the monoidal Borelification principle, it follows that the normed structure $\und{\bbF}^{\Ocoprod}$ is
    equivalently the monoidal Borelification of the cocartesian symmetric monoidal category of finite sets $\bbF^\amalg$,
    and so we denote it $\ul{\bbF}^\amalg$.
\end{remark}

\begin{remark}\label{rk:glo-core}
	The result above and \cref{cor:iota_N_rep} together show
	that we have unique equivalences
	\[
        \Map_{\Span_\Forb(\Fglo)}(1,-)
        \iso \core \ul{\Forb}^1
        \iso \core \und{\F}^\amalg
	\]
	of normed global categories. At a finite group $G$ these are given on objects by
	\[
        \begin{tikzcd}[cramped, column sep=small, row sep=0pt]
            \Map_{\Span_\Forb(\Fglo)}(1,G)
            \arrow[r,"\sim"] & \core (\Forb_{/G})
            \arrow[r, "\sim"] & \core \F_G\\
	        (1 \leftarrow H \hookrightarrow G) \arrow[r, mapsto] & (H \hookrightarrow G) \arrow[r,mapsto] & G/H
        \end{tikzcd}
	\]
    (see \cite{CLL_Global}*{Lemma 5.2.3} for the description of the second map).
\end{remark}

To understand $\bbU\bbP$, it suffices by \Cref{free_omnibus}$(\ref{omnibus:formula_for_free})$ to understand colimits over the global space $\core\ul{\mathbb F}$.

\begin{construction}
    For any $n\ge0$, we write $\chi_n\colon \ul\Sigma_n\to\core\ul{\mathbb F}$ for the map picking out the tautological $\Sigma_n$-set $\{1,\dots,n\}\in\mathbb F_{\Sigma_n}$ under the Yoneda equivalence.

    More concretely, note that the source and target can be obtained as strict functors $\Fglo^\op\to \Cat_1$ of 2-categories. Then we may write down a strictly natural transformation $\ul\Sigma_n\to\core\ul{\mathbb F}$ given at a group $H$ by sending an object $\phi\colon H\to\Sigma_n$ to $\phi^*\{1,\dots,n\}$ and the map in $\ul\Sigma_n$ corresponding to the $2$-cell $\Sigma_n\ni\sigma\colon \phi\Rightarrow\psi$ in $\Fglo$ to the map $\phi^*\{1,\dots,n\}\to\psi^*\{1,\dots,n\}$ given by $\sigma$. As this by design sends the identity of $\Sigma_n$ to $\{1,\dots,n\}$, this necessarily agrees with $\chi_n$.

    We now define a map
    \[
    \chi\colon\coprod_{n\ge0}\ul\Sigma_n\to\core\ul{\mathbb F}
    \]
    of global spaces as the map given on the $n$-th coproduct summand by $\chi_n$.
\end{construction}

\begin{proposition}[cf.~\cite{schwede-k}*{Example 11.5}]\label{prop:decompose-global-sets}
    The map $\chi$ is an equivalence of global categories.
    \begin{proof}
        Let $H$ be a finite group. The above explicit description makes it clear that $\chi_n(H)\colon\ul\Sigma_n(H)\to\core\mathbb F_H$ defines an equivalence onto the full subcategory spanned by $H$-sets of cardinality $n$. As for varying $n$ these form a partition of $\core\mathbb F_H$, the claim follows.
    \end{proof}
\end{proposition}

\begin{notation}\label{not:ext}
    Note that by uniqueness of left adjoints, the composite
    \[
        \Env(\Triv(-))
        \colon \Cat(\Fglo)
        = \Fun^\times(\Fglo^\op,\Cat)
        \to \Fun^\times(\Span_\Forb(\Fglo),\Cat)
        = \NmCat{\Fglo}{\Forb}
    \]
    is equivalent to left Kan extension along $\Fglo^\op \subset \Span_\Forb(\Fglo)$.
    For ease of notation, we will denote the composite by $\Ext \coloneqq \Env \circ \Triv$ for the remainder of this section.
    It will also be useful to distinguish a normed global category from
    its underlying global category notationally. Therefore we will write $\Cc^\otimes$ for a generic normed global category and $\Cc$ for its underlying global category.

    Viewing $\Cc^\otimes$ as a limit preserving functor $\PSh(\Span_\Forb(\Fglo)^\op)^\op\to\Cat$, \cite{CLL_Global}*{Corollary 2.2.8} provides equivalences
    \[
        \cC^\tensor(\Ext(X))
        \simeq  \Fun_{\gl}^{\tensor}(\Ext(X),\cC^\tensor)
        \simeq \Fun_{\gl}(X, \cC)
        \simeq \cC(X)
    \]
    natural in $X$. Recall moreover that for $X$ the represented functor $\und{1}\colon\Fglo^\op\to \Cat$, $\Ext(\ul1) = \Map_{\Span_\Forb(\Fglo)}(1,-)$ admits a unique equivalence to $\iota\und{\F}^\amalg$ by \cref{rk:glo-core}.
\end{notation}

\begin{lemma}
    Let $\Cc^\otimes$ be a weakly distributive normed global precategory. Then $\Cc^\otimes$ has $\Ext(\core\ul{\mathbb F})$-colimits (as a $\Span_{\Forb}(\Fglo)^\op$-precategory) in the sense of Definition~\ref{defi:param-colims}.
\end{lemma}
\begin{proof}
    Note that by Proposition~\ref{prop:decompose-global-sets} we have an equivalence of global spaces
    \[
        \core\und{\F}
        \simeq \mathop{\text{colim}}_{n \ge 0}\underline{\textstyle\coprod_{k= 0}^n\Sigma_k}
    \]
    so applying the left adjoint $\Ext$ yields an equivalence
    \[
        \Ext(\core\ul{\mathbb F})
        \simeq \mathop{\text{colim}}_{n\ge 0}\Ext\!  \left(\underline{\textstyle\coprod_{k = 0}^n\Sigma_k}\right)
    \]
    of normed global categories.
    Since $\cC^\tensor$ is weakly distributive, it admits $\Ext(G)$-colimits for every $G \in \Fglo$ (Example~\ref{ex:distributivity-as-cocompleteness}). Moreover, $\cC^\tensor$ admits fiberwise sequential colimits,
    so the claim now follows from \cref{lemma:colimit-closure}.
\end{proof}

\begin{proof}[Proof of Theorem~\ref{thm:free-algebra-formula}]
	In view of \cref{lemma:CAlg-shift-general,lem:shift-distr} we may assume without loss of generality that $G=1$.

    Suppose first that $\Ee^\otimes$ is a distributive (and not just weakly distributive)
    normed global precategory.
    Then \cref{rk:glo-core,free_omnibus}
    shows that a left adjoint $\mathbb P$ of $\mathbb U$ exists and provides an equivalence $\phi$ between $\mathbb U\mathbb P$ and the composite
    \begin{equation*}
        \Ee(1)
        \simeq\Fun_{\gl}^{\otimes}(\core\ul{\mathbb F}^\amalg, \Ee^\otimes)
        \xrightarrow{\res\,} \Fun_{\gl}(\core\ul{\mathbb F}, \Ee)
        \xrightarrow{\colim} \Ee(1)
    \end{equation*}
    such that $\phi$ is oplax natural in normed functors.
    Using \cref{not:ext} and the $\Cat$-linear adjunction $\Ext \colon \PreCat{\Fglo} \rightleftarrows \PreNmCat{\Fglo}{\Forb} \colon \fgt$ with counit $\epsilon$
    we can rewrite this composite as
    \begin{equation}\label{eq:general-colim-formula}
        \Ee^\tensor(1)
        \simeq\cE^\tensor(\core \und{\F}^\amalg)
        \xrightarrow{\epsilon^*\,} \cE^\tensor(\Ext(\core \und{\F}))
        \xrightarrow{\colim} \Ee^\tensor(1)
    \end{equation}
    If $\Cc^\otimes$ is now any weakly distributive normed global precategory, then the $\Span_{\Forb}(\Fglo)^\op$-parametrized Yoneda embedding \cite{martini2021yoneda} of $\cC^{\tensor,\op}$
    gives rise to the fully faithful \cite{martini2021yoneda}*{Corollary~4.7.16} and colimit-preserving \cite{martiniwolf2021limits}*{Proposition~5.2.9${}^\op$} embedding
    \[
        y_{\cC^{\otimes,\op}}^\op \colon \cC^\otimes \hookrightarrow \widehat{\cC}^\tensor
        \coloneqq \und{\PSh}_{\Span_\Forb(\Fglo)^\op}(\cC^{\tensor,\op})^\op.
    \]
    Note that ${\widehat{\cC}}^\tensor$ is cocomplete by another application of \cite{martiniwolf2021limits}*{Proposition~5.2.9${}^\op$} and therefore in particular distributive (Examples~\ref{ex:fiberwise} and~\ref{ex:distributivity-as-cocompleteness}).

    We now claim that
    $
        \bbP(1) \colon \widehat{\cC}^\tensor(1)
        \to \und{\CAlg}_\gl(\widehat{\cC}^\tensor)(1)
    $
    restricts on full subcategories to a left adjoint
    $
        \cC^\tensor(1)
        \to \und{\CAlg}_\gl(\cC^\tensor)(1)
        \simeq \Fun_\gl^\tensor(\und{\F}^\amalg,\cC^\tensor)
    $ of $\bbU(1)$.
    Note that a normed global functor $\und{\F}^\amalg \to \widehat{\cC}^\tensor$
    factors through $\cC^\tensor$ if and only if it does so after
    precomposing by the essentially surjective $\core \und{\F}^\amalg \to \und{\F}^\amalg$.
    Thus, to see that $\bbP(1)$ restricts as desired, it suffices to check that $\bbU\bbP(1) \colon \widehat{\cC}(1) \to \widehat{\cC}(1)$ restricts to $\cC^\tensor(1) \to \cC^\tensor(1)$.
    But this follows from the description (\ref{eq:general-colim-formula}),
    the fact that $\cC^\tensor \hookrightarrow \widehat{\cC}^\tensor$
    preserves colimits, and that $\cC^\tensor$ admits $\Ext(\core \und{\F})$-colimits by the previous lemma.

    Thus $\bbU(1) \colon \und{\CAlg}_\gl(\cC^\tensor)(1) \to \cC^\tensor(1)$
    admits a left adjoint $\bbP(1)$,
    and $\phi$ (for $\Ee^\tensor=\widehat{\cC}^\tensor$) restricts to an equivalence $\psi$ between the endofunctor $\mathbb U\mathbb P$ of $\Cc^\tensor(1)$ and the composite $(\ref{eq:general-colim-formula})$ for $\Ee^\tensor=\Cc^\tensor$. As the Yoneda embedding is functorial, the equivalence $\psi$ is again oplax natural
    in normed functors.

    We have thus shown that for all weakly distributive normed global precategories, $\bbU\bbP$ is computed by the composite $(\ref{eq:general-colim-formula})$.
    Therefore it only remains to identify this composite for any weakly distributive normed global \emph{category} (as opposed to a precategory) with the genuine symmetric powers.

    Under the equivalence $\cC^\tensor(\Ext(\core \und{\F})) \simeq \Cc(\core\ul{\mathbb F})\simeq\Cc(\coprod_{n\ge0}\ul\Sigma_n)\simeq\prod_{n\ge0}\Cc(\Sigma_n)$, the rightmost arrow corresponds to the functor $(X_n)_{n\in\mathbb N}\mapsto\coprod_{n\ge 0}\text{Sub}^{1}_{\Sigma_n}X_n$.
    It therefore remains to identify the composite
    $\Cc(1)\to\Cc(\core\ul{\mathbb F})\to\Cc(\Sigma_n)$
    or equivalently $\cC^\tensor(\Ext(\und{1})) \to \cC^\tensor(\Ext(\core \und{\F})) \to \cC^\tensor(\Ext(\und{\Sigma}_n))$
    with the map ${\Nm_{\Sigma_{n-1}}^{\Sigma_n}}\circ{\Inf^1_{\Sigma_{n-1}}}$.
    This amounts to saying that the composite
    \[
        \Ext(\und{\Sigma}_n)
        \xto{\Ext(\chi_n)} \Ext(\core \und{\F})
        \xto{\epsilon} \core\und{\F}^{\amalg}
        \simeq \Ext(\und{1})
    \]
    in $\PSh(\Span_{\Forb}(\Fglo))$ is the Yoneda image of the zig-zag $1 \leftarrow \Sigma_{n-1} \hookrightarrow \Sigma_n$ (recall that $\Ext$ sends representable
    functors to corepresentable ones).
    By the non-parametrized Yoneda lemma, this can be checked
    by plugging in the identity span $(\Sigma_n = \Sigma_n = \Sigma_n) \in \Ext(\und{\Sigma}_n)(\Sigma_n)$.
    Since this lies in the image of the unit
    $\eta \colon \und{\Sigma}_n(\Sigma_n) \to \Ext(\Sigma_n)(\Sigma_n)$,
    naturality and the triangle identity tell us that we
    need to compute the image of $\id_{\Sigma_n}$
    under $\und{\Sigma}_n \xto{\chi_n} \core \und{\F} = \core \und{\F}^{\amalg}|_{\Fglo^\op} \simeq \Map_{\Span_\Forb(\Fglo)}(1,-)|_{\Fglo^\op}$,
    where the last equivalence is that of \cref{rk:glo-core}.
    By the remark, we know that the inverse of this equivalence
    sends $1 \leftarrow \Sigma_{n-1} \hookrightarrow \Sigma_n$
    to the tautological $\Sigma_n$-set $\Sigma_n/\Sigma_{n-1}$,
    which is also the image of $\id_{\Sigma_n}$ under $\chi_n$,
    as desired.
\end{proof}

\begin{remark}\label{rk:equivariant-sucks}
    If $G$ is a finite group, then Theorem~\ref{free_omnibus} applies just as well to the distributive context $(\mathbb F_G,\mathbb F_G,\mathbb F_G)$ from Example~\ref{ex:G-distributivity}, showing that for a distributive normed $G$-category $\Cc^\otimes$, the forgetful functor $\ul\CAlg_G(\Cc^\otimes)\to\Cc$ has a left adjoint $\mathbb P$ given in degree $G/G$ by the formula
    \[
        \mathbb U\mathbb P(X)\simeq\colim_{\core\ul{\mathbb F}_G}\ul X\simeq\coprod_{n\ge 0}\colim_{\core\ul{\mathbb F}_{G,n}}\ul X
    \]
    where $\ul X\colon\core\ul{\mathbb F}_G^\amalg\to\Cc^\otimes$ is the normed functor classifying the object $X$, and $\ul{\mathbb F}_{G,n}\subset\ul{\mathbb F}_G$ is the full $G$-subcategory of equivariant sets with precisely $n$ elements; this formula also appears as \cite{NardinShah}*{Example~4.3.7}, where a specific model of the unstraightening of $\ul{\mathbb F}_{G,n}$ is denoted $\textbf{O}_{G\times\Sigma_n,\Gamma_n}$.

    However, unlike in the global case, there does not seem to be a simple way to express this parametrized colimit in terms of restrictions, norms, and inductions (there are no inflations nor subductions for $G$-categories).
\end{remark}

\section{Ultra-commutative global ring spectra as parametrized algebras}\label{sec:global-model}

In this section, we construct the normed global category of global spectra $\und{\Sp}_\gl^\tensor$
and prove \cref{introthm:global-param} from the introduction.
A suitable global comparison functor $\und{\UCom}_\gl \to \und{\Sp}_\gl^\tensor$
was already constructed in \cite{puetzstueck-new},
but we recall the constructions for the convenience of the reader.

\begin{construction}\label{constr:spgl}
    Consider the symmetric monoidal $1$-category of flat symmetric spectra $\Sp^{\Sigma,\tensor}_\textup{flat}$ \cite{hss,shipley-convenient}, with symmetric monoidal product given by the usual smash product. We form the associated Borel normed global category $(\Sp^{\Sigma,\tensor}_\textup{flat})^\flat$ in the sense of Example~\ref{ex:borel-distributive}. Recall that this means that $(\Sp^{\Sigma,\tensor}_\textup{flat})^\flat(G)=\Fun(G,\Sp^{\Sigma}_\textup{flat})$ naturally in $G\in\Fglo^\op$, with the covariant functoriality in fold maps encoding the smash product, and the norm $\Nm^G_H$ for finite groups $H\subset G$ given by the classical, $1$-categorical symmetric monoidal norm (which in this case is known as the \emph{Hill--Hopkins--Ravenel norm}).

    We now equip $(\Sp^{\Sigma,\tensor}_\textup{flat})^\flat$  at level $G$ with the $G$-global weak equivalences of \cite{g-global}*{Definition~3.1.28}. By \cite{g-global}*{Lemma~3.1.50 and Proposition~3.1.62} and \cite{Lenz-Stahlhauer}*{Corollary~5.19} both restrictions and norms preserve $G$-global weak equivalences between flat symmetric spectra, and so $(\Sp^{\Sigma,\tensor}_\textup{flat})^\flat$ improves to a diagram in relative categories. Pointwise inverting the $G$-global weak equivalences, we obtain the \emph{normed global category of global spectra} $\und{\Sp}_\gl^\tensor$.
    By construction, the localization functors assemble into a normed global functor $L\colon (\Sp^{\Sigma,\tensor}_\textup{flat})^\flat \to \und{\Sp}_\gl^\tensor$.
\end{construction}

\begin{remark}\label{rk:who-cares-about-flatness}
    By cofibrant replacement in the flat model structures of \cite{g-global}*{Theorem~3.1.40}, the underlying global category $\und{\Sp}_\gl \coloneqq \und{\Sp}_\gl^\tensor|_{\Fglo^\op}$
    can be identified with the global category of global spectra constructed in \cite{CLL_Global} as the analogous localization of the Borel category of \emph{all} (not necessarily flat) symmetric spectra. In particular, by the main theorem of \emph{op.~cit.} it has a universal property as the `free globally presentable equivariantly stable global category.' We will return to this universal property as well as alternative (less model-categorical) descriptions of $\ul\Sp_\gl$ and $\ul\Sp_\gl^\otimes$ in Section~\ref{sec:tambara}.
\end{remark}

\begin{remark}
    Building on the previous remark, one can equivalently obtain $\und{\Sp}_\gl^\tensor$ as an analogous localization of the Borelification of all symmetric spectra. However, in this case the norms a priori have to be derived, and it is not clear whether the localization functor is actually normed or just lax normed.\footnote{To our knowledge, the question whether the smash product of symmetric spectra is fully homotopical, is still open, see e.g.~the discussion in the introduction of \cite{sagave-schwede}.} For the details of this construction we refer the reader to \cite{puetzstueck-new}*{Construction 4.13}.
\end{remark}

We now recall the construction of the global comparison functor $\und{\UCom}_\gl \to \und{\CAlg}_\gl(\und{\Sp}_\gl^\tensor)$ from \cite[Construction 4.23]{puetzstueck-new}.

\begin{construction}\label{constr:comparison}
    Since $L$ is a normed global functor, we obtain an induced global functor
    \[
        \und{\CAlg}_\gl(L) \colon \und{\CAlg}_\gl((\Sp^{\Sigma,\tensor}_\textup{flat})^\flat )\to \und{\CAlg}_\gl(\und{\Sp}^\tensor_\gl).
    \]
    By \cite{puetzstueck-new}*{Proposition~3.12}, the left hand side is equivalent to the global category $\big(\CAlg(\Sp^{\Sigma,\tensor}_\textup{flat})\big)^\flat$ in a way that is compatible with the evident forgetful functors to $(\Sp^{\Sigma,\tensor}_\textup{flat})^\flat$. In particular, we see that the composite
    \[
        \big({\CAlg(\Sp^{\Sigma,\tensor}_\textup{flat})}\big)^\flat\to\und{\CAlg}_\gl(\und{\Sp}^\tensor_\gl)
    \]
    inverts all $G$-global weak equivalences. It therefore descends to a functor
    \[
        \Phi\colon\ul\UCom_\gl\to\und{\CAlg}_\gl(\und{\Sp}^\tensor_\gl),
    \]
    where the source denotes the localization of $\big({\CAlg(\Sp^{\Sigma,\tensor}_\textup{flat})}\big)^\flat$ at the $G$-global weak equivalences.
    In summary, we then have a commutative diagram
    \[\begin{tikzcd}[cramped]
        {\und{\CAlg}_\gl((\Sp^{\Sigma,\tensor}_{\textup{flat}})^\flat)} &[2em] {\und{\CAlg}_\gl(\und{\Sp}_\gl^\tensor)} \\
        {(\CAlg(\Sp^{\Sigma,\tensor}_\textup{flat}))^\flat} & {\und{\UCom}_\gl}\rlap.
        \arrow["{\und{\CAlg}_\gl(L)}", from=1-1, to=1-2]
        \arrow["\sim"', from=1-1, to=2-1]
        \arrow["\cL"', from=2-1, to=2-2]
        \arrow["\Phi"', from=2-2, to=1-2]
    \end{tikzcd}\]
    As the left hand vertical map lies over the forgetful functors to $(\Sp^{\Sigma,\otimes}_\text{flat})^\flat$, the universal property of localization implies that $\Phi$ fits into a commutative diagram
    \[
        \begin{tikzcd}[column sep=small]
            \ul\UCom_\gl\arrow[rr,"\Phi"]\arrow[dr, bend right=15pt, "\mathbb U"'] &[.2em]&[-.2em] \ul\CAlg_\gl(\ul\Sp_\gl^\otimes)\arrow[dl, bend left=15pt, "\mathbb U"]\rlap.\\
            & \ul\Sp_\gl
        \end{tikzcd}
    \]
\end{construction}

\begin{remark}\label{rk:ucom-model}
    Similarly to Remark~\ref{rk:who-cares-about-flatness}, the category $\CAlg(\text{$G$-Sp}^\Sigma)$ carries a \emph{positive flat $G$-global model structure} \cite{Lenz-Stahlhauer}*{Proposition~5.7}, such that the cofibrant objects are in particular flat spectra when we forget the ring structure \cite{shipley-convenient}*{Corollary~4.3}. Analogously to Remark~\ref{rk:who-cares-about-flatness}, we therefore could have equivalently defined $\ul\UCom_\gl$ by localizing the Borel global category of all commutative algebras in symmetric spectra.
\end{remark}

\begin{remark}
    In \cite{schwede2018global}*{Theorem~5.4.3}, Schwede introduces a \emph{global model structure} on the category of strictly commutative algebras in orthogonal spectra, and calls the objects of the resulting localization \emph{ultra-commutative global ring spectra}. Schwede's model structure records equivariant information for all compact Lie groups, but as we show in Appendix~\ref{app:orthogonal}, once one restricts to information about \emph{finite} groups, the resulting $\infty$-category can be equivalently modelled by Hausmann's global model structure on $\CAlg(\Sp^\Sigma)$ \cite{hausmann-global}*{Theorem~3.5}, which literally agrees with the flat ($1$-)global model structure from Remark~\ref{rk:who-cares-about-flatness}. In other words, the underlying category $\UCom_\gl$ of $\ul\UCom_\gl$ is given by Schwede's category of ultra-commutative global ring spectra, localized at the global family of finite groups.
\end{remark}

We can now state the following precise version of Theorem~\ref{introthm:global-param} from the introduction:

\begin{theorem}\label{thm:global-model}
    The functor $\Phi\colon\ul\UCom_\textup{gl}\to\ul\CAlg_\gl(\ul\Sp_\textup{gl}^\otimes)$ is an equivalence.
\end{theorem}

The proof will occupy the rest of this section. Before we dive into the details, let us give a general overview of the proof strategy: By construction, $\Phi\colon\ul\UCom_\gl\to\ul\CAlg_\gl(\ul\Sp_\gl^\otimes)$ is compatible with the forgetful functors $\mathbb U$ to $\ul\Sp_\gl$. We will show below that both of these forgetful functors are monadic right adjoints; by abstract nonsense about monadicity, the theorem will then follow once we can show that $\Phi$ also commutes with the left adjoints to the forgetful functors (via the Beck--Chevalley map), and the hard part is understanding the two left adjoints.

In §\ref{subsec:spgl-dist} we will show that $\ul\Sp_\gl^\otimes$ is distributive. By our earlier results, this implies that $\mathbb U\colon\ul\CAlg_\gl(\ul\Sp^\otimes_\gl)\to\ul\Sp_\gl$ indeed admits a left adjoint $\Ppar$ (for `parametrized') such that the resulting adjunction is monadic; at the same time, it will also provide a formula for the left adjoint. We will also show that the $1$-categorical model $(\Sp^{\Sigma,\otimes}_\text{flat})^\flat$ is at least weakly distributive, which again provides us with an explicit identification of the left adjoint $\Pmod \colon(\Sp^{\Sigma}_\text{flat})^\flat\to\ul\CAlg_\gl\big((\Sp^{\Sigma,\otimes}_\text{flat})^\flat\big)$ (for `model'). Note that while this description could have been easily obtained by hand, the advantage of our approach is that we immediately get for free that this identification is compatible with the normed functor $L$, which is a key ingredient of the proof of Theorem~\ref{thm:global-model}.

In §\ref{subsec:monadicity} we show that $\Pmod$ can be left derived to a functor $\Pder\colon\ul\Sp_\gl\to\ul\UCom_\gl$ left adjoint to the forgetful functor $\mathbb U$, and that the resulting adjunction is still monadic.

In §\ref{subsec:proof-ucom-vs-nalg} we will then combine all of the above pieces to compare the left adjoints $\Ppar$ and $\Pder$ and deduce the theorem. This will proceed by in turn comparing both of these left adjoints to $\Pmod$: as already mentioned, the comparison between $\Pmod$ and $\Ppar$ will follow almost for free from Theorem~\ref{thm:free-algebra-formula}; the comparison between $\Pmod$ and $\Pder$ on the other hand relies on some model categorical results from \cite{g-global,Lenz-Stahlhauer}.

\subsection{Distributivity of global spectra}\label{subsec:spgl-dist} As promised, we begin by proving that the normed global categories involved are (weakly) distributive. We first consider the easier $1$-categorical case:

\begin{lemma}\label{lemma:model-weak-dist} The normed global category
$(\Sp^{\Sigma,\tensor}_\textup{flat})^{\flat}$ is weakly distributive.
\end{lemma}

\begin{proof}
    We consider the fully faithful normed inclusion $(\Sp^{\Sigma,\tensor}_\textup{flat})^{\flat}\subset (\Sp^{\Sigma,\tensor})^{\flat}$ induced by the symmetric monoidal inclusion $\Sp^{\Sigma,\otimes}_\text{flat}\subset\Sp^{\Sigma,\otimes}$. The right hand side is the monoidal global Borelification of a presentably symmetric monoidal category, hence distributive (Example~\ref{ex:borel-distributive}). To show that the full normed subcategory $(\Sp^{\Sigma,\tensor}_\textup{flat})^{\flat}$ is weakly distributive, it therefore suffices to show that flat spectra are closed under sequential colimits and the left adjoints to restrictions. For the closure under sequential colimits, note that by definition a spectrum is flat if and only if for every $n\ge0$ a certain natural \emph{latching map} $L_n X\to X_n$ is a cofibration of simplicial sets, i.e.~levelwise injective. The \emph{latching object} $L_nX$ is defined as a certain colimit of the entries $X_m$ of $X$, and in particular it preserves sequential colimits. The claim therefore follows as sequential colimits of injections of sets are injections again.

In the same way one shows that flat spectra are closed under coproducts (which is also a formal consequence of them being the cofibrant objects of a model category) as well as quotients by group actions (which is not formal, and uses that injections of $G$-sets have complements). Since any left adjoint to restriction is given on underlying non-equivariant spectra by a quotient of a finite coproduct by a group action, this proves the remaining closure properties.
\end{proof}

\begin{remark}
    We do not know whether $(\Sp^{\Sigma,\tensor}_\text{flat})^\flat$ is even distributive. Note that the subcategory $\Sp^\Sigma_\text{flat}$ is \emph{not} closed under all colimits in $\Sp^\Sigma$ (for example, because it contains all free symmetric spectra and these generate $\Sp^\Sigma$ under colimits)---however, this does not rule out that $\Sp^\Sigma_\text{flat}$ might have all colimits (still preserved by the smash product), which would imply distributivity of $(\Sp^{\Sigma,\tensor}_\text{flat})^\flat$.
\end{remark}

We now move on to the other global category appearing in our comparison.

\begin{theorem}\label{thm:spgl-dist}
    The normed global category $\ul\Sp^\otimes_\textup{gl}$ is distributive.
\end{theorem}

The proof will require some preparations. We begin with the following lemma, which will be instrumental in proving that restrictions and norms for $\ul\Sp_\gl^\tensor$ preserve $\Delta^\op$-indexed colimits.

\begin{lemma}\label{lemma:geometric-real-homotopical}
    For any finite group $G$, geometric realization in $\textup{$G$-Sp}^\Sigma$ preserves $G$-global weak equivalences.
    \begin{proof}
        There is a (simplicial) \emph{injective $G$-global model structure} on $\textup{$G$-Sp}^\Sigma$, whose cofibrations are just the levelwise injections \cite{g-global}*{Corollary~3.1.46}. By definition, a simplicial object in $\textup{$G$-Sp}^\Sigma$ is then Reedy cofibrant with respect to this model structure, if and only if it is levelwise given by Reedy cofibrant diagrams in $\text{SSet}$. But \emph{every} simplicial object in $\text{SSet}$ is Reedy cofibrant, hence so is every simplicial object in $\textup{$G$-Sp}^\Sigma$. As geometric realization in any simplicial model category is left Quillen for the Reedy model structure, the claim follows from Ken Brown's Lemma.
    \end{proof}
\end{lemma}

Verifying the basechange condition for the left adjoints to restrictions will be significantly more involved. Ultimately, we want to reduce this to the case of the normed global $1$-category $(\Sp^{\Sigma,\otimes}_\text{flat})^\flat$ considered above; however, while the localization functor $L\colon (\Sp^{\Sigma,\otimes}_\text{flat})^\flat\to\ul\Sp_\gl^\otimes$ preserves restrictions and norms, it does \emph{not} commute with left adjoints in general as the latter do not preserve weak equivalence between flat spectra as a consequence of \cite{g-global}*{Example~3.1.51}. The following lemma will serve as a replacement:

\begin{lemma}\label{lemma:compute-left-derived}
    Let $f\colon H\to G$ be a homomorphism of finite groups. Then the restriction $f^*\colon \Sp_{G\textup{-gl}}\to\Sp_{H\textup{-gl}}$ admits a left adjoint $f_!$. If $X\in\textup{$H$-Sp}^\Sigma_\textup{flat}$, then the Beck--Chevalley map $f_!LX\to Lf_!X$ is an equivalence in each of the following cases:
    \begin{enumerate}
        \item The $H$-symmetric spectrum $X$ is cofibrant in the \emph{projective} $H$-global model structure of \cite{g-global}*{Theorem~3.1.41}.
        \item The underlying $\ker(f)$-symmetric spectrum of $X$ is cofibrant in the projective $\ker(f)$-global model structure.
        \item The subgroup $\ker(f)$ acts levelwise freely on $X$ outside the basepoint.
    \end{enumerate}
    Moreover, $(1)\Rightarrow(2)\Rightarrow(3)$.
\end{lemma}

To simplify our terminology, we will simply refer to the cofibrant objects of the $H$-global projective model structure as \emph{projective} below.\footnote{The reader familiar the model categorical approach to classical equivariant stable homotopy theory should be warned that being cofibrant in the projective \emph{$G$-global} model structure is very different from being cofibrant in the projective \emph{$G$-equivariant} model structure, cf.~\cite{g-global}*{Warning~3.1.22}. For example, the implication $(1)\Rightarrow(3)$ in the above lemma shows that the sphere is not $G$-globally projectively cofibrant unless $G=1$.}

    \begin{proof}
        \cite{g-global}*{Lemma~3.1.49} shows that $f_!\colon\text{$H$-Sp}^\Sigma\to\text{$G$-Sp}^\Sigma$ is left Quillen for the projective model structures. We conclude that $f^*\colon \Sp_{G\textup{-gl}}\to\Sp_{H\textup{-gl}}$ has a left adjoint given as the left derived functor of $f_!$; in particular, the Beck--Chevalley map is an equivalence on projectively cofibrant spectra.

        We now observe that $(1)\Rightarrow(2)$ by \cite{Lenz-Stahlhauer}*{Proposition~1.44} while $(2)\Rightarrow(3)$ by \cite{g-global}*{Remark~3.1.22}. It then only remains to show that the Beck--Chevalley map $f_!LX\to Lf_!X$ is invertible whenever $\ker(f)$ acts freely on $X$. Unravelling definitions, this amounts to saying that if $p\colon X'\to X$ is a projectively cofibrant replacement, then $f_!(p)$ is a $G$-global weak equivalence, which is a direct consequence of \cite{g-global}*{Proposition~3.1.54}.
    \end{proof}

\begin{remark}
    Note that restriction along $f\colon H\to G$ does \emph{not} preserve projectivity unless $f$ is injective, i.e.~we can't avoid the above problem by defining $\ul\Sp^\otimes_\gl$ in terms of the projective model structures instead. In fact, as explained in a more general context in \cite{Lenz-Stahlhauer}*{Remark~2.13}, having to deal with two different model structures and their interplay is unavoidable: if we could find for every $G$ a model structure on $G$-symmetric spectra with weak equivalences the $G$-global weak equivalences and such that restrictions are both left and right Quillen, this would for every $p\colon G\to 1$ force the inflation $p^*\colon\Sp_\gl\to\Sp_\text{$G$-gl}$ to be fully faithful, so that in particular $p_*p^*\mathbb S\simeq\mathbb S$. However, the tom Dieck splitting shows that the two sides are not even non-equivariantly equivalent unless $G$ is trivial.
\end{remark}

\begin{corollary}\label{cor:forb-homotopical}
    Let $f\colon H\to G$ be any map in $\Forb$. Then $f^*\colon\ul\Sp_\gl(G)\to \ul\Sp_\gl(H)$ has a left adjoint $f_!$. Moreover, the Beck--Chevalley transformation $f_!L\to Lf_!$ is an equivalence.
    \begin{proof}
        The map $f$ can be factored into a sequence of maps each of which is a coproduct of injective homomorphisms together with fold maps $K\amalg K\to K$ (for varying finite groups $K$). Because $\ul\Sp_{\gl}$ preserves products, it therefore suffices to treat the case that $f\colon H\to G$ is an injective group homomorphism or a fold map $f\colon G\amalg G\to G$.

        The first case is an immediate consequence of Lemma~\ref{lemma:compute-left-derived}, as the third condition is vacuous for injective homomorphisms. On the other hand, in the second case $f_!$ is simply given by the coproduct, so the claim follows from the fact that the coproduct of $G$-symmetric spectra preserves $G$-global weak equivalences by \cite{g-global}*{Lemma~3.1.43}.
    \end{proof}
\end{corollary}

Finally, we will need the following lemma producing projective  $(G\times H)$-global spectra for us:

\begin{lemma}\label{lemma:projective-smash}
    Let $G,H$ be finite groups, and let $X\in\textup{$G$-Sp}^\Sigma, Y\in\textup{$H$-Sp}^\Sigma$ be projective. Then the $(G\times H)$-symmetric spectrum $\Inf^G_{G \times H}X\smashp \Inf^{H}_{G \times H} Y$ is again projective.
\end{lemma}

Beware that $\Inf^G_{G \times H}X$ alone will \emph{not} be projective unless $H=1$ or $X=0$, and similarly for $\Inf^H_{G \times H}Y$.

\begin{proof}
    We will more generally show that the pushout product of any $G$-global projective cofibration with an $H$-global projective cofibration is a $(G\times H)$-global projective cofibration. As usual, it suffices to consider the case of generating cofibrations, which by \cite{g-global}*{Proposition~3.1.20} are precisely those of the form
    \[
        \bm\Sigma(A,-)\smashp_{\Sigma_A} X_+\smashp (\partial\Delta^n\hookrightarrow\Delta^n)_+
    \]
    where $\bm\Sigma$ is the indexing category for symmetric spectra, $A\in\bm\Sigma$, $n\ge 0$, and $X$ is a transitive $(\Sigma_A\times G)$-set such that the $G$-action is free. As coproducts of cofibrations are cofibrations, we see that the above is more generally a cofibration for any (not necessarily transitive) $(\Sigma_A\times G)$-set with free $G$-action.

    Similarly, a set of generating cofibrations for the $H$-global projective model structure is given by the maps
    \[
        \bm\Sigma(B,-)\smashp_{\Sigma_B} Y_+\smashp (\partial\Delta^m\hookrightarrow\Delta^m)_+
    \]
    for $(\Sigma_B\times H)$-sets $Y$ with free $H$-action. The pushout product of two such maps is then given by
    \begin{equation}\label{eq:to-show-proj}
        \bm\Sigma(A\amalg B,-)\smashp_{\Sigma_{A\amalg B}} Z_+\smashp \big((\partial\Delta^n\hookrightarrow\Delta^n){\scriptstyle\square}(\partial\Delta^m\hookrightarrow\Delta^m)\big)_+
    \end{equation}
    where $\scriptstyle\square$ refers to the pushout product in simplicial sets and $Z=\Sigma_{A\amalg B}\times_{\Sigma_A\times\Sigma_B}(X\times Y)$. By assumption on $X$ and $Y$, $G\times H$ acts freely on $X\times Y$, and hence on $Z$. Accordingly, the left adjoint functor
    $\bm\Sigma(A\amalg B,-)\smashp_{\Sigma_{A\amalg B}} Z_+\smashp (-)$
    sends generating cofibrations of simplicial sets to $(G\times H)$-global projective cofibrations of $(G\times H)$-symmetric spectra. Thus, it already sends all cofibrations of simplicial sets to $(G\times H)$-global projective cofibrations, and in particular $(\ref{eq:to-show-proj})$ is a $(G\times H)$-global projective cofibration by the pushout product axiom for ordinary simplicial sets.
\end{proof}

\begin{proof}[Proof of Theorem~\ref{thm:spgl-dist}]
    We recommend the reader skips this argument on first reading, and instead enjoys a nice book or a beautiful walk.

    As shown in \cite{CHLL_NRings}*{proof of Proposition~5.5.7}, flat symmetric spectra are closed under geometric realization and the smash product functor $\Sp^\Sigma\times\Sp^\Sigma\to\Sp^\Sigma$ preserves geometric realizations up to \emph{isomorphism}. In particular, we see that all norms in $(\Sp^{\Sigma,\otimes}_{\text{flat}})^\flat$ preserve geometric realizations up to isomorphism. By \cite[Corollary 5.19]{Lenz-Stahlhauer} and \cref{lemma:geometric-real-homotopical} both norms and geometric realization are homotopical, so we conclude that norms in $\ul\Sp_\gl^\otimes$ preserve $\Delta^\op$-shaped colimits. An analogous argument shows that they also preserve all filtered colimits (using that they are again homotopical \cite{g-global}*{Corollary~3.1.46}), hence all sifted colimits. On the other hand, restrictions of $\ul\Sp_\gl$ even preserve arbitrary colimits, as they are left adjoints by \cite{g-global}*{Lemma~3.1.50}. This verifies condition (1) of \Cref{def:distributivity}.

  	As recalled in Remark~\ref{rk:who-cares-about-flatness}, the underlying global category $\ul\Sp_\gl$ is cocomplete, so the left adjoints to restriction satisfy basechange with respect to maps in $\Fglo$. This shows (2) and (3) of \Cref{def:distributivity};  it remains to verify (4), the Beck--Chevalley condition for commuting norms with left adjoints to restrictions. This, however, will involve substantial work.

    We consider a pullback diagram in $\Span_{\Forb}(\Fglo)^\op$ (cf.~\cref{rem:dist_diag_comm_in_span})
    of the form
    \begin{equation}\label{diag:spanop-pb}
        \begin{tikzcd}[column sep=small, row sep=1.25em]
            A & A & B \\
            E & E & D \\
            C & C & D\rlap.
            \arrow[equal, from=1-2, to=1-1]
            \arrow[equal, from=2-2, to=2-1]
            \arrow[equal, from=3-2, to=3-1]
            \arrow[equal, from=2-3, to=3-3]
            \arrow["f", from=1-2, to=1-3]
            \arrow["q"', from=2-1, to=3-1]
            \arrow["q", from=2-2, to=3-2]
            \arrow[norm, from=2-1, to=1-1]
            \arrow[norm, from=2-2, to=1-2]
            \arrow[norm, from=2-3, to=1-3]
            \arrow[from=2-2, to=2-3]
            \arrow[from=3-2, to=3-3]
            \arrow["\lrcorner"{anchor=center, pos=0.125, rotate=90}, draw=none, from=2-2, to=1-3]
            \arrow["\lrcorner"{anchor=center, pos=0.125, rotate=-90}, draw=none, from=2-2, to=3-1]
        \end{tikzcd}
    \end{equation}
    We may factor $C\to D$ as a coproduct of surjective homomorphisms (i.e.~a map in $\mathbb E$) followed by a map in $\Forb$. Because Beck--Chevalley maps compose, it suffices to prove that norms commute with each class of maps separately.

    Assume first that $C\to D$ is in $\Forb$, hence so is $A\to B$ by Remark~\ref{rk:goo-ggg-pb}. By Corollary~\ref{cor:forb-homotopical}, $L\colon (\Sp^{\Sigma,\otimes}_\text{flat})^\flat\to\ul\Sp_\gl^\otimes$ commutes with the left adjoints to restriction along $A\rightarrowmono B$ and $C\rightarrowmono D$. As $L$ is essentially surjective, the adjointability for $(\Sp^{\Sigma,\otimes}_\text{flat})^\flat$ proven in Lemma~\ref{lemma:model-weak-dist} implies the corresponding statement for $\ul\Sp_\gl^\otimes$.

    Next, we consider the case that $C\to D$ is in $\mathbb E$. To begin, we decompose $B$ into its components, which induces decompositions of $D$ and $C$ by pulling back. As coproducts in $\Fglo$ yield products in $\Span_{\Forb}(\Fglo)^\op$ by semiadditivity, and since pullbacks are stable under products, we may then decompose $(\ref{diag:spanop-pb})$ into a (co)product of pullbacks. As adjointability of each of these summands implies adjointability of the original diagram, we have reduced to the case that $B$ is a group (viewed as a $1$-object groupoid).

    We now want to prove adjointability by reducing to the underived statement again. As remarked above, $L$ does not commute with left adjoints to restriction in general, so some care will be needed. As before, we will call a $G$-symmetric spectrum \emph{projective} whenever it is cofibrant in the $G$-global projective model structure. More generally, for $G\in\Fglo$ with components $G_1,\dots,G_n$, we will call $X\in(\Sp^\Sigma_\text{flat})^\flat(G)$ \emph{projective} if its restriction to each individual $G_i$ is projective in the above sense.

    Let now $\coprod_{i=1}^n(p_i\colon C_i\rightarrowepic D_i)$ be the decomposition of $C\rightarrowepic D$. Then the lower right path from $C$ to $B$ through the diagram (\ref{diag:spanop-pb}) sends $(X_1,\dots,X_n)\in\ul\Sp_\gl^\otimes(C_1)\times\cdots\times\ul\Sp_\gl^\otimes(C_n)$ to $\smash{\bigotimes_{i=1}^n \Nm^B_{D_i} p_{i!}(X)}$. By Lemma~\ref{lemma:compute-left-derived} and the fact that $L$ is a normed functor, this can be computed by taking for each $i$ a projective representative $\overline{X}_i$ of $X_i$ in $\text{$C_i$-Sp}^\Sigma$ and performing the quotients and norms on the pointset level. We claim that plugging the same representatives into the upper left path still computes the corresponding derived functors; as adjointability holds on the pointset level by another application of Lemma~\ref{lemma:model-weak-dist}, this will then finish the proof.

    The only step in the upper left path of the diagram that needs deriving (given that we work in \emph{flat} symmetric spectra throughout) is the final step: to compute the left adjoint to restriction along $A\to B$, we ought to take an $A$-projective replacement. We want to show that this cofibrant replacement does not change the homotopy type of the resulting $B$-global spectrum, i.e.~the Beck--Chevalley map $f_!L\overline Y\to Lf_!\overline Y$ is invertible where $\overline Y\coloneqq \Nm^A_E q^*(\overline X_1,\dots,\overline X_n)\in(\Sp^{\Sigma,\otimes}_\textup{flat})^\flat(A)$.

    Write $A=\coprod_{i=1}^nA_i$ and let $K$ be the disjoint union of the kernels of the components $A_i\to B$; by Lemma~\ref{lemma:compute-left-derived} it will be enough to show that the restriction of $\overline{Y}$ to $K$ is projective.
    Consider for this the following pullback in $\Fglo$:
    \begin{equation}\label{eq:pb-no-spans}
        \begin{tikzcd}[cramped]
            L\arrow[r]\arrow[d,mono] & 1\arrow[d,mono]\\
            A\arrow[r] & B
            \arrow["\lrcorner"{anchor=center, pos=0.125}, draw=none, from=1-1, to=2-2]
        \end{tikzcd}
    \end{equation}
    By the universal property of the pullback, $K\to A$ factors through $L$, and by left cancellability the resulting map $K\to L$ is in $\Forb$.\footnote{In fact, it is not hard to show using \cite{CLL_Global}*{Lemma~4.3.2} that it is even the inclusion of a coproduct summand.}

    As restrictions along injective homomorphisms preserve projectivity by \cite{Lenz-Stahlhauer}*{Proposition~3.1.44}, it will suffice to show that the restriction of $\overline Y$ to $L$ is projective. Passing through the pullback-preserving functor $\Fglo \to \Span_{\Forb}(\Fglo)^\op$, we may interpret $(\ref{eq:pb-no-spans})$ as a pullback in $\Span_{\Forb}(\Fglo)^\op$ when we view it as a diagram of forward maps. Pasting pullback squares in $\Span_{\Forb}(\Fglo)^{\op}$, we see that we can compute $L$ as the pullback of $C=C\to D$ along the composite
    \[
        \begin{tikzcd}[column sep=1.75em]
            1 &\arrow[l,norm, "="'] 1 \arrow[r, mono] & B &\arrow[l, norm] D\arrow[r, "="] & D
        \end{tikzcd}
    \]
    in $\Span_{\Forb}(\Fglo)^\op$, which is a certain span \begin{tikzcd}[cramped,column sep=small] 1 &\arrow[l,norm] 1\times_B D \arrow[r] & D\end{tikzcd}. While it is not hard to compute $S\coloneqq 1\times_BD$ explicitly, all that we will need below is that it is a set (i.e.~a discrete groupoid), as it admits a faithful map to $1$.

    To summarize, we have now reduced to showing that for a pullback diagram
    \[\begin{tikzcd}[column sep=small, row sep=1.25em]
        L & L & 1 \\
        M & M & S \\
        C & C & D
        \arrow[equal, from=1-2, to=1-1]
        \arrow[equal, from=2-2, to=2-1]
        \arrow[equal, from=3-2, to=3-1]
        \arrow[norm, from=2-1, to=1-1]
        \arrow[norm, from=2-2, to=1-2]
        \arrow[norm, from=2-3, to=1-3]
        \arrow["q"', from=2-1, to=3-1]
        \arrow["q", from=2-2, to=3-2]
        \arrow[mono, from=2-3, to=3-3]
        \arrow[epic, from=3-2, to=3-3]
        \arrow[from=1-2, to=1-3]
        \arrow[from=2-2, to=2-3]
        \arrow["\lrcorner"{anchor=center, pos=0.125, rotate=90}, draw=none, from=2-2, to=1-3]
        \arrow["\lrcorner"{anchor=center, pos=0.125, rotate=-90}, draw=none, from=2-2, to=3-1]
    \end{tikzcd}\]
    in $\Span_{\Forb}(\Fglo)^\op$ with $S$ discrete, the composite
    $\Nm^L_Mq^*$ preserves projectivity. As usual, we can reduce this further by factoring the right vertical span into its forwards and backwards part. The pullback along $S=S\rightarrowmono D$ can then be simply computed in $\Fglo$, where it has the form
    \[
        \begin{tikzcd}[cramped]
            T\arrow[r,epic]\arrow[d,mono] & S\arrow[d,mono]\\
            C\arrow[r,epic] & D
            \arrow["\lrcorner"{anchor=center, pos=0.125}, draw=none, from=1-1, to=2-2]
        \end{tikzcd}
    \]
    with $T\rightarrowepic S$ again a coproduct of surjections and with $T\rightarrowmono C$ faithful; in particular, restriction along the latter preserves projectivity. As an upshot (replacing $D$ by $S$ and $C$ by $T$), we have reduced further to the case that $S\rightarrowmono D$ is the identity, i.e.~our diagram has the form
    \[\begin{tikzcd}[column sep=small, row sep=1.25em]
        L & L & 1 \\
        M & M & D \\
        C & C & D
        \arrow[equal, from=1-2, to=1-1]
        \arrow[equal, from=2-2, to=2-1]
        \arrow[equal, from=3-2, to=3-1]
        \arrow[equal, from=2-3, to=3-3]
        \arrow[norm, from=2-1, to=1-1]
        \arrow[norm, from=2-2, to=1-2]
        \arrow[norm, from=2-3, to=1-3]
        \arrow["q"', from=2-1, to=3-1]
        \arrow["q", from=2-2, to=3-2]
        \arrow[epic, from=3-2, to=3-3]
        \arrow[from=1-2, to=1-3]
        \arrow[from=2-2, to=2-3]
        \arrow["\lrcorner"{anchor=center, pos=0.125, rotate=90}, draw=none, from=2-2, to=1-3]
        \arrow["\lrcorner"{anchor=center, pos=0.125, rotate=-90}, draw=none, from=2-2, to=3-1]
    \end{tikzcd}\]
    with $D$ a finite set, and $C\to D$ of the form $\coprod_{d\in D}C_d\to D$ with groups $C_d$.

    We are now finally at the point where all the maps can easily be made explicit via the translation between pullback squares in spans and distributivity diagrams from Remark~\ref{rem:dist_diag_comm_in_span}: we have $L=\prod_{d\in D}C_d$, $M=L\times D$, and the map $q\colon M\to C$ is given on $\{d\}\times L$ by projecting to the $d$-th factor and then including. Thus, if $\overline X_d\in\text{$C_d$-Sp}^\Sigma_\text{flat}$, then $\Nm^L_Mq^*(\overline X_d)_{d\in D}=\bigwedge_{d\in D}\Inf^L_{C_d} \overline X_d$. Applying Lemma~\ref{lemma:projective-smash} inductively shows that this is indeed projective. This, finally, completes the proof.
\end{proof}

\subsection{Monadicity}\label{subsec:monadicity} In this subsection we will introduce the left adjoint of the forgetful functor $\mathbb U\colon\ul\UCom_\gl\to\ul\Sp_\gl$ and show that the resulting adjunction is monadic.

\begin{construction}\label{constr:p-derived-vs-mod}
    As recalled in Remark~\ref{rk:ucom-model}, $\CAlg(\text{$G$-Sp}^\Sigma)$ admits a model structure transferred from the \emph{positive flat $G$-global model structure} on $\text{$G$-Sp}^\Sigma$; in particular, the adjunction
    \[
        \Pmod\colon \text{$G$-Sp}^\Sigma \rightleftarrows \CAlg(\text{$G$-Sp}^\Sigma)\noloc\mathbb U
    \]
    is a Quillen adjunction for the positive flat model structure on the source. We conclude that $\mathbb U\colon\ul\UCom_\gl(G)\to \ul\Sp_\gl(G)$ admits a left adjoint $\Pder$, given by the left derived functor of $\Pmod$. The Beck--Chevalley transformation
    \begin{equation}\label{eq:Pder-vs-mod}
        \alpha\colon\Pder\circ L\to \Ll\circ\Pmod
    \end{equation}
     associated to the commutative square
    \[
        \begin{tikzcd}
            \CAlg(\text{$G$-Sp}^\Sigma_\text{flat})\arrow[d,"\mathbb U"']\arrow[r, "\mathcal L"] & \und{\UCom}_\gl(G)\arrow[d,"\mathbb U"]
            \\
            \text{$G$-Sp}^\Sigma_\text{flat} \arrow[r, "L"'] &
            \ul\Sp_\gl(G)
        \end{tikzcd}
    \]
    is an equivalence whenever $X\in\text{$G$-Sp}^\Sigma_\text{flat}$ is actually \emph{positively flat}, i.e.~cofibrant in the positive flat $G$-global model structure.
\end{construction}

\begin{lemma}
    The left adjoints $\mathbb P^\textup{der}$ assemble into a global functor $\ul\Sp_\gl\to\ul\UCom_\gl$ left adjoint to $\bbU$.
    \begin{proof}
        It suffices to verify the Beck--Chevalley condition, which follows immediately from the observation that also the restrictions are left Quillen for the positive flat model structures, see \cite{Lenz-Stahlhauer}*{Lemma~5.10}.
    \end{proof}
\end{lemma}

\begin{lemma}\label{lemma:U-derived-monadic}
    The adjunction
    \[
        \mathbb P^\textup{der}\colon\ul\Sp_\gl(G)\rightleftarrows\ul\UCom_\gl(G)\noloc\mathbb U
    \]
    is monadic for every finite group $G$.
    \begin{proof}
        By Lemma~\ref{lemma:geometric-real-homotopical}, geometric realizations in $\text{$G$-Sp}^\Sigma$ are homotopical. As the forgetful functor $\mathbb U\colon\CAlg(\text{$G$-Sp}^\Sigma)\to\text{$G$-Sp}^\Sigma$ of simplicial $1$-categories preserves geometric realization up to isomorphism, we conclude that also geometric realizations in $\CAlg(\text{$G$-Sp}^\Sigma)$ are homotopical, and that the resulting functor $\mathbb U\colon\ul\UCom_\gl(G)\to  \ul\Sp_\gl(G)$ on localizations preserves $\Delta^\op$-indexed colimits. As it is also conservative by definition, the claim follows from Barr--Beck--Lurie \cite[Theorem 4.7.3.5]{HA}.
    \end{proof}
\end{lemma}

\subsection{The comparison}\label{subsec:proof-ucom-vs-nalg}

We have constructed $\Pder\colon\Sp_\gl\to\UCom_\gl$ as the left derived functor of
\[
    \Pmod\colon\Sp^\Sigma_\text{flat}\to\CAlg(\Sp^\Sigma_\text{flat}),\
    X\mapsto\bigvee_{n\ge0} X^{\smashp n}/\Sigma_n
\]
Note that as predicted by Theorem~\ref{thm:free-algebra-formula}, this agrees with the formula
\begin{equation}\label{eq:Pmod-explicit}
    X\mapsto\bigvee_{n\ge 0} \Sub_{\Sigma_n}^1\Nm^{\Sigma_n}_{\Sigma_{n-1}}\Inf^1_{\Sigma_{n-1}}X,
\end{equation}
interpreted in the normed global category $(\Sp^{\Sigma,\otimes}_{\text{flat}})^\flat$.
We want to establish the analogous description for $\Pder$, which will ultimately amount to showing that each of the individual constituents of $(\ref{eq:Pmod-explicit})$ already computes the corresponding derived functor when $X$ is positively flat. As usual, the only tricky part will be the final subduction $\Sub_{\Sigma_n}^1 = (-)/\Sigma_n$, and in light of Lemma~\ref{lemma:compute-left-derived} this boils down to a question about the freeness of the action.
The following folklore result will turn out to be precisely what we need:

\begin{lemma}\label{lemma:smash-power-free}
    Let $Y$ be a positively flat symmetric spectrum. Then the $\Sigma_n$-action on $Y^{\smashp n}$ via permuting the factors is levelwise free.
    \begin{proof}
        Applying \cite{harper-corrigendum}*{Proposition~7.7*} with $B=\mathbb S,X=0$ shows that $(Y^{\smashp n})_m$ is cofibrant in the projective model structure on $\Fun(\Sigma_n,\text{SSet}_*)\cong\Fun(\Sigma_n,\text{SSet})_*$. As the cofibrant objects of the latter are precisely the $\Sigma_n$-simplicial sets on which $\Sigma_n$ acts freely outside the basepoint \cite{cellular}*{Proposition~2.16}, this proves the claim.
    \end{proof}
\end{lemma}

Putting everything together, we are now ready to prove our comparison theorem:

    \begin{proof}[{Proof of Theorem~\ref{thm:global-model}}]
        Fix a finite group $G$. By construction, we have a commutative diagram
        \[
            \begin{tikzcd}[column sep=small]
                \ul\UCom_\gl(G)\arrow[rr, "\Phi"]\arrow[dr, bend right=15pt, "\mathbb U"'] &[.5em]&[-.5em] \ul\CAlg_\gl(\ul\Sp_\gl^\otimes)(G)\arrow[dl, "\mathbb U", bend left=15pt]\rlap.\\
                &\ul\Sp_\gl(G)
            \end{tikzcd}
        \]
        By Lemma~\ref{lemma:U-derived-monadic} and \Cref{cor:par_U_monadic}, respectively, the two forgetful functors are part of monadic adjunctions
        \begin{align*}
            \Pder\colon\ul\Sp_\gl(G) &\rightleftarrows\ul\UCom_\gl(G)\noloc\mathbb U\\
            \Ppar\colon\ul\Sp_\gl(G) &\rightleftarrows\ul\CAlg_\gl(\ul\Sp^\otimes_\gl)(G)\noloc\mathbb U,
        \end{align*}
        so it will suffice by \cite{HA}*{Corollary 4.7.3.16} that the Beck--Chevalley map $\beta\colon\Ppar\to \Phi\circ\Pder$ is an equivalence of functors $\ul\Sp_\gl(G)\to\ul\CAlg_\gl(\ul\Sp_\gl^\otimes)(G)$.
        For this we consider the pasting of the square from \cref{constr:p-derived-vs-mod} with the above:
        \begin{equation}\label{diag:phi-extended}
            \begin{tikzcd}
                {\CAlg(\text{$G$-Sp}^{\Sigma}_\text{flat})}\arrow[r, "\mathcal L"]\arrow[d,"\mathbb U"'] &\ul\UCom_\gl(G) \arrow[r,"\Phi"]\arrow[d,"\mathbb U"] & \ul\CAlg_\gl(\ul\Sp_\gl^\otimes)(G)\arrow[d,"\mathbb U"]\\
                \text{$G$-Sp}^\Sigma_\text{flat}\arrow[r, "L"'] & \ul\Sp_\gl(G)\arrow[r,equal] &\ul\Sp_\gl(G)\rlap.
            \end{tikzcd}
        \end{equation}
        By the compatibility of mates with pastings, the Beck--Chevalley map associated to the total square agrees with the composite
        \[
            \Ppar LX
            \xrightarrow{\;\beta_{LX}\;}
            \Phi\Pder LX
            \xrightarrow{\;\Phi\alpha_{X}\;}\Phi\mathcal L\Pmod X,
        \]
        where $\alpha$ is the Beck--Chevalley transformation of the left square.
        As every object of $\ul\Sp^\tensor_\gl(G)$ is of the form $LX$ for a \emph{positively} flat $G$-symmetric spectrum $X$ (by cofibrant replacement in the $G$-global positive model structure), and since $\alpha_X$ is an equivalence for any such $X$ by Construction~\ref{constr:p-derived-vs-mod}, it suffices by 2-out-of-3 that the Beck--Chevalley map associated to the pasting $(\ref{diag:phi-extended})$ is an equivalence when restricted to positively flat $G$-symmetric spectra $X$.

        For this, recall the definition of $\Phi$ from \cref{constr:comparison},
        and note that the composite rectangle $(\ref{diag:phi-extended})$
        is equivalent to
        \[
            \begin{tikzcd}
                {\text{$G$-CAlg}(\Sp^{\Sigma}_\text{flat})}\arrow[r, "\sim"]\arrow[d,"\mathbb U"'] & \ul\CAlg_\gl((\Sp^{\Sigma,\otimes}_\text{flat})^\flat)(G)\arrow[d,"\mathbb U"]\arrow[r, "\ul\CAlg_\gl(L)"] &[2em] \ul\CAlg_\gl(\ul\Sp^\otimes_\gl)(G)\arrow[d,"\mathbb U"]\\
                \text{$G$-Sp}^\Sigma_\text{flat}\arrow[r,equal] & (\Sp^\Sigma_\text{flat})^\flat(G)\arrow[r, "L"'] & \ul\Sp_\gl(G)\rlap.
            \end{tikzcd}
        \]
        The Beck--Chevalley map of the left hand square is invertible as both horizontal maps are equivalences; using once more that Beck--Chevalley maps compose, it will therefore suffice that the Beck--Chevalley map $\gamma$ of the right hand square is an equivalence for every \emph{positively} flat $X\in\text{$G$-Sp}^\Sigma_\text{flat}$.

        For this, we write $\mathbb P^\text{bor}$ for the left adjoint of the middle vertical functor. Note that the adjunctions
        \begin{align*}
            \mathbb P^\text{bor}\colon(\Sp^{\Sigma}_\text{flat})^\flat &\rightleftarrows \ul\CAlg_\gl\big((\Sp^{\Sigma,\otimes}_\text{flat})^\flat\big) \noloc\mathbb U\\
            \mathbb P^\text{par}\colon\ul\Sp_\gl&\rightleftarrows \ul\CAlg_\gl(\ul\Sp_\gl^\otimes) \noloc\mathbb U
        \end{align*}
        are now both instances of Theorem~\ref{thm:free-algebra-formula}. Whether the Beck--Chevalley map $\gamma$ is an equivalence can be checked after applying the conservative functor $\mathbb U$, and the aforementioned Theorem~\ref{thm:free-algebra-formula} provides us with a commutative diagram
        \[
            \begin{tikzcd}[column sep=-15,cramped]
                \mathbb U\Ppar LX\arrow[r, "\mathbb U(\gamma)"]\arrow[d, "\sim"'] & \mathbb U\ul\CAlg_\gl(L)\mathbb P^\text{bor} X\arrow[r, "\sim"] & L\mathbb U\mathbb P^\text{bor} X\arrow[d, "\sim"]\\
                \bigoplus\limits_{n\ge0}\Sub_{\Sigma_n}^1\Nm^{\Sigma_n}_{\Sigma_{n-1}}\Inf_{\Sigma_{n-1}}^1 LX\arrow[rr, "\BC_!"']
                &&
                L\Big(\hskip1pt{\bigvee\limits_{n\ge 0} \Sub_{\Sigma_n}^1\Nm^{\Sigma_n}_{\Sigma_{n-1}}\Inf_{\Sigma_{n-1}}^1 X
                }\Big)
            \end{tikzcd}
        \]
        (where we write $\Nm^{\Sigma_n}_{\Sigma_{n-1}}$ instead of $\Nm^{\Sigma_n\times G}_{\Sigma_{n-1}\times G}$ etc.~for simplicity).
        It will therefore be enough by $2$-out-of-$3$ that the lower horizontal map is an equivalence for positively cofibrant $X$. This follows at once by combining the following observations:
        \begin{enumerate}
            \item $L$ commutes with all restrictions/inflations and norms (by construction).
            \item $L$ preserves coproducts by \cite{g-global}*{Lemma~3.1.43}.
            \item For every \emph{positively} flat $X$, the Beck--Chevalley map
                \[
                    \Sub_{\Sigma_n}^1 L(X^{\smashp n})
                    = \Sub_{\Sigma_n}^1L\Nm^{\Sigma_n}_{\Sigma_{n-1}}\Inf_{\Sigma_{n-1}}^1X
                    \to L(\Sub_{\Sigma_n}^1\Nm^{\Sigma_n}_{\Sigma_{n-1}}\Inf_{\Sigma_{n-1}}^1 X)
                \]
            is an equivalence by Lemma~\ref{lemma:compute-left-derived} combined with Lemma~\ref{lemma:smash-power-free}.\qedhere
        \end{enumerate}
    \end{proof}

    \begin{remark}
        It might be tempting to try and simplify the above argument by replacing $\Sp^{\Sigma,\otimes}_\text{flat}$ with the subcategory of positively flat symmetric spectra throughout. However, this is \emph{not} a symmetric monoidal subcategory of $\Sp^\Sigma$ as the unit (the sphere) is not \emph{positively} flat.
    \end{remark}

    \begin{remark}
        Roughly speaking, the above proof relied on being able to give a `uniform' description for the free algebra functors that simultaneously captures the case of strict algebras on the pointset level as well as the case of parametrized algebras on the $\infty$-categorical level, with this description being simple enough to control the process of deriving it. In view of Remark~\ref{rk:equivariant-sucks}, it is not immediately clear how to adapt this strategy to give a parametrized description of strictly commutative $G$-equivariant ring spectra. Instead, we will deduce the latter comparison in Section~\ref{sec:globalize-equivariant-alg} from the global comparison,  building on the results of \cite{CLL_Clefts}.
    \end{remark}

\section{Norms in a globalized world}\label{sec:norms-and-globalization}

Having proven the equivalence of ultra-commutative global ring spectra and normed algebras, we will now begin our process of understanding the situation at a fixed group $G$. As we just mentioned, our approach will be to reduce to the global statement, by exploiting the strong connection between global and equivariant spectra. This connection is for our purposes best explained by the main result of \cite{Linskens2023globalization}, which expresses the global category of global spectra as a \emph{globalization} of the global category of equivariant spectra. We will begin by recalling this construction.

\subsection{Partially lax functors and rigidification}
Consider a pair $(T,M)$ of a category $T$ together with a wide subcategory $M \subset T$.
In the following we will need to consider functors of $T$-(pre)categories which are lax \emph{away} from $M$.
More specifically, in the unstraightened picture, this consists of a functor between cocartesian fibrations
over $T^\op$ which is only required to preserves cocartesian edges over $M^\op$.
We will refer to such functors as \emph{$\neg M$-lax $T$-functors}.

\begin{definition}
    We write $\PreCat{T}^{\negMlax}$ for the subcategory of $\Cat_{/T^{\op}}$ spanned by the cocartesian fibrations and morphisms which preserve cocartesian edges over maps in $M^{\op}$.
    If $T$ furthermore has finite coproducts, we write $\iCat{T}^{\negMlax}$ for the full subcategory of $\PreCat{T}^{\negMlax}$ spanned by the (unstraightenings of) $T$-categories.
\end{definition}

\begin{notation}
    As usual, these categories canonically enhance to $(\infty,2)$-categories,
    and we denote the hom category between $\Cc$ and $\Dd$ by $\Fun_T^{\negMlax}(\Cc,\Dd)$.
\end{notation}

\begin{example}
    We may apply the previous definition to the pair $(\Span_{\cN}(\cF)^{\op},\cF)$, obtaining a category $\PreCat{\Span_{\cN}(\cF)^{\op}}^{\neglax{\cF}}$. This is equivalent to the full subcategory of $\PreOp{\cF}{\cN}$ spanned by the $\cN$-normed $\cF$-precategories. Similarly, the category $\iCat{\Span_{\cN}(\cF)^{\op}}^{\neglax{\cF}}$ is equivalent to the full subcategory of $\Op{\cF}{\cN}$ spanned by the $\cN$-normed $\cF$-categories.
    We will therefore denote these categories by $\PreNmCat{\cF}{\cN}^{\cN\textup{-lax}}$ and $\NmCat{\cF}{\cN}^{\cN\textup{-lax}}$ respectively.
\end{example}

\begin{example}
    In the case $\Span_\cN(\cF) = \Span(\F)$,
    we see that $\NmCat{\cF}{\cN}^{\cN\textup{-lax}}$
    is the category of symmetric monoidal categories and lax symmetric monoidal functors.
\end{example}

\begin{definition}
    Observe that the (unstraightened) Yoneda embedding defines a functor $\und{(-)}\colon T^\op \to \iCat{T}^{\negMlax}$.
    We define the \emph{$M$-rigidification} of a $T$-precategory $\Cc$ to be the functor
    \[
        \Rig_M^T\cC \colon T^\op \to \Cat,\quad A\mapsto \Fun^{\negMlax}_T(\und{A},\Cc).
    \]
    In the special case $(T,M) = (\Fglo,\Forb)$, we will refer to $\Rig_M^T\Cc$ as the \emph{globalization} of $\Cc$, following \cite{Linskens2023globalization}, and denote it by $\Glob(\Cc)$.
    This construction clearly assembles into a functor $\Rig_M^T \colon \PreCat{T}^{\negMlax} \to \PreCat{T}$.
\end{definition}

\begin{remark}
It is worthwhile to unwind the definition of $\Rig_M^T\Cc(A)$. This category is given by the category of $\neg M$-lax functors $s\colon \underline{A}\rightarrow \Cc$, i.e.~by the category of those functors
\[\begin{tikzcd}
	{(T_{/A})^{\op}} && {\Un^\co (\Cc)} \\
	& {T^{\op}}
	\arrow["s", from=1-1, to=1-3]
	\arrow[from=1-1, to=2-2]
	\arrow[from=1-3, to=2-2]
\end{tikzcd}\]
which preserve cocartesian lifts of maps in $M^\op$.
Given an object $B\to A$ over $A$, $s(B)$ is an object of $\Cc(B) \subset \Un^\co(\Cc)$, the fiber over $B$. Given a map $f\colon B\to B'$ of objects over $A$, we obtain a map $s_f\colon f^*s(B')\to s(B)$ by factoring the map $s(B')\to s(B)$ into a cocartesian edge followed by an edge in the fiber $\Cc(B)$. The condition for $s$ to be $\neg M$-lax implies that $s_f$ is an equivalence whenever $f$ is in $M$.
\end{remark}

\begin{definition}\label{def:universal-not-M-lax-funct}
Consider the $T$-functor $\Delta\colon \Cc\to \Rig_M^T\Cc$ given at $A\in T$ by
\[
    \Delta(A)\colon \Cc(A)
    \simeq \Fun_T(\und{A}, \Cc)
    \hookrightarrow \Fun_T^{\negMlax}(\und{A}, \Cc)
    \buildrel\text{def}\over= \Rig_M^T\Cc(A),
\]
the inclusion of honest $T$-functors into $\neg M$-lax $T$-functors $\und{A}\to \Cc$. One can show, see \cite[Construction 4.7]{Linskens2023globalization}, that each functor $\Delta(A)$ admits a right adjoint
\[
    \mathrm{ev}_A\colon \Rig_M^T\Cc(A)\rightarrow \Cc(A),
\]
which is given by evaluating at $\id_A \in \und{A}(A)$. Evaluating the Beck--Chevalley transformation $f^* \mathrm{ev}_{B'}\to\mathrm{ev}_Bf^*$ associated to a map $f\colon B\rightarrow B'$ at some $s \in \Rig_M^T\cC(B')$, we find that it is precisely given by the lax structure map $s_f\colon f^*s(B')\rightarrow s(B)$. In particular, it is an equivalence for $f\in M$. Therefore \cite{Linskens2023globalization}*{Lemma 3.19} implies that the functors $\mathrm{ev}_A$ assemble into an $\neg M$-lax functor, and ultimately into a natural transformation
\[
    \mathrm{ev}\colon \Rig_M^T \Rightarrow \id_{\PreCat{T}^\negMlax}.
\]
\end{definition}

A beautiful observation of Abell\'an gives the following universal property of this construction.

\begin{theorem}[{\cite[Theorem 4]{Abellan23}}]
The natural transformation $\mathrm{ev}\colon \Rig_M^T \Rightarrow \id$ exhibits
\[
    \Rig_M^T \colon \PreCat{T}^{\negMlax}\to \PreCat{T}
\]
as a right adjoint to the inclusion $\PreCat{T} \hookrightarrow \PreCat{T}^{\negMlax}$. Moreover, the adjunction $\incl\dashv\Rig_M^T$ is $\Cat$-linear.
\end{theorem}

\begin{proof}
    We claim that the (honest) $T$-functor $\Delta\colon \Cc\to \Rig_M^T\Cc$ is a compatible unit of the putative adjunction of categories. The triangle identities are verified in the proof of \cite[Theorem 4.10]{Linskens2023globalization}. Finally, note that the left adjoint $\incl$ is obviously $\Cat$-linear.
\end{proof}

Let us now mention the key example for our purposes, justifying the name `globalization' for $\Glob \coloneqq \Rig^{\Fglo}_{\Forb}$:

\begin{example}\label{ex:globalization}
    For every finite group $G$, we define $\text{$G$-Sp}^\Sigma_\text{equiv.~proj.}$ as the full subcategory of the category of $G$-symmetric spectra spanned by those objects that are cofibrant in the equivariant projective model structure of \cite{hausmann-equivariant}*{Theorem~4.8}. By \cite{CLL_Clefts}*{Lemma~9.3} these assemble into a global subcategory $\ul\Sp^{\Sigma}_\text{equiv.~proj.}$ of the global $1$-category $(\Sp^\Sigma)^\flat$, and this improves to a diagram in relative categories by equipping $G$-symmetric spectra with the \emph{equivariant stable weak equivalences} of \cite{hausmann-equivariant}*{Definition~2.35}. We denote the corresponding localization by $\ul\Sp$ and call it the \emph{global category of equivariant spectra}.

    The inclusions $\text{$G$-Sp}^\Sigma_\text{equiv.~proj.}\hookrightarrow\text{$G$-Sp}_\text{flat}^\Sigma$ are then homotopical for the $G$-global weak equivalences on the target by \cite{g-global}*{Proposition~3.3.1}, so we obtain a global functor $\iota_!\colon\ul\Sp\hookrightarrow\ul\Sp_\gl$. This admits an $\neg\Forb$-lax right adjoint $\iota^*\colon\ul\Sp_\gl\to\ul\Sp$, induced pointwise by the localization functors, see~\cite{CLL_Clefts}*{Lemma~9.12}.

    Comparing the universal properties of both sides (see \cite{CLL_Global}*{Theorem~7.3.2} and \cite{CLL_Clefts}*{Theorem~9.4}, respectively), the second author showed \cite{Linskens2023globalization}*{Theorem~5.14} that we have an equivalence $\ul\Sp_\gl\simeq\Glob(\ul\Sp)$ fitting into a commutative diagram
    \[
        \begin{tikzcd}[column sep=small]
            &[-.4em]\ul\Sp\arrow[dl, bend right=10pt, "\Delta"']\arrow[dr, bend left=10pt, "\iota_!"]\\
            \Glob(\ul\Sp)\arrow[rr,"\sim"']&&\ul\Sp_\gl\rlap.
        \end{tikzcd}
    \]
    In other words, $\iota^*\colon\ul\Sp_\gl\to\ul\Sp$ is the universal $\neg\Forb$-lax functor to $\ul\Sp$.
\end{example}

\begin{remark}\label{rem:glob-of-cat-is-cat}
    Note that in the previous example, the globalization of the global $\infty$-category $\ul\Sp$ turned out to be a global $\infty$-category again. This is not a coincidence: \cite{Linskens2023globalization}*{Lemma 4.4} shows that $\Glob$ restricts to a functor $\Cat(\Fglo)^{\lnot\Forb\text{-lax}}\to\Cat(\Fglo)$.
\end{remark}

The previous example does not yet give us any control of the multiplicative norm functors present on both sides of the equivalence. We will now explain how to achieve this. Suppose $(\cF,\cN,\Mm)$ is a distributive context, and recall that this demands the existence of a factorization system $(\cE,\Mm)$ on $\cF$. Given an $\Nn$-normed $\Ff$-precategory, we can form the $\Mm$-rigidification of its underlying $\cF$-precategory; on the other hand, we can also consider its `$\Mm$-rigidification as an $\cN$-normed $\cF$-precategory,' i.e.~the rigidification with respect to the pair $(\Span_{\cN}(\cF)^{\op},\Span_{\Mm,\Nn}(\cF)^{\op})$. Our next goal is to compare the two, for which we first fix some notation:

\begin{notation}\label{not:E-lax}
We will use the notation
\[
    \PreNmCat{\cF}{\cN}^{\lax{\Ee}}
    \coloneqq \PreCat{\Span_{\cN}(\cF)^{\op}}^\neglax{\Span_{\cM,\cN}(\cF)^{\op}}
\]
and denote $\Cat^*(\cF)^\neglax{\cM}$ by $\PreCat{\cF}^\lax{\cE}$. We will refer to morphisms in either category as $\cE$-lax functors. We will also denote functor categories in either category by $\Fun^{\lax{\Ee}}(-,-)$, potentially with a subscript if there is ambiguity.

Finally, we will denote the functor
\[
    \Rig_{\Span_{\cM,\cN}(\cF)^{\op}}^{\Span_{\cN}(\cF)^{\op}}
    \colon \PreNmCat{\cF}{\cN}^{\lax{\cE}}\to \PreNmCat{\cF}{\cN}
\]
by $\Rig^\otimes$ and the functor $\Rig_{\cM}^{\cF}\colon \PreCat{\cF}^{\lax{\cE}} \to \PreCat{\cF}$ by $\Rig$.
\end{notation}

Below, we will show that the canonical functor $\Rig\to \fgt \circ \Rig^\otimes$ is an equivalence. However, this will rely on a slightly different universal property of the envelope construction than we have previously used, which we will quickly discuss now.

\begin{definition}\label{def:fact_stable}
	Let $S$ be a category equipped with a factorization system $(E,M)$. We say a wide subcategory $E_0\subset E$ is \emph{factorization stable} if for every commutative square
	\[\begin{tikzcd}
		A & A' \\
		{B} & {B'}
		\arrow[epic, from=1-1, to=1-2]
		\arrow[mono, from=1-1, to=2-1]
		\arrow[mono, from=1-2, to=2-2]
		\arrow[epic, from=2-1, to=2-2]
	\end{tikzcd}\]
	in which the horizontal edges are in $E$ and the vertical edges are in $M$, if the map $B\rightarrowepic B'$ is in $E_0$ then so is the map $A\rightarrowepic A'$. We will denote by $(E_0,M)$ the wide subcategory of $S$ spanned by arrows which may be written as composites of the form
\[\begin{tikzcd}
	\cdot & \cdot & \cdot,
	\arrow["{\in E_0}\;\;", epic, from=1-1, to=1-2]
	\arrow["{\in M}", mono, from=1-2, to=1-3]
\end{tikzcd}\]
which is indeed a subcategory of $S$ as $E_0$ is factorization stable.
\end{definition}

Recall that given a subcategory $S_0\subset S$ we write $\Cat^{{S_0}\dcocart}_{/S}$ for the subcategory of $\Cat_{/S}$ spanned by the $S_0$-cocartesian fibrations and on morphisms by those functors which preserve cocartesian edges over morphisms in $S_0$.

\begin{proposition}
The envelope construction $\Env(-)$ defines a $\Cat$-linear left adjoint to the inclusion $\Cat^{{E_0}\dcocart}_{/S}\hookrightarrow\Cat^{{(E_0,M)}\dcocart}_{/S}$.
\end{proposition}

Note that the maximal case $E_0=E$ is precisely \cite{BHS_Algebraic_Patterns}*{Theorem~E}, recalled as Theorem~\ref{thm:BHS} above. As the other extreme, the minimal case $E_0=\core S$ appears as \cite{BHS_Algebraic_Patterns}*{Proposition~2.1.4}.

\begin{proof}
    By the special case $E_0=\core S$ just cited, it suffices to show that $\Env$ preserves $E_0$-cocartesian fibrations, and that for each $E_0$-cocartesian fibration $\Xx$ and $(E_0,M)$-cocartesian fibration $\Yy$, a map $\Env(\Xx)\to\Yy$ in $\Cat_{/S}^{M\dcocart}$ preserves $E_0$-cocartesian lifts if and only if the composite $\Xx\to\Env(\Xx)\to\Yy$ with the unit does so.

    For the first claim, we let $p\colon\Xx\to S$ be $E_0$-cocartesian. Then the basechange $\Env(\Xx)=\Xx\times_{S}\Ar_M(S)\to\Ar_M(S)$ admits cocartesian lifts over all maps in $\Ar_M(S)$ whose source is a map in $E_0$, hence in particular over all maps of the form
    \begin{equation}\label{diag:E_0-cc}
        \begin{tikzcd}
            \cdot\arrow[d,mono]\arrow[r,epic, "{}\in E_0\;\;"] &  \cdot\arrow[d,mono]\\
            \cdot\arrow[r,epic,"{}\in E_0\;\;"'] & \cdot\rlap,
        \end{tikzcd}
    \end{equation}
    i.e.~where both source and target lie in $E_0$. On the other hand, by the special case $E_0=E$ (Theorem~\ref{thm:BHS} above), each of these maps is itself cocartesian for $t\colon\Ar_M(S)=\Env(\id_S)\to S$, and conversely any cocartesian lift of a map in $E_0$ has this form by factorization stability. Thus, the composite $\Env(\Xx)\to\Ar_M(S)\to S$ has cocartesian lifts for maps in $E_0$, as claimed.

    For the second claim, we can essentially repeat the argument of Barkan--Haugseng--Steinebrunner: Retracing our steps shows that the $E_0$-cocartesian edges in $\Env(\Xx)$ are precisely given by the pairs of an $E_0$-cocartesian edge in $\Xx$ and a map in $\Ar_M(S)$ of the form $(\ref{diag:E_0-cc})$. In particular, the unit $\Xx\to\Env(\Xx)$ preserves cocartesian edges over $E_0$, and we see that \emph{any} $\Env(\Xx)\to\Yy$ over $S$ preserves $E_0$-cocartesian lifts to objects of the form $(x,\id_{p(x)})$ if and only if its restriction to $\Xx$ preserves $E_0$-cocartesian lifts. However, a general object $(x, m\colon p(x)\rightarrowmono y)$ is itself the target of the cocartesian lift of $m$ to $(x,\id_{p(x)})$. If we now let $e\colon y\to z$ be any map in $E_0$, then factoring $em$ with respect to $(E,M)$ yields, by factorization stability, a diagram of the form $(\ref{diag:E_0-cc})$, which we can lift to a diagram in $\Env(\Xx)_\cocart$ sending the top left vertex to $(x,\id_{p(x)})$, and hence sending the lower left vertex to our fixed general object $(x,m)$. If $\Env(\Xx)\to\Yy$ is also a map in $\Cat_{/S}^{M\dcocart}$, then it sends the given cocartesian lifts of the vertical maps to cocartesian edges, and by the above it sends the lift of the top horizontal map to a cocartesian edge. Thus, right cancellability of cocartesian edges in $\Yy$ shows that it also sends the fixed lift of the lower horizontal map to a cocartesian edge; this is precisely what we had to prove.
\end{proof}

We may apply this to the factorization system $(\cF^{\op},\cN)$ on $\Span_{\cN}(\cF)$ equipped with the subcategory $\cM^\op\subset \cF^\op$; one easily checks that this is factorization stable, since $\Mm$ is closed under basechange in $\cF$. Note that with these choices, an $\cE$-lax functor of $\cN$-normed $\cF$-precategories, as defined in \Cref{not:E-lax}, is precisely a functor which preserves cocartesian edges over morphisms in $(E_0,M) \simeq \Span_{\cM,\cN}(\cF)$. In keeping with previous notation, we will write $\PreOp{\cN}{\cF}^{\neglax{\cM}}$ for the category of $\cF^{\op}$-cocartesian fibrations over $\Span_\cN(\cF)$ and functors which preserve cocartesian edges over $\cM^{\op}$. We may specialize the previous proposition to this context, and passing to full subcategories gives the following corollary.

\begin{corollary}\label{cor:lax-env}
The envelope construction defines a $\Cat$-linear left adjoint to the inclusion
\[
    \PreNmCat{\cF}{\cN}^{\lax{\cE}}
    \hookrightarrow \PreOp{\cF}{\cN}^{\neglax{\cM}}.
\]
In particular, we obtain a natural equivalence
\[
    \Fun^{\lax{\cE}}(\Env(\cO),\Cc) \simeq \Fun^{\neglax{\cM}}(\cO,\Cc)
\]
for all $\cN$-normed $\cF$-preoperads $\cO$ and $\cN$-normed $\cF$-precategories $\Cc$.\qed
\end{corollary}

We can now prove the main result of this subsection.

\begin{theorem}\label{thm:Normed_Rig}
Suppose $(\cF,\cN,\Mm)$ is a distributive context. Then the Beck--Chevalley transformation filling the square
\begin{equation}\label{diag:bc-rig}
\begin{tikzcd}[column sep=large]{\PreNmCat{\cF}{\cN}^{\lax{\cE}}} & {\PreNmCat{\cF}{\cN}} \\
    {\PreCat{\cF}^{\lax{\cE}}} & {\PreCat{\cF}}
	\arrow["{\mathrm{Rig^\otimes}}", from=1-1, to=1-2]
	\arrow["{\fgt}"', from=1-1, to=2-1]
    \arrow["{\mathrm{BC}_*}"{description}, shorten <=9pt, shorten >=7pt, Rightarrow, from=1-2, to=2-1]
	\arrow["{\fgt}", from=1-2, to=2-2]
	\arrow["{\mathrm{Rig}}"', from=2-1, to=2-2]
\end{tikzcd}
\end{equation}
is an equivalence. In particular, given any $\cN$-normed $\cF$-precategory $\Cc$, there is a unique pair of an $\cN$-normed structure on $\Rig(\Cc)$ together with an extension of $\ev \colon \Rig(\Cc)\to\Cc$ to an $\Ee$-lax $\cN$-normed functor.
\end{theorem}

\begin{proof}
    Let $\Cc^\tensor$ be an $\cN$-normed $\cF$-precategory,
    with underlying $\cF$-precategory $\cC$.
    Recall from \cref{cor:iota_N_rep} that the represented $\Span_{\cN}(\cF)^\op$-precategory $\und{A}$ is equivalent to $\iota\und{\cN}^A\simeq \Env(\Triv(\und{A}))$, where the latter $\und{A}$ refers to the $\cF$-precategory represented by $A$. So we may compute
    \begin{align*}
        \Rig^\otimes(\Cc^\tensor)(A)\buildrel\text{def}\over=\Fun^{\lax{\Ee}}(\und{A},\Cc^\tensor)&\simeq \Fun^{\lax{\Ee}}( \Env(\Triv(\und{A})),\Cc^\tensor)
        \\&\simeq \Fun^{\neglax{\cM}}(\Triv(\und{A}),\Cc^\tensor),
	\end{align*}
	where the second equivalence is precisely that of the previous corollary.
    Let us write $i\colon \cF^{\op} \hookrightarrow \Span_{\cN}(\cF)$ for the canonical inclusion.
    Recall from the proof of \Cref{prop:calg} that the adjunction $\Triv\dashv \fgt$ was obtained by restricting the adjunction $i_!\dashv i^*$ on slice categories.
    However, a simple inspection shows that this adjunction also restricts to a $\Cat$-linear adjunction
	\[
	    \Triv(-)\colon \Cat^{\Mm^{\op}\dcocart}_{/\cF^\op}
        \rightleftarrows \Cat^{\Mm^{\op}\dcocart}_{/\Span_\cN(\cF)}\noloc \fgt.
	\]
    Therefore we obtain further natural equivalences
	\[
        \Fun^{\neglax{\cM}}_{\Span_{\cN}(\cF)^{\op}}(\Triv(\und{A}),\Cc^\tensor)
        \simeq \Fun^{\neglax{\cM}}_{\cF}(\und{A},\Cc)
        \buildrel\text{def}\over= \Rig(\cC)(A).
	\]
	By inspection, the resulting equivalence
    $\Rig^\otimes(\Cc^\tensor)|_{\cF^\op} \simeq \Rig(\Cc)$ of $\cF$-precategories fits into a commutative diagram
    \[
        \begin{tikzcd}
            \Rig^\otimes(\Cc^\tensor)|_{\cF^\op}\arrow[d,"\ev|_{\cF^\op}"']\arrow[r,"\sim"] & \Rig(\Cc)\arrow[d,"\ev"]\\
            \Cc^\tensor|_{\cF^\op}\arrow[r,equals] & \Cc\rlap,
        \end{tikzcd}
    \]
    i.e.~it agrees with the Beck--Chevalley transformation filling the square $(\ref{diag:bc-rig})$.

    For the final statement, we see that $\Rig^\tensor(\cC^\tensor)$ and the counit of the adjunction provide the claimed $\cN$-normed structure on $\Rig(\cC)$ and $\ev \colon \Rig(\cC) \to \cC$, and uniqueness of this extension follows from the universal property of $\Rig^\tensor$ together with conservativity of the forgetful functor.
\end{proof}

\begin{notation}
Justified by the previous results, given an $\cN$-normed $\cF$-precategory $\cC$, we will typically identify $\Rig(\Cc)$ and $\Rig^\otimes(\Cc)$.
As a special case, given a normed global precategory $\Cc$, we may canonically view $\Glob(\Cc)$ as a normed global precategory (only occasionally denoted $\Glob^\otimes$ for emphasis). Note that Remark~\ref{rem:glob-of-cat-is-cat} shows that if $\Cc$ was a normed global category, then so is $\Glob(\Cc)$.
\end{notation}

\Cref{thm:Normed_Rig} also gives us a different description of the category of normed algebras in a rigidification, which we record now.

\begin{theorem}\label{thm:algebras_in_Rig}
	Let $\Cc$ be an $\cN$-normed $\cF$-precategory. Then
	\[
	\ul\CAlg^\cN_\cF(\Rig(\Cc))(A)
	\simeq \Fun^{\neglax{\cM}}(\und{A}^\Ncoprod,\Cc)
	= \plaxlim\limits_{B\in\Span_\cN(\cF_{/A})}\Cc(B),
	\]
	naturally in $\Cc$ and $A$, where the partially lax limit is with respect to the marking $(\cF_{/A}\times_\cF \Mm)^{\op}$. Moreover, this equivalence fits into a diagram
	\begin{equation}\label{diag:compatibility-CAlg-in-Rig}
		\begin{tikzcd}
			\ul\CAlg^\cN_\cF(\Rig(\Cc))\arrow[r,"\sim"]\arrow[d,"\mathbb U"'] & \Fun^{\neglax{\cM}}(\und{(-)}^\Ncoprod,\Cc)\arrow[d,"\incl^*"]\\
			\Rig(\Cc)\arrow[r,equals] & \Fun^{\neglax{\cM}}(\Triv\und{(-)},\Cc)
		\end{tikzcd}
	\end{equation}
	commuting naturally in $\Cc$, where $\incl$ is the map from Lemma~\ref{lem:unfurl}(2).
\end{theorem}
\begin{proof}
	We may compute:
	\begin{align*}
		\ul{\CAlg}^{\cN}_{\cF}(\Rig(\Cc))(A)
		&\simeq \Fun^{\tensor}(\Env(\ul{A}^\Ncoprod), \Rig(\Cc)) \\
		&\simeq \Fun^{\lax{\cE}}(\Env(\und{A}^{\Ncoprod}),\Cc) \\
		&\simeq \Fun^{\neglax{\cM}}(\und{A}^\Ncoprod,\Cc),
	\end{align*}
	where the first equivalence follows by adjunction, the second by the universal property of $\Rig(\Cc)$,
	and the third by \cref{cor:lax-env}. Commutativity of $(\ref{diag:compatibility-CAlg-in-Rig})$ is then a direct consequence of naturality of the individual adjunction equivalences.

	For the equivalence of the last term with the partially lax limit in the statement,
	recall that the cocartesian unstraightening of $\und{A}^{\Ncoprod}$ is $\Span_\cN(\cF_{/A})\to \Span_{\cN}(\cF)$,
	and that an edge is cocartesian over $\Mm^{\op}$ if it is in the subcategory $(\cF_{/A}\times_\cF \Mm)^{\op}\subset \Span_\cN(\cF_{/A})$.
	Recall further that $\Fun^{\neglax{\cM}}(\und{A}^\Ncoprod,\Cc)$ is by definition the category of functors $\Span_{\cN}(\cF_{/A})\to \Un^{\cocart}(\Cc)$ over $\Span_{\cN}(\cF)$ which preserve cocartesian edges over $\Mm^{\op}$.
	Pulling back to $\Span_{\cN}(\cF_{/A})$, this is equivalent to the full subcategory of sections of the cocartesian fibration
	\[
	\Un^\cocart(\cC)\times_{\Span_{\cN}(\cF)}\Span_{\cN}(\cF_{/A}) \to \Span_{\cN}(\cF_{/A})
	\]
	spanned by those sections which send maps in $(\cF_{/A}\times_\cF \Mm)^{\op}$ to cocartesian edges.
	By definition, this is the partially lax limit of the functor
	\[
	\Span_{\cN}(\cF_{/A}) \to \Span_{\cN}(\cF) \xrightarrow{\Cc} \Cat,
	\]
	where we mark $\Span_{\cN}(\cF_{/A})$ by the subcategory $(\cF_{/A}\times_\cF \Mm)^{\op}$.
\end{proof}

\subsection{Applications to (global) equivariant homotopy theory}
We may now specialize the above discussion to $(\Fglo,\Forb,\Forb)$ to obtain a normed version of the identification $\Glob(\ul{\Sp})\simeq \ul{\Sp}_{\gl}$.

\begin{construction}\label{constr:spotimes}
	We begin by constructing a normed enhancement
	$\und{\Sp}^\tensor$ of the global category
	of equivariant spectra $\und{\Sp}$ introduced in \cref{ex:globalization}. Namely, consider once more for each finite group $G$ the full subcategory $\text{$G$-Sp}_\text{equiv.proj.}^\Sigma\subset \text{$G$-Sp}^\Sigma$ spanned by those $G$-symmetric spectra that are cofibrant in the projective equivariant model structure \cite{hausmann-equivariant}*{Theorem~4.8}. We have seen in \cref{ex:globalization} that the restriction $f^*\colon \text{$G$-Sp}^\Sigma\to \text{$H$-Sp}^\Sigma$ along any homomorphism {$f\colon H\to G$} of finite groups defines a homotopical functor $\text{$G$-Sp}_\text{equiv.proj.}^\Sigma\to \text{$H$-Sp}_\text{equiv.proj.}^\Sigma$. On the other hand, \cite{hausmann-equivariant}*{Proposition 6.1 and Theorem 6.7} show that the smash product and Hill--Hopkins--Ravenel norms similarly restrict to homotopical functors. Thus, we obtain a normed global subcategory $\ul\Sp^{\Sigma,\otimes}_\text{equiv.~proj.}\subset(\Sp^{\Sigma,\otimes})^\flat$ which we can localize to a normed global category $\ul\Sp^\otimes$ with underlying global category $\ul\Sp$. This admits a normed global functor $\ul\Sp^\otimes\hookrightarrow\ul\Sp_\gl^\otimes$ induced by the inclusion of equivariantly projective into flat symmetric spectra.
\end{construction}

Recall from Example~\ref{ex:goo} the factorization system $(\Forb,\mathbb E)$ on $\Fglo$.

\begin{corollary}\label{cor:globalize-sp-monoidal}
	The inclusion $\ul\Sp^\otimes\hookrightarrow\ul\Sp_\gl^\otimes$ admits an $\mathbb E$-lax normed right adjoint, and this exhibits $\ul\Sp^\otimes_\gl$ as the normed globalization of $\ul\Sp^\otimes$.
	\begin{proof}
		By \cite[Proposition 4.21]{puetzstueck-new} the pointwise right adjoints assemble into an $\mathbb E$-lax normed functor. As its underlying functor $\ul{\Sp}_\gl\to\ul{\Sp}$ is an (un-normed) globalization by \cref{ex:globalization}, the result is a consequence of \Cref{thm:Normed_Rig}.
	\end{proof}
\end{corollary}

Combining \Cref{thm:global-model,thm:algebras_in_Rig}, we obtain Theorem~\ref{introthm:ucomgl-from-sp} from the introduction:

\begin{corollary}\label{cor:ucom-via-equivariant-spectra}
    There is a natural (in $G \in \Fglo^\op$) equivalence
    \[
        \UCom_\textup{$G$-gl}
        \simeq \plaxlim\limits_{H\in\Span_{\Forb}(\Fglo_{/G})} \textup{Sp}_H,
    \]
    where we mark the faithful backwards maps. Moreover, this lifts the equivalence
    \[
        \Sp_\textup{$G$-gl}\simeq\plaxlim\limits_{H\in(\Fglo_{/G})^\op} \textup{Sp}_H
    \]
    from Example~\ref{ex:globalization}.\qed
\end{corollary}

\begin{remark}\label{rk:BMY}
    In \cite{deflations}*{Definition~6.5}, Blumberg, Mandell and Yuan define the category of \emph{global commutative ring spectra with multiplicative deflations} as the partially lax limit of a certain diagram $\Span(\Fglo)\to\Cat, G\mapsto\Sp_G$. As we will show in Theorem~\ref{thm:equivariant-uniqueness}, the restriction of this diagram to $\Span_\Forb(\Fglo)$ agrees with $\ul\Sp^\otimes$, so specializing the above Corollary to $G=1$, we see that any global commutative ring spectrum with multiplicative deflations indeed has an underlying ultra-commutative ring spectrum in the sense of \cite{schwede2018global}, affirming \cite{deflations}*{Remark~6.6}. We will say more about multiplicative deflations in §\ref{subsec:mult-defl}.

    On the other hand, also the partially lax limit
    \[
        \CAlg_{\gl}(\ul{\Sp}^\otimes)=\plaxlim_{H\in\Span_{\Forb}(\Fglo)}\Sp_H
    \]
    where we mark \emph{all} backwards maps has been considered before: these are the \emph{global algebras} of \cite{yuan-frobenii}. The above corollary then exhibits these global algebras as the full subcategory of $\UCom_\gl$ spanned by those ultra-commutative global ring spectra whose underlying global spectrum is in the essential image of the inclusion $\Sp\hookrightarrow\Sp_\gl$. Note that this is different from being contained in the essential image of the fully faithful left adjoint $\CAlg(\Sp)\to\UCom_\gl$ of the forgetful functor, see \cite{Lenz-Stahlhauer}*{Warning~6.28}. Thus, a global algebra does indeed carry more structure than just an ordinary $E_\infty$ ring spectrum.
\end{remark}

\section{Globalizing ultra-commutative $G$-ring spectra}\label{sec:globalize-equivariant-alg}
Once again, for this section we fix a distributive context $(\cF,\cN,\Mm)$. However, this time we make the additional assumption that $\cN$ is contained in $\Mm$. The reader should have $(\Fglo,\Forb,\Forb)$ in mind.

\subsection{Global categories of equivariant algebras}
Let $\Cc$ be an $\Mm$-distributive $\cN$-normed $\cF$-precategory, and write $\Cc|_{\Mm}$ for the restriction of $\Cc$ to a (distributive) $\cN$-normed $\Mm$-precategory. We are going to investigate the connection between $\ul\CAlg_{\cF}^{\cN}(\Rig^\otimes \cC)$ and $\ul\CAlg_{\Mm}^{\cN}(\cC|_{\Mm})$, i.e.~between the $\cF$-precategory of $\Nn$-normed algebras in $\Rig^\otimes\Cc$ versus the $\cM$-precategory of $\cN$-normed algebras in $\Cc|_\cM$. We will see that the latter is naturally a right Bousfield localization of the former. We begin with the following definition.

\begin{definition}
    We say a functor $\cX\to \Span_{\cN}(\cF)$ is an \emph{$\Mm$-partial $(\cF,\cN)$-operad} if it is a $(\Span_\cN(\cF), \cR^\op, \Span_{\Ee,\cN}(\cF))$-operad in the sense of Definition~\ref{def:operad}, i.e.
	\begin{enumerate}
		\item $\cX\to \Span_{\cN}(\cF)$ admits cocartesian lifts of maps in $\cR^{\op}$,
		\item $\cX$ has finite products which are preserved by the map to $\Span_{\cN}(\cF)$, \emph{and}
		\item for every two objects $X,Y\in \cX$, the projection maps $X \times Y\to X,Y$ are cocartesian.
	\end{enumerate}
    We denote the category of $\Mm$-partial $(\cF,\cN)$-operads by $\Op{\cF\sbar \Mm}{\Nn}$.
    Maps of $\Mm$-partial $(\cF,\cN)$-operads are given by maps of $\Mm^{\op}$-cocartesian fibrations.
\end{definition}

\begin{example}
    For $\Mm = \cF$ this recovers the $(\cF,\cN)$-operads
    of \cref{def:f-n-operad}.
\end{example}

Connected to this definition is the following lemma.

\begin{lemma}
The inclusion $i\colon \Span_{\cN}(\Mm) \to \Span_{\cN}(\cF)$ induces a $\Cat$-linear adjunction
\[
\mathrm{Triv}\colon \Op{\Mm}{\cN} \rightleftarrows \Op{\cF\sbar \Mm}{\cN}\noloc (-)|_{\Mm}.
\]
\end{lemma}

\begin{proof}
    The existence of the adjunction above on $\cM^\op$-cocartesian fibrations is analogous to the construction of the adjunction $\Triv\dashv \fgt$ in \Cref{prop:calg}. As the inclusion $i$ preserves finite products, we conclude that the adjunction restricts to the categories of operads, analogously to \Cref{thm:oper_operad}(\ref{triv_on_operad}).
\end{proof}

\begin{construction}\label{constr:gamma}
    Consider the inclusion $\cM_{/A} \subset \cF_{/A}$, which gives the square
    \[
    \begin{tikzcd}
        {\Span_{\cN}(\Mm_{/A})} & {\Span_{\cN}(\cF_{/A})} \\
        {\Span_{\cN}(\Mm)} & {\Span_{\cN}(\cF)}
        \arrow[from=1-1, to=1-2]
        \arrow[from=1-1, to=2-1]
        \arrow[from=1-2, to=2-2]
        \arrow[from=2-1, to=2-2]
    \end{tikzcd}
    \]
    naturally in $A\in \Mm$.
    Recall that the right vertical map is precisely $\und{A}^{\Ncoprod}\in \Op{\cF}{\cN}$. Amusingly, the left hand map is also $\und{A}^{\Ncoprod}$ in $\Op{\Mm}{\cN}$, but now $\und{A}$ is the $\Mm$-groupoid represented by $A\in \Mm$ and $(-)^{\Ncoprod}$ is the functor $\iCat{\Mm}\to \Op{\Mm}{\cN}$. To avoid confusion, we will denote the former by $\und{A}^{\Ncoprod}_{\cF}$ and the latter by $\und{A}^{\Ncoprod}_{\Mm}$.
With this notation, the square above gives a morphism
\[
    \gamma_A \colon \Triv(\und{A}^{\Ncoprod}_{\Mm})\to\und{A}^{\Ncoprod}_{\cF}
\]
in $\Op{\cF\sbar \Mm}{\cN}$, naturally in $A \in \cM$. This in turn induces a functor
\begin{multline*}
    \ul\CAlg^{\cN}_{\cF}(\Rig^\otimes\Cc)(A)
    \simeq \Fun^{\neglax{\cM}}_{\Span_{\cN}(\cF)^{\op}}(\und{A}^\Ncoprod_{\cF},\Cc)\\
    \xrightarrow{\gamma_A^*} \Fun^{\neglax{\cM}}_{\Span_{\cN}(\cF)^{\op}}(\Triv(\und{A}_\cM^\Ncoprod),\Cc)
    \simeq \ul\CAlg^{\cN}_{\Mm}(\Cc|_{\Mm})(A),
\end{multline*}
where the first equivalence above is \Cref{thm:algebras_in_Rig} and the second follows by the adjunction of the previous lemma.
Moreover, this is natural in $A\in \Mm$, and so we obtain an $\cM$-functor
\[
    \fgt\colon \ul\CAlg^{\cN}_{\cF}(\Rig^\otimes\Cc)|_{\Mm}
    \to \ul\CAlg^{\cN}_{\Mm}(\Cc|_{\Mm}).
\]
\end{construction}

\begin{theorem}\label{thm:global-extension}
    Let $\Cc$ be an $\Mm$-distributive $\Nn$-normed $\cF$-precategory.
    Then the $\Mm$-functor $\fgt$ is a localization.
    Moreover, it admits a parametrized left adjoint
    \[
        \mathcal L\colon\ul\CAlg^{\cN}_{\Mm}(\Cc|_{\Mm})
        \hookrightarrow \ul\CAlg^{\cN}_{\cF}(\Rig^\otimes\Cc)|_{\Mm}
    \]
    whose essential image is even an $\cF$-subprecategory of $\ul\CAlg^{\cN}_{\cF}(\Rig^\otimes\Cc)$.
\end{theorem}

Before we prove the theorem, let us note the following immediate consequence:

\begin{corollary}
Under the above assumptions, there is a unique pair of an extension of $\ul\CAlg^{\cN}_{\Mm}(\Cc|_{\Mm})$ to an $\cF$-precategory $\ul\CAlg^\cN_{\Mm\triangleright \cF}(\Cc)$ together with an extension of $\mathcal L$ to an $\cF$-functor
    $
        \Ll\colon\ul\CAlg^\cN_{\Mm\triangleright \cF}(\Cc)
        \hookrightarrow \ul\CAlg^{\cN}_{\cF}(\Rig^\otimes\Cc).
    $\qed
\end{corollary}

\begin{definition}
We call this extension $\ul\CAlg^\cN_{\Mm\triangleright \cF}(\Cc)$ the \emph{$\cF$-precategory of $\cN$-normed $\Mm$-algebras}~in~$\Cc$.
\end{definition}

\begin{remark}\label{rk:fgt-lax-extension}
    Passing to right adjoints, the corollary shows that the forgetful functors assemble into an $\Ee$-lax $\cF$-functor $\fgt\colon\ul\CAlg^{\cN}_{\cF}(\Rig^\otimes\Cc)\to \ul\CAlg^\cN_{\Mm\triangleright \cF}(\Cc)$  which is pointwise a Dwyer--Kan localization. We can therefore interpret the functoriality of $\ul\CAlg_{\Mm\triangleright \cF}^{\cN}(\Cc)$ with respect to general maps in $\cF$ as the left derived functors of the corresponding restrictions in $\ul\CAlg^{\cN}_{\cF}(\Rig^\otimes\Cc)$.
\end{remark}

We now turn to the proof of Theorem~\ref{thm:global-extension}. We will use \cref{thm:cleft-Kan} to construct the pointwise left adjoints.

\begin{lemma}\label{lemma:LKE-CAlgRig}
    Let $\Cc$ be an $\Mm$-distributive $\Nn$-normed $\Ff$-precategory, and let $f\colon\cO\to\cO'$ be a map in $\Op{\Ff\sbar\Mm}{\Nn}$. Then $f^*\colon\Fun^{\neglax\Mm}(\cO',\Cc)\to\Fun^{\neglax\Mm}(\cO,\Cc)$ has a left adjoint $f_!$ given by pointwise left Kan extension along $\Env(f)$, where $\Env$ denotes the envelope with respect to the factorization system $(\Mm^\op,\Span_{\Ee,\Nn}(\Ff))$.
\end{lemma}

Below we will be able to treat the notion of a \emph{pointwise} Kan extension as a black box; the curious reader is referred back to Remark~\ref{rk:pointwise-Kan}.

\begin{proof}
    To prove this, observe that by adjunction $f^*$ is equivalent to the restriction
    \begin{equation}\label{eq:restriction-enveloped}
        \Fun_\cF^{\Nstr}(\Env(\cO'),\Cc)
        \to \Fun_\cF^{\Nstr}(\Env(\cO),\Cc)
    \end{equation}
    along $\Env(f)$.
    We apply \cref{thm:cleft-Kan} to exhibit a left adjoint to restriction along $\Env(\gamma_A)$.
    To check the hypotheses, note that
    \begin{itemize}
        \item $\Cc$ is $\Mm$-distributive and so $\ul{\bfU}_{\Mm}^{\times}$-cocomplete by \Cref{prop:distributivity}.
        \item The envelopes $\Env(\cO)$ and $\Env(\cO')$ are $\cN$-normed $\cF$-categories by \Cref{prop:envelope-operad} (applied to $(S,E,M) = (\Span_\cN(\cF),\cM^\op,\Span_{\cE,\cN}(\cF))$).
        \item That the required maps are right fibrations and that the required squares are pullbacks follows exactly as in the proof of \Cref{lem:left_kan_along_env}.
    \end{itemize}
    Having checked the hypotheses, we obtain that $(\ref{eq:restriction-enveloped})$
    admits a left adjoint given by pointwise left Kan extension.
\end{proof}

\begin{remark}\label{rk:Env-factorization-system}
    We continue to write $\Env$ for the envelope with respect to the factorization system $(\Mm^\op,\Span_{\Ee,\Nn}(\Ff))$. By \cite{HHLNa}*{Corollary~2.22}, we can identify $\Env(1)=\Ar_{\Span_{\Ee,\Nn}(\Ff)}(\Span_\Nn(\Ff))$ with the category of spans in (a subcategory of) $\Fun(\Lambda^2_0,\Ff)$, whose morphisms are of the form
    \[
        \begin{tikzcd}
            \cdot &\cdot\arrow[l] \arrow[r,norm] & \cdot\\
            \cdot\arrow[u,epic]\arrow[d,norm] & \arrow[d,norm]\arrow[l]\arrow[u,epic]\cdot\arrow[r,norm]\arrow[dl, "\llcorner"{very near start},phantom]\arrow[ur,"\urcorner"{very near start},phantom] & \cdot\arrow[u,epic]\arrow[d,norm]\\
            \cdot &\arrow[l]\cdot\arrow[r,norm] & \cdot\rlap,
        \end{tikzcd}
    \]
    and hence the envelope of $\Span_\Nn(\Ff_{/A})=\ul A_\Ff^{\Ncoprod}$ is given for any $A\in\Ff$ by the span category
    \[
        \begin{tikzcd}
            A & \arrow[l,equals] A\arrow[r,equals] &A\\
            \arrow[u]\cdot &\arrow[u]\cdot\arrow[l] \arrow[r,norm] & \cdot\arrow[u]\\
            \cdot\arrow[u,epic]\arrow[d,norm] & \arrow[d,norm]\arrow[l]\arrow[u,epic]\cdot\arrow[r,norm]\arrow[dl, "\llcorner"{very near start},phantom]\arrow[ur,"\urcorner"{very near start},phantom] & \cdot\arrow[u,epic]\arrow[d,norm]\\
            \cdot &\arrow[l]\cdot\arrow[r,norm] & \cdot\vphantom{Y}
        \end{tikzcd}
        \qquad\text{with fibers}\qquad
        \begin{tikzcd}
            A & \arrow[l,equals] A\arrow[r,equals] &A\\
            \arrow[u]\cdot &\arrow[u]\cdot\arrow[l] \arrow[r,norm] & \cdot\arrow[u]\\
            \cdot\arrow[u,epic]\arrow[d,norm] & \arrow[d,norm]\arrow[l,equals]\arrow[u,epic]\cdot\arrow[r,norm]\arrow[dl, "\llcorner"{very near start},phantom]\arrow[ur,"\urcorner"{very near start},phantom] & \cdot\arrow[u,epic]\arrow[d,norm]\\
            B &\arrow[l,equals] B \arrow[r,equals] & B\rlap.
        \end{tikzcd}
    \]
    The map $\gamma_A$ from Construction~\ref{constr:gamma} is fully faithful as $\Mm\subset\Ff$ is left cancellable and $\Nn\subset\Mm$, hence also $\Env(\gamma_X)$ is fully faithful. We may therefore identify the envelope of $\Triv(\ul A_\Mm^{\Ncoprod})$ with the full subcategory of the above given by those objects where the topmost vertical maps to $A$ belong to $\Mm$.
\end{remark}

\begin{observation}
    The description from Remark~\ref{rk:Env-factorization-system} makes it clear that each category $\Env(\ul{A}_\Ff^{\Ncoprod})(B)$ carries a factorization system, given by the standard factorization system on a span category consisting of backwards and forwards maps.
\end{observation}

\begin{proposition}\label{prop:Env-param-fs}
The above defines a \emph{parametrized factorization system} in the sense of \cite{shah2022parametrizedII}*{Definition~3.1}, i.e.~given any span $B\gets B'\rightarrownorm B''$, the structure map
$
    \Env(\und{A}_\cF^{\Ncoprod})(B)\to \Env(\und{A}_\cF^{\Ncoprod})(B'')
$
restricts to a map between the left parts of the factorization systems as well as to a map between the right parts.

Similarly, if $A\to A'$ is any map in $\Ff$, then the functor
$
    \Env(\und{A}_\cF^\Ncoprod)(B)\to \Env(\und{A'}_\cF^\Ncoprod)(B)
$
induced by the pushforward preserves the above factorization systems.
    \begin{proof}
        We begin by observing that the final statement about the pushforward functoriality in $A$ is immediate from the definitions.

        Next, we consider the functoriality with respect to backwards maps $B\gets B'$. The cocartesian pushforward $g$ of a map $f$ in $\Env(\und{A}_\cF^\Ncoprod)(B)$ is then uniquely characterized by being a map in the fiber over $B'$ that fits into a square
        \begin{equation*}
            \begin{tikzcd}
                \cdot\arrow[d, "f"']\arrow[r, "\text{co}"] & \cdot\arrow[d, "g"]\\
                \cdot\arrow[r, "\text{co}"'] &\cdot
            \end{tikzcd}
        \end{equation*}
        where the horizontal maps are cocartesian over $B\gets B'$. By Theorem~\ref{thm:BHS}, these cocartesian lifts are precisely given by the backwards maps of the form
                \[
            \begin{tikzcd}
                A\arrow[r, equals] & A\\
                \cdot\arrow[u] & \arrow[l,mono] \cdot\arrow[u]\\
                \cdot\arrow[u,epic]\arrow[d,norm] & \cdot\arrow[l]\arrow[d,norm]\arrow[u,epic]\arrow[dl, "\llcorner"{very near start}, phantom]\\
               B & \arrow[l] B'\llap.
            \end{tikzcd}
        \]
        As backwards maps are right cancellable, we immediately see that if $f$ is a backwards map, then so is $g$.

        Now assume that $f$ is a forwards map instead. We first compute the composite
        \[
            \begin{tikzcd}
                A\arrow[r, equals] & A\\
                U\arrow[u] & \arrow[l,mono] U'\arrow[u]\\
                V\arrow[u,epic]\arrow[d,norm] & V'\arrow[l]\arrow[d,norm]\arrow[u,epic]\arrow[dl, "\llcorner"{very near start}, phantom]\\
                B & \arrow[l] B'
            \end{tikzcd}
            \;\;\circ\;\;
            \begin{tikzcd}
                A\arrow[r, equals] & A\\
                \vphantom{A'}W \arrow[u]\arrow[r,norm] & U\arrow[u]\\
                \vphantom{A'}X\arrow[u,epic]\arrow[r,norm]\arrow[d,norm] \arrow[ur,phantom, "\urcorner"{very near start}] & V\arrow[u,epic]\arrow[d,norm]\\
                B\arrow[r,equals] & B\vphantom{A'}
            \end{tikzcd}
            \;\;=\;\;
            \begin{tikzcd}
                A\arrow[r,equals] & A\arrow[r,equals] & A\\
                W\arrow[u] &\arrow[l,mono] \smash{U'\times_UW}\vphantom{A}\arrow[r,norm]\arrow[u] & U'\arrow[u]\\
                X\arrow[d,norm]\arrow[u,epic] &\arrow[l]\arrow[dl,"\llcorner"{very near start},phantom]\arrow[ur, "\urcorner",very near start,phantom] \smash{V'\times_VX}\vphantom{B}\arrow[u,epic]\arrow[r,norm]\arrow[d,norm] & V'\arrow[u,epic]\arrow[d,norm]\\
                B & \arrow[l] B' \arrow[r,equals] & B'
            \end{tikzcd}
        \]
        of $f$ followed by a cocartesian lift of $B\gets B'$. The map $U'\times_UW\to W$ is in $\Mm$ (being a basechange of a map in $\Mm$), i.e.~the backwards part of the resulting span is a cocartesian lift of $B\gets B'$. As the forward part is evidently fiberwise, we conclude that it computes the cocartesian pushforward of $f$; in particular, the latter is indeed a forward map, as we wanted to show.

        Finally, we consider the functoriality in forward maps $B'\rightarrownorm B''$. Using Theorem~\ref{thm:BHS} once more, we see that its cocartesian lifts are precisely the forward maps of the form
        \[
            \begin{tikzcd}
                A\arrow[r,equals] & A\\
                \cdot\arrow[u]\arrow[r,equals] &\cdot\arrow[u]\\
                \cdot\arrow[d,norm]\arrow[u,epic]\arrow[r,equals]\arrow[ur,"\urcorner"{very near start}, phantom] &\cdot\arrow[d,norm]\arrow[u,epic]\\
                B'\arrow[r, norm] & B''\llap.
            \end{tikzcd}
        \]
        Thus, the functoriality is simply given by postcomposition in the lowermost square, which evidently preserves backwards maps and forward maps separately.
    \end{proof}
\end{proposition}

\begin{proposition}\label{prop:left-adj-to-fgt-for-rig}
    Let $\Cc$ be an $\cM$-distributive $\cN$-normed $\cF$-precategory. Then $\mathbb U\colon\ul\CAlg_\Ff^\Nn\Rig(\Cc)\to\Rig(\Cc)$ has a (strong) left adjoint $\mathbb P$.
    \begin{proof}
        By \Cref{thm:algebras_in_Rig}, $\mathbb U$ is given in degree $A$ up to natural equivalence by
        \[
            \Fun^{\neglax{\Mm}}(\ul{A}^\Ncoprod_{\cF},\Cc)
            \xrightarrow{\incl^*}\Fun^{\neglax{\Mm}}(\Triv(\ul{A}_\cF),\Cc),
        \]
        hence equivalently by
        \[
            \Fun^{\Nstr}_{\Ff}(\Env\ul{A}^\Ncoprod_{\cF},\Cc)
            \xrightarrow{\Env(\incl)^*}\Fun^{\Nstr}_{\Ff}(\Env\Triv(\ul{A}_\cF),\Cc),
        \]
        where we take envelopes with respect to the factorization system $(\Mm^\op,\Span_{\Ee,\Nn}(\Ff))$ on $\Span_\Nn(\Ff)$. By \cref{lemma:LKE-CAlgRig}, this has a left adjoint $\mathbb P(A)$ given by pointwise left Kan extension, so it only remains to verify the Beck--Chevalley condition with respect to restriction along any $f\colon A\to B$ in $\cF$.

        Arguing as before, this amounts to showing that applying $\Fun^{\Nstr}_\Ff(-,\Cc)$ to
        \[
            \begin{tikzcd}
                \Env\Triv(\ul{A}_\cF)\arrow[r,hook] \arrow[d, "\Env\Triv(\ul f)"'] & \Env(\ul{A}^\Ncoprod_{\cF})\arrow[d, "\Env(\ul f^\Ncoprod)"]\\
                \Env\Triv(\ul{B}_\cF) \arrow[r,hook] & \Env(\ul{B}^\Ncoprod_{\cF})
            \end{tikzcd}
        \]
        results in a horizontally left adjointable square. For this we note that by \cref{prop:Env-param-fs}, the horizontal maps are given by the inclusions of the left halves of parametrized factorization systems, and that the right hand vertical map also preserves the right parts of these factorization systems. The claim will therefore follow from Lemma~\ref{lemma:rough-basechange} (applied with $T=\Span_\Nn(\Ff)^\op$) once we show that the map induced by $\Env(\ul f^\Ncoprod)$ on the right halves of the factorization systems is a (pointwise) right fibration. But this map is simply a basechange of $\Ff_{/f}\colon\Ff_{/A}\to\Ff_{/B}$, hence indeed a right fibration.
    \end{proof}
\end{proposition}

\begin{remark}
    We conjecture that for an $\cM$-distributive $\cN$-normed $\cF$-category $\Cc$, $\Rig(\Cc)$ is distributive. This would fit into the general pattern observed in \cite{Linskens2023globalization}, that the rigidification of $\Cc$ can in some cases equivalently be understood as the free cocompletion of $\Cc$ under certain (parametrized) colimits. Given such a fact, the previous proposition would follow immediately from our general results about the existence of free functors for normed algebras in a distributive category.
\end{remark}

\begin{proof}[Proof of Theorem~\ref{thm:global-extension}]
    Fix any $A\in \cF$.
    As the first step, we will prove that the functor
    $\fgt(A)\colon\und{\CAlg}^\cN_\cF(\Rig^\tensor\Cc)(A)\to\und{\CAlg}^\cN_\Mm(\Cc|_{\Mm})(A)$
    admits a fully faithful left adjoint $\mathcal L(A)$; in particular, this shows that $\fgt(A)$ is a localization.

    Recall that the map $\Env(\gamma_A)$ from Construction~\ref{constr:gamma} is actually fully faithful (Remark~\ref{rk:Env-factorization-system}), hence so is the pointwise left Kan extension along it (cf.~Theorem~\ref{thm:cleft-Kan}), giving the desired left adjoint $\cL(A)$ to $\fgt(A)$.
    We will now show that the essential images of these fully faithful left adjoints form an $\cF$-subprecategory; already from them being an $\cM$-subprecategory it will then follow by abstract nonsense about Bousfield localizations that the $\mathcal L(A)$'s assemble into an $\cM$-functor left adjoint to $\fgt$.

    To prove the claim, we first recall that $\mathbb U\colon\ul\CAlg_{\Ff}^\Nn\Rig(\Cc)\to\Rig^\otimes\Cc$ has an $\Ff$-strong left adjoint $\bbP$ by \Cref{prop:left-adj-to-fgt-for-rig}. Then the following
    diagram commutes for each $A\in\Ff$ because the corresponding diagram of right adjoints commutes:
    \[
        \begin{tikzcd}
            \Cc(A)\arrow[d,"\mathbb P"']\arrow[r, hook, "\Delta"] & \Rig(\Cc)(A)\arrow[d, "\mathbb P"]\\
            \ul\CAlg_\Mm^\cN(\Cc|_{\Mm})(A)\arrow[r, "\mathcal L"', hook]&\ul\CAlg^\cN_\cF(\Rig^\tensor\Cc)(A)\rlap;
        \end{tikzcd}
    \]
    here $\Delta$ is the strong $\Ff$-functor from \cref{def:universal-not-M-lax-funct}.
    As both the top horizontal and right vertical arrows are part of $\cF$-functors,
    it follows that for every $f\colon B\to A$ in $\cF$ and every $X\in\Cc(A)$ we have
    \[
        f^*\mathcal L\mathbb P(X)\simeq f^*\mathbb P\Delta (X)\simeq \mathbb P\Delta (f^*X)\simeq \mathcal L\mathbb P(f^*X)\in\essim(\mathcal L).
    \]
    Moreover, by \cite[Proposition 4.7.3.14]{HA}, the monadicity of $\und{\CAlg}^{\cN}_{\Mm}(\Cc|_{\Mm})(A)$ over $\Cc(A)$ (see \Cref{cor:par_U_monadic}) implies that $\und{\CAlg}^\cN_\Mm(\Cc|_{\Mm})(A)$ is generated under sifted colimits by the essential image of $\mathbb P$, so that the essential image of $\cL(A)$ is generated under sifted colimits by objects of the form $\cL\mathbb P(X)$. As the essential image of the fully faithful left adjoint $\cL(B)$ is closed under arbitrary colimits and since $f^*$ preserves sifted colimits by \cref{lem:calg-fib-colims}(2), the claim follows.
\end{proof}

\begin{remark}\label{rem:compat-with-p}
    The above proof shows that we can extend the $\Mm$-functor $\mathbb P\colon\Cc|_{\Mm}\to \ul\CAlg^{\cN}_{\Mm}(\Cc|_{\Mm})$ in a unique way to an $\cF$-functor $\mathbb P\colon\Cc\to\ul\CAlg^\cN_{\Mm\triangleright \cF}(\cC)$ such that $\mathcal L\mathbb P\simeq\mathbb P\Delta$ as $\cF$-functors. In particular, we get a natural extension of $\mathbb U\colon\ul\CAlg^{\cN}_{\Mm}(\Cc|_{\Mm})\to\Cc|_{\Mm}$ to an $\Ee$-lax $\cF$-functor $\mathbb U\colon\ul\CAlg^\cN_{\Mm\triangleright \cF}(\Cc)\to\Cc$ by passing to right adjoints. However, this extension will in general not be a strong $\cF$-functor, see Warning~\ref{warn:U-ext-no-strict}.
\end{remark}

\subsection{Globalizing equivariant algebras}
With the above results at hand, it is now no longer difficult to prove the main abstract result of this section:

\begin{theorem}\label{thm:globalize-algebras}
    Let $(\cF,\cN,\Mm)$ be a distributive context such that $\cN \subset \Mm$ and let $\Cc$ be an $\Mm$-distributive $\cN$-normed $\cF$-precategory.
    Then
    $
        \fgt\colon\ul\CAlg^{\cN}_{\cF}(\Rig(\Cc))\to \ul\CAlg^\cN_{\Mm\triangleright \cF}(\Cc)
    $
    (see Remark~\ref{rk:fgt-lax-extension}) is the universal $\Ee$-lax $\cF$-functor.
\end{theorem}

The proof of the theorem will rely on a monadicity argument, and the following lemma will turn out to precisely provide the required Beck--Chevalley condition:

\begin{proposition}\label{prop:mysterious-Beck-Chevalley}
    Let $\Cc$ be an $\Mm$-distributive $\Nn$-normed $\Ff$-precategory. Then the Beck--Chevalley transformation
    \[
        \mathbb P_\Mm\circ\ev\to\fgt\circ\mathbb P_\Ff
    \]
    induced by the commutative square
    \begin{equation}\label{diag:so-many-forgetful-functors}
        \begin{tikzcd}
            \ul\CAlg^\Nn_\Ff(\Rig\,\Cc)|_\Mm\arrow[d,"\mathbb U"']\arrow[r, "\fgt"] & \ul\CAlg^\Nn_\Mm(\Cc|_\Mm)\arrow[d, "\mathbb U"]\\
            (\Rig\,\Cc)|_\Mm\arrow[r, "\ev"'] & \Cc|_\Mm
        \end{tikzcd}
    \end{equation}
    of $\Mm$-precategories is an equivalence.
    \begin{proof}
        We can check the Beck--Chevalley condition after evaluating at each $A\in\Mm$. For this we observe that we may identify $(\ref{diag:so-many-forgetful-functors})$ in degree $A$ with the square obtained by applying $\Fun^{\Nn\text-\otimes}_\Ff(-,\Cc)$ to the square
        \begin{equation}\label{diag:enveloped-and-franked}
            \begin{tikzcd}
                \mathord{\Env}(\ul{A}_{\cF}^{\Ncoprod}) &\arrow[l,hook'] \mathord{\Env}(\ul{A}_{\cM}^{\Ncoprod})\\
                \mathord{\Env}(\Triv(\ul{A}_\Ff))\arrow[u,hook] &\arrow[l,hook']\arrow[u,hook]\mathord{\Env}(\Triv(\ul{A}_\cM)),
            \end{tikzcd}
        \end{equation}
        where $\ul{A}_\cM$ again refers to the presheaf on $\cM$ represented by $A$ (with unstraightening $(\cM_{/A})^{\op}$), $\ul A_\Ff$ is the corresponding $\Ff$-presheaf, and $\Env$ always refers to the envelope for $\Mm^\op\subset\Span_\Nn(\Ff)$. By Lemma~\ref{lemma:LKE-CAlgRig}, the restrictions along the vertical maps admit left adjoints given by pointwise left Kan extension. Recall now from Remark~\ref{rk:Env-factorization-system} that we may identify ${\Env}(\ul{A}_{\cM}^\Ncoprod)(B)\hookrightarrow{\Env}(\ul{A}_{\cF}^\Ncoprod)(B)$ for any $B\in\Ff$ with the inclusion of the span categories whose morphisms look as follows:
        \[
            \left\{\,\begin{tikzcd}[cramped]
                A\arrow[r,equal] & A\arrow[r,equal] & A\\
                \cdot\arrow[u,mono] &\arrow[l,mono]\arrow[u,mono]\cdot\arrow[r,norm] & \cdot\arrow[u,mono]\\
                \cdot\arrow[u,epic]\arrow[d,norm] &\arrow[l,equal]\cdot\arrow[u,epic]\arrow[d,norm]\arrow[r,norm]\arrow[ur,phantom,"\urcorner"{very near start}] & \cdot\arrow[u,epic]\arrow[d,norm]\\
                B\arrow[r,equal] & B\arrow[r,equal] &B\vphantom g
            \end{tikzcd}\,\right\}
            \lhook\joinrel\longrightarrow
            \left\{\,\begin{tikzcd}[cramped]
                A\arrow[r,equal] & A\arrow[r,equal] & A\\
                \cdot\arrow[u] &\arrow[l]\arrow[u]\cdot\arrow[r,norm] & \cdot\arrow[u]\\
                \cdot\arrow[u,epic]\arrow[d,norm] &\arrow[l,equal]\cdot\arrow[u,epic]\arrow[d,norm]\arrow[r,norm]\arrow[ur,phantom,"\urcorner"{very near start}] & \cdot\arrow[u,epic]\arrow[d,norm]\\
                B\arrow[r,equal] & B\arrow[r,equal] &B\vphantom g
            \end{tikzcd}\,\right\}.
        \]
        Under this identification, the vertical maps in $(\ref{diag:enveloped-and-franked})$ are precisely given by the inclusion of the backwards arrows of these two span categories. By Proposition~\ref{prop:Env-param-fs}, the backwards and forwards maps assemble into a parametrized factorization system on ${\Env}(\ul{A}_{\cF}^\Ncoprod)$; it is then clear that this restricts to a parametrized factorization system on  ${\Env}(\ul{A}_{\cM}^\Ncoprod)$. The basechange condition will therefore once more be an instance of Lemma~\ref{lemma:rough-basechange}, once we show that the induced functor on forward maps is a pointwise right fibration.  But this map is given in degree $B$ by the full inclusion
        \[
            \left\{\,
                \begin{tikzcd}[cramped]
                    A\arrow[r,equal] & A\\
                   \arrow[u,mono]\cdot\arrow[r,norm] & \cdot\arrow[u,mono]\\
                    \cdot\arrow[u,epic]\arrow[d,norm]\arrow[r,norm]\arrow[ur,phantom,"\urcorner"{very near start}] & \cdot\arrow[u,epic]\arrow[d,norm]\\
                    B\arrow[r,equal] &B\vphantom g
                \end{tikzcd}
            \,\right\}
            \lhook\joinrel\longrightarrow
            \left\{\,
            \begin{tikzcd}[cramped]
                A\arrow[r,equal] & A\\
               \arrow[u]\cdot\arrow[r,norm] & \cdot\arrow[u]\\
                \cdot\arrow[u,epic]\arrow[d,norm]\arrow[r,norm]\arrow[ur,phantom,"\urcorner"{very near start}] & \cdot\arrow[u,epic]\arrow[d,norm]\\
                B\arrow[r,equal] &B\vphantom{g}
            \end{tikzcd}
            \,\right\},
        \]
        which is a sieve as $\Nn\subset\Mm$, hence in particular a right fibration.
    \end{proof}
\end{proposition}

\begin{proof}[Proof of Theorem~\ref{thm:globalize-algebras}]
    We may equivalently show that the strong $\cF$-functor
    \[
        \widehat{\fgt}\colon\ul\CAlg^{\cN}_{\cF}\Rig^\otimes\Cc
        \to\Rig\,\ul\CAlg_{\Mm\triangleright \cF}^\cN (\Cc)
    \]
    induced by $\fgt$ is an equivalence. For this we consider the commutative diagram
    \[
    \begin{tikzcd}
        \ul\CAlg^{\cN}_{\cF}\Rig^{\otimes}\Cc\arrow[rr, "\widehat{\fgt}"]\arrow[dr, bend right=15pt, "\mathbb U"', end anchor={[yshift=.6pt]west}] && \Rig\, \ul\CAlg_{\Mm\triangleright \cF}^{\cN}\,\Cc\arrow[dl, bend left=15pt, "\Rig(\mathbb U)", end anchor={[yshift=.6pt]east}]\\[.33ex]
        & \Rig\,\Cc
    \end{tikzcd}
    \]
    of strong $\cF$-functors. Recall that the left functor is a monadic right adjoint at every level by \Cref{cor:par_U_monadic}. On the other hand, we claim $\Rig(\bbU)$ is also a monadic right adjoint:
    It is conservative because $\bbU$ is and admits a left adjoint since $\bbU$ does and $\Rig$ is a 2-functor;
    moreover, since $\Rig$ is the right adjoint in a $\Cat$-linear adjunction,
    it is compatible with the cotensoring, so it follows from \cref{rem:fiberwise-colim-cotensor}
    that also $\Rig(\bbU)$ preserves fiberwise sifted colimits.

    Therefore it suffices by \cite[Corollary 4.7.3.16]{HA} to show that the Beck--Chevalley map $\Rig(\bbP) \to \smash{\widehat{\smash{\fgt}\vphantom{t}}}\circ\bbP$ is invertible. By the (2-)universal property of $\Rig$, this can be checked after postcomposing with the universal $\Ee$-lax functor $\Rig\,\ul\CAlg^{\Nn}_{\Mm\triangleright\Ff}(\Cc)\to\ul\CAlg^{\Nn}_{\Mm\triangleright\Ff}(\Cc)$. As moreover the Beck--Chevalley map $\mathbb P\circ\ev\to{\ev}\circ{\Rig(\mathbb P)}$ associated to the right hand square in
    \[
        \begin{tikzcd}
            \ul\CAlg^{\cN}_{\cF}\Rig^{\otimes}\Cc\arrow[d,"\mathbb U"']\arrow[r, "\widehat{\fgt}"] & \Rig\, \ul\CAlg_{\Mm\triangleright \cF}^{\cN}(\Cc) \arrow[r, "\ev"]\arrow[d,"\Rig(\mathbb U)"] & \ul\CAlg_{\Mm\triangleright \cF}^{\cN}(\Cc)\arrow[d,"\mathbb U"]\\
            \Rig(\Cc)\arrow[r,equal] & \Rig(\Cc)\arrow[r,"\ev"'] & \Cc
        \end{tikzcd}
    \]
    is invertible by $2$-functoriality of $\Rig$, we altogether see that the Beck--Chevalley map for the left hand square (i.e.~the one in question) is invertible if and only if the Beck--Chevalley map for the total rectangle is so. But the latter can be checked on underlying $\Mm$-precategories, which is precisely the content of Proposition~\ref{prop:mysterious-Beck-Chevalley}.
\end{proof}

We now specialize to the case $(\Ff,\Nn,\Mm)=(\Fglo,\Forb,\Forb)$, corresponding to (global) equivariant homotopy theory.

\begin{definition}
    Suppose $\Cc$ is an equivariantly distributive normed global category.
    Then we obtain a global category\footnote{Since $(\Fglo,\Forb)$ is even extensive, \cref{thm:oper_operad}(5) shows that this is really is a global category, and not just a precategory.}
    $
        \ul\CAlg_{\eq\triangleright\gl}(\cC)
        \coloneqq \ul\CAlg_{\Forb\triangleright\Fglo}^{\Forb}(\Cc)
    $
    of $\Forb$-normed $\Forb$-algebras in $\Cc$.
    \end{definition}

    \begin{remark}\label{rem:g-algebras}
        Given a normed global category $\cC$ and a finite group $G$,
        note that
        \begin{align*}
            \und{\CAlg}_{\eq\triangleright\gl}(\cC)(G)
            &= \und{\CAlg}_\Forb^{\Forb}(\cC|_{\Forb})(G)\\
            &\simeq \Fun_{\Span(\Forb)}^{\lax{\Forb}}(\Span(\F_G), \cC|_{\Span(\Forb)})\\
            &\simeq \Fun_{\Span(\F_G)}^{\lax{\F_G}}(\Span(\F_G), \cC|_{\Span(\F_G)})
            \eqqcolon \CAlg_G(\cC|_G).
        \end{align*}
        The latter agrees with the definition of \emph{normed $G$-algebras} considered in \cite[Section 9.2]{BachmannHoyois2021Norms}. Similarly, one can use \cite[Section 5.2]{BHS_Algebraic_Patterns} to see that this also agrees with the definition of \emph{$G$-commutative algebras} in a \emph{$G$-symmetric monoidal $G$-$\infty$-category} (\Cref{ex:G-distributivity}) used in \cite{NardinShah}, see \cite{CHLL_NRings}*{Remark~5.6.4}.
    \end{remark}

We can combine the above results with the other universal property of $\Rig=\Glob$ from \cite{Linskens2023globalization} to obtain a universal property of the inclusion $\Ll\colon\ul\CAlg_{\Mm\triangleright\Ff}^\Nn(\Cc)\hookrightarrow\ul\CAlg_{\Ff}^\Nn(\Glob^\otimes\Cc)$. To state this we will need a couple of definitions:

\begin{definition}
    A global precategory is called \emph{fiberwise presentable} if it factors through the non-full subcategory $\Pr^\text{L}\subset\Cat$ of presentable categories and left adjoint functors.
\end{definition}

\begin{definition}
    A global precategory $\Cc\colon\Fglo^\op\to\Cat$ is called \emph{equivariantly presentable} if it is fiberwise presentable and left $\Forb$-adjointable (i.e.~$\Forb$-cocomplete in the sense of Example~\ref{ex:distributivity-as-cocompleteness}). We call $\Cc$ \emph{globally presentable} if it is fiberwise presentable and left $\Fglo$-adjointable (i.e.~$\Fglo$-cocomplete).
\end{definition}

\begin{remark}
    Recall that a commutative square of functors all of which admit both adjoints is horizontally right adjointable
    if and only if it is vertically left adjointable. In particular, globally presentable categories are also right $\Fglo$-coadjointable (`$\Fglo$-complete').
\end{remark}

\begin{corollary}
    Let $\Cc$ be an equivariantly distributive normed global category whose underlying global category is fiberwise presentable (hence equivariantly presentable). Then:
    \begin{enumerate}
        \item $\ul\CAlg_{\eq\triangleright\gl}(\Cc)$ is equivariantly presentable.
        \item $\ul\CAlg_{\gl}\Glob(\Cc)$ is globally presentable.
        \item $\mathcal L\colon \ul\CAlg_{\eq\triangleright\gl}(\Cc)\hookrightarrow\ul\CAlg_{\gl}\Glob(\Cc)$ is the initial equivariantly cocontinuous (i.e.~left $\Forb$-adjointable) functor from $\ul\CAlg_{\eq\triangleright\gl}(\Cc)$ to a globally presentable category.
    \end{enumerate}
\end{corollary}

\begin{proof}
    By Theorem \ref{thm:oper_operad}(5), Example \ref{ex:gog}, and Remark \ref{rem:glob-of-cat-is-cat} all of these are indeed global categories (as opposed to precategories).
        Now given (1) and (2), point (3) will immediately follow from Theorem~\ref{thm:globalize-algebras} via \cite[Theorem 4.10]{Linskens2023globalization}.

    The accessibility of the categories in points (1) and (2) follows from \cite[Proposition 5.4.7.11]{HTT}, see also Remark 5.4.7.13 there. Applying \Cref{cor:calg-cocomplete}, we see that $\und{\CAlg}_\gl(\Glob(\cC))$
    is also fiberwise cocomplete, hence fiberwise presentable.

    On the other hand, by \cite[Theorem 4.6]{Linskens2023globalization}, the underlying global category of $\Glob(\Cc)$ is globally presentable. Now recall that $\mathbb U\colon\ul\CAlg_{\gl}\Glob(\Cc)\to\Glob(\Cc)$ is conservative and admits a global left adjoint $\mathbb P$. So we may show that more generally a fiberwise presentable global precategory $\Dd$ which admits a conservative global right adjoint $\mathbb{U}\colon \Dd\to \Cc$ to a globally presentable global precategory $\Cc$ is again globally presentable.

    Let $f \colon X \to Y$ in $\Fglo$. Then $\bbU f^* \simeq f^*\bbU$ preserves limits, so by conservativity also $f^* \colon \cD(Y) \to \cD(X)$ preserves limits. As it also preserves colimits by assumption, the adjoint functor theorem shows that it admits both adjoints. Thus, instead of proving that $\cD$ is left $\Fglo$-adjointable, we may equivalently show that it is right $\Fglo$-coadjointable.
    Using that $\bbU$ is a conservative $\Fglo$-right adjoint,
    this reduces to the fact that $\cC$ is right $\Fglo$-coadjointable;
    namely, consider the following pullback square in $\Fglo$ and two induced commutative rectangles:
    \[\begin{tikzcd}[cramped,sep=scriptsize]
        {X'} & X & {\cC(X')} & {\cD(X')} & {\cD(X)} & {\cC(X')} & {\cC(X)} & {\cD(X)} \\
        {Y'} & Y & {\cC(Y')} & {\cD(Y')} & {\cD(Y)} & {\cC(Y')} & {\cC(Y)} & {\cD(Y)}
        \arrow["{g'}", from=1-1, to=1-2]
        \arrow["{f'}"', from=1-1, to=2-1]
        \arrow["\lrcorner"{anchor=center, pos=0.125}, draw=none, from=1-1, to=2-2]
        \arrow["f", from=1-2, to=2-2]
        \arrow["\bbU"', from=1-4, to=1-3]
        \arrow["{g'^*}"', from=1-5, to=1-4]
        \arrow["{g'^*}"', from=1-7, to=1-6]
        \arrow["\bbU"', from=1-8, to=1-7]
        \arrow["g"', from=2-1, to=2-2]
        \arrow["{f'^*}", from=2-3, to=1-3]
        \arrow["{f'^*}", from=2-4, to=1-4]
        \arrow["\bbU", from=2-4, to=2-3]
        \arrow["{f^*}"', from=2-5, to=1-5]
        \arrow["{g^*}", from=2-5, to=2-4]
        \arrow["{f'^*}", from=2-6, to=1-6]
        \arrow["{f^*}", from=2-7, to=1-7]
        \arrow["{g^*}", from=2-7, to=2-6]
        \arrow["{f^*}"', from=2-8, to=1-8]
        \arrow["\bbU"', from=2-8, to=2-7]
    \end{tikzcd}\]
    Evidently, the rectangles have equivalent pastings. We want to show that the right square of the middle
    rectangle is vertically right adjointable. The squares with horizontal maps given by $\bbU$ are vertically right adjointable
    because they are horizontally left adjointable by existence of $\bbP$.
    Moreover, the left square in the right rectangle is right adjointable by global presentability of $\cC$.
    Since Beck--Chevalley maps compose and $\bbU$ is conservative, we conclude.

    Now we turn to the equivariant presentability of $\ul\CAlg_{\eq\triangleright\gl}(\Cc)$. As mentioned, each  $\ul\CAlg_{\eq\triangleright\gl}(\Cc)(A)$ is accessible. Since $\cL$ is a fully faithful global functor which is pointwise left adjoint,
    it then follows that $\und{\CAlg}_{\eq\triangleright\gl}(\cC)$ is closed under fiberwise colimits in $\ul\CAlg_\gl(\Glob(\Cc))$; we conclude that it is fiberwise cocomplete, hence fiberwise presentable. Moreover, since $\cL$ is an $\Forb$-left adjoint, the fact that $\und{\CAlg}_\gl(\Glob(\cC))$ is left $\Forb$-adjointable implies that also $\und{\CAlg}_{\eq\triangleright\gl}(\cC)$ is so, analogously to how we verified right $\Fglo$-coadjointability of $\ul\CAlg_\gl(\Glob(\Cc))$ above.
\end{proof}

\subsection{Equivariant ultra-commutative ring spectra as normed algebras} We will now compare the global category $\ul\CAlg_{\eq\triangleright\gl}(\ul\Sp^\tensor)$ to the more classical definition of \emph{genuine commutative $G$-ring spectra} in terms of model categories. We start with the following (thankfully easier) equivariant analogue of Theorem~\ref{thm:spgl-dist}:

\begin{proposition}\label{prop:sp-equ-distr}
    The normed global category $\ul\Sp^\otimes$ of Construction~\ref{constr:spotimes} is equivariantly distributive.
    \begin{proof}
        By \cite{CHLL_NRings}*{Proposition~5.5.7} the norm $\Nm^G_H\colon\Sp_H\to\Sp_G$ preserves sifted colimits for every group $G$ and every subgroup $H$; as each $\Sp_G$ is presentably symmetric monoidal (see e.g.~Example~5.4.7 of \emph{op.\ cit.}), we conclude that all norms preserve sifted colimits. On the other hand, \cite{CLL_Clefts}*{Theorem~9.4} shows that restrictions in $\und{\Sp}$ (along maps in $\Fglo$) even preserve all colimits.

        Next, we recall that since $\Sp^{\Sigma,\tensor}$ is presentably
        symmetric monoidal, $(\Sp^{\Sigma,\tensor})^\flat$ is distributive (Example~\ref{ex:borel-distributive}).
        Since $\und{\Sp}^\tensor$ is defined as a localization of the full subcategory $\und{\Sp}^{\Sigma,\tensor}_{\text{equiv.~proj.}} \subset (\Sp^{\Sigma,\tensor})^\flat$,
        it now suffices to observe that by \cite{hausmann-equivariant}*{§5.2}
        also the left adjoints $\Ind_H^G \colon(\Sp^{\Sigma,\tensor})^\flat(H)\to(\Sp^{\Sigma,\tensor})^\flat(G)$
        restrict to homotopical functors on projectively cofibrant objects
        for every subgroup inclusion $H \subset G$,
        and similarly for the coproducts (as in any model category).
    \end{proof}
\end{proposition}

\begin{construction}
    Just like its $G$-global (flat) counterpart, the equivariant projective model structure on $G$-symmetric spectra has a positive analogue, also constructed in \cite{hausmann-equivariant}*{Theorem~4.8}, whose cofibrant objects are those $G$-equivariantly projectively cofibrant spectra $X$ such that $X_0=0$. Accordingly, we have a further global subcategory $\ul\Sp^\Sigma_\text{pos.~equiv.~proj.}\subset\ul\Sp^\Sigma_\text{equiv.~proj.}$ spanned by these objects, and the inclusion induces an equivalence on Dwyer--Kan localizations.

    By \cite{hausmann-equivariant}*{Theorem~6.15}, the positive projective model structure transfers to the category of commutative algebras. Thus, the categories $\CAlg(\text{$G$-Sp}^{\Sigma,\tensor})_\text{pos.~equiv.~proj.}$ for varying $G$ assemble into a full global subcategory \[\ul{\CAlg(\Sp^{\Sigma,\tensor})}_\text{pos.~equiv.~proj.}\subset\big({\CAlg(\Sp^{\Sigma,\tensor})}\big)^\flat,\] improving to a diagram in relative categories. We caution the reader that the underlying $G$-symmetric spectrum of an object of $\CAlg(\text{$G$-Sp}^{\Sigma,\tensor})_\text{pos.~equiv.~proj.}$ need \emph{not} be cofibrant in the (positive) $G$-equivariant projective model structure again.\footnote{This is not just a pointset level artifact, but has a homotopy invariant interpretation, see Warning~\ref{warn:U-ext-no-strict} below.} On the other hand, the left adjoints $\mathbb P\colon\text{$G$-Sp}^\Sigma\to\CAlg(\text{$G$-Sp}^\Sigma)$ do indeed restrict accordingly by definition of a transferred model structure, so they assemble into a global functor
    \[
        \mathbb P\colon\ul\Sp^\Sigma_\text{pos.~equiv.~proj.}\to\ul{\CAlg(\Sp^{\Sigma,\tensor})}_\text{pos.~equiv.~proj.},
    \]
    and by Ken Brown's Lemma they preserve $G$-equivariant weak equivalences. Thus, writing $\ul\UCom$ for the corresponding localization of the target and conflating the localization of the source with $\ul\Sp$, we arrive at a global functor $\mathbb P\colon\ul\Sp\to\ul\UCom$.
\end{construction}

\begin{construction}
    By \cite{Lenz-Stahlhauer}*{Lemma~6.27} the inclusions
    \[
        \CAlg(\text{$G$-Sp}^{\Sigma,\tensor})_\text{pos.~equiv.~proj.}\hookrightarrow
        \CAlg(\text{$G$-Sp}^{\Sigma,\tensor})_\text{flat}
    \]
    are homotopical for the $G$-equivariant weak equivalences on the source and the $G$-global weak equivalences on the target, and as explained after \emph{loc.\ cit.},
    the resulting functor on Dwyer--Kan localizations is fully faithful, with its right adjoint being a localization at the $G$-equivariant weak equivalences.

    As in the previous construction, we can then localize this to a fully faithful inclusion $\ul\UCom\hookrightarrow\ul\UCom_\gl$, which by construction fits into a commutative diagram
    \[
        \begin{tikzcd}
            \ul\Sp\arrow[r,hook]\arrow[d,"\mathbb P"'] & \ul\Sp_\gl\arrow[d,"\mathbb P"]\\
            \ul\UCom\arrow[r,hook] & \ul\UCom_\gl\rlap.
        \end{tikzcd}
    \]
\end{construction}

We can now finally prove Theorem~\ref{introthm:equivariant-param} from the introduction:

\begin{theorem}\label{thm:equivariant-model}
    There is an equivalence of global categories
    \[
        \ul\CAlg_{\eq\triangleright\gl}(\ul\Sp^\otimes)
        \simeq\ul\UCom
    \]
    compatible with the free functors from $\ul\Sp$.
    In view of \cref{rem:g-algebras}, this gives for every finite group $G$
    an equivalence
    \[
        \CAlg_G(\und{\Sp}^\tensor_G) \simeq \UCom_G
        = \CAlg(G\text{-}\Sp^{\Sigma,\tensor})[W_{G\textup{-equiv.~stable}}^{-1}]
    \]
    between normed algebras in
    $\und{\Sp}_G^\tensor \coloneqq \und{\Sp}^\tensor|_{\Span(\F_G)}$
    and the Dwyer--Kan localization of the 1-category
    of strictly commutative algebras in $G$-symmetric spectra,
    compatible with the forgetful functors to $\Sp_G$.
\end{theorem}
\begin{proof}
    By Theorem~\ref{thm:global-extension} together with Remark~\ref{rem:compat-with-p}, we have a fully faithful inclusion
    $
        \ul\CAlg_{\eq\triangleright\gl}(\ul\Sp^\otimes)
        \hookrightarrow \ul\CAlg_{\gl}(\ul\Sp^\otimes_\gl)
    $
    compatible with the free functors.
    On the other hand, we also have a fully faithful inclusion
    $\ul\UCom\hookrightarrow\ul\UCom_\gl$ compatible with the free functors by the previous construction. It thus suffices to show that the essential images of these two embeddings match up under the equivalence
    $\und{\UCom}_\gl \simeq \ul\CAlg_\gl(\ul\Sp^\otimes_\gl)$
    from Theorem~\ref{thm:global-model}.

    To prove the claim, we may work at a fixed group $G$, and we may equivalently show that the right adjoints of the above embeddings invert the same maps in
    $\UCom_\text{$G$-gl} \simeq \ul\CAlg_\gl(\ul\Sp^\otimes_\gl)(G)$.
    The right adjoint of
    $
        \ul\CAlg_{\eq\triangleright\gl}(\ul\Sp^\otimes)(G)
        \hookrightarrow \ul\CAlg_\gl(\ul\Sp^\otimes_\gl)(G)
    $
    is simply the forgetful functor $\fgt$, so it inverts precisely those maps whose underlying map in $\Sp_\text{$G$-gl}$ forgets to an equivalence in $\Sp_G$.
    On the other hand, the right adjoint to $\UCom_G\hookrightarrow \UCom_{\text{$G$-gl}}$ is also just the forgetful functor, so the maps inverted by it admit the same description. The claim follows, as the equivalence
    $\ul\UCom_\gl \simeq \ul\CAlg_\gl(\ul\Sp^\otimes_\gl)$
    from Theorem~\ref{thm:global-model} is compatible with the forgetful functors to $\ul\Sp_\gl$.
\end{proof}

\begin{corollary}
    There are equivalences
    \[
        \UCom_G \simeq \plaxlim\limits_{H\in\Span(\mathbb{F}_{G})}\Sp_H
    \]
    for all finite groups $G$, where we mark the backwards maps.
\end{corollary}
\begin{proof}
    The above theorem gives an equivalence
    \[
        \UCom_G
        \simeq \CAlg_G(\und{\Sp}_G^\tensor)
        = \Fun_{\Span(\F_G)}^{\lax{\F_G}}(\Span(\F_G),\und{\Sp}_G^\tensor),
    \]
    and the latter is precisely the definition of the partially lax limit in the statement.
\end{proof}

\begin{warn}\label{warn:U-ext-no-strict}
    The forgetful functors $\mathbb U\colon\UCom_{G}\to\Sp_G$ assemble into an $\mathbb E$-lax global functor $\ul\UCom\to\ul\Sp$ right adjoint to $\mathbb P$, but as we will now show, this lax functor is \emph{not} strong, and hence by the above theorem neither is the forgetful functor $\ul\CAlg_{\eq\triangleright\gl}(\ul\Sp^\otimes)\to\ul\Sp$.
    This highlights the subtlety in constructing
    a \emph{global} category of normed $\Forb$-algebras in the global category of \emph{equivariant} spectra.\footnote{In fact, the normed global category corepresenting $\CAlg_1$ is the wide subcategory of $\ul{\mathbb F}^{\amalg}$ given by the isovariant maps, and so a Yoneda computation shows that for a general normed global category $\Cc$ the restriction $\CAlg_G(\Cc)\to\CAlg_1(\Cc)$ has no section unless $G=1$ and the only $(\infty,2)$-natural maps in the other direction are tensor powers of the norm $\Nm^G_1$.}

    If $X$ is cofibrant in the model structure on $\CAlg(G\text{-}\Sp^\Sigma)$ transferred from the positive projective model structure, then its restriction along $\alpha\colon H\to G$ can be computed directly on the pointset level. However, the restriction of its underlying $G$-spectrum is obtained by first taking a projectively cofibrant replacement $f\colon X'\to X$ in $G\text{-}\Sp^\Sigma$, and then restricting $X'$; unravelling definitions, the Beck--Chevalley map is precisely given by $\alpha^*(f)\colon\alpha^*(X')\to\alpha^*(X)$.

    Now let $F_1=\bm\Sigma(1,-)\smashp S^1$ be the symmetric spectrum representing the functor $Y\mapsto\Omega Y_1$ and let $X=\mathbb P(F_1)$. As $F_1$ is cofibrant in the positive projective model structure, $X$ is cofibrant in the transferred model structure. On the other hand, one directly computes that $X\cong\bigvee_{n\ge 0}\bm\Sigma(n,-)\smashp_{\Sigma_n}S^n$ in the $1$-category of symmetric spectra. Fix now a cofibrant replacement $X'\to X$ in the non-equivariant projective model structure; we want to show that there is some group $G$ such that $X'\to X$ is not a weak equivalence with respect to the trivial action, i.e.~$X'\to X$ is not a \emph{global} weak equivalence.

    We will prove more generally that there is no global weak equivalence $X'\simeq X$ at all. To this end, observe that by \cite{g-global}*{Proposition~3.3.1} the global stable homotopy type modelled by $X'$ is contained in the essential image of the left adjoint $\Sp\hookrightarrow \Sp_\gl$ of the forgetful functor, i.e.~$X'$ maps trivially into any global spectrum $Y$ with trivial underlying non-equivariant spectrum. It therefore suffices to find such a $Y$ into which $X$ maps non-trivially, or equivalently, into which one of the $\bm\Sigma(n,-)\smashp_{\Sigma_n}S^n$ maps non-trivially.

    We will do this for the summand $Z\coloneqq\bm\Sigma(2,-)\smashp_{\Sigma_2}S^2$ (an analogous argument works for every $n\ge2$). Namely, observe that $Z$ corepresents the functor $Y\mapsto \textup{maps}^{\Sigma_2}(S^2, Y_2)$ in the simplicially enriched $1$-category of symmetric spectra, hence the functor $Y\mapsto\pi_0^{\Sigma_2}(Y)$ on the global stable homotopy category. It therefore suffices to take any global spectrum with trivial underlying spectrum but non-trivial $\pi_0^{\Sigma_2}$, and of course such global spectra abound. A concrete example of such a $Y$ is as follows: let $\mathbb A\coloneqq\ul{\pi}_0\mathbb S$ be the Burnside ring global Mackey functor (corepresenting evaluation at the trivial group $1$), and let $f\colon\mathbb A\to\ul{\mathbb Z}$ be the functor to the constant global Mackey functor at $\mathbb Z$ classifying $1\in\mathbb Z$. Then $f$ is an isomorphism at the trivial group, but not injective at the group $\Sigma_2$ (as $\mathbb A(\Sigma_2)$ has rank $2$), so its kernel is a global Mackey functor $M$ with $M(1)=0, M(\Sigma_2)\not=0$. Then the global Eilenberg--MacLane spectrum $HM$ has the desired properties.
\end{warn}

\begin{theorem}\label{thm:ucom-glob}
    The inclusion $\ul{\textup{UCom}}\hookrightarrow\ul{\textup{UCom}}_{\textup{gl}}$ has an $\mathbb E$-lax right adjoint, and this exhibits $\ul{\UCom}_\gl$ as the globalization of $\ul\UCom$. In other words, we have equivalences
    \[
        \UCom_\textup{$G$-gl}\simeq\plaxlim_{H\in(\Fglo_{/G})^{\op}}\UCom_\text{$H$},
    \]
    where $(\Fglo_{/G})^{\op}$ is marked by the maps over $G$ whose underlying map in $\Fglo$ is in $\Forb$. Moreover, this equivalence is natural in $G\in\Fglo^\op$ and compatible with the forgetful functors to $\UCom_G$.
    \begin{proof}
        By the proof of Theorem~\ref{thm:equivariant-model} we have a commutative diagram
        \begin{equation*}
            \begin{tikzcd}
                \ul\UCom\arrow[r,hook]\arrow[d,"\sim"'] & \ul\UCom_\gl\arrow[d,"\sim"]\\
                \ul\CAlg_{\eq\triangleright\gl}(\ul\Sp^\otimes)\arrow[r,"\Ll"'] & \ul\CAlg_\gl(\ul\Sp^\otimes_\gl)\rlap.
            \end{tikzcd}
        \end{equation*}
        The right adjoint of the lower horizontal map is a globalization by Theorem~\ref{thm:globalize-algebras} together with Corollary~\ref{cor:globalize-sp-monoidal}, so the claim follows.
    \end{proof}
\end{theorem}

\section{Norms, Mackey functors, and Tambara functors}\label{sec:tambara}
While we defined our normed categories of global and equivariant spectra by suitably deriving constructions on the level of model categories (as these are easiest to construct, and enable the comparison to ultra-commutative ring spectra), one can also attempt to construct them in a purely $\infty$-categorical fashion, using the Mackey functor description of equivariant and global spectra \cite{cmnn,CLL_Spans}. The purpose of this section is to first carry out the Mackey functor approach, and then to show that this is equivalent to our model categorical construction.
Beyond the relevance of comparing these two a priori very different approaches, this will also allow us to deduce a description for connective ultra-commutative global and $G$-equivariant ring spectra as \emph{spectral Tambara functors}, see Theorems~\ref{thm:equivariant-Tambara} and~\ref{thm:global-Tambara}.

For the normed structure on equivariant spectra/Mackey functors, this will basically amount to reorganizing the arguments given for a fixed group $G$ in \cite{CHLL_NRings}. However, these arguments ultimately rely on the fact that the suspension spectrum functor $\Sigma^\infty\colon\Spc_{G,*}\to\Sp_{G}$ already has a non-parametrized universal property in $\CAlg(\Pr^\text{L})$, so this line of reasoning cannot be applied to the categories of $G$-global spectra. Instead, the key new input for the global case will be \cref{thm:Normed_Rig} above, which will allow us to reduce this to the equivariant case.

\subsection{Mackey functor models of equivariant and global stable homotopy theory}
We begin by recalling the Mackey functor models of equivariant and global spectra from \cite{cmnn,CLL_Spans}, at this point without norms.

\begin{construction}
    We write $\ul\Mack$ for the global category
    \[
        \Fglo^\op\hskip-1pt
        \xrightarrow{\null\;\hskip1.5pt({\mathbb F}^\flat\hskip-.5pt,\hskip1pt{\mathbb F}^\flat\hskip-.5pt,\hskip1pt{\mathbb F}^\flat)\hskip1.5pt\;\null}\AdTrip\xrightarrow{\;\Span\;}\Cat
        \xrightarrow{\;P_\Sigma\;}{\Cat}(\sift)
    \]
    and $\ul\Mack_\gl$ for the global category
    \[
        \Fglo^\op
        \xrightarrow{\;(\Fglo^\flat,\Fglo^\flat,\Forb^\flat)\;}\AdTrip
        \xrightarrow{\;\Span\;}\Cat\xrightarrow{\;P_\Sigma\;}{\Cat}(\sift)\rlap.
    \]
    In particular, the functoriality is given in both cases by the universal property of $P_\Sigma$ or, equivalently \cite{HHLNa}*{Theorem~8.1}, by left Kan extension.
\end{construction}

\begin{remark}
    Via straightening--unstraightening, we can identify $\Fglo^\flat$ with the contravariant functor $\Fglo_{/-}\colon X\mapsto \Fglo_{/X}$, where the functoriality is via pullback. We claim that this restricts to an equivalence between $\Forb^\flat$ and the full subcategory $\Fglo_{/-}\times_{\Fglo}\Forb$, or equivalently, that a map $A\to B$ of (finite) groupoids over $X$ is faithful if and only if the induced map $A_x\to B_x$ on fibers is a faithful for every $x\in X$. But indeed, since a map of spaces is a summand inclusion iff its fibers are either empty or contractible, $f\colon A\to B$ is faithful if and only if for every $b\in B$ the fiber $f^{-1}(b)$ is a discrete groupoid. As this agrees with the fiber of $A_x\to B_x$ for $x$ the image of $b$ by the pasting law, the claim then follows.

    In particular, one may equivalently define $\ul\Mack_\gl(X) \coloneqq \Fun^\times(\Span_{\Forb}(\Fglo_{/X}),\Spc)$ with the evident functoriality. This `unstraightened' perspective on global Mackey functors is used in \cite{CLL_Spans}, and we will freely apply results from \emph{loc.~cit.} for our present definition, leaving the straightening--unstraightening translation implicit.

    Similarly, one observes that $\Fglo^\flat\simeq\Fglo_{/-}$ identifies $\mathbb F^\flat$ with the full global subcategory $\Forb_{/-}\subset\Fglo_{/-}$ whose objects are the faithful functors. Combining this with the equivalence $({\Forb_{/X}})_{/Y\to X}\simeq\Forb_{/Y}$ for every $Y\to X$ in $\Forb$, we can then identify the restriction of $\ul\Mack$ along $\mathbb F_G\simeq\Forb_{/G}\to\Forb$ with $X\mapsto{\Fun}^\times\big({\Span}\big((\mathbb F_G)_{/X}\big),{\Spc}\big)$, which is the perspective taken in \cite{CHLL_NRings}.
\end{remark}

As promised, the above is equivalent to our model categorical approach:

\begin{theorem}\label{thm:unnormed-comparison}
	~
    \begin{enumerate}
        \item $\ul\Mack$ is equivariantly presentable, and there is a unique equivalence \[\ul\Sp\iso{\Sp}\otimes\ul\Mack\eqqcolon\ul\SpMack\] sending the sphere $\mathbb S$ to the delooping of $\hom(1,-)\colon\Span(\mathbb F)\to\CMon$.
        \item $\ul\Mack_\textup{gl}$ is globally presentable, and there exists a unique equivalence \[\ul\Sp_\textup{gl}\iso{\Sp}\otimes\ul\Mack_\textup{gl}\eqqcolon\ul\SpMack_\textup{gl}\] sending $\mathbb S$ to the delooping of $\hom(1,-)\colon\Span_{\Forb}(\Fglo)\to\CMon$.
        \item The inclusion $i\colon\mathbb F\hookrightarrow\Fglo$ induces an equivariantly cocontinuous fully faithful functor $i_!\colon\ul\SpMack\hookrightarrow\ul\SpMack_\textup{gl}$ sitting inside a commutative diagram
        \begin{equation}\label{diag:globalization-square}
            \begin{tikzcd}
                \ul\Sp\arrow[r,hook]\arrow[d,"\sim"'] & \ul\Sp_\textup{gl}\arrow[d,"\sim"]\\
                \ul\SpMack\arrow[r, hook, "i_!"'] & \ul\SpMack_\textup{gl}\rlap.
            \end{tikzcd}
        \end{equation}
        Moreover, $i_!$ has an $\mathbb E$-lax right adjoint $i^*$, and this exhibits $\ul\SpMack_\gl$ as the globalization of $\ul\SpMack$.
    \end{enumerate}
    \begin{proof}
        We will prove the claims in a peculiar order.

        We begin by proving the second statement, which follows immediately by combining \cite{CLL_Global}*{Theorem~7.3.2} with \cite{CLL_Spans}*{Corollary~9.17}.

        Next, we will show that $\ul\Mack$ is equivariantly presentable and that $i_!\colon\ul\Mack\hookrightarrow\ul\Mack_\gl$ is equivariantly cocontinuous; from this the corresponding statement about the pointwise stabilizations will then follow via \cite{CLL_Clefts}*{Lemma~8.5}. To prove both claims, it suffices to show that the essential image of $\ul\Mack\hookrightarrow\ul\Mack_\gl$ is closed under ordinary colimits and the left adjoints to restrictions along maps in $\Forb$. The first statement is clear since $i_!$ is a fully faithful left adjoint. For the second statement, it suffices to observe that for any \emph{injective} homomorphism $\alpha\colon H\rightarrowmono G$, a left adjoint to $\Span(\alpha^*)\colon \Span_{\Forb}(\Fun(G,\Fglo))\to\Span_{\Forb}(\Fun(H,\Fglo))$ is given by $\Span(\alpha_!)$, by a straightforward application of \cite{BachmannHoyois2021Norms}*{Corollary~C.21}, and that the functor $\alpha_!\colon \Fun(H,\Fglo)\to\Fun(G,\Fglo)$ restricts to $\Fun(H,\mathbb F)\to\Fun(G,\mathbb F)$ (being given on underlying non-equivariant objects simply by a finite coproduct).

        Next, we note that $\ul\Mack$ is \emph{equivariantly semiadditive} in the sense of \cite{CLL_Global}*{Example~4.5.2}: namely, since we have just shown that it is equivariantly cocomplete, we may check this on underlying $\Forb$-categories, where this is an instance of \cite{CLL_Spans}*{Corollary~9.14} (for $P=T=\Orb$). Thus, $\ul\SpMack$ is \emph{equivariantly stable} by \cite{CLL_Clefts}*{Lemma~8.8}, and so the universal property of $\ul\Sp$ \cite{CLL_Clefts}*{Theorem~9.4} shows that there is a unique equivariantly cocontinuous functor $\Phi\colon\ul\Sp\to\ul\SpMack$ sending $\mathbb S$ to the delooping of $\Hom(1,-)$. We want to show that $\Phi$ is an equivalence, which can be checked on underlying $\mathbb O$-categories, where this follows by combining \cite{CLL_Clefts}*{Theorem~9.5} with \cite{CLL_Spans}*{Corollary~9.14} (for $P=T=\Orb$ again).

        Now we can prove the rest of the third statement. All maps in $(\ref{diag:globalization-square})$ are equivariantly cocontinuous, so to verify commutativity it will be enough to compare the images of the sphere spectrum. By construction, the upper path sends this to the delooping $\psi$ of $\hom(1,-)\colon \Span_\Forb(\Fglo) \to \CMon$. An easy Yoneda argument now shows that $i_!$ sends the delooping of $\hom(1,-)\colon\Span(\mathbb F)\to\CMon$ to $\psi$, so also the lower path through the diagram also sends $\mathbb S$ to $\psi$, as claimed. Finally, with commutativity of $(\ref{diag:globalization-square})$ established, $\ul\Sp\hookrightarrow\ul\Sp_\gl$ being left adjoint to a globalization by Example~\ref{ex:globalization} immediately shows that so is $i_!$.
    \end{proof}
\end{theorem}

\subsection{Norms for equivariant Mackey functors}
Before we can describe the norms on $\ul\SpMack$ and $\ul\SpMack_\gl$, we have to deal with the unstable situation. For this, recall that a map $f\colon X\to Y$ of (pointed or unpointed) simplicial sets with $G$-action is called a \emph{$G$-equivariant weak equivalence} if the induced map $f^H\colon X^H\to Y^H$ is a weak homotopy equivalence for every $H\subset G$, or equivalently if the geometric realization $|f|$ is a \emph{$G$-equivariant homotopy equivalence} (i.e.~there exists an equivariant map $g\colon|Y|\to |X|$ together with homotopies through equivariant maps from $|f|g$, $g|f|$ to the respective identities).

\begin{proposition}
    Write $\SSet_*^\otimes$ for the symmetric monoidal $1$-category of finite pointed simplicial sets, where the symmetric monoidal structure is the smash product. Then the associated Borel normed $1$-category $(\SSet_*^\otimes)^\flat$ can be localized with respect to the equivariant weak equivalences, yielding a normed global category $\ul\Spc_*^\otimes$.
\end{proposition}

    The restriction along $\Span(\mathbb F_G)\simeq\Span(\Forb_{/G})\to\Span_{\Forb}(\Fglo)$ of this normed global category is denoted $\ul{\mathfrak S}_{G,*}$ in \cite{CHLL_NRings}*{Section~5}.

\begin{proof}
    Factoring a general map in $\Fglo$ as usual, we have to show that for every homomorphism $f\colon G\to H$ of finite groups the restriction $f^*\colon\Fun(G,\SSet_*)\to\Fun(H,\SSet_*)$ is homotopical, that the smash product of $G$-simplicial sets preserves weak equivalences in each variable, and that for every $H\subset G$ the symmetric monoidal norm $\Fun(H,\SSet_*)\to\Fun(G,\SSet_*)$ is again homotopical. The first statement is clear from the definitions, while the remaining statements are verified in \cite{CHLL_NRings}*{Proposition~5.2.11}.
\end{proof}

\begin{lemma}\label{lem:normed-pt-spc-free-sift}
    The inclusion $(\mathbb F_*^\otimes)^\flat\hookrightarrow\ul\Spc_*^\otimes$ is pointwise a sifted cocompletion.
    \begin{proof}
        Note that the functors $\Span(\mathbb F_G)\simeq\Span(\Forb_{/G})\to\Span_{\Forb}(\Fglo)$ are jointly essentially surjective: if $X$ is a general object in $\Fglo$ with decomposition into components $X\simeq\coprod_{i=1}^n X_i$, then $X$ lives over the group $\prod_{i=1}^nX_i$ via the faithful map $\coprod_{i=1}^nX_i\to\prod_{i=1}^nX_i$ given by the `identity matrix.' Thus, we may check the claim after restricting to each $\Span(\mathbb F_G)$, where this appears as \cite{CHLL_NRings}*{Corollary~5.2.7}.
    \end{proof}
\end{lemma}

We can now introduce our normed enhancement of $\ul\Mack$:

\begin{construction}
    The cartesian symmetric monoidal structure on $\mathbb F$ induces a normed structure $\ul{\mathbb F}^\times$ on the Borelification as before.
    Applying $\Span$ pointwise, we obtain a normed global category $\Span(\ul{\mathbb F})^\otimes$. Finally, passing to pointwise sifted cocompletion yields
    $
        \ul\Mack^\otimes
        \coloneqq \Fun^\times(-,\Spc)\circ \Span(\ul{\mathbb F})^\otimes
        \colon\Span_{\Forb}(\Fglo)
        \to{\Cat}(\sift)
    $.
\end{construction}

\begin{remark}
    If $G$ is any finite group, then the restriction of $\ul\Mack^\otimes$ along $\Span(\mathbb F_{G})\simeq\Span(\Forb_{/G})\to\Span_{\Forb}(\Fglo)$ recovers the normed $G$-category $\ul\NMon_G$ of \cite{CHLL_NRings}*{Section~5} (up to straightening--unstraightening as before).
\end{remark}

\begin{proposition}
    There is a unique equivariantly cocontinuous normed global functor $\mathbb P\colon \ul\Spc_*^\otimes\to\ul\Mack^\otimes$ sending $S^0$ to $\hom(1,-)\colon\Span(\mathbb F)\to\Spc$.
    \begin{proof}
        The argument is essentially the same as for a fixed group, cf.~\cite{CHLL_NRings}*{Proposition~5.3.1}.

        For uniqueness, we first claim that there is at most one equivariantly cocontinuous $\ul\Spc_*\to\ul\Mack$ sending $S^0$ to $\hom(1,-)$. Indeed, by the universal property of sifted cocompletion, this is equivalent to saying that there is a unique functor $\psi\colon \ul{\mathbb F}_*\to\ul\Mack$ sending $S^0$ to $\hom(1,-)$ and commuting with left adjoints to restrictions. As every element of $\ul{\mathbb F}_*$ can be written as $n_!f^*S^0$ for $f$ a map in $\Fglo$ and $n$ a map in $\Forb$, any such $\psi$ will necessarily land in the Yoneda image of $\Span(\ul{\mathbb F})$, so it will in particular be enough to show that there is a unique functor $\ul{\mathbb F}_*\to\Span(\ul{\mathbb F})$ sending $S^0$ to $1$ and commuting with left adjoints to restrictions. As $\Span(\ul{\mathbb F})$ is Borel (since $\ul{\mathbb F}=\mathbb F^\flat$ is Borel and $\Span$ preserves limits), this is equivalent to saying that there is a unique finite coproduct preserving functor $\mathbb F_*\to\Span(\mathbb F)$ sending $S^0$ to $1$. This follows again from the universal property of $\mathbb F_*$ in $\Cat^\amalg$.

        For uniqueness of the normed structure, we observe that it is similarly enough to show uniqueness of a normed structure on $\ul{\mathbb F}_*\to\Span(\ul{\mathbb F})$. Using the monoidal Borelification principle once more, this is equivalent to uniqueness of a normed structure on $\mathbb F_*\to\Span(\mathbb F)$, which follows from idempotency of both sides in $\CAlg(\Cat^\amalg)$, see \cite{harpaz2020ambidexterity}*{Corollary~5.5}.

        It remains to construct one such normed global left adjoint, for which we basically run the above argument backwards: the unique coproduct preserving functor $\mathbb F_*\to\Span(\mathbb F)$ sending $S^0$ to $1$ has a unique symmetric monoidal structure by idempotency, inducing $\ul{\mathbb F}_*^\otimes\to\Span(\ul{\mathbb F})^\otimes\hookrightarrow\ul\Mack^\otimes$. Using the universal property of sifted cocompletion, this then extends to $\ul\Spc_*^\otimes\to\ul\Mack^\otimes$, and it only remains to check that this is equivariantly cocontinuous. As in the proof of \cref{lem:normed-pt-spc-free-sift}, this can be checked after restricting to $\Span(\mathbb F_G)$ for every finite group $G$; the resulting functor of normed $G$-categories then admits the analogous description to the above, using the monoidal Borelification for $G$-categories instead of the one for global categories. However, for this functor it was verified in \cite{CHLL_NRings}*{proof of Proposition~5.3.1} that it is a $G$-equivariant left adjoint.
    \end{proof}
\end{proposition}

\begin{theorem}\label{thm:equivariant-uniqueness}
    There exists a unique pair of
    \begin{enumerate}
        \item a normed structure on $\ul\SpMack$ such that all norms preserve sifted colimits and each $\ul\SpMack(X)$ is a presentably symmetric monoidal category, and
        \item a normed structure on the stabilization map $\ell\colon\ul\Mack\to\ul\SpMack$.
    \end{enumerate}
    Moreover, there exists a unique normed equivalence $\ul\Sp^\otimes\simeq\ul\SpMack^\otimes$ making the following diagram of normed global categories and normed functors commute:
    \[
        \begin{tikzcd}
            \ul\Spc_{*}^\otimes\arrow[d,"\Sigma^\infty"']\arrow[r,"\mathbb P"] & \ul\Mack^\otimes\arrow[d,"\ell"]\\
            \ul\Sp^\otimes\arrow[r, "\sim"'] & \ul\SpMack^\otimes\kern-1.5pt\rlap.\kern1.5pt
        \end{tikzcd}
    \]
    \begin{proof}
        By \cite{CHLL_NRings}*{Lemma~5.5.3 and Proposition~5.5.4}, the stabilization map $\ul\Mack(X)\to\ul\SpMack(X)$ lifts uniquely to a map in $\CAlg(\Pr^\text{L})$, and this is the universal example of a map inverting $\Sigma\bbone$ in either $\CAlg(\Pr^\textup{L})$ or $\CAlg({\Cat}(\sift))$. This immediately proves the first half.

        For the second half, we observe that by Proposition 5.5.6 of \emph{op.~cit.} the composite $\ul\Spc_*(X)\to\ul\Mack(X)\to\ul\SpMack(X)$ is given by universally inverting $\Nm^X_1S^1$ in $\CAlg({\Cat}(\sift))$ whenever $X$ is connected. The same then holds for general $X$ by product preservation. As also the map $\Sigma^\infty\colon\ul\Spc^\otimes_*\to\ul\Sp^\otimes$ admits the same pointwise description as a consequence of \cite{gepnermeier2020equivTMF}*{Theorem~C.7} (cf.~\cite{CHLL_NRings}*{proof of Proposition 5.5.6}), the claim follows.
    \end{proof}
\end{theorem}

\begin{warn}
    Since the norms of $\ul\Mack^\tensor$ do not preserve \emph{all} colimits, the norms of $\ul{\SpMack}^\tensor$ cannot be simply obtained by taking the Lurie tensor product of $\und{\Mack}^\tensor$ with $\Sp$. Moreover, the norms for $\ul{\SpMack}^\tensor$ will no longer be given by left Kan extension along maps of span categories.
\end{warn}

\subsection{Norms for global Mackey functors}
Let us enhance $\ul\Mack_\gl$ to a normed global category:

\begin{lemma}\label{lem:goo-ring-context}
    The triple $(\Fglo,\Forb,\Forb)$ is a \emph{ring context} in the sense of \cite{CHLL_NRings}*{Definition~4.3.2}.

    \begin{proof}
        It is clear that $\Fglo$ has pullbacks, and we already mentioned in \cref{rem:orbital-dist} that $(\Fglo,\Forb)$ and $(\Fglo,\Fglo)$ are extensive span pairs by virtue of being free finite coproduct completions of orbital pairs. It then only remains to show that for every $m\colon X\to Y$ in $\Forb$ the pullback functor $m^*\colon\Fglo_{/Y}\to\Fglo_{/X}$ has a right adjoint $m_*$ mapping $\Fglo_{/X}\times_{\Fglo}\Forb$ to $\Fglo_{/Y}\times_{\Fglo}\Forb$. After straightening, this corresponds to saying that $m^*\colon\Fun(Y,\Fglo)\to\Fun(X,\Fglo)$ has a right adjoint restricting to $\Fun(X,\Forb)\to\Fun(Y,\Forb)$. This is immediate from the pointwise formula for right Kan extensions, as faithful functors are closed under limits.
    \end{proof}
\end{lemma}

\begin{construction}\label{constr:Mackgl}
    As $(\Fglo,\Forb,\Forb)$ is a (semi)ring context, \cite{CHLL_NRings}*{Lemma~4.1.7} extends $\Span_{\Forb}(\Fglo^\flat)$ to a normed global category $\Span_{\Forb}(\Fglo^\flat)^\otimes$, such that the norm $n_\otimes$ along $n\colon A\to B$ is given by left Kan extension along $\Span(n_*)$.\footnote{In general, $\Span(n_*)$ is not adjoint to $\Span(n^*)$, so $n_\otimes$ is \emph{not} just the right adjoint of $n^*$.}

    By the monoidal Borelification principle, this normed structure can also be uniquely characterized by what it does on the underlying category $\Span_{\Forb}(\Fglo)$. Simply by definition, this is obtained via the functoriality of $\Span$ from the cartesian symmetric monoidal structure on $\Fglo$.

    By \cite{CHLL_NRings}*{Proposition~3.3.1} this then again yields a normed structure on $\ul\Mack_\gl$ via free sifted cocompletion, see also \cite{CHLL_NRings}*{Definition~4.1.9}.
\end{construction}

\begin{remark}\label{rk:GGO}
    The argument from the proof of Lemma~\ref{lem:goo-ring-context} applies verbatim to show that the pullback functor $f^*\colon\Fglo_{/Y}\to\Fglo_{/X}$ more generally has a right adjoint $f_*$ restricting to $\Fglo_{/X}\times_{\Fglo}\Forb\to\Fglo_{/Y}\times_{\Fglo}\Forb$ for \emph{every} $f$ in $\Fglo$ (and not just those maps in $\Forb$), i.e.~also $(\Fglo,\Fglo,\Forb)$ is a ring context.
    One then obtains a `globally normed' global category $\ul\Mack_\gl^\botimes\colon\Span(\Fglo)\to{\Cat}$ extending $\ul\Mack_\gl$ in the same way as the previous construction (i.e.~with norms for arbitrary maps of groupoids, not only for the faithful ones).
    We will return to this observation in §\ref{subsec:mult-defl}.
\end{remark}

\begin{theorem}\label{thm:global-existence}
    There exists a unique pair of
    \begin{enumerate}
        \item an enhancement of $\ul\SpMack_\textup{gl}$ to a normed global category $\ul{\SpMack}_\textup{gl}^\otimes$ such that each $\ul{\SpMack}_\textup{gl}^\otimes(X)$ is presentably symmetric monoidal and all norms preserve sifted colimits, together with
        \item an extension of the stabilization map $\ul\Mack_\textup{gl}\to\ul\SpMack_\textup{gl}$ to a normed functor.
    \end{enumerate}
\end{theorem}

\begin{proof}
    A significant part of this argument is analogous to the equivariant case handled in \cite{CHLL_NRings}*{§5.5} and recalled in \cref{thm:equivariant-uniqueness} above, and accordingly we will be a bit terse about this part.

    Fix $X\in\Fglo$ and equip $\und{\Mack}_\gl(X)$ with the symmetric monoidal structure encoded in the normed structure. By \cite{CHLL_Bispans}*{Lemma 5.5.3}, the stabilization map $\und{\Mack}_\gl(X)\to\und{\SpMack}_\gl(X)$ lifts uniquely to a map in $\CAlg(\text{Pr}^\text{L})$, and the resulting map is the symmetric monoidal inversion of $\Sigma\bbone$ in $\CAlg(\Pr^\text{L})$. Combining \cite{BachmannHoyois2021Norms}*{Proposition~4.1} with \cite{CHLL_Bispans}*{Lemma~5.5.2}, we then see that this is also a symmetric monoidal inversion in $\CAlg({\Cat}(\sift))$; in particular, not only do symmetric monoidal left adjoints inverting $\Sigma\bbone$ lift uniquely through it, but this holds more generally for any sifted colimit preserving symmetric monoidal functor inverting $\Sigma\bbone$.

    With this universal property at hand, it will be enough to show that for any map
    \[\begin{tikzcd}
        X &\arrow["f"',l] Y \arrow["n",r,norm] & Z
    \end{tikzcd}\]
    in $\Span_{\Forb}(\Fglo)$ the functor $n_\otimes f^*\colon\und{\Mack}_\gl(X)\to\und{\Mack}_\gl(Z)$ preserves sifted colimits, and that $n_\otimes f^*\Sigma\bbone$ is inverted by $\und{\Mack}_\gl(Z)\to\und{\SpMack}_\gl(Z)$.

    For the statement about sifted colimits, it suffices to observe that $n_\otimes f^*$ is given by restricting a left adjoint
    $
        \Fun(\Span_{\Forb}(\Fglo^\flat(X)),\Spc)
        \to \Fun(\Span_{\Forb}(\Fglo^\flat(Y)),\Spc)
    $
    to the categories of product preserving functors, and the latter are closed under sifted colimits.

    For the invertibility of $n_\otimes f^*\Sigma\bbone$ in $\ul\SpMack_{\gl}^\tensor(Z)$, note first that $\Sigma\bbone$ is invertible in $\SpMack(Z)$ and that both $f^*$ and $n_\otimes$ admit natural symmetric monoidal structures. It therefore suffices to show that each $\ul\SpMack(X)\to\ul\SpMack_\gl(X)$ admits a symmetric monoidal structure. Comparing the universal properties of symmetric monoidal inversion of $S^1$ and of usual (non-monoidal) stabilization, it suffices to construct such a structure on the composite $\ul\Mack(X)\to\ul\SpMack(X)\to\ul\SpMack_\gl(X)$. But we can also factor this as $\ul\Mack(X)\to\ul\Mack_\gl(X)\to\ul\SpMack_\gl(X)$, where the second map comes with a symmetric monoidal structure by design, while $i_!\colon\ul\Mack\to\ul\Mack_\gl$ even has a normed structure induced by the unique symmetric monoidal structure on $\mathbb F\hookrightarrow\Fglo$.
\end{proof}

Using this, we can now state our comparison result:

\goodbreak
\begin{theorem}\label{thm:normed-comparison}\
    \begin{enumerate}
        \item The normed global functor $i_!\colon\ul\Mack^\otimes\hookrightarrow\ul\Mack_\gl^\otimes$ lifts uniquely to a normed global functor $i_!\colon\ul\SpMack^\otimes\hookrightarrow\ul\SpMack_\gl^\otimes$.
        \item This normed global functor has an $\mathbb E$-lax right adjoint $i^*\colon\ul\SpMack^\otimes_\gl\to\ul\SpMack^\otimes$, which exhibits $\ul\SpMack^\otimes_\gl$ as the normed globalization of $\ul\SpMack^\otimes$.
        \item There is a preferred equivalence $\ul\Sp_\gl^\otimes\simeq\ul\SpMack^\otimes_\gl$ of normed global categories.
    \end{enumerate}
\end{theorem}

We will need one additional ingredient for the proof:

\begin{lemma}\label{lemma:right-adjoint-strong-normed}
    The normed global functor $i_!\colon\ul\Mack^\otimes\hookrightarrow\ul\Mack^\otimes_\gl$ has an $\mathbb E$-lax right adjoint $i^*$.
    \begin{proof}
        Note that $i_!$ has a pointwise right adjoint $i^*$, which is simply given by restriction along $\Span(i^\flat)$ (as $i$ preserves coproducts).

        If $m\colon A\rightarrowmono B$ is a map in $\Forb$, then $m^*\colon\ul\Mack(B)\to\ul\Mack(A)$ is given by left Kan extension along $\Span(m^*)\colon\Span(\mathbb F^\flat(B))\to\Span(\mathbb F^\flat(A))$, and similarly for $\ul\Mack_\gl$. By an easy application of \cite{BachmannHoyois2021Norms}*{Proposition~C.21}, the vertical arrows in
        \begin{equation}\label{diag:Forb-shriek}
            \begin{tikzcd}
                \Span(\mathbb F^\flat(B))\arrow[r, "i"] & \Span_\Forb(\Fglo^\flat(B))\\
                \Span(\mathbb F^\flat(A))\arrow[u, "\Span(m^*)"]\arrow[r,"i"'] & \Span_\Forb(\Fglo^\flat(A))\arrow[u, "\Span(m^*)"']
            \end{tikzcd}
        \end{equation}
        have a left adjoints induced by the left adjoints $\mathbb F^\flat(A)\to\mathbb F^\flat(B)$ and $\mathbb G^\flat(A)\to\mathbb G^\flat(B)$. The diagram without the spans is vertically left adjointable by direct inspection, hence so is $(\ref{diag:Forb-shriek})$. By $2$-functoriality of $\Fun^\times(-,\Spc)$, we conclude that $i_!$ commutes with left adjoints to restrictions along $m$, so that $i^*$ commutes with restrictions along $m$.

        It remains to show that $i^*$ commutes with norms, for which it suffices that for every map $n\colon A\rightarrownorm B$ in $\Forb$, the square
        \begin{equation}\label{diag:i^*-normed}
            \begin{tikzcd}
                \Fun(\Span(\mathbb F^\flat(A)),\Spc)\arrow[d,"i_!"']\arrow[r, "\Span(n_*)_!"] &[1.5em] \Fun(\Span(\mathbb F^\flat(B)),\Spc)\arrow[d,"i_!"]\\
                \Fun(\Span_\Forb(\Fglo^\flat(A)),\Spc)\arrow[r, "\Span(n_*)_!"'] & \Fun(\Span_\Forb(\Fglo^\flat(B)),\Spc)
            \end{tikzcd}
        \end{equation}
        is vertically right adjointable, or equivalently that the diagram where we passed to right adjoints everywhere is vertically right adjointable (for this to make sense, we had to get rid of product preservation). For this we want to apply \cite{CHLL_Bispans}*{Theorem~4.1.5} to the diagram
    \begin{equation}\label{diag:chill-basechange}
            \begin{tikzcd}
                \Span(\mathbb F^\flat(A))\arrow[d,"i"']\arrow[r, "\Span(n_*)"] &[1.5em] \Span(\mathbb F^\flat(B))\arrow[d,"i"]\\
                \Span_\Forb(\Fglo^\flat(A))\arrow[r, "\Span(n_*)"'] & \Span_\Forb(\Fglo^\flat(B))
            \end{tikzcd}
        \end{equation}
        with the standard factorization systems on the span categories:
        \begin{enumerate}
            \item For any $C\in\Fglo$, the map $i\colon\mathbb F^\flat(C)\hookrightarrow\Forb^\flat(C)$ is a sieve and hence in particular a right fibration. Setting $C=A,B$ we see that both vertical maps in $(\ref{diag:chill-basechange})$ restrict to right fibrations on the monic parts of the factorization systems.
            \item The diagram
            \[
                \begin{tikzcd}
                    \PSh(\mathbb F^\flat(A))\arrow[r, "(n_*)_!"]\arrow[d,"i_!"'] & \PSh(\mathbb F^\flat(B))\arrow[d, "i_!"]\\
                    \PSh(\Fglo^\flat(A))\arrow[r, "(n_*)_!"'] & \PSh(\Fglo^\flat(B))
                \end{tikzcd}
            \]
            is horizontally left adjointable (as it is so before applying presheaves), hence vertically right adjointable. Thus, also the diagram obtained by passing to right adjoints everywhere is vertically right adjointable.
        \end{enumerate}
        With these established, the aforementioned \cite{CHLL_Bispans}*{Theorem~4.1.5} applies to show that $(\ref{diag:i^*-normed})$ is vertically right adjointable, as claimed.
    \end{proof}
\end{lemma}

\begin{proof}[Proof of Theorem~\ref{thm:normed-comparison}]
    By the universal property of symmetric monoidal inversion in $\CAlg({\Cat}(\sift))$, $i_!\colon\ul\Mack^\otimes\hookrightarrow \ul\Mack^\otimes_\gl$ induces a normed global functor $i_!\colon\ul\SpMack^\otimes\hookrightarrow\ul\SpMack^\otimes_\gl$ sitting in a commutative diagram
    \begin{equation}\label{diag:mack-spmack}
        \begin{tikzcd}
            \ul\Mack^\otimes\arrow[d,"\ell"']\arrow[r, "i_!", hook] & \ul\Mack^\otimes_\gl\arrow[d,"\ell"]\\
            \ul\SpMack^\otimes\arrow[r, "i_!"', hook] & \ul\SpMack^\otimes_\gl\rlap.
        \end{tikzcd}
    \end{equation}
    In particular, the underlying global functor $\ul\SpMack\hookrightarrow\ul\SpMack_\gl$ is the functor $i_!$ from Theorem~\ref{thm:unnormed-comparison}, hence left adjoint to a globalization. \Cref{thm:Normed_Rig} will then show that also the right adjoint to $\ul\SpMack^\otimes\hookrightarrow\ul\SpMack_\gl^\otimes$ is a normed globalization, as soon as we show that this right adjoint commutes strongly with norms.

    We will show that the $\mathbb E$-lax right adjoint $i^*\colon\ul\Mack_\gl^\otimes\to\ul\Mack^\otimes$ from the previous lemma lifts to $\ul\SpMack_\gl^\otimes\to\ul\SpMack^\otimes$; as this lift will then be the desired right adjoint by $2$-functoriality of symmetric monoidal inversion, this will in particular imply the claim.

    Observe first that $i^*\colon\ul\Mack_\gl\to\ul\Mack$ preserves fiberwise colimits: namely, sifted colimits are computed in the whole functor category (on which $i^*$ admits a right adjoint via right Kan extension), while coproducts are computed as products (which it preserves as $i^*$ at the same time \emph{is} a right adjoint). Thus, we in particular see that $i^*\Sigma\bbone\simeq \Sigma i^*\bbone\simeq\Sigma\bbone$, so that $i^*$ indeed induces a functor $\ul\SpMack_\gl^\otimes\to\ul\SpMack^\otimes$ by the universal property of symmetric monoidal inversion.

    Altogether, we have shown that $i_!\colon\ul\SpMack^\otimes\hookrightarrow\ul\SpMack_\gl^\otimes$ is left adjoint to the universal $\mathbb E$-lax functor to $\ul\SpMack^\otimes$.  As $\ul\Sp^\otimes\hookrightarrow\ul\Sp^\otimes_\gl$ is similarly left adjoint to the universal $\mathbb E$-lax functor to $\ul\Sp^\otimes$ by Corollary~\ref{cor:globalize-sp-monoidal}, the equivalence $\ul\Sp^\otimes\simeq \ul\SpMack^\otimes$ from Theorem~\ref{thm:equivariant-uniqueness} then globalizes to $\ul\Sp^\otimes_\gl\simeq\ul\SpMack^\otimes_\gl$, as desired.
\end{proof}

\subsection{Ultra-commutative global ring spectra as global Tambara functors} \ \\
In \cite{CHLL_NRings}*{Theorems~5.6.1 and~5.6.3}, Cnossen, Haugseng and a non-empty subset of the present authors showed that \emph{connective} equivariant spectra assemble into a full normed subcategory of $\ul\Sp^\otimes$, and that equivariant normed algebras in this can be identified with \emph{space-valued Tambara functors}. In light of our Theorem~\ref{thm:equivariant-model}, we can restate this as follows:

\begin{theorem}\label{thm:equivariant-Tambara}
	For every finite group $G$, the subcategory $\UCom_\textup{$G$,$\ge0$}$ of connective ultra-commutative $G$-equivariant ring spectra is equivalent to the full subcategory of
	\[
	\Fun^\times(\Bispan(\mathbb F_{G}),\Spc)
	\]
	spanned by the \emph{$\bm G$-Tambara functors}, i.e.~those functors $F$ such that for every $X\in\mathbb F_{G}$ the h-space structure $F(X)\times F(X)\to F(X)$ induced by the bispan
	\[
	X\amalg X\xleftarrow{\;=\;} X\amalg X\xrightarrow{\;=\;}X\amalg X\xrightarrow{\;(\id,\id)\;} X
	\]
	is grouplike.\qed
\end{theorem}

\begin{remark}
    For more details about categories of bispans we refer the reader to \cite{CHLL_NRings}*{§2.3}.
\end{remark}

Recall that a $G$-global spectrum $X$ is called \emph{connective} \cite{g-global}*{Definition~3.4.11} if the underlying equivariant spectrum of $\phi^*X$ is connective for every $\phi\colon H\to G$. As an equivariant spectrum is connective iff its associated spectral Mackey functor factors through $\Sp_{\ge0}\subset\Sp$, an easy diagram chase shows that the equivalence $\ul\Sp_\gl\simeq\ul\SpMack_\gl$ restricts to an equivalence between the global subcategory given in degree $G$ by the connective global spectra and the one given in degree $G$ by the product preserving functors $\Span_{\Forb}(\Fglo^\flat(G))\to\Sp_{\ge0}$.

\begin{proposition}\label{prop:deloc-normed}
    The full global subcategory $\und{\SpMack}_{\gl,\geq0}$ is closed under norms,
    and the normed stabilization map factors as
    \[
        \und{\Mack}_\gl^\tensor \xto{S^\tensor} \und{\SpMack}_{\gl,\geq0}^\tensor \subset \und{\SpMack}_{\gl}^\tensor.
    \]
    Moreover, $S^\tensor$ has a fully faithful lax normed right adjoint $\iota \colon \und{\SpMack}_{\gl,\geq0}^\tensor \hookrightarrow \und{\Mack}_\gl^\tensor$.
    \begin{proof}
        Identifying
        $
            \Fun^\times(\Span_{\Forb}(\Fglo^\flat(X)),\Spc)
            \simeq \Fun^\times(\Span_{\Forb}(\Fglo^\flat(X)),\CMon),
        $
        the stabilization map is simply given by postcomposition with the Bousfield localization $\CMon\to\Sp_{\ge 0}$.
        From this description, it is evident that the stabilization map factors through a
        left Bousfield localization $S \colon \und{\Mack}_\gl \to \und{\SpMack}_{\gl,\geq0}$.
        Since the stabilization map is normed, $\und{\SpMack}_{\gl,\geq0}$ inherits the normed
        structure from $\und{\SpMack}_{\gl}^\tensor$ and $S$ canonically upgrades to a normed functor,
        as desired.
    \end{proof}
\end{proposition}

We can now prove the following global version of Theorem~\ref{thm:equivariant-Tambara}:

\begin{theorem}\label{thm:global-Tambara}
    For every finite group $G$, the category $\UCom_{G\text{-}\gl,\geq0}$ of connective ultra-commutative $G$-global ring spectra is equivalent to the full subcategory of
    \[
        \Fun^\times(\Bispan(\Fglo_{/G},\Fglo_{/G}\times_{\Fglo}\Forb,\Fglo_{/G}\times_{\Fglo}\Forb),\Spc)
    \]
    spanned by the \emph{$\bm G$-global Tambara functors}, i.e.~those functors $F$ such that for every $X\in\Fglo_{/G}$ the h-space structure $F(X)\times F(X)\to F(X)$ induced by the bispan
    \[
        X\amalg X\xleftarrow{\;=\;} X\amalg X\xrightarrow{\;=\;}X\amalg X\xrightarrow{\;(\id,\id)\;} X
    \]
    is grouplike.
\end{theorem}

\begin{proof}
    We claim that we have a commutative diagram
    \[\begin{tikzcd}
        {\UCom_{G\text{-}\gl,\geq0}} & {\und{\CAlg}_\gl(\und{\SpMack}_{\gl,\geq0}^\tensor)(G)} &[1em] {\und{\CAlg}_\gl(\und{\Mack}_\gl^\tensor)(G)} \\
        {\Sp_{G\text{-}\gl,\geq0}} & {\und{\SpMack}_{\gl,\geq0}(G)} & {\und{\Mack}_{\gl}(G)}\rlap.
        \arrow["\sim", from=1-1, to=1-2]
        \arrow["\bbU"', from=1-1, to=2-1]
        \arrow["\iota\circ-", from=1-2, to=1-3]
        \arrow["\bbU", from=1-2, to=2-2]
        \arrow["\bbU", from=1-3, to=2-3]
        \arrow["\sim"', from=2-1, to=2-2]
        \arrow["\iota"', from=2-2, to=2-3]
    \end{tikzcd}\]
    The left square uses the normed equivalence $\und{\Sp}_\gl^\tensor \simeq \und{\SpMack}^\tensor_\gl$
    of \cref{thm:normed-comparison} together with the fact that the equivalence of \cref{thm:global-model} is compatible with the forgetful functors $\bbU$, and we can check connectivity after applying $\bbU$.
    The right square uses the fully faithful lax-normed right adjoint $\iota$ of the above proposition;
    note that the induced map on cocartesian unstraightenings is again fully faithful as its left adjoint is a localization. Hence the postcomposition functor $\iota\circ-$ is also fully faithful.

    It is now clear that the top horizontal composite identifies $\UCom_{G\text{-}\gl,\geq0}$ with the full subcategory $\cC \subset \und{\CAlg}_\gl(\und{\Mack}_\gl^\tensor)(G)$ on those normed algebras whose underlying $G$-global Mackey functor $M$ factors through $\text{CGrp}\subset\CMon$, i.e.~those such that the h-space structure $M(X)\times M(X)\to M(X)$ induced by
    \[
        X\amalg X\xleftarrow{\;=\;}X\amalg X\xrightarrow{\;(\id,\id)\;}X
    \]
    is grouplike. Finally, applying \cite{CHLL_NRings}*{Theorem 4.3.6} to the ring context $(\Fglo_{/G}, \Fglo_{/G}\times_{\Fglo}\Forb,\Fglo_{/G}\times_{\Fglo}\Forb)$ identifies $\Cc$ with the category of $G$-global Tambara functors.
\end{proof}

\subsection{Multiplicative deflations}\label{subsec:mult-defl}
Recall from Remark~\ref{rk:GGO} that also $(\Fglo,\Fglo,\Forb)$ is a ring context. Applying Construction~\ref{constr:Mackgl} in this context then provides an extension of $\ul\Mack_\gl^\otimes$ to a `globally normed global category,' i.e.~a functor $\ul\Mack_\gl^\botimes\colon\Span(\Fglo)\to\Cat$, such that for every surjective group homomorphism $p\colon G\to G/N$ the functor
\begin{equation}\label{eq:geometric-fixed-points}
    \Phi^N\coloneqq p_\otimes\colon \und{\Mack}_\gl(G) \to \und{\Mack}_\gl(G/N)
\end{equation}
is given by left Kan extension along
\[
    \Span(p_*)\colon\Span_{\Forb}(\Fglo^\flat(G))\to\Span_{\Forb}(\Fglo^\flat(G/N)).
\]
\begin{lemma}
    There exists a unique pair of
    \begin{enumerate}
        \item an extension of $\ul\SpMack_\gl^\otimes$ to $\ul\SpMack_\gl^\botimes\colon\Span(\Fglo)\to{\Cat}(\sift)$, and
        \item an extension of the stabilization map $\ell\colon\ul\Mack_\gl^\otimes\to\ul\SpMack_\gl^\otimes$ to a globally normed functor $\ul\Mack_\gl^\botimes\to\ul\SpMack_\gl^\botimes$.
    \end{enumerate}
    Moreover, in this extension the norms $\Phi^N$ along surjective group homomorphisms actually preserve all colimits.
    \begin{proof}
        Arguing as before (and with Theorem~\ref{thm:global-existence} at hand), for the existence and uniqueness it suffices that $\ell\Phi^N(\Sigma\bbone)$ is invertible for every normal subgroup $N\leq G$. We will show that $(\ref{eq:geometric-fixed-points})$ is actually a left adjoint; this will immediately imply $\Phi^N(\Sigma\bbone)\simeq\Sigma\Phi^N(\bbone)\simeq\Sigma\bbone$, which is inverted by $\ell$ by definition, and it will also imply that its stabilization $\und{\SpMack}_\gl(G)\to\und{\SpMack}_\gl(G/N)$ is again cocontinuous.

        To prove that $(\ref{eq:geometric-fixed-points})$ is a left adjoint, we observe that $p_*=(-)^{hN}\colon\Fglo^\flat(G)\to\Fglo^\flat(G/N)$ preserves finite coproducts, so that $\Span(p_*)$ preserves finite products. Accordingly,
        \[
            \Span(p_*)^*\hskip-1pt\colon{\Fun}\big({\Span}_{\Forb}\big(\Fglo^\flat(G/N)\big),{\Spc}\big)\to{\Fun}\big({\Span}_{\Forb}\big(\Fglo^\flat(G)\big),{\Spc}\big)
        \]
        restricts to the desired right adjoint.
    \end{proof}
\end{lemma}

\begin{construction}
    Analogously, one obtains globally normed categories $\ul\Mack^\botimes$ and $\ul\SpMack^\botimes$. As the inclusion $i\colon\mathbb F\hookrightarrow \Fglo$ preserves all limits, we can then lift $(\ref{diag:mack-spmack})$ to a commutative diagram
    \[
        \begin{tikzcd}
            \ul\Mack^\botimes\arrow[d,"\ell"']\arrow[r, "i_!", hook] & \ul\Mack_\gl^\botimes\arrow[d,"\ell"]\\
            \ul\SpMack^\botimes \arrow[r, "i_!"', hook] & \ul\SpMack^\botimes_\gl
        \end{tikzcd}
    \]
    of globally normed categories.
\end{construction}

\begin{theorem}
    \
    \begin{enumerate}
        \item The functor $i_!\colon\ul\SpMack^\botimes\hookrightarrow \ul\SpMack^\botimes_\gl$ admits an $\mathbb E$-lax right adjoint $i^*$, and this is the universal $\mathbb E$-lax $\Span(\Fglo)$-functor to $\ul\SpMack^\botimes$.
        \item There is an equivalence
        \begin{equation}
            \begin{aligned}\label{eq:comparison-globally-normed-calg}
            \und{\CAlg}_\Fglo^\Fglo(\ul\SpMack_\gl^\botimes)(G)
            \coloneqq{}&\Fun_{\Span(\Fglo)}^\lax{\Fglo}(\Span(G\text{-}\Fglo), \ul\SpMack_\gl^\botimes)\\
            \simeq{}&\Fun_{\Span(\Fglo)}^{\neglax{\Forb^\op}}(\Span(G\text-\Fglo),\ul\SpMack^\botimes)
            \end{aligned}
        \end{equation}
        natural in $G\in\Fglo^\op$ and sitting in a commutative triangle
        \[
            \begin{tikzcd}[column sep=small]
                \und{\CAlg}_\Fglo^\Fglo(\ul\SpMack_\gl^\botimes)\arrow[dr, bend right=15pt, "i^*\circ\mathbb U"']\arrow[rr,"\sim"] &&[-1em] \Fun^{\neg\Forb^\op\textup{-lax}}_{\Span(\Fglo)}(\Span(\Fglo^\flat),\ul\SpMack^\botimes)\arrow[dl, "\ev_{\id}", bend left=15pt]\\
                & \ul\SpMack^\botimes
            \end{tikzcd}
        \]
        of partially lax functors.
    \end{enumerate}
\end{theorem}

For $G=1$, the right hand side in $(\ref{eq:comparison-globally-normed-calg})$ is the category of \emph{global commutative ring spectra with multiplicative deflations} from \cite{deflations} mentioned in Remark~\ref{rk:BMY}.

\begin{proof}
    Arguing precisely as in the proof of Lemma~\ref{lemma:right-adjoint-strong-normed} (which, by direct inspection, never used that the map denoted $n\colon A\to B$ there was actually contained in $\Forb\subset\Fglo$) we see that $i_!$ has an $\mathbb E$-lax right adjoint $i^*\colon\ul\SpMack_\gl^\botimes\to\ul\SpMack^\botimes$. By \cref{thm:Normed_Rig}, the universal property can then be checked on underlying global categories, which is part of Theorem~\ref{thm:unnormed-comparison}.
    With this established, the second statement is an instance of Theorem~\ref{thm:algebras_in_Rig}.
\end{proof}

Finally, we can also give a Tambara functor description for the connective objects:

\begin{theorem}\label{thm:BMY-Tambara}\
    \begin{enumerate}
        \item The connective $G$-global spectra for varying $G$ assemble into a globally normed global subcategory $\ul\SpMack_{\gl,\ge0}^\botimes\subset\ul\SpMack_{\gl,\ge0}^\botimes$.
        \item The stabilization map $\ul\Mack_\gl^\botimes\to\ul\SpMack_\gl^\botimes$ induces a Bousfield localization onto $\ul\SpMack_{\gl,\ge0}^\botimes$. The essential image of the (fully faithful, lax normed) right adjoint consists precisely of the grouplike Mackey functors.
        \item For every $G\in\Fglo$ there is an equivalence between $\ul\CAlg^\botimes_\gl(\ul\SpMack_{\gl,\ge0}^\botimes)(G)$ and the full subcategory of
        \[
            \Fun^\times(\Bispan(\Fglo_{/G},\Fglo_{/G},\Forb_{/G}),\Spc)
        \]
        spanned by the Tambara functors (defined as before).
    \end{enumerate}
\end{theorem}

In particular, for $G=1$ this means that connective global commutative ring spectra with multiplicative deflations can equivalently be viewed as Tambara functors $\Bispan(\Fglo,\Fglo,\Forb)\to\Spc$.

\begin{proof}
    As we have already improved the stabilization to a globally normed functor $\ul\Mack_\gl^\botimes\to\ul\SpMack_\gl^\botimes$, the first two statements follow precisely as in \cref{prop:deloc-normed}. The final statement is then again a consequence of \cite{CHLL_NRings}*{Theorem~4.3.6}.
\end{proof}

\appendix
\section{Orthogonal spectra vs.\ symmetric spectra}\label{app:orthogonal}
\begin{convention}
    In this appendix only, we will refer to ordinary $1$-categories as \emph{categories} and to $(\infty,1)$-categories as $\infty$-categories
    (unlike in the rest of this paper, where category means $(\infty,1)$-category).
\end{convention}

Schwede originally introduced \emph{ultra-commutative global ring spectra} in terms of a certain \emph{global model structure} on the category of strictly commutative algebras in orthogonal spectra \cite{schwede2018global}*{Theorem 5.4.3}. We will write $\UCom_\gl^\text{Lie}$ for the associated $\infty$-category (with the superscript indicating that Schwede's model comes with genuine fixed points for all compact Lie groups, and not just all finite groups). On the other hand, our reference model for ultra-commutative global ring spectra (and all other kinds of spectra) is based on \emph{symmetric} spectra, following \cite{hausmann-equivariant, hausmann-global}, and the goal of this short appendix is to compare these two.

\begin{construction}
    The category $\Sp^O$ of orthogonal spectra in the category $\text{Top}$ of compactly generated weak Hausdorff spaces comes with a forgetful functor $\Sp^O\to\Sp^\Sigma(\text{Top})$ to the category of symmetric spectra in $\text{Top}$. Postcomposing with the singular complex functor, we obtain a functor $i^*\colon \Sp^O\to\Sp^\Sigma$ to the category of symmetric spectra in simplicial sets, as considered in the main text.

    The functor $i^*$ has a left adjoint $i_!$, given explicitly as a composition
    \[
        \Sp^\Sigma\xrightarrow{\;|\cdot|\;}\Sp^\Sigma(\text{Top})\xrightarrow{\text{enriched left Kan extension}}\Sp^O.
    \]
    The second arrow has a natural symmetric monoidal structure by \cite{MMSS}*{Proposition~3.3}, while the first inherits one from the fact that geometric realization preserves products (see also §19 of \emph{op.~cit.}). In summary, we have an adjunction
    \begin{equation}\label{eq:sym-vs-orth}
        i_!\colon\Sp^\Sigma\rightleftarrows\Sp^O\noloc i^*
    \end{equation}
    with symmetric monoidal left adjoint, so that $i^*$ admits a unique \emph{lax} symmetric monoidal structure making unit and counit into symmetric monoidal transformations. In particular, $(\ref{eq:sym-vs-orth})$ lifts to an adjunction
    \begin{equation}\label{eq:CAlg-sym-vs-orth}
        i_!\colon\CAlg(\Sp^\Sigma)\rightleftarrows\CAlg(\Sp^O)\noloc i^*.
    \end{equation}
\end{construction}

We can now state our comparison:

\begin{theorem}\label{thm:orth-vs-sym}
    The right adjoint $i^*$ in $(\ref{eq:CAlg-sym-vs-orth})$ descends to a functor of $\infty$-categories $\UCom_\gl^\textup{Lie}\to\UCom_\gl$, and this functor is a right Bousfield localization at the `$\Fin$-global weak equivalences,' i.e.~the weak equivalences of the $\Fin$-global model structure of \cite{schwede2018global}*{Theorem 4.3.17}.
\end{theorem}

In other words, once one forgets that manifolds of positive dimension exist, there is no difference between Schwede's original approach and Hausmann's model used in this paper.

For the proof of the theorem, we recall several model category structures on the categories involved:

\goodbreak
\begin{proposition}\label{prop:model-structures}\
    \begin{enumerate}
        \item The flat cofibrations of orthogonal spectra and the global (weak) equivalences of \cite{schwede2018global}*{Definition~4.1.3} are part of a model structure on $\Sp^O$, called the \emph{global model structure}. This also has a variant where the cofibrations are instead only the \emph{positive} flat cofibrations (i.e.~flat cofibrations $f\colon X\to Y$ such that $f_0$ is a homeomorphism).
        \item The positive global model structure on $\Sp^O$ transfers to $\CAlg(\Sp^O)$, i.e.~the latter has a model structure where the weak equivalences and fibrations are created by the forgetful functor.
        \item The flat cofibrations of symmetric spectra and the global weak equivalences (as considered in this paper) are part of a \emph{global model structure} on $\Sp^\Sigma$. Again, this has a variant where the cofibrations are the \emph{positive} flat cofibrations.
        \item The positive global model structure on $\Sp^\Sigma$ transfers to $\CAlg(\Sp^\Sigma)$.
    \end{enumerate}
    \begin{proof}
        The first statement is \cite{schwede2018global}*{Theorem~4.3.18 and Proposition~4.3.33}, while the second one can be found as Theorem~5.4.3 of \emph{op.~cit.}
        The statements about the symmetric spectrum models appear as \cite{hausmann-global}*{Theorem~2.18 and Theorem~3.5}.
    \end{proof}
\end{proposition}

\begin{theorem}[Hausmann]
    The adjunction $(\ref{eq:sym-vs-orth})$ is a Quillen adjunction when we equip both sides with the global model structures, or when we equip both sides with the positive global model structures. In each case, the right adjoint is homotopical, and the induced functor on Dwyer--Kan localizations is a right Bousfield localization at the $\Fin$-global weak equivalences.
    \begin{proof}
        Orthogonal spectra also carry a \emph{$\Fin$-global model structure} with the $\Fin$-global weak equivalences as weak equivalences and fewer cofibrations, see \cite{schwede2018global}*{Theorem~4.3.17}. Combining \cite{hausmann-global}*{Proposition~2.19 and Theorem~5.3} shows that $(\ref{eq:sym-vs-orth})$ is a Quillen equivalence for this model structure, and as noted in the proof there, the right adjoint is homotopical. This proves the statement about the usual global model structures.

        For the part about the positive global model structures it only remains to show that $i_!\dashv i^*$ is a Quillen adjunction. As the positive global model structure has the same weak equivalences as, but fewer cofibrations than the usual global one, it is clear that $i_!$ sends acyclic cofibrations to weak equivalences. It therefore only remains to show that $i_!$ sends cofibrations to cofibrations. For this it suffices to observe that the slice of the inclusion $i\colon \bm\Sigma\to\textbf{O}$ between the topologically enriched indexing categories for symmetric and orthogonal spectra, respectively, over the trivial inner product space is contractible, so that $(i_!X)_0\cong |X_0|$ naturally in $X\in\Sp^\Sigma$.
    \end{proof}
\end{theorem}

\begin{proof}[Proof of Theorem~\ref{thm:orth-vs-sym}]
    We equip both sides of the adjunction (\ref{eq:CAlg-sym-vs-orth}) with the transferred model structures recalled in Proposition~\ref{prop:model-structures}. The previous theorem then immediately shows that $i^*$ is homotopical right Quillen, and hence $i_!$ is in particular left Quillen.

    To see that the left derived functor $\textbf{L}i_!$ is fully faithful, it then suffices to show that the derived unit $\id\to i^*\textbf{L}i_!$ is an equivalence, or equivalently, that the ordinary unit $X\to i^*i_!X$ is a global weak equivalence for any cofibrant $X\in\CAlg(\Sp^\Sigma)$. This can be checked on underlying symmetric spectra, where this is a consequence of the previous theorem as $X$ is in particular flat as a symmetric spectrum by \cite{shipley-convenient}*{Corollary~4.3}.

    We conclude that $\textbf{L}i_!\dashv i^*$ is a right Bousfield localization. As $i^*$ inverts precisely the $\Fin$-global weak equivalences by yet another application of the previous theorem, this completes the proof.
\end{proof}

\begin{remark}
    Using analogous arguments, one deduces from \cite{hausmann-equivariant}*{Theorem 7.5} that $\UCom_G$ can be equivalently described as the Dwyer--Kan localization at the $G$-equivariant weak equivalences of the category of strictly commutative algebras in orthogonal spectra with $G$-action, as considered in \cite{mandell-may}*{Section~III.8}.
\end{remark}

\section{Basechange lemmas}\label{app:basechange}
\subsection{Smooth and proper basechange} This appendix follows §§\ref{subsec:cleft_colimits}--\ref{subsec:cleft-Kan} in that our basic objects of study are $T$-precategories for some fixed category $T$, cf.\ \cref{warn:t-cat}.
The reader should think of the cases $T = \cF$ and $T= \Span_\cN(\cF)^\op$ for a span pair $(\cF,\cN)$,
which recovers the setting of ($\cN$-normed) $\cF$-precategories from the main text.
Our first goal is to prove a $T$-parametrized version of the classical
smooth/proper basechange theorem for Kan extensions,
see e.g.~\cite[Theorem 6.4.13]{cisinski-book}.\footnote{Unfortunately, the convention for the terminology of smooth and proper functors in \cite{cisinski-book} is opposite to that of \cite{HTT} and \cite{martini2021yoneda,martiniwolf2021limits} (although they agree on initial / final functors). Since we heavily cite the latter, we have decided to adopt their conventions.}

Recall from \cite[Definition 4.4.4]{martini2021yoneda} and the discussion right below it that a functor $p \colon \cX \to \cY$ of $T$-precategories is \emph{proper} if for any diagram of pullback squares
\[\begin{tikzcd}
    \mathcal{S} & \cV & \cX \\
    \mathcal{R} & \cZ & \cY
	\arrow["j", from=1-1, to=1-2]
	\arrow[from=1-1, to=2-1]
	\arrow["\lrcorner"{anchor=center, pos=0.125}, draw=none, from=1-1, to=2-2]
	\arrow[from=1-2, to=1-3]
	\arrow["q"', from=1-2, to=2-2]
	\arrow["\lrcorner"{anchor=center, pos=0.125}, draw=none, from=1-2, to=2-3]
	\arrow["p", from=1-3, to=2-3]
	\arrow["i", from=2-1, to=2-2]
	\arrow[from=2-2, to=2-3]
\end{tikzcd}\]
if $i$ is initial in the sense of \cite[Section 4.3]{martini2021yoneda}, then so is $j$.
Dually, we say that $p$ is \emph{smooth} if $p^\op \colon \cX^\op \to \cY^\op$ is proper,
or equivalently if for any pullback square as above, if $i$ is final then so is $j$.
The most important class of such functors come from left and right fibrations, i.e.~those functors internally right orthogonal to initial and final functors, respectively:

\begin{proposition}[{{\cite[Proposition 4.4.7]{martini2021yoneda}}}]\label{prop:lfib-smooth}
    Any left (right) fibration of $T$-precategories is smooth (proper).\qed
\end{proposition}

Importantly, being a left or right fibration of $T$-precategories is a pointwise condition:

\begin{proposition}\label{prop:param-lfib}
    A $T$-functor $p \colon \cX \to \cY$ is a left (right) fibration of $T$-precategories if and only if $p(A) \colon \cX(A) \to \cY(A)$ is a left (right) fibration for each $A \in T$.
\end{proposition}
\begin{proof}
    By \cite[Section 3.5]{martini2021yoneda} we can equivalently view $T$-precategories as complete Segal objects in $\PSh(T)$.
    Let $\cX_\bullet \to \cY_\bullet$ be the map of complete Segal objects in $\PSh(T)$
    corresponding to $p$ under this equivalence, so that for every $A \in T$ the category $\cX(A)$ corresponds to the complete Segal space $\cX_\bullet(A)$.

    By \cite[Proposition 4.1.3]{martini2021yoneda} $p$ is a left fibration
    of $T$-precategories if and only if a certain map $\cX_n \to \cX_0 \times_{\cY_0} \cY_n$ defined using the simplicial structure is an equivalence, for every $n \geq 1$.
    This can be checked after evaluating at each $A \in T$,
    which by another application of the cited Proposition (this time for $T = 1$)
    is equivalent to $\cX(A) \to \cY(A)$ being a left fibration for every $A \in T$.
\end{proof}

\begin{lemma}\label{lem:basechange-omega}
    Consider a cartesian square of small $T$-precategories
    \begin{equation}\label{diag:basechange-cart-square}\begin{tikzcd}
        \cV & \cX \\
        \cZ & \cY\rlap.
        \arrow["g", from=1-1, to=1-2]
        \arrow["q"', from=1-1, to=2-1]
        \arrow["\lrcorner"{anchor=center, pos=0.125}, draw=none, from=1-1, to=2-2]
        \arrow["p", from=1-2, to=2-2]
        \arrow["f"', from=2-1, to=2-2]
    \end{tikzcd}\end{equation}
    If $p$ is smooth, then the commutative square
    \[\begin{tikzcd}
        {\Fun_T(\cV,\und{\Spc}_T)} & {\Fun_T(\cX,\und{\Spc}_T)} \\
        {\Fun_T(\cZ,\und{\Spc}_T)} & {\Fun_T(\cY,\und{\Spc}_T)}
        \arrow["{g^*}"', from=1-2, to=1-1]
        \arrow["{q^*}", from=2-1, to=1-1]
        \arrow["{p^*}"', from=2-2, to=1-2]
        \arrow["{f^*}", from=2-2, to=2-1]
    \end{tikzcd}\]
    is vertically left
    and horizontally right adjointable.

    Dually, if $p$ is proper, the square is vertically right and horizontally left adjointable.
\end{lemma}
\begin{proof}
    Note that since $\und{\Spc}_T$ is a complete and cocomplete $T$-precategory \cite{martiniwolf2021limits}*{Proposition~5.2.9},
    it follows that all functors in the second square admit both adjoints \cite{martiniwolf2021limits}*{Corollary~6.3.7}.
    It then follows by abstract nonsense that the square is horizontally left (right)
    adjointable if and only if it is vertically right (left) adjointable.
    Thus it suffices to focus on vertical adjointability.

    By the (natural) straightening--unstraightening equivalence of \cite[Theorem 4.5.1]{martini2021yoneda} (and passing to global sections) we can identify the square in question with
    \[\begin{tikzcd}
        {\textsf{LFib}(\cV)} & {\textsf{LFib}(\cX)} \\
        {\textsf{LFib}(\cZ)} & {\textsf{LFib}(\cY)}
        \arrow["{g^*}"', from=1-2, to=1-1]
        \arrow["{q^*}", from=2-1, to=1-1]
        \arrow["{p^*}"', from=2-2, to=1-2]
        \arrow["{f^*}", from=2-2, to=2-1]
    \end{tikzcd}\]
    where $\textsf{LFib}(-)$ is the category of left fibrations of $T$-precategories with given base, and the maps are given by pullback.
    Now by \cite[Remark 5.5.5]{martiniwolf2021limits}, if $p$ is smooth (proper), then it is $\mathsf{LFib}$-smooth ($\mathsf{LFib}$)-proper in the sense of Definition 5.5.1 of \emph{op.~cit.}, as discussed in the remark.
    This allows us to apply Proposition 5.5.3 of \emph{op.~cit.}\ to deduce that the above square is vertically left (right) adjointable.
\end{proof}

\begin{lemma}\label{lem:yoneda-kan-adjointable}
    Consider a functor $f \colon \cX \to \cY$ of $T$-precategories,
    and let $y \colon \cC \hookrightarrow \und{\PSh}_T(\cC)$ denote the $T$-parametrized Yoneda
    embedding \cite{martini2021yoneda}*{§4.7} of some $T$-precategory $\cC$.
    We have a commutative square
    \[\begin{tikzcd}
        {\Fun_T(\cX,\cC)} & {\Fun_T(\cX,\und{\PSh}_T(\cC))} \\
        {\Fun_T(\cY,\cC)} & {\Fun_T(\cY,\und{\PSh}_T(\cC))}
        \arrow["{y}\circ-", hook, from=1-1, to=1-2]
        \arrow["{f^*}", from=2-1, to=1-1]
        \arrow["{y}\circ-"', hook, from=2-1, to=2-2]
        \arrow["{f^*}"', from=2-2, to=1-2]
    \end{tikzcd}\]
    which is vertically right adjointable if the pointwise right Kan extension (see Remark~\ref{rk:pointwise-Kan}${}^\op$)
    along $f$ exists in $\cC$. Dually, if we use the fully faithful colimit-preserving
    $y_{\cC^\op}^\op \colon \cC \hookrightarrow \und{\PSh}_T(\cC^\op)^\op$
    instead, then the resulting square is left adjointable if the pointwise left Kan extension along $f$ exists in $\cC$.
\end{lemma}
\begin{proof}
    This is immediate from the pointwise formula for right Kan extensions (Remark~\ref{rk:pointwise-Kan}) and the fact that the Yoneda embedding $y$ preserves all limits by \cite{martiniwolf2021limits}*{Proposition 4.4.8}.
\end{proof}

\begin{proposition}\label{prop:smooth-proper-basechange}
    Consider again a cartesian square of $T$-precategories as in (\ref{diag:basechange-cart-square}) and let $\cC$ be another $T$-precategory. We obtain a commutative diagram
    \begin{equation}\begin{tikzcd}
        {\und{\Fun}_T(\cV,\cC)} & {\und{\Fun}_T(\cX,\cC)} \\
        {\und{\Fun}_T(\cZ,\cC)} & {\und{\Fun}_T(\cY,\cC)}\rlap.
        \arrow["{g^*}"', from=1-2, to=1-1]
        \arrow["{q^*}", from=2-1, to=1-1]
        \arrow["{p^*}"', from=2-2, to=1-2]
        \arrow["{f^*}", from=2-2, to=2-1]
    \end{tikzcd}\end{equation}
    If $p$ is proper (smooth), then the square is vertically right (left) adjointable
    or horizontally left (right) adjointable as soon as the requisite adjoints exist
    and are given by pointwise Kan extensions.
\end{proposition}

\begin{proof}
    We will prove that if $p$ is proper and the pointwise right Kan extensions
    $p_*,q_*$ exist in $\cC$, then the square is vertically right adjointable.
    The other cases are handled analogously.

    We can check the equivalence of the Beck--Chevalley map $f^*p_* \Rightarrow q_*g^*$
    in each degree $A \in T$ separately.
    Using that we have an equivalence $\und{\Fun}_T(\cX,\cC)(A) \simeq \Fun_T(\cX,\und{\Fun}_T(\und{A},\cC))$ which is natural in $\cX$ and $A$,
    we may replace $\und{\Fun}_T(\und{A},\cC)$ with $\cC$ and prove that the right square in the following pasting is vertically right adjointable,
    given that both vertical right adjoints exist and $p$ is proper:
    \[\begin{tikzcd}
        {\Fun_T(\cV \times \cC^\op,\und{\Spc}_T)} & {\Fun_T(\cV,\cC)} & {\Fun_T(\cX,\cC)} \\
        {\Fun_T(\cZ \times \cC^\op,\und{\Spc}_T)} & {\Fun_T(\cZ,\cC)} & {\Fun_T(\cY,\cC)}
        \arrow[hook', from=1-2, to=1-1]
        \arrow["{g^*}"', from=1-3, to=1-2]
        \arrow["{(q \times \cC^\op)^*}", from=2-1, to=1-1]
        \arrow["{q^*}", from=2-2, to=1-2]
        \arrow[hook', from=2-2, to=2-1]
        \arrow["{p^*}"', from=2-3, to=1-3]
        \arrow["{f^*}", from=2-3, to=2-2]
    \end{tikzcd}\]
    Here the fully faithful left horizontal inclusions are induced by the $T$-parametrized
    Yoneda embedding $y \colon \cC \hookrightarrow \und{\PSh}_T(\cC)$ and currying.

    We can check the vertical right-adjointability of the right square
    after postcomposing with the bottom left fully faithful functor.
    Moreover, note that the left square is vertically right adjointable
    by \cref{lem:yoneda-kan-adjointable}.
    Since Beck--Chevalley maps compose, we are thus reduced to proving that the entire
    rectangle is vertically right adjointable.
    However, note that since pre- and postcomposition commute, this rectangle can be rewritten as
    \[\begin{tikzcd}
        {\Fun_T(\cV \times \cC^\op,\und{\Spc}_T)} &[1.33em] {\Fun_T(\cX \times \cC^\op,\und{\Spc}_T)} & {\Fun_T(\cX,\cC)} \\
        {\Fun_T(\cZ \times \cC^\op,\und{\Spc}_T)} & {\Fun_T(\cY \times \cC^\op, \und{\Spc}_T)} & {\Fun_T(\cY,\cC)}
        \arrow["{(g \times \cC^\op)^*}"', from=1-2, to=1-1]
        \arrow[hook', from=1-3, to=1-2]
        \arrow["{(q \times \cC^\op)^*}", from=2-1, to=1-1]
        \arrow["{(p \times \cC^\op)^*}"', from=2-2, to=1-2]
        \arrow["{(f \times \cC^\op)^*}", from=2-2, to=2-1]
        \arrow["{p^*}"', from=2-3, to=1-3]
        \arrow[hook', from=2-3, to=2-2]
    \end{tikzcd}\]
    Here, the right square is again vertically right adjointable by \cref{lem:yoneda-kan-adjointable}.
    Since $p \times \cC^\op$ is a pullback of $p$ it is again proper,
    and using once more that Beck--Chevalley maps compose, we are thus reduced to the case $\cC = \und{\Spc}_T$, which was handled in \cref{lem:basechange-omega}.
\end{proof}

\subsection{Demure basechange}
Recall from \cite[Definition 3.1]{shah2022parametrizedII}, that a \emph{factorization system} on a $T$-precategory $\cX$ consists of two wide subcategories $E,M\subset\cX$ such that $(E(A),M(A))$ is a factorization system on $\cX(A)$ for every $A\in T$. In this section we will prove the following application of proper basechange:

\begin{lemma}\label{lemma:rough-basechange}
    Let $\cX$ and $\cX'$ be two $T$-precategories both equipped with parametrized factorization systems $(E,M)$ and $(E',M')$, respectively, and let $f\colon\cX\to \cX'$ be a $T$-functor with the following properties:
    \begin{enumerate}
        \item $f$ maps $E$ to $E'$.
        \item $f$ maps $M$ to $M'$, and the resulting map $M\to M'$ is a right fibration.
    \end{enumerate}
    Let $\cC$ be another $T$-precategory and consider the commutative diagram
    \[\begin{tikzcd}[cramped]
        {\und{\Fun}_T(\cX',\cC)} & {\und{\Fun}_T(\cX,\cC)} \\
        {\und{\Fun}_T(E',\cC)} & {\und{\Fun}_T(E,\cC)}
        \arrow["{f^*}", from=1-1, to=1-2]
        \arrow["{i'^*}"', from=1-1, to=2-1]
        \arrow["{i^*}", from=1-2, to=2-2]
        \arrow["{f|_E^*}", from=2-1, to=2-2]
    \end{tikzcd}\]
    where $i \colon E \hookrightarrow \cX$ and $i' \colon E' \hookrightarrow \cX'$ denote the inclusions.
    Then the square is vertically left adjointable or horizontally right adjointable
    as soon as the requisite adjoints exist and are given by pointwise Kan extensions.
\end{lemma}

For the proof we will need another lemma:

\begin{lemma}
	Let $\cX$ be a $T$-precategory with a parametrized factorization system $(E,M)$, and let $\cC$ be any $T$-precategory. Assume that the pointwise left Kan extension along $i\colon \core\cX=\core M\hookrightarrow M$ exists in $\cC$. Then also the left Kan extension along $j\colon E\hookrightarrow\cX$ exists and is pointwise. Moreover, the Beck--Chevalley transformation
	\[
	 \begin{tikzcd}            	         \ul\Fun_T(E,\cC)\arrow[r, "j_!"]\arrow[d,"\res"'] & \ul\Fun_T(\cX,\cC)\arrow[d,"\res"]\\            \ul\Fun_T(\core\cX,\cC)\arrow[r, "i_!"']\arrow[ur, Rightarrow,shorten=7pt] & \ul\Fun_T(M,\cC)
      \end{tikzcd}
    \]
    is an equivalence.
    \begin{proof}
        We will show that the inclusion
        \[
            \core\big((\pi_A^*M)_{/X}\big)=(\pi_A^*M)_{/X}\times_{\pi_A^*M}\core(\pi_A^*\cX)\hookrightarrow
            (\pi_A^*\cX)_{/X}\times_{\pi_A^*\cX}\pi_A^*E
        \]
        is final in the sense of \cite{martiniwolf2021limits}*{Proposition~4.6.1}. The lemma will then follow from the pointwise formula.

        To prove the claim, we replace $\cX$ by $\pi_A^*\cX$; we then have to show that the inclusion $\core(M_{/X})\hookrightarrow\cX_{/X}\times_\cX E$ is final, for which we will show that it admits a parametrized left adjoint. By \cite{CHLL_Bispans}*{proof of Proposition~4.1.4} $\incl$ has a pointwise left adjoint $\lambda$. As $\incl$ is fully faithful, $\lambda$ is a localization; the Beck--Chevalley condition with respect to restrictions of $\cX$ then amounts to saying that if $f\colon B\to B'$ is any map in $T$, then $f^*$ sends maps inverted by $\lambda(B')$ to maps inverted by $\lambda(B)$. But \emph{every} map is inverted by $\lambda(B)$ as its target is a groupoid, so the claim follows.
    \end{proof}
\end{lemma}

\begin{proof}[Proof of Lemma~\ref{lemma:rough-basechange}]
    Assume that the pointwise left Kan extensions along $i$ and $i'$ exist.
    In view of \cref{lem:yoneda-kan-adjointable}, we may assume that $\cC$ is cocomplete (and complete).
    We can now simply repeat the proof of the non-parametrized version of the lemma given in \cite{CHLL_Bispans}*{Lemma~4.1.6}, replacing references to Proposition~4.1.4 of \emph{op.\ cit.}\ by references to the previous lemma, and replacing the appeal to (ordinary, non-parametrized) smooth/proper basechange by a reference to \cref{prop:smooth-proper-basechange}.

    Now assume instead that the pointwise right Kan extensions along $f$ and $f|_E$ exist. We may dually reduce to the case that $\cC$ is complete and cocomplete. But then the two Beck--Chevalley maps are actually total mates of each other, so one is invertible if and only if the other one is so. Thus, the claim follows from the above.
\end{proof}

\section{Ultra-commutativity and ordinary higher algebra}\label{app:ordinary-HA}
The goal of this appendix is to prove a precise version of the slogan that ultra-commutativity cannot be encoded in ordinary higher algebra. We begin with the equivariant case:

\begin{proposition}
    Let $G\not=1$ be any finite group, let $\mathcal O$ be an operad (in the non-parametrized sense), and let $X\in\mathcal O_1$ be any object. Then there does \emph{not} exist an equivalence $\UCom_G\simeq\Alg_{\mathcal O}(\Sp_G)$ fitting into a commutative diagram
    \[
        \begin{tikzcd}[column sep=small]
            \UCom_G\arrow[rr,"\sim"]\arrow[dr, bend right=15pt,"\fgt"'] && \Alg_{\mathcal O}(\Sp_G)\rlap.\arrow[dl, bend left=15pt,"\ev_X"]\\
            & \Sp_G
        \end{tikzcd}
    \]
    \begin{proof}
        As the geometric fixed point functor $\Phi^G\colon\Sp_G\to\Sp$ is symmetric monoidal and cocontinuous, it admits a lax symmetric monoidal right adjoint $R$, yielding an induced adjunction $\Alg_{\mathcal O}(\Phi^G)\dashv\Alg_{\mathcal O}(R)$. We note that both adjoints commute with the forgetful functors $\ev_X$; thus, passing to total mates shows that the left adjoint $\Alg_{\mathcal O}(\Phi_G)$ commutes with both $\ev_X$ and its left adjoint $\mathbb P^{\mathcal O}_X$, so that 
        \begin{equation}\label{eq:Phi-G-vs-monad}
            \Phi^G\circ{\ev_X}\circ\mathbb P^{\mathcal O}_X\simeq{\ev_X}\circ\mathbb P^{\mathcal O}_X\circ\Phi^G.
        \end{equation} 
        In the same way one shows that the inclusion $\Sp_{\ge 0}\hookrightarrow\Sp$ commutes with ${\ev_X}\circ\mathbb P^{\mathcal O}_X$; in particular, ${\ev_X}\circ\mathbb P^{\mathcal O}\colon\Sp\to\Sp$ preserves the full subcategory of connective spectra. Combining this with the equivalence $(\ref{eq:Phi-G-vs-monad})$, we conclude that ${\ev_X}\circ\mathbb P^{\mathcal O}\colon\Sp_G\to\Sp_G$ preserves the full subcategory $\Xx$ spanned by all $G$-spectra such that $\Phi^GX$ is connective. To prove the proposition, it will then suffice to show that the analogous composite $\mathbb U\mathbb P\colon\Sp_G\to\UCom_G\to\Sp_G$ does \emph{not} preserve $\Xx$.

        For this, we let $X\coloneqq\Omega\Sigma^\infty_+G$. Then $\Phi^G(X)\simeq \Omega\Phi^G\Sigma^\infty_+G\simeq\Omega\Sigma^\infty_+(G^G)=0$, so $X\in\Xx$.    To show that $\mathbb U\mathbb P(X)\notin\Xx$, we begin by observing that $X$ is modelled on the pointset level by the $G$-symmetric spectrum $\bm\Sigma(1,-)\smashp G_+$, so we compute as in Warning~\ref{warn:U-ext-no-strict} that $\mathbb U\mathbb P(X)$ is modelled by $\bigvee_{n\ge 0}\bm\Sigma(n,-)\smashp_{\Sigma_n}(G^n)_+$. We now let $n=|G|$, and we pick an arbitrary enumeration $g_1,\dots,g_n$ of the elements of $G$. This then defines an injective homomorphism $\sigma\colon G\to\Sigma_n$ characterized by $g.g_i=g_{\sigma(g)(i)}$ for all $g\in G$ and $1\le i\le n$ (in other words, $\sigma$ classifies the left regular action on $G$, transported to $\{1,\dots,n\}$ via the chosen enumeration). The isotropy of $(g_1,\dots,g_n)$ in the $(G\times\Sigma_n)$-set $G^n$ (with $G$ acting via left multiplication and $\Sigma_n$ via permuting the compontents) consists precisely of the elements $(g,\sigma(g))$, so the inclusion of the orbit of $(g_1,\dots,g_n)$ exhibits $D_\sigma\coloneqq\sigma^*\bm\Sigma(n,-)$ as a summand of $\bm\Sigma(n,-)\smashp_{\Sigma_n}(G^n)_+$ and hence of $\mathbb U\mathbb P(X)$; thus, it will suffice to show that $\Phi^G(D_\sigma)$ is not connective.
        
        To this end we observe that the $\Sigma_n$-spectrum $\bm\Sigma(n,-)$ is isomorphic on the pointset level to the norm of the $\Sigma_{n-1}$-spectrum $\bm\Sigma(1,-)$,where $\Sigma_{n-1}$ acts trivially. It follows that $\bm\Sigma(n,-)$ is an inverse with respect to the smash product of the $\Sigma_n$-spectrum $\Sigma^\infty S^n$ where $\Sigma_n$ acts on $S^n$ by permutation. Thus, $D_\sigma$ is inverse to the $G$-spectrum $\sigma^*\Sigma^\infty S^n$, and the ordinary spectrum $\Phi^G(D_\sigma)$ is inverse to $\Phi^G(\sigma^*\Sigma^\infty S^n)\simeq\Sigma^\infty(\sigma^*S^n)^G$. The fixed-point space $(\sigma^*S^n)^G$ is a sphere again, and its dimension $d$ equals the dimension of the $\sigma(G)$-fixed points of the permutation representation $\R^n$. We therefore have $\Phi^G(D_\sigma)\simeq\Omega^d\mathbb S$ and hence in particular $\pi_{-d}\Phi^G(D_\sigma)\cong\Z$. However, $d\ge1$ as $(1,\dots,1)$ is a non-zero $\Sigma_n$-fixed point of $\R^n$; thus, $\Phi^G(D_\sigma)$ is non-connective, and hence so is $\Phi^G(\mathbb U\mathbb P(X))$.
    \end{proof}
\end{proposition}

Similarly we have the following global version:

\begin{proposition}
    Let $G$ be any finite group (the case $G=1$ is allowed), let $\mathcal O$ be an operad, and let $X\in\mathcal O_1$ be arbitrary. Then there does \emph{not} exist an equivalence $\UCom_\textup{$G$-gl}\simeq\Alg_{\mathcal O}(\Sp_\textup{$G$-gl})$ fitting into a commutative diagram
    \[
        \begin{tikzcd}[column sep=small]
            \UCom_\textup{$G$-gl}\arrow[rr,"\sim"]\arrow[dr, bend right=15pt,"\fgt"'] && \Alg_{\mathcal O}(\Sp_\textup{$G$-gl})\rlap.\arrow[dl, bend left=15pt,"\ev_X"]\\
            & \Sp_\textup{$G$-gl}
        \end{tikzcd}
    \]
    \begin{proof}
        Write $\mathbb P^{\mathcal O}_X$ for the left adjoint to $\ev_X$. If $G\not=1$ then one argues as in the previous proposition to show that ${\ev_X}\circ\mathbb P^{\mathcal O}_X$ preserves the full subcategory $\Xx$ of those $G$-global spectra whose underlying $G$-spectrum has connective geometric fixed points, and $X\coloneqq\Omega\Sigma^\infty_+G$ is again an example of an element of $\Xx$ whose image under $\mathbb U\mathbb P\colon\Sp_\textup{$G$-gl}\to\UCom_\textup{$G$-gl}\to\Sp_\textup{$G$-gl}$ is no longer contained in $\Xx$.

        It therefore only remains to treat the case $G=1$. For this we argue as before to see that ${\ev_X}\circ\mathbb P^{\mathcal O}_X$ preserves the essential image $\Yy$ of the unique symmetric monoidal left adjoint $\Sp\to\Sp_\gl$. However, we saw in Warning~\ref{warn:U-ext-no-strict} that the symmetric spectrum $\bm\Sigma(1,-)\smashp S^1$ models an element of $\Yy$ (in fact, this is just the sphere spectrum) such that its image under $\mathbb U\mathbb P$ does not belong to $\Yy$.
    \end{proof}
\end{proposition}

\bibliography{reference}
\end{document}